\documentclass{jams-l}

\usepackage{pifont}
\usepackage{amssymb}
\usepackage{mathabx}
\usepackage[makeroom]{cancel}
\usepackage{graphicx}
\usepackage[cmtip,all]{xy}
\usepackage[usenames,dvipsnames]{color}
\usepackage{upgreek}
\usepackage{multirow}
\usepackage{hyperref}
\definecolor{linkcolour}{rgb}{0,0.2,0.6}
\hypersetup{colorlinks,breaklinks,urlcolor=linkcolour, linkcolor=linkcolour}
\usepackage{multicol}
\usepackage{colortbl}
\definecolor{bobby}{gray}{0}
\definecolor{LightCyan}{rgb}{0.88,1,1}
\usepackage{arydshln}
\usepackage{soul}
\usepackage{bbm}
\newtheorem{theorem}{Theorem}[section]
\newtheorem{lemma}[theorem]{Lemma}
\newtheorem{proposition}[theorem]{Proposition}
\newtheorem{corollary}[theorem]{Corollary}

\theoremstyle{definition}
\newtheorem{definition}[theorem]{Definition}
\newtheorem{convention}[theorem]{Convention}

\newtheorem{example}[theorem]{Example}

\theoremstyle{remark}
\newtheorem{remark}[theorem]{Remark}

\numberwithin{equation}{section}

\DeclareMathSymbol{\cmemptyset}{\mathord}{symbols}{59}
\newcommand{\prematrix}[3]{\left(\!\!\def\arraystretch{1.2} \begin{array}{#1} #2 \\\hline #3 \end{array}\!\!\right)}

\usepackage{geometry}
\begin{document}

\large

\title{Category theory for genetics II:\\genotype, phenotype and haplotype}

\author{R\'{e}my Tuy\'{e}ras}
\curraddr{}
\email{rtuyeras@gmail.com}
\thanks{This research was supported by the AFOSR grants FA9550-14-1-0031 and FA9550-17-1-0058 (email David Spivak at dspivak@mit.edu) and by the NIH grants R01-HG008155 and R01-AG058002 (email Manolis Kellis at manoli@mit.edu).}

\date{}

\dedicatory{}

\begin{abstract}
The overarching goal of this paper is to solve the word problem for a class of idempotent commutative monoids whose elements model population haplotypes. More specifically, we design an algebraic framework in which it is possible to unravel population stratification relationships and infer linkage disequilibrium in terms of algebraic equations of haplotypes expressed in idempotent commutative monoids. We show how these relations can be used to clarify haplotype-phenotype associations through the consideration of intermediate phenotypes and genetic mechanisms such as segregation and homologous recombination. The present work paves the way for the implementation of combinatorial GWAS in the study of complex traits, and for a framework in which one can infer genetic variants interactions along with the corresponding regulatory circuitry. Throughout the paper, we formalize concepts such as genotypes, haplotypes and haplogroups, and model segregation and homologous recombination in the language of pedigrads (introduced in previous work). The benefit of using pedigrads is that they provide a computational framework in which one can reason about haplotype dynamics over several generations.
\end{abstract}

\subjclass[2010]{18C99, 92B99, 16Y60, 20M99, 92D10}

\keywords{
Limit sketches, Idempotent commutative monoids, Semiring, Word problem, Genetics, Linear algebra.
}

\maketitle

\section{Introduction}

\subsection{Short presentation}
The applied scope of the present article is to propose a set of novel techniques taking place in idempotent commutative monoids to better encode and clarify linkage disequilibrium relations across genomic datasets (this is further explained below, in section \ref{ssec:motivations}).
At a more theoretical level, the paper takes on the task of solving the word problem for a specific class of idempotent commutative monoids. These idempotent commutative monoids are defined in a category of models for a certain class of limit sketches. We use the universal property of the models to determine conditions in which the monoids are (locally) isomorphic to free finitely-generated idempotent commutative monoids and use this information to solve the word problem through a specific algorithm.

As will be seen, the use of idempotent commutative monoid structures shall be motivated by the modeling of homologous recombination and the uncertainty that lies in retracing the emergence of new haplotypes. Overall, this paper tackles the question of inferring the intermediate-phenotypic circuitry linking genotypes to phenotypes by encoding the space of population-stratification dependencies through formal algebraic equations of haplotypes. For this purpose, we develop a framework in which one can find minimal algebraic relations between haplogroups to explain how the corresponding haplotypes can lead to specific combinations of phenotypes. In particular, the present work paves the way for a framework in which one can represent and analyze combinatorial genetic interactions leading to complex phenotypes.

\subsection{Motivations} \label{ssec:motivations}
The holy grail of epidemiology is to (1) accurately predict the occurrence of diseases within individuals and (2) to detect the factors that are responsible for the occurrence of these diseases \cite[Sec. \emph{Biological mechanisms versus risk prediction}]{Lander_intro}. By being able to do so, researchers hope to create new opportunities to find tractable ways to act on those factors, specifically to prevent, stop, or limit the development or continuation of the diseases. Because complex diseases usually arise as a combination of a multitude of genetic effects, research in genomics and epidemiology has concentrated its efforts on understanding the likelihood with which sets of genetic variants are inherited together. It turns out that this phenomenon is not rare and underlies many challenges faced in research today. For this very reason, the fact that two genetic variants may be more likely to be transmitted together through generations underlies an important problem, which is referred to as \emph{linkage disequilibrium} \cite{LD1, LD2, LD3, LD4, LD5}.

When ignoring natural selection, linkage disequilibrium can be seen as a consequence of Mendel's first and second inheritance laws \cite{Mendel_nature}, namely segregation and homologous recombination \cite{recomb1,recomb2,McPeek,Haldane,Zhao_mapF,Speed_GMF}. Meanwhile, Mendel's third inheritance law would state that genotypes are linked phenotypes through non-linear relationships. While the first two laws will constitute an inherent part of our axiomatization, we will show how the third law can be regarded as a logical deduction of our theory through the use of pedigrads \cite{Seqali}. For the sake of modeling genetic mechanisms, we argue that these pedigrads are to be enriched in the category of idempotent commutative monoids. To motivate the use of such objects in the paper, this introduction aims to give insights regarding the extent to which the first and second Mendelian laws complicates the discovery of genotype-phenotype associations. In particular, we discuss the concept of linkage disequilibrium from the lens of current techniques in computational biology and intuitions developed by Mendel, the father of genetics (see \cite{Mendel_history}), to eventually justify the necessity of using pedigrads.

First, recall that \emph{haplotypes} are collections of DNA variants that are inherited together. As a result, haplotypes constitute a sound concept to characterize the collection of variants that a given population can possess \cite{Hap_map,Hap_LD}. In the spirit of Mendel's experiments for finding the laws of inheritance, let us initiate our discussion on haplotypes by using an example of a progeny. Note that the point of considering a progeny is to ensure that the amount of variations that could confound our deductions is minimal. In this respect, consider a situation in which two individuals $M$ and $F$ produce an offspring $C_1$ -- we shall suppose that the offsprings do not undergo mutations. In theory, the genomic information possessed by $C_1$ is a mixture of the genomic information possessed by $M$ and $F$. To explain this from the point of view of a specific position $p$ in the genome, we could assume that $M$ has a genotype\footnote{A genotype is made of a pair of \emph{alleles} encoded by nucleotide symbols (\emph{e.i.} $\mathtt{A}$, $\mathtt{C}$, $\mathtt{G}$, $\mathtt{T}$). Note that every person possesses two chromosomes of a same type (these chromosomes are said to be \emph{homologous}) and each letter refers to the nucleotide (or allele) of one of the two homologous chromosomes.} $\mathtt{AT}$ at position $p$ and $F$ has a genotype $\mathtt{AT}$ at the same position. In that case, the child $C_1$ would have one of the following genotypes at position $p$: $\mathtt{AA}$, $\mathtt{AT}$ and $\mathtt{TT}$. To set a directionality in the way these genotypes are generated through generations, we can write this process with an arrow as follows.
\begin{equation}\label{eq:intro:prepare:icm-equations}
\mathtt{AT} + \mathtt{AT} \to \mathtt{AA}+\mathtt{AT}+\mathtt{TT}
\end{equation}
More generally, if the two individuals $F$ and $M$ begets a progeny $C_1$, $C_2$, \dots, $C_n$ of $n$ individuals, then the genotypes of the progeny run over the set $\{\mathtt{AA},\mathtt{AT},\mathtt{TT}\}$. While the family generated by $F$ and $M$ has a haplotype equal to $\{\mathtt{A},\mathtt{T}\}$, there are subgroups of children of $F$ and $M$ whose haplotypes are either $\{\mathtt{A}\}$ or $\{\mathtt{T}\}$, namely those children with the genotype $\mathtt{AA}$ or $\mathtt{TT}$. If we forget the fact that we are looking at a progeny, and instead think of this progeny as a population, the two subpopulations whose haplotypes are either $\{\mathtt{A}\}$ or $\{\mathtt{T}\}$ would not be able to produce individuals with genotypes $\mathtt{AT}$ by themselves. This bottleneck effect suggests a principle discussed in the paragraph below regarding genotypes-phenotypes associations.

For illustration, suppose that $F$ and $M$ are both healthy individuals and that all the progeny of $F$ and $M$ is healthy up to $C_{n-1}$, while $C_n$ is the first individual showing signs of a specific disease. Assuming that this disease is not age dependent and has a single genetic cause (meaning that there is a function from the genotypes to the disease), such combinations of phenotypes suggests that $F$ or $M$ must possess a heterozygous\footnote{Heterozygous genotypes consist of pairs of two different nucleotides while homozygous genotypes consist of pairs of the same nucleotide.} genotype (\emph{e.g.} $\mathtt{AT}$) at the causal position and that the disease can be caused by one of the two possible homozygous genotype (\emph{e.g.} $\mathtt{AA}$) such that all the healthy children $C_1,C_2,\dots,C_{n-1}$ possess the other homozygous genotype (\emph{e.g.} $\mathtt{TT}$) or the heterozygous genotype (\emph{e.g.} $\mathtt{AT}$). This illustrates how one can learn information about the complexity of certain genotype-phenotype associations by considering dynamics resulting from Mendel's first and second laws.

However, current research directions have taken a slightly less combinatorial approach than the one taken by Mendel and instead focus on what is known as additive genetic effects. From now on, our goal will be to review techniques in the latter case in order to emphasize the contribution of the combinatorial framework developed in this paper.

In the additive setting, genetic effects are expected to be proportional to numerical values typically encoding genetic variants, specifically the values 0 or 2 would be given to homozygous genotypes and the value 1 would be given to heterozygous genotypes. In this setting, the transformation shown in (\ref{eq:intro:prepare:icm-equations}) do not have much sense anymore, and cannot even be translated in terms of an equation involving the genotype values. For example, the numerical translation of transformation (\ref{eq:intro:prepare:icm-equations}) is the equation $1+1 = 0+1+2$, which does not hold. In fact, a lot of equations in the additive setting do not translate back to the combinatorial setting discussed earlier. Yet, the additive framework has one major advantage, which is that it allows one to model genetic effects as weighted sums of the genotype values. These weighted sums are known as \emph{polygenic risks scores} (PRS) \cite{PRS_survey, PRS_tutorial} and are commonly used for predicting the occurrence of a disease through a scoring system. However, this advantage also presents a trade-off in accuracy because the calculation of the weights entails approximations that often results in prediction errors.

Nevertheless, one lesson that the additive hypothesis teaches us is that genetic effects are the results of a multitude of genetic variations. Specifically, the PRS of an individual $j$ for a given disease is computed as a linear equation of the genotype values of this individual for a set $I$ of genetic variants. This means that if we denote the genotype values of individual $j$ as $g_i(j)$ for each variant $i \in I$, then the PRS equation multiplies each value $g_i(j)$ by a coefficient $\beta_i$ as shown in equation (\ref{eq:prs-equation:intro}) below.
\begin{equation}\label{eq:prs-equation:intro}
\mathsf{PRS}(j) = \sum_{i \in I} g_i(j) \beta_i
\end{equation}
While the variables $g_i(j)$ take their values in the set $\{0,1,2\}$, the coefficients $\beta_i$ are real values measuring the correlation existing between the genetic variant $i \in I$ and the disease. For the readers learned in statistical modeling, this should suggest that the coefficients $\beta_i$ can be approximated from linear regressions computed with respect to each genetic variant $i \in I$. In practice, when these linear regressions are computed with respect to a large set of variants across the genome, the list of resulting weights $\beta_i$ is referred to as a genome-wide association study (GWAS) \cite{GWAS_1,GWAS_2}. In such studies, each coefficient $\beta_i$ is the effect size of a linear regression between the local genotype matrix $Z_i = (g_i(j))_{j}$ -- or ideally the global genotype matrix $Z = (g_i(j))_{j,i}$ ---- and a vector $y = (y_j)_j$ of disease values $y_j$ for each individual $j$, as shown below.
\[
\def\arraystretch{1.5}
\left\{
\begin{array}{ll}
y = Z_i\beta_i+\varepsilon_i&(\textrm{pointwise linear regression for each $i \in I$})\\
y = Z\beta+\varepsilon&(\textrm{global linear regression over $I$})
\end{array}
\right.
\]
After this, the obtained coefficients $\beta = (\beta_i)_i$ would be used in the computation of the PRS equation. In practice, the pointwise linear regression $y = Z_i\beta_i+\varepsilon$ is faster to compute, and for this reason, would be preferred over the global regression $y = Z\beta+\varepsilon$. However, the global regression would be more accurate as it takes into account the genetic similarities existing between each individual. This additional information is useful to eliminate certain biases that cannot be taken into account by the local regressions $y = Z_i\beta_i+\varepsilon$, and thus provides more accurate estimations of the coefficients $\beta_i$.

The previous paragraph underlies an important challenge regarding the construction of PRS equations, namely the weights provided by GWAS would predict the occurrence of a given genetic effect solely through the lens of one variant at a time \cite{GWAS_prob2,GWAS_prob1}. On the other hand, multivariate linear regressions would measure that genetic effect through the lens of a multitude of variants.  The advantage of considering a greater number of variants is that one has more chances to explain the occurrence of comorbidities that arise along with a given disease -- this phenomenon is known as \emph{pleiotropy} \cite{Pleiotropy1,Pleiotropy2,Pleiotropy3}. Ideally, one wants pleiotropic models to explain comorbidities for a better classification of the associated phenotypes (\emph{e.g.} in the case of Alzheimer's disease, being able to do so would allow one to better understand and classify the disease's subtypes and hence design personalized treatment strategies \cite{Intermediate_Alz}). However, the GWAS procedures used to calculate PRS weights makes RPS equations rather disease-specific. As a result, the use of PRS in the prediction of multiple diseases can introduce further prediction errors. Still, the consideration of models involving multiple phenotypes is expected to increase the overall accuracy of these predictions. Indeed, while the correspondence from genotypes to phenotypes is not one-to-one, it is more likely to be so if we take intermediate phenotypes associated with a given phenotype \cite{Intermediate1,Intermediate2,Intermediate3}.
\[
\xymatrix@C+70pt{
\fbox{\textrm{Genetics}} \ar@<-.7ex>@{->}[r] \ar@<+.7ex>@{<-}[r]^-{\small\begin{array}{c}\textrm{pleiotropic}\\\textrm{relationship}\end{array}} & \fbox{\textrm{Intermediates}}\ar[r]^-{\small\begin{array}{c}\textrm{A phenotype is proxy}\\\textrm{for other phenotypes}\end{array}} & \fbox{\textrm{Phenotype}}
}
\]

Even though non-linear methods are being developed to take into account multiple phenotypes, these usually combine several layers of techniques that make their framework less accessible to human reasoning \cite{non-linear1,non-linear2,non-linear3}. Meanwhile, let us notice that the type of transformations discussed in (\ref{eq:intro:prepare:icm-equations}) can be used to study combinations of phenotypes in a non-linear fashion. To illustrate this, let $\mathsf{A}$, $\mathsf{B}$, $\mathsf{C}$, and $\mathsf{D}$ be four related phenotypes and let us consider four individuals $I_1$, $I_2$, $I_3$ and $I_4$, each possessing the combinations of phenotypes $\{\mathsf{A},\mathsf{B},\mathsf{C}\}$, $\{\mathsf{A},\mathsf{D}\}$, $\{\mathsf{D},\mathsf{C}\}$, and $\{\mathsf{A},\mathsf{B}\}$, respectively. We can use these combinations to infer transformations as shown in (\ref{eq:intro:prepare:icm-equations}), but between the four individuals $I_1$, $I_2$, $I_3$ and $I_4$. Specifically, every equation induced from forming unions on sets of phenotypes, as shown below on the left of (\ref{eq:intro:icm-equations:phenotypes}), suggests a bidirected transformation on the corresponding individuals, as shown on the right.
\begin{equation}\label{eq:intro:icm-equations:phenotypes}
\{\mathsf{A},\mathsf{B},\mathsf{C}\} \cup \{\mathsf{A},\mathsf{D}\} = \{\mathsf{D},\mathsf{C}\} \cup \{\mathsf{A},\mathsf{B}\} \quad\quad\quad \Rightarrow \quad\quad\quad I_1 + I_2 \leftrightarrows I_3 + I_4
\end{equation}
Then, one would ideally expect to be able to explain the occurrence of the phenotypes $\mathsf{A}$, $\mathsf{B}$, $\mathsf{C}$, and $\mathsf{D}$ genetically -- at least to some extent. This means that the transformations inferred on the individuals $I_1$, $I_2$, $I_3$ and $I_4$ should have a translation at the level of the genome of these individuals -- in the same fashion as in transformation (\ref{eq:intro:prepare:icm-equations}). To give an example, let us assume that the phenotypes $\mathsf{A}$, $\mathsf{B}$, $\mathsf{C}$, and $\mathsf{D}$ are caused by mutations on six different positions in the genomes of the individuals $I_1$, $I_2$, $I_3$ and $I_4$. More specifically, let us consider a situation as follows, where the six positions are numbered from 1 to 6 for convenience:
\begin{itemize}
\item[1)] phenotype $\mathsf{A}$ occurs in every individual possessing \emph{at least one} nucleotide $\mathtt{A}$ at position $\mathtt{1}$ and \emph{at least one} nucleotide $\mathtt{T}$ at position $\mathtt{5}$;
\item[2)] phenotype $\mathsf{B}$ occurs in every individual possessing \emph{two} nucleotides $\mathtt{C}$ at position $\mathtt{2}$ and \emph{at least one} nucleotide $\mathtt{C}$ at position $\mathtt{3}$;
\item[3)] phenotype $\mathsf{C}$ occurs in every individual possessing \emph{at least one} nucleotide $\mathtt{A}$ at position $\mathtt{4}$ and \emph{two} nucleotides $\mathtt{T}$ at position $\mathtt{6}$;
\item[4)] phenotype $\mathsf{D}$ occurs in every individual possessing \emph{at least one} nucleotide $\mathtt{A}$ at position $\mathtt{2}$ and \emph{two} nucleotides $\mathtt{C}$ at position $\mathtt{3}$ on the same chromosome.
\end{itemize}
Per the previous rules, the left-hand side equation of (\ref{eq:intro:icm-equations:phenotypes}) can be explained by a transformation of the form shown in (\ref{eq:intro:icm-equations:haplotypes(recomb)}).

\begin{equation}\label{eq:intro:icm-equations:haplotypes(recomb)}
\begin{array}{rcccccccc}
\textrm{positions}&I_1&+&I_2&\leftrightarrows&I_3&+&I_4\\
\rotatebox[origin=c]{-90}{$
\begin{array}{l}
\uparrow\uparrow\uparrow\uparrow\uparrow\uparrow\\
\mathtt{123456}
\end{array}$}
&\rotatebox[origin=c]{-90}{$
\begin{array}{l}
\mathtt{ACCATT}\\
\mathtt{ACCAGT}\\
\end{array}$}
&+&
\rotatebox[origin=c]{-90}{$
\begin{array}{l}
\mathtt{TCCGGT}\\
\mathtt{AACATC}\\
\end{array}$}
&\leftrightarrows&
\rotatebox[origin=c]{-90}{$
\begin{array}{l}
\mathtt{ACCGGT}\\
\mathtt{AACAGT}\\
\end{array}$}
&+&
\rotatebox[origin=c]{-90}{$
\begin{array}{l}
\mathtt{TCCATC}\\
\mathtt{ACCATT}\\
\end{array}$}\\
\end{array}
\end{equation}
For the same reasons as those that would push us to study causal inference models through elementary variations, which are also the same reasons as those that would make Mendel consider progenies in its experiments, one wants to make sure that the amount of variations contained in transformation (\ref{eq:intro:icm-equations:haplotypes(recomb)}) is modeled according to the most elementary reproduction dynamics. Importantly, by doing so, one can limit the amount of confounding and non-heritable factors capable to correlate with the occurrence of the phenotypes. For example, observe that transformation (\ref{eq:intro:icm-equations:haplotypes(recomb)}) was generated by shuffling nucleotides at each given position so that the same nucleotides appear on the two sides of the transformation. The astute reader might even notice that the shuffling operations involved in transformation (\ref{eq:intro:icm-equations:haplotypes(recomb)}) were not only applied at each position, but more generally according to two regions, namely the DNA block indexed by the positions $\mathtt{1}$, $\mathtt{2}$ and $\mathtt{3}$ and that indexed by the positions $\mathtt{4}$, $\mathtt{5}$ and $\mathtt{6}$. In other words, transformation (\ref{eq:intro:icm-equations:haplotypes(recomb)}) is a shuffling of the DNA blocks $\mathtt{ACC}$, $\mathtt{AAC}$, $\mathtt{TCC}$ on the upper part and a shuffling of the DNA blocks $\mathtt{AGT}$, $\mathtt{ATT}$, $\mathtt{GGT}$, $\mathtt{ATC}$. These shuffling operations on regions of genomic materials are meant to model the action of homologous recombination  and segregation on one's genomic materials through sexual reproduction.

It is important to note that, through generations, homologous recombination will not favor any specific direction in which a transformation of the form (\ref{eq:intro:icm-equations:haplotypes(recomb)}) happens. This should therefore push us to express these transformations as proper equations (or equivalence relations). Furthermore, recombination events may not happen for the same block decomposition. This means that individuals of a given population with common ancestors will likely be related through multiple equations of haplotypes that uses different genomic block decompositions. If we were to model these different recombination events within a single algebraic framework, the types of equations that we would obtain would likely hold up to mixtures of decompositions, but none that actually occur in individuals. Hence, if we were to model and infer the amount of linkage disequilibrium that exists in a population, it would certainly be beneficial to keep track of the different block decompositions associated with each homologous recombination event. In this paper, we shall account for these questions through the use of pedigrads, whose sheaf-like structures will allow us to (1) consider multiple haplotype dynamics indexed by different genomic block decompositions and (2) to relate each of these dynamics through adequate mappings.

To conclude this introduction, we have shown how an algebraic framework inspired from Mendel's intuitions could give us additional insights regarding genotype-phenotype associations deduced from usual linear models. Overall, we want to take advantage of these insights to provide a intuitive paradigm in which it is possible to reason about complex genetic diseases. Specifically, the inference of causal pathways -- such as the four rules presented above regarding the occurrence of the phenotypes $\mathsf{A}$, $\mathsf{B}$, $\mathsf{C}$ and $\mathsf{D}$  -- is still quite challenging for linear methods such as PRS, and the beginning of this introduction has demonstrated -- through our example of the progeny $M+F \to C_1+\dots+C_n$ -- how equations of the form shown in (\ref{eq:intro:icm-equations:haplotypes(recomb)}) and (\ref{eq:intro:icm-equations:phenotypes}) can guide us toward a better understanding of the mechanisms that lead to a given combination of phenotypes. To address these considerations, the present article proposes a framework in which we can find and describe the space of transformations relating the genomic materials of individuals in a parametrized manner. However, contrary to additive framework, the techniques used to find these parametrizations will not be compatible with the usual algebraic structures used in standard linear algebra. As a result, the paper will shift from the conventional linear algebra paradigm (defined for rings and fields) to a computational paradigm involving operations that more closely model biological mechanisms, namely Mendel's first and second laws.

The overall achievement of this paper is to show how pedigrads \cite{Seqali} enriched in the category of idempotent commutative monoids give a symbolic and computational framework to solve the word problems for relations of the form shown in (\ref{eq:intro:icm-equations:haplotypes(recomb)}). In particular, these relations accounts for population stratification in a more combinatorial and biological fashion than the linear models. Such techniques could be used along with Mendelian-randomization-based methods \cite{Smith03,Smith05,Smith14} to improve disease predictions and create new opportunities to tailor treatments to specific types of patients.

\subsection{Road map and results}
In this paper, we aim to solve the word problems for a class of monoids whose elements provide models for population haplotypes. Specifically, these monoids are computed through localization-like techniques as coequalizers of diagrams resulting from images of pedigrads enriched in idempotent commutative monoids. The underlying pedigrad structures associated with the construction of these monoids allow us to address two computational challenges arising in population genomics. First, they allow us to reason about population dynamics through a lightweight formalism that implicitly accounts for the action of homologous recombination and segregation on genomic information. Second, they allow us to determine conditions in which we can identify pairs of groups of haplotypes (\emph{i.e.} solve the word problem) and, by doing so, clarify the stratification structure of a population. These two advantages give us a third computational advantage, which comes in the form of a variational principle for linking genotypes to phenotypes and unravel combinatorial interactions existing between genetic variations.

Since the present work builds on previous work (see \cite{Seqali}), we start by recalling the basic definitions of our formalism in section \ref{sec:Chromologies_and_Pedigrads} -- the main concepts being those of a pedigrad and a chromology. As was the case in \cite{Seqali}, section \ref{sec:Chromologies_and_Pedigrads} includes the presentation of a general example that is used throughout the paper to illustrate the various definitions introduced herein (see section \ref{ssec:Main_example}). In section \ref{ssec:truncation_functor}, we recall some of the main constructions of \cite{Seqali}, which were originally introduced to study sequence alignments. In this paper, we will use these structures to model the diploid structure of the genome, which involves the consideration of pairs of alleles for each genetic location.

We then proceed to section \ref{sec:Pedigrads_in_idempotent_commutative_monoids}, which constitutes one of the two main parts of this paper. Specifically, the goal of this section is to show that pedigrads enriched in idempotent commutative monoids can be used to model haplotype dynamics such as homologous recombination and segregation. To do so, we start by recalling the definition of idempotent commutative monoids in section \ref{ssec:ICMonoids}. We then introduce the category of idempotent commutative (denoted as $\mathbf{Icm}$) in section \ref{ssec:Category_Icm} and its associated universal property in section \ref{ssec:ICMonoids_universal_construction}. To prepare for the characterization of the universal property associated with our pedigrads in $\mathbf{Icm}$, we recall, in section \ref{ssec:Reminder_monomorphisms_epimorphisms} and section \ref{ssec:Coequalizers_of_ic_monoids}, general facts about epimorphisms, monomophisms and coequalizers in the category $\mathbf{Icm}$. Then, in section \ref{ssec:biology_algebraic-operations}, we define the concepts of genotypes, haplotypes and haplogroups in terms of the internal language associated with our pedigrads. To express the universal property of our pedigrads, we introduce recombination chromologies in section \ref{ssec:recombination_chromologies}, which are a type of chromology (a type of limit sketch -- see \cite{Seqali}) suited for the modeling homologous recombination. Finally, in section \ref{ssec:recombination_monoids} and section \ref{ssec:Recombination_schemes_and_pedigrads}, we show how to construct canonical pedigrads in $\mathbf{Icm}$ for given recombination chromologies (see Corollary \ref{cor:D_ET_is_a_mon_pedigrad}, which is deduced from Theorem \ref{theo:morphism_to_mon_pedigrad} and Theorem \ref{theo:representable_pedigrad_E_b_varepsilon}) and introduce the concept of recombination schemes, which allows us to characterize the images of our pedigrads in terms of freely generated idempotent commutative monoids. In particular, recombination schemes provide us with an environments in which we can use linear algebra intuitions to find parametrization of haplotypes in terms of other haplotypes.

We develop these linear algebra intuitions throughout section \ref{sec:solving_our_problem} -- the other main part of this paper. Specifically, this section tackles the development of a matrix framework in idempotent commutative monoids to solve the world problem in recombination schemes. First, to be able to reason about equations of haplotypes, we introduce, in section \ref{ssec:Number_systems_for_icmonoids}, a formal subtraction operation for idempotent commutative monoids. For convenience, we denote this subtraction operation by using a fractional notation (as opposed to an additive one). We show in section \ref{ssec:idempotent_commutative_semirings} that these fractions can be encoded through the action of a certain semiring structure on the underlying idempotent commutative monoid. This action structure suggests that our formal subtraction operation can be modeled through a tensor-like structure. From section \ref{ssec:formal_series-polynomials} to section \ref{ssec:Multiplicative-atomic-structures}, we show how we can encode and recover this tensor-like structure in terms of a semiring of polynomials quotiented by a certain equivalence relation (see Theorem \ref{theo:tensor-congruence:semiring-compatiblity}). Once this tensor-like structure is established, we further develop the concept of action of a semiring on a commutative monoid in section \ref{ssec:action-semirings}. This allows us to develop, in section \ref{ssec:linear algebra}, a general matrix framework for commutative monoids and semirings. In order to solve the word problem in recombination schemes, we provide three successive notions of null spaces, the last two notions improving their predecessors through more practical algorithms. Each of the proposed null spaces requires different reformulations of the underlying matrix framework, which we describe from section \ref{secc:skew:linear-algebra} to section \ref{ssec:Calculus_and_algorithm}. More specifically, these sections provide a sequence of theorems (namely, Theorem \ref{theo:correspondence:null-spaces:1}, Theorem \ref{theo:correspondence:null-spaces:2} and Theorem \ref{theo:correspondence:null-spaces:3}) that ultimately allow us to develop the right notion of null space to find parametrizations of haplotypes in terms of other haplotypes (see Convention \ref{conv:solution_ud}, Remark \ref{rem:algogirthm:linear-systems} and Example \ref{exa:pedigrad-unification}).

We conclude the paper in section \ref{sec:conclusion} by summarizing the main points discussed in each of our examples based on the main example provided in section \ref{ssec:Main_example}.


\subsection{Acknowledgments}
The difficulty of translating concepts of biology, and its statistical nature, into category theory has made the present work go through a series of reformulations over these last five years. While the mathematical content has not changed much, the story told in this last version is quite different from the one told in the early versions. In this acknowledgment section, I would like to thank anyone who returned feedback, suggestions, questions, or engaged in discussions about the present work. In particular, I would like to thank (in alphabetical order) David Spivak, Eric Neumann, Gregory Ginot, Gregory Grant, Jean Clairambault, Jean-Fran\c{c}ois Mascari, Manolis Kellis, Nils Baas, Sharon Spivak, Soumyashant Nayak, and Yongjin Park, who, in various ways, all helped to reach the present version of this manuscript.


\section{Chromologies and Pedigrads}\label{sec:Chromologies_and_Pedigrads}

\subsection{Main example}\label{ssec:Main_example}
As was done in \cite{Seqali}, we will make use of a main example to motivate the different concepts introduced herein. Most of our subsequent examples will aim to demonstrate how the different definitions and results of the present paper help understand the main example (stated at the end of this section). In this paper, our main example will showcase a situation in which one is interested to \emph{identify, for a given population, the parts of the genome that are responsible for a set of phenotypes observed in the individuals of that population}.

Ideal conditions to establish a mapping from genotypes to phenotypes would be the consideration of a population that show a great amount of variations at the level of the phenotypes, but little variation at the level of the genotypes. For instance, Mendel was able to reach such conditions by considering progenies in his experiments (see section \ref{ssec:motivations}). Unfortunately, such conditions can rarely be achieved in the case of humans, since members of human progenies can only coexist through five or six successive generations. For this reason, we will center all our reasoning on a slightly broader type of populations called \emph{haplogroups}, which gather groups of individuals sharing the same haplotype.

Before introducing the premise of our main example, let us use the next couple of paragraphs to recall some important challenges about genetics, which will guide of endeavor. First, let us emphasize that there is a practical difficulty to studying haplogroups, as opposed to progenies, in that haplogroups need to be computed through comparisons of genomic information while progenies can just be recorded through time. There is also a challenge in trying to link genotypes to phenotypes in that the occurrence of most phenotypes would be explained by not only genetic information but also external factors.

Despite these difficulties, one often hopes to find genetic causes by comparing genotypes with phenotypes without considering of external factors. To do so, on often focuses on a specific type genetic variant known as \emph{single nucleotide polymorphisms} (SNPs). These SNPs are variants that can be found within at least one percent of the population (see \cite{Dictionary}). Interestingly, SNPs also correspond to the types of variants used to define haplotypes (see \cite{Hap_map}). As a result, our approach, which mostly focus on haplotypes and not SNPs, can be seen as a natural extension of current practices. Furthermore, we could argue that haplogroups are more adequate to describe genotype-phenotype associations due to linkage disequilibrium (see section \ref{ssec:motivations} above). Indeed, it is important to understand that each variant in linkage disequilibrium is not necessarily biologically responsible for the disease, but that it statistically appears to be so because it is likely to be on the same segment as a causal variant. For these reasons, disease-associated variants are usually only referred to as \emph{markers} for the disease. In our approach, we want to more adequately characterize genetic markers in terms of haplotypes.

Without further introduction, we shall now introduce our problem. Since we already build our intuitions on a practical example involving four individuals $I_1$, $I_2$, $I_3$ and $I_4$ in section \ref{ssec:motivations}, we will further build on this example for a broader population and a greater number of variants. For clarity, these added variants will be written in lower case while the causal variants considered in section \ref{ssec:motivations} will be indicated in upper case. For our example, we shall consider a set of 12 individuals and a set of 15 variants, all assumed to be SNPs and each made of two alleles encoded by nucleotide symbols. From now on, the individuals $I_1$, $I_2$, $I_3$, and $I_4$ will be referred to as $\mathtt{p}_1$, $\mathtt{p}_2$, $\mathtt{p}_3$, and $\mathtt{p}_4$.

\[
\begin{array}{|c|l|l|}
\hline
\multicolumn{1}{|c|}{\cellcolor[gray]{0.8}\textrm{Ind.}}& \multicolumn{1}{c|}{\cellcolor[gray]{0.8}\textrm{Gen.}} & \multicolumn{1}{c|}{\cellcolor[gray]{0.8}\textrm{Phen.}}\\
\hline
\multirow{2}{*}{$\mathtt{p}_{1}$}
& \mathtt{aAtCgCtAtTtcaaT} & \multirow{2}{*}{\textsf{ABC}} \\ %
& \mathtt{gAtCcCcAaGtgacT} &\\ %
\hline
\multirow{2}{*}{$\mathtt{p}_{2}$}
& \mathtt{gTtCcCtGaGtgatT} & \multirow{2}{*}{\textsf{AD}} \\ %
& \mathtt{aAtAcCcAtTgcctC} &\\ %
\hline
\multirow{2}{*}{$\mathtt{p}_{3}$}
& \mathtt{aAtCgCtGaGtgaaT} & \multirow{2}{*}{\textsf{CD}} \\ %
& \mathtt{aAtAcCcAaGtgatT} &\\ %
\hline
\end{array}
\quad\quad\quad
\begin{array}{|c|l|l|}
\hline
\multicolumn{1}{|c|}{\cellcolor[gray]{0.8}\textrm{Ind.}}& \multicolumn{1}{c|}{\cellcolor[gray]{0.8}\textrm{Gen.}} & \multicolumn{1}{c|}{\cellcolor[gray]{0.8}\textrm{Phen.}}\\
\hline
\multirow{2}{*}{$\mathtt{p}_{4}$}
& \mathtt{gTtCcCcAtTgcctC} & \multirow{2}{*}{\textsf{AB}} \\ %
& \mathtt{gAtCcCtAtTtcacT} &\\ %
\hline
\multirow{2}{*}{$\mathtt{p}_{5}$}
& \mathtt{gAtCcCtAtGtcaaT} & \multirow{2}{*}{\textsf{ACD}} \\ %
& \mathtt{aAcAgTcAtTgcatT} &\\ %
\hline
\multirow{2}{*}{$\mathtt{p}_{6}$}
& \mathtt{gAtCcCtGaTgcctC} & \multirow{2}{*}{\textsf{AB}} \\ %
& \mathtt{gTtCcCtGaGtgaaaT} &\\ %
\hline
\end{array}
\]
\[
\begin{array}{|c|l|l|}
\hline
\multicolumn{1}{|c|}{\cellcolor[gray]{0.8}\textrm{Ind.}}& \multicolumn{1}{c|}{\cellcolor[gray]{0.8}\textrm{Gen.}} & \multicolumn{1}{c|}{\cellcolor[gray]{0.8}\textrm{Phen.}}\\
\hline
\multirow{2}{*}{$\mathtt{p}_{7}$}
& \mathtt{gTtCcCtAtGtcaaT} & \multirow{2}{*}{\textsf{ABC}} \\ %
& \mathtt{gAtCcCcAtTgcatT} &\\ %
\hline
\multirow{2}{*}{$\mathtt{p}_{8}$}
& \mathtt{aAcAgTtGaTgcctC} & \multirow{2}{*}{\textsf{AD}} \\ %
& \mathtt{gAtCcCtGaGtgaaT} &\\ %
\hline
\multirow{2}{*}{$\mathtt{p}_{9}$}
& \mathtt{aTtAgTcAaGtgaaT} & \multirow{2}{*}{\textsf{CD}} \\ %
& \mathtt{aTtCgTtGaTgcatT} &\\ %
\hline
\end{array}
\quad\quad\quad
\begin{array}{|c|l|l|}
\hline
\multicolumn{1}{|c|}{\cellcolor[gray]{0.8}\textrm{Ind.}}& \multicolumn{1}{c|}{\cellcolor[gray]{0.8}\textrm{Gen.}} & \multicolumn{1}{c|}{\cellcolor[gray]{0.8}\textrm{Phen.}}\\
\hline
\multirow{2}{*}{$\mathtt{p}_{10}$}
& \mathtt{aTtCgTcAaTgcacT} & \multirow{2}{*}{\textsf{AB}} \\ %
& \mathtt{gAtCcCtGaTgcctC} &\\ %
\hline
\multirow{2}{*}{$\mathtt{p}_{11}$}
& \mathtt{aAcAgTcAaTgcatT} & \multirow{2}{*}{\textsf{ACD}} \\ %
& \mathtt{aTtAgTcAaGtgaaT} &\\ %
\hline
\multirow{2}{*}{$\mathtt{p}_{12}$}
& \mathtt{gAtCcCcAaGtgatT} & \multirow{2}{*}{\textsf{CD}} \\ %
& \mathtt{aTtAgTcAaGtgacT} &\\ %
\hline
\end{array}
\]

In the rest of this paper, our goal will be to use the previous dataset to illustrate how we can use pedigrads in idempotent commutative monoids to (1) understand the population structure existing between individuals and (2) help identify genetic marker interactions associated with combinations of phenotypes. In this respect, we shall start by integrating the previous tables into a single pedigrad enriched in the category of idempotent commutative monoids.

\subsection{Pre-ordered sets}\label{sec:pre-ordered_sets}
This section recalls concepts associated with pre-order relations.

\begin{definition}[Pre-ordered sets]
A \emph{pre-ordered set} consists of a set $\Omega$ and a binary relation $\leq$ on $\Omega$ satisfying the following logical statements.
\begin{itemize}
\item[1)] (reflexivity) for every $x \in \Omega$, the relation $x \leq x$ holds;
\item[3)] (transitivity) for every $x,y,z \in \Omega$, if $x \leq y$ and $y \leq z$ hold, then so does $x \leq z$.
\end{itemize}
\end{definition}

\begin{example}\label{exa:pre-ordered_set_0_and_1}
The set $\{0,1\}$ is a pre-ordered set when it is equipped with the relations $0 \leq 1$; $0 \leq 0$ and $1 \leq 1$. In the literature, this pre-ordered set is known as the \emph{Boolean} pre-ordered set and the values 0 and 1 are often understood as encoding the Boolean values $\mathtt{false}$ and $\mathtt{true}$, respectively.
\end{example}

\begin{example}\label{exa:product_pre-order_set_0_1}
For every positive integer $n$, the $n$-fold Cartesian product $\{0,1\}^{\times n}$ of the pre-ordered set $\{0,1\}$ defined in Example \ref{exa:pre-ordered_set_0_and_1} is equipped with a pre-order relation $\leq$ that compares two tuples in $\{0,1\}^{\times n}$ such that the relation $(x_1,\dots,x_n) \leq (y_1,\dots,y_n)$ holds if, and only if, the relation $x_i \leq y_i$ holds for every $1 \leq i \leq n$.
\end{example}

\begin{remark}[Pre-order categories]
A pre-ordered set is equivalently a category in which there exists at most one arrow between every pair of objects. In the sequel, a pre-ordered set will sometimes be called a \emph{pre-order category} to emphasize its categorical structure.
\end{remark}

\begin{definition}[Order-preserving functions]
Let $(\Omega_1,\leq_1)$ and $(\Omega_2,\leq_2)$ be two pre-ordered sets. We shall speak of an \emph{order-preserving function} from $(\Omega_1,\leq_1)$ to $(\Omega_2,\leq_2)$ to refer to a function $f:\Omega_1 \to \Omega_2$ for which every relation $x \leq_1 y$ in $\Omega_1$ gives rise to a relation $f(x) \leq_2 f(y)$ in $\Omega_2$.
\end{definition}

\begin{convention}[Notation]
We shall denote by $\mathbf{pOrd}$ the category whose objects are pre-ordered sets and whose morphisms are order-preserving functions.
\end{convention}

\begin{example}[Projection]\label{exa:product_morphisms_set_0_1}
For every positive integer $n$, the $n$-fold Cartesian product $\{0,1\}^{\times n}$ of Example \ref{exa:product_pre-order_set_0_1} is equipped with a canonical collection of $n$ functions $\pi_i:\{0,1\}^{\times n} \to \{0,1\}$, for each $i \in \{1,\dots,n\}$, where a function $\pi_i$ sends a tuple $(x_1,\dots,x_n)$ in $\{0,1\}^{\times n}$ to its $i$-th component $x_i$ in $\{0,1\}$. It is straightforward to verify that these functions preserve the order relations of $\{0,1\}^{\times n}$ in $\{0,1\}$ and thus define morphisms in $\mathbf{pOrd}$.
\end{example}

\subsection{Finite sets of integers}
For every positive integer $n$, we will denote by $[n]$ the finite set of integers $\{1,2,\dots,n\}$. We will also let $[0]$ denote the empty set. In the sequel, for every non-negative integer $n$, the set $[n]$ will implicitly be equipped with the underlying order defined on integers (note that the restriction of this order on $[0]$ is the empty order).

\subsection{Segments}
Let $(\Omega,\preceq)$ denote a pre-ordered set. 
A \emph{segment} over $\Omega$ consists of a pair of non-negative integers $(n_1,n_0)$, an order-preserving surjection\footnote{\emph{i.e.} an order-preserving function that is a surjection.} $t:[n_1] \to [n_0]$ and a function $c:[n_0] \to \Omega$. In the sequel, we shall often refer to the elements of $\Omega$ as \emph{colors}.

\begin{remark}[Representation]
Segments are equipped with a canonical graphical representation (see \cite{Seqali}). For a segment $(t,c)$, say as defined above, the finite set $[n_1]$ represents the ``genomic positions'' composing the segment while the fibers $t^{-1}(1), \dots, t^{-1}(n_0)$ of the surjection $t:[n_1] \to [n_0]$ gather these elements into patches (see the brackets below).
\[
t = \xymatrix@C-30pt{(\bullet&\bullet&\bullet)&(\bullet&\bullet&\bullet&\bullet)&(\bullet&\bullet&\,\cdots\, &\bullet)&(\bullet& \bullet)}
\]
Finally, each patch of the segment is associated with `colors' in $\Omega$ that are specified by the map $c:[n_0] \to \Omega$. For instance, taking $\Omega$ to be the Boolean pre-ordered set $\{\mathtt{false}\leq \mathtt{true}\}$, as defined in Example \ref{exa:pre-ordered_set_0_and_1}, and taking $c$ such that $c(1) = \mathtt{false}$, $c(2) = \mathtt{true}$, $\dots$, $c(n_0-1) = \mathtt{true}$, and $c(n_0) = \mathtt{true}$ will color all the elements of $[n_1]$ belonging to the fiber $t^{-1}(1)$ in the $\mathtt{false}$-color (represented in \emph{white}) and will color the elements of the fibers $t^{-1}(2)$, $\dots$, $t^{-1}(n_0-1)$, and $t^{-1}(n_0)$ in the $\mathtt{true}$-color (represented in black), as shown below.
\[
(t,c) = \xymatrix@C-30pt{(\circ&\circ&\circ)&(\bullet&\bullet&\bullet&\bullet)&(\bullet&\bullet&\,\cdots\, &\bullet)&(\bullet& \bullet)}
\]
\end{remark}

\begin{remark}[Notations]
In the definition of a segment $(c,t)$ as defined above, the specification of the data $n_1$ and $n_0$ is redundant with the data of the function $t$ and $c$. Yet, for the sake of conciseness and clarity, such a segment $(c,t)$ will sometimes be referred to as an arrow $(t,c):[n_1] \multimap [n_0]$.
\end{remark}

\begin{convention}[Domains, topologies \& types]
For every segment $(t,c):[n_1] \multimap [n_0]$, the data $[n_1]$ will be called the \emph{domain} of $(t,c)$, the data $t$ will be called the \emph{topology} of $(t,c)$ and the data $(n_1,n_0)$ will be called the \emph{type} of $(t,c)$. The type of a segment will be formally specified through the notation $[n_1] \multimap [n_0]$.
\end{convention}

\begin{definition}[Homologous segments]
Two segments $(t,c)$ and $(t',c')$ over $\Omega$ will be said to be \emph{homologous} if their topologies $t$ and $t'$ are equal.
\end{definition}

\begin{definition}[Quasi-homologous segments]
Two segments $(t,c):[n_1] \multimap [n_0]$ and $(t',c'):[n_1'] \multimap [n_0']$ over $\Omega$ will be said to be \emph{quasi-homologous} if their domains $[n_1]$ and $[n_1']$ are equal.
\end{definition}

\subsection{Morphisms of segments}\label{ssec:Morphisms_of_chromosomal_patches}
Let $(\Omega,\preceq)$ be a pre-ordered set and let $(t,c):[n_1] \multimap [n_0]$ and $(t',c'):[n_1'] \multimap [n_0']$ be two segments over $\Omega$. A morphism of segments from $(t,c)$ to $(t',c')$ consists of a pair $(f_1,f_0)$ made of
\begin{itemize}
\item[1)] an order-preserving injection $f_1:[n_1] \to [n_1']$;
\item[2)] an order-preserving function $f_0:[n_0] \to [n_0']$;
\end{itemize}
such that the inequality $c' \circ f_0(i) \preceq c(i)$ holds for every $i \in [n_0]$ and the following diagram commutes.
\[
\xymatrix{
[n_1]\ar@{->>}[r]^{t}\ar@{)->}[d]_{f_1}&[n_0]\ar[d]^{f_0}\\
[n_1']\ar@{->>}[r]^{t'}&[n_0']
}
\]
It is straightforward to check that the obvious component-wise composition on morphisms of segments over $\Omega$ defines a category whose identities are given by the pairs of identities. We will denote this category by $\mathbf{Seg}(\Omega)$; its objects are segments over $\Omega$ and its arrows are the morphisms between them.

\subsection{Pre-orders on homologous segments}
Let $(\Omega,\preceq)$ be a pre-ordered set and let $t:[n_1] \to [n_0]$ be an order-preserving surjection. The subcategory of $\mathbf{Seg}(\Omega)$ whose objects are the homologous segments of topology $t$ and whose arrows are the morphisms of segments for which the components $f_0$ and $f_1$ are identities will be denoted as $\mathbf{Seg}(\Omega:t)$
and referred to as the \emph{category of homologous segments (over $\Omega$) of topology $t$}.

\begin{proposition}[From \cite{Seqali}]
For every order-preserving surjection $t:[n_1] \to [n_0]$, the category $\mathbf{Seg}(\Omega:t)$ is a pre-order category.
\end{proposition}

\subsection{Pre-orders on quasi-homologous segments}
Let $(\Omega,\preceq)$ be a pre-ordered set and let $n_1$ be a non-negative integer. The subcategory of $\mathbf{Seg}(\Omega)$ whose objects are the quasi-homologous segments of domain $[n_1]$ and whose arrows are the morphisms of segments for which the component $f_1$ is an identity will be denoted as $\mathbf{Seg}(\Omega\,|\,n_1)$ and called the \emph{category of quasi-homologous segments (over $\Omega$) of domain $n_1$}.

The phrasing of following statement will become relevant in proof of Proposition \ref{prop:compare_cone_irreducible}, which is related to one of our main theorems (see Theorem \ref {theo:representable_pedigrad_E_b_varepsilon})

\begin{lemma}\label{lem:quasi_homologous_preordered_category}
If there exists a morphism $(t,c) \to (t',c')$ in $\mathbf{Seg}(\Omega\,|\,n_1)$, then it is the only morphism of type $(t,c) \to (t',c')$ in $\mathbf{Seg}(\Omega)$.
\end{lemma}
\begin{proof}
Let $(\mathrm{id},f_0):(t,c) \to (t',c')$ be the morphism of the statement in $\mathbf{Seg}(\Omega\,|\,n_1)$ and let $(g_1,g_0):(t,c) \to (t',c')$ be another morphism in $\mathbf{Seg}(\Omega)$. Because $g_1$ is an order-preserving inclusion of type $[n_1] \to [n_1]$, it must be an identity, so that the identity $g_0 \circ t = t'$ holds. On the other hand, the identity $f_0 \circ t = t'$ also holds, which means that $g_0 \circ t = f_0 \circ t$. Because $t$ is an epimorphism, the identity $g_0=f_0$ must hold.
\end{proof}

Even though the statement of Lemma \ref{lem:quasi_homologous_preordered_category} is intended to be used in subsequent sections, we can use it here to give an alternative proof for the fact that a category of quasi-homologous segments is a pre-order category.

\begin{proposition}[Also proved in \cite{Seqali}]\label{prop:Seg_Omega_n_porder}
For every non-negative integer $n_1$, the category $\mathbf{Seg}(\Omega\,|\,n_1)$ of quasi-homologous segments is a pre-order category.
\end{proposition}
\begin{proof}
Since $\mathbf{Seg}(\Omega\,|\,n_1)$ is a subcategory of $\mathbf{Seg}(\Omega)$, Lemma \ref{lem:quasi_homologous_preordered_category} implies that there exists at most one arrow between each pair of objects in $\mathbf{Seg}(\Omega\,|\,n_1)$.
\end{proof}

\subsection{Cones}\label{ssec:cones}
Recall that a \emph{cone} in a category $\mathcal{C}$ consists of an object $X$ in $\mathcal{C}$, a small category $A$, a functor $F:A \to \mathcal{C}$ and a natural transformation $\Delta_{A}(X) \Rightarrow F$ where $\Delta_{A}(X)$ denotes the constant functor $A \to \mathbf{1} \to \mathcal{C}$ mapping every object in $A$ to the object $X$ in $\mathcal{C}$.

\begin{definition}[Wide spans]\label{def:wide_spans}
In the sequel, we shall speak of a \emph{wide span} to refer to a cone $\Delta_{A}(X) \Rightarrow F$ defined over a finite discrete small category $A$ whose objects are ordered with respect to a total order (this will allow us to have canonical choices of limits).
\end{definition}

\begin{example}[Wide spans]\label{exa:wide_spans}
Giving a wide span in a category $\mathcal{C}$ amounts to giving a finite collection of arrows $\mathbf{S} :=\{f_i:X \to F_i\}_{i \in [n]}$ in $\mathcal{C}$. When the category $\mathcal{C}$ has products, the implicit order of the set $[n]=\{1,\dots,n\}$ can be used to give a specific representative for the product of the collection $\{F_i\}_{i \in [n]}$ in $\mathcal{C}$.
\end{example}

\subsection{Chromologies}\label{sec:chromologies}
A \emph{chromology} is a pre-ordered set $(\Omega,\preceq)$ that is equipped, for every non-negative integer $n$, with a set $D[n]$ of cones in the category $\mathbf{Seg}(\Omega\,|\,n)$. Such a chromology will later be denoted as a pair $(\Omega,D)$.

\begin{remark}[Future example]
See section \ref{ssec:recombination_chromologies} for an example of a particular chromology.
\end{remark}

\begin{convention}[Subchromologies]\label{conv:subchromologies}
For every pre-ordered set $(\Omega,\preceq)$, we will say that a chromology $(\Omega,D)$ is a \emph{subchromology} of another chromology $(\Omega,D')$ if for every non-negative integer $n$, the inclusion $D[n] \subseteq D'[n]$ holds. We will then write $(\Omega,D) \subseteq (\Omega,D')$.
\end{convention}

\subsection{Logical systems}\label{ssec:Logical_systems}
We will speak of a \emph{logical system} to refer to a category $\mathcal{C}$ that is equipped with a subclass of its cones $\mathcal{W}$ (see section \ref{ssec:cones}).

\subsection{Pedigrads} \label{ssec:Pedigrads}
Let $(\Omega,D)$ be a chromology and $(\mathcal{C},\mathcal{W})$ be a logical system. A \emph{pedigrad} in $(\mathcal{C},\mathcal{W})$ is a functor $\mathbf{Seg}(\Omega) \to \mathcal{C}$ that sends, for every non-negative integer $n$, each cone in $D[n]$ to a cone in $\mathcal{W}$.

\begin{convention}[$\mathcal{W}$-pedigrads]
As was done in \cite{Seqali}, we will often refer to a pedigrad in logical system $(\mathcal{C},\mathcal{W})$ as a $\mathcal{W}$-pedigrad.
\end{convention}

\subsection{Morphisms of pedigrads} \label{ssec:Morphisms_of_pedigrads}
Recall that, for every pair of categories $\mathcal{C}$ and $\mathcal{D}$, the notation $[\mathcal{C},\mathcal{D}]$ denotes the category whose objects are functors $\mathcal{C} \to \mathcal{D}$ and whose arrows are natural transformations in $\mathcal{D}$ over $\mathcal{C}$.
Let $(\Omega, D)$ be a chromology and $(\mathcal{C},\mathcal{W})$ be a logical system. A \emph{morphism of pedigrads} from a pedigrad $A:\mathbf{Seg}(\Omega)\to\mathcal{C}$ in $(\mathcal{C},\mathcal{W})$ for the chromology $(\Omega, D)$ to a pedigrad $B:\mathbf{Seg}(\Omega)\to\mathcal{C}$ in $(\mathcal{C},\mathcal{W})$ for the chromology $(\Omega, D)$ is an arrow  $A \Rightarrow B$ in the category $[\mathbf{Seg}(\Omega),\mathcal{C}]$.

\section{Environment functors}\label{sec:Examples_of_pedigrads_in_sets}
The goal of this section is to recall more constructions made in \cite{Seqali}. Throughout the section, we shall let $(E,\varepsilon)$ denote a pointed set and $(\Omega,\preceq)$ denote a pre-ordered set.

\subsection{Truncation functors}\label{ssec:truncation_functor}
First, we recall the definition of the truncation operation given in \cite{Seqali}. This operation can be compared to a filtering operation that would allow us to only select information marked by ``colors'' above a certain threshold.

\begin{definition}[Truncation]\label{def:truncation}
For every segment $(t,c):[n_1] \multimap [n_0]$ over $\Omega$ and element $b \in \Omega$, we will denote by $\mathsf{Tr}_b(t,c)$ the subset $\{i \in [n_1]~|~ b \preceq c \circ t(i)\}$ of $[n_1]$. This is the set of all elements in $[n_1]$ whose images via $c \circ t$ is greater than or equal to $b$ in $\Omega$.
\end{definition}

\begin{example}[Truncation]\label{exa:truncation}
Let $(\Omega,\preceq)$ be the Boolean pre-ordered set $\{0 \leq 1\}$. If we consider the segment in $\mathbf{Seg}(\Omega)$ given below, on the left, the truncation operation takes the values shown on the right depending on the values of $b \in \{0,1\}$.
\[
(t,c)=\xymatrix@C-30pt{
(\bullet&\bullet&\bullet)&(\circ&\circ)&(\bullet&\bullet&\bullet&\bullet)&(\circ&\circ&\circ&\circ&\circ)&(\bullet&\bullet&\bullet)&(\circ)
}
\quad\quad\quad
\left\{
\begin{array}{ll}
\mathsf{Tr}_1(t,c) = \{1,2,3,6,7,8,9,15,16,17\}\\
\mathsf{Tr}_0(t,c) = {[18]}
\end{array}
\right.
\]
\end{example}

As is done in \cite{Seqali}, we extend the previous operation to a functor on categories of segments by extending the category $\mathbf{Set}$, in which the operation $\mathsf{Tr}_b$ lands, to the category $\mathbf{Set}_{\ast}$ of pointed sets and point-preserving maps. Before doing so, recall from \cite{MacLane} that there is an adjunction
\[
\xymatrix{
\mathbf{Set} \ar@<+1.2ex>[r]^{\mathbb{F}} \ar@{}[r]|{\bot}\ar@<-1.2ex>@{<-}[r]_{\mathbb{U}} & \mathbf{Set}_{\ast}
}
\]
whose right adjoint $\mathbb{U}:\mathbf{Set}_{\ast} \to \mathbf{Set}$ forgets the pointed structure (\emph{i.e.} $\mathbb{U}:(X,p) \mapsto X$) and whose left adjoint 
$\mathbb{F}:\mathbf{Set} \to \mathbf{Set}_{\ast}$ maps a set $X$ to the obvious pointed set $(X+\{\ast\},\ast)$ and maps a function $f:X \to Y$ to the coproduct map $f+\{\ast\}:X+\{\ast\} \to Y+\{\ast\}$.

\begin{proposition}[From \cite{Seqali}]\label{prop:Tr_functor_pointed_Set_op}
For every element $b \in \Omega$, the mapping $(t,c) \mapsto \mathbb{F}\mathsf{Tr}_b(t,c)$ extends to a functor $\mathsf{Tr}_b^{\ast}:\mathbf{Seg}(\Omega) \to \mathbf{Set}_{\ast}^{\mathrm{op}}$ mapping every function $(f_1,f_0):(t,c) \to (t',c')$ in $\mathbf{Seg}(\Omega)$ to the following map of pointed sets.
\[
\begin{array}{llllll}
\mathsf{Tr}_b^{\ast}(f_1,f_0)&:&\mathbb{F}\mathsf{Tr}_b(t',c')&\to&\mathbb{F}\mathsf{Tr}_b(t,c)&\\
&&j&\mapsto & i &\textrm{if } \exists i \in \mathsf{Tr}_b(t,c) \,: \, j = f_1(i);\\
&&j&\mapsto & \ast &\textrm{otherwise.}\\
\end{array}
\]
\end{proposition}

\subsection{Environment functors}
In this section, we construct functors of the form $\mathbf{Seg}(\Omega) \to \mathbf{Set}$ for every pointed set $(E,\varepsilon)$ and every parameter $b \in \Omega$ (Definition \ref{def:set_E_b_varepsilon}). We do so by using the truncation functor defined in section \ref{ssec:truncation_functor}.

The following convention recapitulates various notations used in \cite{Seqali}.

\begin{convention}[Notation]
The hom-set of a category $\mathcal{C}$ from an object $X$ to an object $Y$ will be denoted as $\mathcal{C}(X,Y)$. For instance, the set of functions from a set $X$ to a set $Y$ will be denoted by $\mathbf{Set}(X,Y)$. Recall that, for any category $\mathcal{C}$, the hom-sets give rise to a functor $\mathcal{C}(\_,\_):\mathcal{C}^{\mathrm{op}}\times \mathcal{C} \to \mathbf{Set}$ called the \emph{hom-functor} \cite[page 27]{MacLane}.
\end{convention}

\begin{definition}[Environment functors]\label{def:set_E_b_varepsilon}
For every element $b \in \Omega$, we will denote by $E_{b}^{\varepsilon}$ the functor $\mathbf{Seg}(\Omega) \to \mathbf{Set}$ defined as the composition of the following pair of functors.
\[
\xymatrix@C+20pt{
\mathbf{Seg}(\Omega)\ar[r]^{\mathsf{Tr}_b^{\ast}}&\mathbf{Set}_{\ast}^{\mathrm{op}}\ar[rr]^{\mathbf{Set}_{\ast}(\_,(E,\varepsilon))}&&\mathbf{Set}
}
\]
\end{definition}

\begin{example}[Objects]\label{exa:elements_as_words}
Suppose that $(\Omega,\preceq)$ denotes the Boolean pre-ordered set $\{0\leq 1\}$ and let $(E,\varepsilon)$ be the pointed set $\{\mathtt{A},\mathtt{C},\mathtt{G},\mathtt{T},\varepsilon\}$. If we consider the segment
\[
(t,c) = \xymatrix@C-30pt{
(\bullet&\bullet&\bullet)&(\circ&\circ)&(\bullet&\bullet&\bullet&\bullet)&(\bullet&\bullet&\bullet&\bullet&\bullet)&(\circ&\circ&\circ)&(\circ)
}
\]
then the set $E_1^{\varepsilon}(t,c)$ (where $b = 1$) will contain the following words (which have been parenthesized for clarity), among many others.
\[
\begin{array}{c}
\xymatrix@C-30pt{
(\mathtt{A}&\mathtt{G}&\varepsilon)&(\mathtt{T}&\mathtt{C}&\mathtt{A}&\mathtt{A})&(\mathtt{T}&\mathtt{A}&\mathtt{G}&\mathtt{G}&\varepsilon);
}\\
\xymatrix@C-30pt{
(\mathtt{G}&\mathtt{T}&\varepsilon)&(\varepsilon&\varepsilon&\varepsilon&\mathtt{C})&(\mathtt{A}&\mathtt{G}&\mathtt{T}&\mathtt{A}&\mathtt{C});
}\\
\xymatrix@C-30pt{
(\mathtt{T}&\mathtt{A}&\mathtt{A})&(\mathtt{G}&\mathtt{A}&\mathtt{T}&\mathtt{C})&(\mathtt{A}&\mathtt{G}&\mathtt{T}&\mathtt{T}&\mathtt{T});
}\\
\textrm{etc.}\\
\end{array}
\]
\end{example}

\begin{example}[Morphisms]\label{exa:morphisms_as_inclusion_of_words}
Suppose that $\Omega$ denotes the Boolean pre-ordered set $\{0\leq 1\}$ and let $(E,\varepsilon)$ be the pointed set $\{\mathtt{A},\mathtt{C},\mathtt{G},\mathtt{T},\varepsilon\}$. If we consider the morphism of segments given below, in which we use adequate labeling to show how the first segment is included in the second one,
\begin{equation}\label{eq:morphisms-segment:env-functors}
\xymatrix@C-30pt{
(\mathop{\bullet}\limits^1&\mathop{\bullet}\limits^2&\mathop{\bullet}\limits^3)&(\mathop{\circ}\limits^4&\mathop{\circ}\limits^5)&(\mathop{\bullet}\limits^6&\mathop{\bullet}\limits^7&\mathop{\bullet}\limits^8&\mathop{\bullet}\limits^9)&(\mathop{\bullet}\limits^{10\,}&\mathop{\bullet}\limits^{11})
} \to \xymatrix@C-30pt{
(\mathop{\bullet}\limits^1&\mathop{\bullet}\limits^2&\mathop{\bullet}\limits^3&\mathop{\bullet}\limits^{\ast}&\mathop{\bullet}\limits^{\ast})&(\mathop{\circ}\limits^4&\mathop{\circ}\limits^5&\mathop{\circ}\limits^{\ast})&(\mathop{\bullet}\limits^6&\mathop{\bullet}\limits^7&\mathop{\bullet}\limits^8&\mathop{\bullet}\limits^9)&(\mathop{\bullet}\limits^{\ast})&(\mathop{\circ}\limits^{10\,}&\mathop{\circ}\limits^{11})
}
\end{equation}
then the image of the previous arrow via $E_1^{\varepsilon}$ is a function whose mappings rules look as follows.
\[
\begin{array}{ccc}
\xymatrix@C-30pt{
(\mathtt{A}&\mathtt{G}&\varepsilon)&(\mathtt{T}&\mathtt{C}&\mathtt{A}&\mathtt{A})&(\mathtt{G}&\mathtt{C})
} &\mapsto &
\xymatrix@C-30pt{
(\mathtt{A}&\mathtt{G}&\varepsilon&\varepsilon&\varepsilon)&(\mathtt{T}&\mathtt{C}&\mathtt{A}&\mathtt{A})&(\varepsilon);
}\\
\xymatrix@C-30pt{
(\mathtt{G}&\mathtt{T}&\varepsilon)&(\varepsilon&\varepsilon&\varepsilon&\mathtt{C})&(\mathtt{T}&\mathtt{A})
} &\mapsto &
\xymatrix@C-30pt{
(\mathtt{G}&\mathtt{T}&\varepsilon&\varepsilon&\varepsilon)&(\varepsilon&\varepsilon&\varepsilon&\mathtt{C})&(\varepsilon);
}\\
\xymatrix@C-30pt{
(\mathtt{T}&\mathtt{A}&\mathtt{A})&(\mathtt{G}&\mathtt{A}&\mathtt{T}&\mathtt{C})&(\mathtt{A}&\mathtt{A})
} &\mapsto& 
\xymatrix@C-30pt{
(\mathtt{T}&\mathtt{A}&\mathtt{A}&\varepsilon&\varepsilon)&(\mathtt{G}&\mathtt{A}&\mathtt{T}&\mathtt{C})&(\varepsilon);
}\\
&\textrm{etc.}&\\
\end{array}
\]
Note how consistent appearances of the pointed element $\varepsilon$ in the previous mappings match the newly inserted elements marked by the symbol $\ast$ in the target of the arrow shown in (\ref{eq:morphisms-segment:env-functors}).
\end{example}

\subsection{Sequence alignments}
In this section, we use a particular case of a more general definition given in \cite{Seqali}. Specifically, the construction of Convention \ref{conv:aligned_pedigrad}, given below, corresponds to an example of what was called an \emph{aligned pedigrad} in \cite{Seqali} while the subsequent construction given in Definition \ref{conv:sequence_alignment_sample} corresponds to an example of what was called a \emph{sequence alignment} in \cite{Seqali}.

\begin{convention}[Notation]\label{conv:aligned_pedigrad}
Let $b$ be an element in $\Omega$. We denote by $\mathbf{2}E_{b}^{\varepsilon}$ the functor $\mathbf{Seg}(\Omega) \to \mathbf{Set}$ resulting from the composition of the three functors given in (\ref{eq:aligned_pedigrad}), where 
\begin{itemize}
\item[-] the rightmost functor is the obvious Cartesian functor of $\mathbf{Set}$;
\item[-] the middle functor is the Cartesian product of the functors $E_{b}^{\varepsilon}:\mathbf{Seg}(\Omega) \to \mathbf{Set}$;
\item[-] and the leftmost functor is the obvious Cartesian diagonal functor.
\end{itemize}
\begin{equation}\label{eq:aligned_pedigrad}
\xymatrix@C+35pt{
\mathbf{Seg}(\Omega) \ar[r]^-{(\mathrm{id},\mathrm{id})} & \mathbf{Seg}(\Omega) \prod \mathbf{Seg}(\Omega) \ar[r]^-{E_{b}^{\varepsilon} \prod E_{b}^{\varepsilon}}& \mathbf{Set} \prod \mathbf{Set} \ar[r]^-{\times} & \mathbf{Set}
}
\end{equation}
\end{convention}

\begin{definition}[Sequence alignments]\label{conv:sequence_alignment_sample}
Let $b$ be an element in $\Omega$. We define a \emph{sequence alignment} over $\mathbf{2}E^{\varepsilon}_{b}$ as a triple $(\iota,T,\sigma)$ where $\iota$ is an inclusion functor $\iota:B \to \mathbf{Seg}(\Omega)$, where $T$ is a functor $B \to \mathbf{Set}$ and $\sigma$ is a natural monomorphism $T \Rightarrow \mathbf{2}E_b^{\varepsilon} \circ \iota$.
\end{definition}

\begin{example}[Sequence alignments]\label{exa:Sequence alignments}
The goal of this example is to illustrate the concepts of sequence alignments (Definition \ref{conv:sequence_alignment_sample}) by using the data given in section \ref{ssec:Main_example}.
Take $(\Omega,\preceq)$ to be the Boolean pre-ordered $\{0 \leq 1 \}$ and $(E,\varepsilon)$ to be the pointed set $\{\mathtt{A},\mathtt{C},\mathtt{G},\mathtt{T},\varepsilon\}$. For any given segment $\tau$ in $\mathbf{Seg}(\Omega)$, the set $\mathbf{2}E_b^{\varepsilon}(\tau)$ is equal to the product $E_b^{\varepsilon}(\tau) \times E_b^{\varepsilon}(\tau)$. The table shown below gives examples of elements in the set $\mathbf{2}E_b^{\varepsilon}(\tau)$ for the corresponding segment $\tau$ shown on the left-hand side, while $b$ is taken to be equal to $1$. As was done in previous examples, parentheses are added to the elements of $\mathbf{2}E_1^{\varepsilon}(\tau)$ for clarity.
\[
\begin{array}{|c|c|c|}
\hline
\cellcolor[gray]{0.8}\tau&\cellcolor[gray]{0.8}\mathbf{2}E_1^{\varepsilon}(\tau)\\
\hline
\!\xymatrix@C-30pt{(\bullet&\bullet&\bullet)&(\bullet&\bullet&\bullet)&(\bullet&\bullet&\bullet)}\!  & 
\!\!\!\!\begin{array}{c}
\xymatrix@-30pt{
(\mathtt{G}&\mathtt{G}&\varepsilon)&(\mathtt{G}&\mathtt{A}&\varepsilon)&(\mathtt{G}&\mathtt{G}&\mathtt{A})\\
(\mathtt{A}&\mathtt{C}&\mathtt{G})&(\mathtt{C}&\mathtt{C}&\mathtt{T})&(\mathtt{C}&\mathtt{T}&\mathtt{G})\\
}
\end{array}\!\!\!\!;\!\!\!\!\!\!
\quad
\begin{array}{c}
\xymatrix@-30pt{
(\varepsilon&\mathtt{A}&\mathtt{A})&(\mathtt{C}&\mathtt{A}&\mathtt{C})&(\mathtt{T}&\mathtt{T}&\mathtt{C})\\
(\mathtt{C}&\mathtt{C}&\mathtt{A})&(\mathtt{G}&\mathtt{G}&\mathtt{T})&(\mathtt{G}&\mathtt{A}&\varepsilon)\\
}
\end{array}\!\!\!\!;\!\!\!\!\!\!
\quad
\begin{array}{c}
\xymatrix@-30pt{
(\mathtt{G}&\mathtt{G}&\mathtt{A})&(\mathtt{C}&\mathtt{T}&\mathtt{C})&(\mathtt{G}&\mathtt{A}&\mathtt{T})\\
(\mathtt{T}&\mathtt{T}&\varepsilon)&(\mathtt{T}&\mathtt{C}&\mathtt{C})&(\varepsilon&\mathtt{A}&\mathtt{C})\\
}
\end{array}\!\!\!\!; \!\!\!
\quad
\textrm{etc.}\\
\hline
\!\begin{array}{c}
\xymatrix@C-30pt{(\circ&\circ&\circ)&(\bullet&\bullet&\bullet)&(\bullet&\bullet&\bullet)}\\
\xymatrix@C-30pt{(\bullet&\bullet&\bullet)&(\circ&\circ&\circ)&(\bullet&\bullet&\bullet)}\\
\xymatrix@C-30pt{(\bullet&\bullet&\bullet)&(\bullet&\bullet&\bullet)&(\circ&\circ&\circ)}
\end{array}
\! &
\!\!\!\!\begin{array}{c}
\xymatrix@-30pt{
(\mathtt{A}&\mathtt{G}&\mathtt{G})&(\mathtt{C}&\mathtt{G}&\mathtt{T})\\
(\mathtt{A}&\mathtt{T}&\mathtt{G})&(\mathtt{T}&\mathtt{C}&\mathtt{G})\\
}
\end{array}\!\!\!\!;\!\!\!\!\!\!
\quad
\begin{array}{c}
\xymatrix@-30pt{
(\mathtt{T}&\mathtt{A}&\mathtt{T})&(\mathtt{A}&\mathtt{A}&\mathtt{T})\\
(\mathtt{C}&\mathtt{T}&\mathtt{G})&(\mathtt{G}&\mathtt{T}&\varepsilon)\\
}
\end{array}\!\!\!\!;\!\!\!\!\!\!
\quad
\begin{array}{c}
\xymatrix@-30pt{
(\mathtt{C}&\mathtt{A}&\mathtt{A})&(\mathtt{A}&\mathtt{A}&\mathtt{C})\\
(\mathtt{C}&\mathtt{A}&\mathtt{C})&(\varepsilon&\mathtt{A}&\mathtt{C})\\
}
\end{array}\!\!\!\!; \!\!\!
\quad
\textrm{etc.}\\
\hline
\!\xymatrix@C-30pt{(\circ&\circ&\circ)&(\bullet&\bullet&\bullet)&(\circ&\circ&\circ)}\!  & 
\!\!\!\!\begin{array}{c}
\xymatrix@-30pt{
(\mathtt{G}&\mathtt{A}&\varepsilon)\\
(\mathtt{C}&\mathtt{C}&\mathtt{T})\\
}
\end{array}\!\!\!\!;\!\!\!\!\!\!
\quad
\begin{array}{c}
\xymatrix@-30pt{
(\mathtt{C}&\mathtt{A}&\mathtt{C})\\
(\mathtt{G}&\mathtt{G}&\mathtt{T})\\
}
\end{array}\!\!\!\!;\!\!\!\!\!\!
\quad
\begin{array}{c}
\xymatrix@-30pt{
(\mathtt{C}&\mathtt{T}&\mathtt{C})\\
(\mathtt{T}&\mathtt{C}&\mathtt{C})\\
}
\end{array}\!\!\!\!; \!\!\!
\quad
\textrm{etc.}\\
\hline
\end{array}
\]
Note that the tables describing the progeny of section \ref{ssec:Main_example} can be seen as a subpart of the functor $\mathbf{2}E_1^{\varepsilon}$ and, more specifically, as a sequence alignment $(\iota,T,\sigma)$ over $\mathbf{2}E_1^{\varepsilon}$.
To define this sequence alignment, we can take the functor $\iota:B \hookrightarrow \mathbf{Seg}(\Omega)$ to be the inclusion of categories whose domain $B$ is the singleton category consisting of the following segment, which we denote as $\mathtt{seg}_{15}$.
\[
\xymatrix@C-30pt{(\bullet)&(\bullet)&(\bullet)&(\bullet)&(\bullet)&(\bullet)&(\bullet)&(\bullet)&(\bullet)&(\bullet)&(\bullet)&(\bullet)&(\bullet)&(\bullet)&(\bullet)}
\]
Before describing $T$, note that, in the real world, the two alleles encoding a genotype cannot be distinguished as the \emph{top allele} and the \emph{bottom allele}. This means that one may want to take $T(\mathtt{seg}_{15})$ to be invariant under swapping of the alleles. This will be done later in the paper through the modeling of segregation events -- but since all the elements of the tables given in section \ref{ssec:Main_example} are distinct under swapping of the alleles, we do not need to address these concerns here.

To define our sequence alignment $T$, we take $T(\mathtt{seg}_{15})$ to be the subset of $E_1^{\varepsilon}(\mathtt{seg}_{15})$ that exactly contains the pair of alleles given in section \ref{ssec:Main_example}, which we recall in the following table.
\[
\begin{array}{|rrr|}
\hline
\multicolumn{3}{|c|}{\cellcolor[gray]{0.8}T(\mathtt{seg}_{15})}\\
\hline
\mathtt{p}_{1} = \!\!
\begin{array}{l}
\mathtt{aAtCgCtAtTtcaaT}\\
\mathtt{gAtCcCcAaGtgacT}\\
\end{array}
&\mathtt{p}_{5} = \!\!
\begin{array}{l}
\mathtt{gAtCcCtAtGtcaaT}\\
\mathtt{aAcAgTcAtTgcatT}\\
\end{array}
&\mathtt{p}_{9} = \!\!
\begin{array}{l}
\mathtt{aTtAgTcAaGtgaaT}\\
\mathtt{aTtCgTtGaTgcatT}\\
\end{array}\\
\hline
\mathtt{p}_{2} = \!\!
\begin{array}{l}
\mathtt{gTtCcCtGaGtgatT} \\
\mathtt{aAtAcCcAtTgcctC} \\
\end{array}
&\mathtt{p}_{6} = \!\!
\begin{array}{l}
\mathtt{gAtCcCtGaTgcctC}\\
\mathtt{gTtCcCtGaGtgaaT}\\
\end{array}
&\mathtt{p}_{10} = \!\!
\begin{array}{l}
\mathtt{aTtCgTcAaTgcacT} \\
\mathtt{gAtCcCtGaTgcctC} \\
\end{array}\\
\hline
\mathtt{p}_{3} = \!\!
\begin{array}{l}
\mathtt{aAtCgCtGaGtgaaT}\\
\mathtt{aAtAcCcAaGtgatT}\\
\end{array}
&\mathtt{p}_{7} = \!\!
\begin{array}{l}
\mathtt{gTtCcCtAtGtcaaT}\\
\mathtt{gAtCcCcAtTgcatT} \\
\end{array}
&\mathtt{p}_{11} = \!\!
\begin{array}{l}
\mathtt{aAcAgTcAaTgcatT}\\
\mathtt{aTtAgTcAaGtgaaT}\\
\end{array}\\
\hline
\mathtt{p}_{4} = \!\!
\begin{array}{l}
\mathtt{gTtCcCcAtTgcctC}\\
\mathtt{gAtCcCtAtTtcacT}\\
\end{array}
&\mathtt{p}_{8} =\!\!
\begin{array}{l}
\mathtt{aAcAgTtGaTgcctC}\\
\mathtt{gAtCcCtGaGtgaaT}\\
\end{array}
&\mathtt{p}_{12} =\!\!
\begin{array}{l}
\mathtt{gAtCcCcAaGtgatT}\\
\mathtt{aTtAgTcAaGtgacT}\\
\end{array}\\
\hline
\end{array}
\]
By definition, the functor $T:B \to \mathbf{Set}$ is equipped with an natural monomorphism $\sigma:T \Rightarrow \mathbf{2}E_1^{\varepsilon}$ and the resulting triple $(\iota,T,\sigma)$ defines an obvious sequence alignment.
\end{example}

\section{Using pedigrads in idempotent commutative monoids to model haplotypes}\label{sec:Pedigrads_in_idempotent_commutative_monoids}
The goal of this section is to refine the concept of environment functors, seen in section \ref{sec:Examples_of_pedigrads_in_sets}, to a more sophisticated environment in which we can carry out efficient computations on haplogroups. Specifically, we use sequence alignments, and their environment functors, to construct pedigrads in the category of idempotent commutative monoids. The elements belonging to the images of these pedigrads will be used to model haplogroups. The main difference between environment functors (equipped with sequence alignments) and the pedigrads that result form them lies in the semantics used to interpret the elements belonging to their respective images. Specifically, to account for homologous recombination and segregation events, we will define our pedigrads in a localization-like fashion by forcing certain limit-dependent properties to hold through coequalizer constructions (see section \ref{ssec:recombination_monoids}). Because the employed coequalizers gather various types of congruences, the resulting pedigrads may not have straightforward presentation in idempotent commutative monoids. To resolve these presentation problems, we introduce, in section \ref{ssec:Recombination_schemes_and_pedigrads}, the concept of recombination schemes. In particular, we provide a series of theorems that allow us to determine the presentability of pedigrads associated with recombination schemes. To prepare for these theorems, we will recall, from section \ref{ssec:ICMonoids} to \ref{ssec:Coequalizers_of_ic_monoids}, concepts associated with idempotent commutative monoids, and we will formalize, from section \ref{ssec:biology_algebraic-operations} to section \ref{ssec:recombination_chromologies}, the concept of homologous recombination and segregation event.

\subsection{Idempotent commutative monoids}\label{ssec:ICMonoids}
Recall that a \emph{monoid} is a set $M$ equipped with a binary operation $\star:M \times M \to M$ and a particular element $e \in M$ which must satisfy the following two axioms:
\begin{itemize}
\item[1)] (Associativity) for every $x,y,z \in M$, the equation $(x \star y) \star z = x \star (y \star z)$ holds; 
\item[2)] (Identity) for every $x \in M$, the equations $x \star e = x = e \star x$ hold; 
\end{itemize}
Below, we give the axioms defining idempotent and commutative monoids.

\begin{definition}[Idempotent]
A monoid $(M,\star,e)$ is said to be \emph{idempotent} if for every $x \in M$, the equation $x \star x = x$ holds.
\end{definition}

\begin{definition}[Commutative]
A monoid $(M,\star,e)$ is said to be \emph{commutative} if for every $x,y \in M$, the equation $x \star y = y \star x$ holds.
\end{definition}

\begin{convention}[Notation]
We will usually follow conventions used in the literature and denote an abstract commutative monoid as $(M,+_M,0_M)$, or sometimes as $(M,+,0)$ for conciseness. The binary operation $+_M$ will be called the \emph{addition} while the particular element $0_M$ will be called the \emph{zero element}.
\end{convention}

\begin{convention}[Naming]
For convenience, we will shorten the name \emph{idempotent commutative monoid} to \emph{ic-monoid}.
\end{convention}

\begin{example}[Generator]\label{exa:icmonoid_B_2}
The set $B_2 = \{0,1\}$ has an ic-monoid structure whose associated addition operation is described by the following Cayley table.
\[
\begin{array}{c|c|c}
\cellcolor[gray]{0.7}+&\cellcolor[gray]{0.8}0&\cellcolor[gray]{0.8}1\\
\hline
\cellcolor[gray]{0.8}0&0&1\\
\hline
\cellcolor[gray]{0.8}1&1&1
\end{array}
\]
Note that this addition mimics the $\mathsf{OR}$-operation acting on the Boolean values $\mathsf{True}$ and $\mathsf{False}$ (which model the values 0 and 1, respectively). In the sequel, the resulting idempotent commutative monoid will be referred to as a triple $(B_2,+,0)$ and called the \emph{Boolean ic-monoid}.
\end{example}

\begin{example}[Power sets]\label{exa:power_set_icmonoids}
Recall that the \emph{power set} of a set $S$ is the set of subsets of $S$, which we will denote as $\mathcal{P}(S)$. This set defines an ic-monoid whose addition operation is given by the union operation and whose zero element is the empty subset. For instance, if we take $S$ to be the set of elements $\{\texttt{x},\texttt{y}\}$, then the addition operation of $\mathcal{P}(S)$ is described by the following Cayley table.
\[
\begin{array}{c|c|c|c|c}
\cellcolor[gray]{0.7}+&\cellcolor[gray]{0.8}\cmemptyset&\cellcolor[gray]{0.8}\{\texttt{x}\}&\cellcolor[gray]{0.8}\{\texttt{y}\}&\cellcolor[gray]{0.8}\{\texttt{x},\texttt{y}\}\\
\hline
\cellcolor[gray]{0.8}\cmemptyset&\cmemptyset&\{\texttt{x}\}&\{\texttt{y}\}&\{\texttt{x},\texttt{y}\}\\
\hline
\cellcolor[gray]{0.8}\{\texttt{x}\}&\{\texttt{x}\}&\{\texttt{x}\}&\{\texttt{x},\texttt{y}\}&\{\texttt{x},\texttt{y}\}\\
\hline
\cellcolor[gray]{0.8}\{\texttt{y}\}&\{\texttt{y}\}&\{\texttt{x},\texttt{y}\}&\{\texttt{y}\}&\{\texttt{x},\texttt{y}\}\\
\hline
\cellcolor[gray]{0.8}\{\texttt{x},\texttt{y}\}&\{\texttt{x},\texttt{y}\}&\{\texttt{x},\texttt{y}\}&\{\texttt{x},\texttt{y}\}&\{\texttt{x},\texttt{y}\}\\
\end{array}
\]
\end{example}

\subsection{Category of idempotent commutative monoids}\label{ssec:Category_Icm}
Let $(M,+_M,0_M)$ and $(N,+_N,0_N)$ be two monoids. Recall that a \emph{morphism of monoids} from $(M,+_M,0_M)$ to $(N,+_N,0_N)$ is a function $f:M \to N$ of sets satisfying the following two axioms:
\begin{itemize}
\item[1)] for every $x,y \in M$, the equation $f(x +_M y) = f(x) +_N f(y)$ holds;
\item[2)] the equation  $f(0_M) = 0_N$ holds.
\end{itemize}

\begin{convention}[Category structure]
We shall denote by $\mathbf{Icm}$ the category whose objects are idempotent commutative monoids and whose arrows are the morphisms of monoids between them.
\end{convention}

\begin{example}[Yoneda correspondence]\label{exa:Yoneda_correspondence_icmonoids}
Any morphism of the form $(B_2,+,0) \to (M,+_M,0_M)$ in $\mathbf{Icm}$ amounts to give two mappings rules of the form $1 \mapsto x$ and $0 \mapsto 0_M$. In other words, giving such a morphism is equivalent to picking out an element $x \in M$.
\end{example}

\begin{convention}[Submonoids]\label{conv:submonoids}
For every monoid $(M,+_M,0_M)$, we will say that a subset $U \subseteq M$ is a \emph{submonoid} of $M$ if the element $0_M$ belongs to $U$ and the restriction of the monoid operation $+_M:M \times M \to M$ lands in the subset $U$ (\emph{i.e.} induces a function $U \times U \to U$. Note that for such a submonoid $U$, the associated inclusion $U \hookrightarrow M$ induces a morphism $(U,+_M,0_M) \to (M,+_M,0_M)$ in $\mathbf{Icm}$.
\end{convention}

\begin{example}[Generators]\label{exa:submonoids}
For every ic-monoid $(M,+_M,0_M)$ and every element $x \in M$, the subset $\{0_M,x\}$ defines an obvious submonoid of $M$ (note that $x+_M x = x$) that is isomorphic to the ic-monoid $B_2$ through the mapping rules $0_M \mapsto 0$ and $x \mapsto 1$.
\end{example}

\begin{remark}[Products]\label{rem:products}
The category $\mathbf{Icm}$ has all products. Indeed, for every set $A$, the product of a collection $\{M_i\}_{i \in A}$ of idempotent commutative monoids $(M_i,+_{M_i},0_{M_i})$ is given by the product monoid $\prod_{i \in A} M_i$ whose elements are collections $(m_i)_{i \in A}$ of elements $m_i \in M_i$ for every $i \in A$. The addition operation is given by the equation $(m_i)_{i \in A} + (m_i')_{i \in A} = (m_i+_{M_i}m_i')_{i \in A}$ and the zero element is given by the collection $(0_{M_i})_{i \in A}$. It is straightforward to check that this structure defines a product object in $\mathbf{Icm}$.
\end{remark}

\begin{convention}[Notation]\label{def:S-fold-product:ic-monoid}
For every ic-monoid $(M,+,0)$ and every set $A$, we will usually denote the product $\prod_{i \in A} M$ as $M^A$ when the cardinal of the set $A$ is unknown. However, if we take $A$ to be the set $[2]=\{1,2\}$, then the notation $M^A$ will preferentially be replaced with the other usual notation $M \times M$.
\end{convention}

\begin{example}[Power set as a functor]
Every function $f:S\to T$ gives rise to a morphism $\mathcal{P}(f):\mathcal{P}(S) \to \mathcal{P}(T)$ in $\mathbf{Icm}$ that maps a subset $U \subseteq S$ to the set $f(U)$ of images of $f$ on $U$. Since for every function $g:T \to R$, the equation $g(f(U)) = g \circ f(U)$ holds, the mapping $f \mapsto \mathcal{P}(f)$ is natural and hence defines a functor $\mathcal{P}:\mathbf{Set} \to \mathbf{Icm}$.
\end{example}

\begin{example}[Products]\label{exa:products:powers-of-B_2}
Let $S$ be a set. The ic-monoid $(\mathcal{P}(S),\cup,\cmemptyset)$ is isomorphic to the ic-monoid $B_2^S$ (see Convention \ref{def:S-fold-product:ic-monoid}) through the following mapping.
\[
\varphi:
\left(
\begin{array}{llll}
\mathcal{P}(S)&\to&B_2^S&\\
U&\mapsto&(\delta_{i,U})_{i \in S}&\\
\end{array}
\right)
\quad
\textrm{ where }
\quad
\delta_{i,U} =
\left\{
\begin{array}{ll}
1&\textrm{if }i \in U\\
0&\textrm{if }i \notin U\\
\end{array}
\right.
\]
It is straightforward to show that the function $\varphi$ defines a morphism in $\mathbf{Icm}$. Indeed, the equation $\delta_{i,\cmemptyset} = 0$ holds for every element $i \in S$, and we can show that an element $i \in S$ belongs to a union $U \cup V$ if, and only if, $i \in U$ or $i \in V$, which implies that the equation $\delta_{i ,U \cup V} = \delta_{i ,U } + \delta_{i ,V}$ holds (see Example \ref{exa:icmonoid_B_2}). The function $\varphi$ also defines an obvious bijection whose inverse is shown below.
\[
\varphi^{-1}:
\left(
\begin{array}{llll}
B_2^S&\to&\mathcal{P}(S)&\\
(m_i)_{i \in S}&\mapsto&\{i \in S~|~m_i = 1\}&\\
\end{array}
\right)
\]
The function $\varphi^{-1}$ induces a morphism in $\mathbf{Icm}$ because the equation $x+y = 1$ in $B_2$ is equivalent to the logical statement: $x = 1$ \emph{or} $y = 1$. Note that for every function $f:S \to T$, the composition $\varphi_T \circ \mathcal{P}(f) \circ \varphi_S^{-1}$ provides a morphism $B_2^S \to B_2^T$ in $\mathbf{Icm}$ that maps a tuple $(x_i)_{i \in S}$ to the tuple $(\sum_{i \in f^{-1}(j)} x_i)_{j \in T}$.
\end{example}

\begin{example}[Addition as a morphism]\label{exa:Addition_as_morphism}
Let $(M,+,0)$ be an object in $\mathbf{Icm}$. The function $\mathsf{sum}:M \times M \to M$ that maps a pair $(x,y)$ to the element $x+y$ in $M$ defines a morphism in $\mathbf{Icm}$. Indeed, we can verify that the identity $\mathsf{sum}(0,0) = 0$ holds and that the following equations are satisfied for every $(x,y) \in M$ and $(x',y') \in M$.
\begin{align*}
\mathsf{sum}((x,y)+(x',y')) &= \mathsf{sum}(x+x',y+y')\\
& = x+x'+y+y' \\
&= x+y + x'+y'\\
& = \mathsf{sum}(x,y) + \mathsf{sum}(x',y')
\end{align*}
Furthermore, we can show that the morphism $\mathsf{sum}:M \times M \to M$ is natural in $M$. Indeed, for every morphism $f:(M,+_{M},0_M) \to (N,+_{N},0_N)$ in $\mathbf{Icm}$, the equation given below, on the left, hold for every $(x,y) \in M$, which is equivalent to saying that the diagram given on the right commutes.
\[
\begin{array}{ll}
f(x) + f(y) = f(x+y)
\end{array}
\quad\quad\quad\quad\quad
\begin{array}{l}
\xymatrix{
M \times M \ar[r]^-{\mathsf{sum}}\ar[d]_{f \times f} & M \ar[d]^{f}\\
N \times N \ar[r]_-{\mathsf{sum}} & N
}
\end{array}
\]
\end{example}

\subsection{Universal construction of ic-monoids}\label{ssec:ICMonoids_universal_construction}
There is an adjunction of the form (\ref{adjunction_set_icm}) between sets and ic-monoids, for which the left adjoint $F$ maps a set $S$ to the free ic-monoid generated over $S$ (see the discussion below) while the right adjoint $U$ is the obvious structural forgetful functor.
\begin{equation}\label{adjunction_set_icm}
\xymatrix{
\mathbf{Set}\ar@<+1.2ex>[r]^-F\ar@{<-}@<-1.2ex>[r]_-U\ar@{}[r]|-{\bot}&\mathbf{Icm}
}
\end{equation}

For every set $S$, the ic-monoid $F(S)$ correspond to the set of finite subsets of $S$ for which the addition is given by the union operation and the zero element is given by the empty set. For every function $f:S \to T$, the image $F(f)$ maps every finite subset $U \subseteq S$ to the set $f(U) \subseteq T$ of images of $f$ on $U$.

\begin{example}[Power sets \emph{versus} free ic-monoids]\label{exa:free_ic_monoid}
For every set $S$, we have an inclusion $F(S) \subseteq \mathcal{P}(S)$ that induces a morphism in $\mathbf{Icm}$. If we take $S$ to be a finite set, then the image $F(S)$ is equal to the power set of $S$. For instance, the free idempotent commutative monoid $F(\{\texttt{x},\texttt{y}\})$ correspond to the monoid described in Example \ref{exa:power_set_icmonoids}.
\end{example}

For every set $S$, the unit $\eta_S:S \to UF(S)$ of adjunction (\ref{adjunction_set_icm}) sends every element $x$ in $S$ to the singleton $\{x\}$ in $F(S)$ while, for every ic-monoid $(M,+,0)$, the counit $\mu_S:FU(M,+,0) \to (M,+,0)$ of adjunction (\ref{adjunction_set_icm}) sends the empty set to 0 and sends every non-empty finite set $\{a_1,a_2,\dots,a_n\} \subseteq M$ to the following finite sum in $(M,+,0)$.
\begin{equation}\label{eq:ic_monoids_adjunction_representative}
\sum_{i=1}^n a_i = a_1+a_2 + \dots + a_n
\end{equation}

\begin{remark}[Representation]
Let $S$ be a set. If one identifies every singleton $\{a\}$ in $F(S)$ with the element $a$ in $S$, then a finite subset $\{a_1,a_2,\dots,a_n\}$ of $S$ can be represented as a finite sum of its elements, as shown in (\ref{eq:ic_monoids_adjunction_representative}).
\end{remark}

\subsection{Reminder on monomorphisms and epimorphisms}\label{ssec:Reminder_monomorphisms_epimorphisms}
In this section, we recall general facts on how to detect or construct epimorphisms and monomorphisms. These properties will be particularly useful to prove one of our main results (in Theorem \ref{theo:representable_pedigrad_E_b_varepsilon}). First, recall that a \emph{monomorphism} (in any category) is an arrow $m:A \to B$ such that for every pair of parallel arrows $f,g:X \rightrightarrows A$ for which the equation $m \circ f = m \circ g$ holds, the two arrows $f$ and $g$ must be equal.

\begin{proposition}\label{prop:characterization_mono}
If a category $\mathcal{C}$ has pullbacks, then an arrow $m:A \to B$ in $\mathcal{C}$ is a monomorphism if the pullback $p_1,p_2:P \rightrightarrows A$ of the arrow $m$ along itself is such that $p_1 = p_2$.
\end{proposition}
\begin{proof}
Let $f,g:X \rightrightarrows A$ be a pair of parallel arrows for which the equation $m \circ f = m \circ g$ holds. Let us show that $f = g$. First, the universality of the pullback $P$ gives an arrow $h:X \to P$ for which the identities $f = p_1 \circ h$ and $g = p_2 \circ h$ hold. Then, because $p_1 = p_2$, we have $f= g$ and the statement follows.
\end{proof}

\begin{definition}[Relative monomorphisms]\label{def:relative-monomorphisms}
Let $I$ be an object in a category $\mathcal{C}$. A morphism $m:X \to Y$ in $\mathcal{C}$ will be said to be an \emph{$I$-monomorphism} if for every pair of arrows $f,g:I \to X$ in $\mathcal{C}$ for which the equation $m \circ f = m \circ g$ holds, the two arrows $f$ and $g$ are equal in $\mathcal{C}$.
\end{definition}

\begin{proposition}\label{prop:characterization_B2_mono}
In the category $\mathbf{Icm}$, an arrow is a monomorphism if and only if it is a $B_2$-monomorphism (see Example \ref{exa:icmonoid_B_2}).
\end{proposition}
\begin{proof}
By Definition \ref{def:relative-monomorphisms}, monomorphisms in $\mathbf{Icm}$ must be $B_2$-monomorphisms. Conversely, Example \ref{exa:Yoneda_correspondence_icmonoids} implies that $B_2$-monomorphisms of the form $f:(M,+_M,0_M) \to (N,+_N,0_N)$ are given by injective functions $f:M \to N$. Because injections are monomorphisms in $\mathbf{Set}$, we can use a diagrammatic reasoning to deduce that that injective morphisms in $\mathbf{Icm}$ are monomorphisms in $\mathbf{Icm}$.
\end{proof}

\begin{example}[Monomorphisms and encoding]\label{exa:monomorphism-encoding}
For every $S$ set, Remark \ref{exa:free_ic_monoid} provides a morphism $F(S) \to \mathcal{P}(S)$ in $\mathbf{Icm}$ and Example \ref{exa:products:powers-of-B_2} provides an isomorphism $\mathcal{P}(S) \cong B_2^S$ in $\mathbf{Icm}$. Since the morphism $F(S) \to \mathcal{P}(S)$ is defined by an injective function, it also defines a $B_2$-monomorphism (see Example \ref{exa:Yoneda_correspondence_icmonoids}) and hence a monomorphism (by Proposition \ref{prop:characterization_B2_mono}). The composition of the monomorphism $F(S) \to \mathcal{P}(S)$ with the isomorphism $\mathcal{P}(S) \cong B_2^S$ provides a monomorphism $F(S) \to B_2^S$ in $\mathbf{Icm}$ satisfying the following mapping rule.
\[
\varphi_S:
\left(
\begin{array}{llll}
F(S)&\to&B_2^S&\\
U&\mapsto&(\delta_{i,U})_{i \in S}&\\
\end{array}
\right)
\quad
\textrm{ where }
\quad
\delta_{i,U} =
\left\{
\begin{array}{ll}
1&\textrm{if }i \in U\\
0&\textrm{if }i \notin U\\
\end{array}
\right.
\]
It follows from Remark \ref{exa:free_ic_monoid} that if $S$ is finite, then the monomorphism $\varphi_S:F(S) \to B_2^S$ corresponds to the isomorphism $F(S) \cong B_2^S$ of Example \ref{exa:products:powers-of-B_2}. Note that, for every $S$ set, the monomorphism $\varphi_S:F(S) \to B_2^S$ gives an easy way to encode a free ic-monoid in a computer as a piece of binary code (see Example \ref{exa:icmonoid_B_2}). More generally, we can use this morphism to \emph{computerize} any  ic-monoid as its free presentation. This means that the ic-monoid is encoded up to its non-free relations. Then, these relations can be recovered as coequalizer constructions on the free presentation (see section \ref{ssec:Coequalizers_of_ic_monoids}).
\end{example}

Let us now recall the definition for epimorphisms. Specifically, an \emph{epimorphism} (in any category) is an arrow $e:B \to A$ such that for every pair of parallel arrows $f,g:A \rightrightarrows X$ for which the equation $f \circ e = g \circ e$ holds, the two arrows $f$ and $g$ must be equal.

\begin{definition}[Orthogonality]
A morphism $f:X \to Y$ in a category $\mathcal{C}$ will be said to be \emph{orthogonal} to an object $I$ in $\mathcal{C}$ if for every arrow $i:I \to Y$ in $\mathcal{C}$, there exists a dashed arrow (as shown below) making the following diagram commute in $\mathcal{C}$. The dashed arrow will be referred to as a \emph{lift} of the arrow $e$ \emph{along} the arrow $i$.
\[
\xymatrix@R-5pt{
&X\ar[d]^{e}\\
I\ar[r]_{i}\ar@{-->}[ru]&Y
}
\]
\end{definition}

\begin{proposition}\label{prop:characterization_epi}
Every morphism in $\mathbf{Icm}$ that is orthogonal with respect to the Boolean ic-monoid  $B_2$ (Example \ref{exa:icmonoid_B_2}) is an epimorphism. 
\end{proposition}
\begin{proof}
We can use Example \ref{exa:Yoneda_correspondence_icmonoids} to see that if an arrow $e:X \to Y$ is orthogonal with respect to the Boolean ic-monoid $(B_2,+,0)$, then it is surjective: for every $y \in Y$, there exists $x \in X$ for which the identity $e(x) = y$ holds. Because surjections are epimorphisms in $\mathbf{Set}$, we can use a diagrammatic reasoning to deduce that surjective morphisms in $\mathbf{Icm}$ are epimorphisms in $\mathbf{Icm}$.
\end{proof}

\subsection{Coequalizers of ic-monoids}\label{ssec:Coequalizers_of_ic_monoids}
An introduction (or refresher) on coequalizers in $\mathbf{Set}$ can be found in \cite[def. 7.2.1.4]{SpivakBook} -- also see \cite{MacLane}. Here, we give a definition for coequalizers in $\mathbf{Icm}$. The main idea behind forming coequalizers in $\mathbf{Icm}$ is to force certain linear equations to hold in a given ic-monoid.

\begin{definition}[Coequalizers]
Let $f,g:X \rightrightarrows Y$ be a pair of morphisms in some category $\mathcal{C}$. A \emph{coequalizer} for the pair $(f,g)$ is an arrow $e:Y \to Q$ in $\mathcal{C}$ such that
\begin{itemize}
\item[-] the identity $e \circ f = e \circ g$ holds;
\item[-] for every morphism $h:X \to Q'$ for which the identity holds $h \circ f = h \circ g$, there exists a unique arrow $h':Q \to Q'$ making the following diagram commute.
\[
\xymatrix{
Y\ar[d]_{e}\ar[r]^{h'} & Q'\\
Q \ar[ru]_{h'}& 
}
\]
\end{itemize}
\end{definition}

\begin{convention}
We shall use conventional terminology and say that the following diagram commutes whenever the identity $e \circ f = e \circ g$ holds.
\[
\xymatrix{
X\ar@<-1ex>[r]_{g}\ar@<+1ex>[r]^{f}&Y\ar[r]^{e}&Q
}
\]
\end{convention}

It is well-known \cite{Golan,Pareigis} that the category $\mathbf{Icm}$ is isomorphic to the category of \emph{semimodules over the Boolean semiring}. While there does not seem to be any specific references describing coequalizers of idempotent commutative monoids, we can find in the literature an explicit description of coequalizers for Boolean semimodules \cite{Pareigis}. The remainder of this section focuses on giving an explicit description of coequalizers in $\mathbf{Icm}$ by translating the constructions given thereof.

\begin{definition}[Quotient]\label{def:quotient_icm}
For every pair of morphisms $f,g:(X,+_X,0_X) \rightrightarrows (Y,+_Y,0_Y)$ in $\mathbf{Icm}$, define the binary relation $R(f,g)$ on $Y$ containing the pair of elements $(m,m')$ in $Y\times Y$ such that for every ic-monoid $(Z,+_Z,0_Z)$ and morphism $h:(Y,+_Y,0_Y) \to (Z,+_Z,0_Z)$ in $\mathbf{Icm}$ for which the equation $h \circ f = h \circ g$ holds, the relation $h(m) = h(m')$ holds too.
\end{definition}

\begin{remark}[Equivalence relation]
For every pair of arrows $f,g:(X,+_X,0_X) \rightrightarrows (Y,+_Y,0_Y)$ in $\mathbf{Icm}$, it is direct to verify that the binary relation $R(f,g)$ on $Y$ is an equivalence relation. Below, we use the notation $[m]_{f,g}$ to denote the equivalence class of an element $m$ in $Y$ for this relation.
\end{remark}

\begin{proposition}\label{prop:quotient_map_icm}
Given the notations used in Definition \ref{def:quotient_icm}, the quotient set $Y/R(f,g)$ defines an ic-monoid for the addition $[m]_{f,g} + [m']_{f,g} := [m+_Ym']_{f,g}$ and the zero element $[0_Y]_{f,g}$. The quotient map $e:Y \to Y/R(f,g)$ then defines a morphism of monoids.
\end{proposition}
\begin{proof}
To show that the quotient $Y/R(f,g)$ defines an ic-monoids, we only need to show that the addition $[m]_{f,g} + [m']_{f,g} := [m+_Ym']_{f,g}$ is well-defined -- all the other axioms directly follow from those of $(Y,+_Y,0_Y)$. In this respect, let $(m_1,m'_1)$ and $(m_2,m_2')$ be two pairs in $R(f,g)$. First, for every morphism $h:(Y,+_Y,0_Y) \to (Z,+_Z,0_Z)$ in $\mathbf{Icm}$ for which the equation $h \circ f = h \circ g$ holds, we have the relations $h(m_1) = h(m'_1)$ and $h(m_2) = h(m'_2)$. This implies that the following equations hold.
\[
h(m_1+_Ym_2) = h(m_1)+_Zh(m_2)= h(m_1')+_Zh(m_2') = h(m_1'+_Ym_2')
\]
The previous equations show that the pair $(m_1+_Ym_2,m'_1+_Ym'_2)$ is in $R(f,g)$. In other words, the structure $(Y/R(f,g),+,[0_Y]_{f,g})$ is well-defined as an ic-monoid. Also, note that the definition of the ic-monoid structure on $Y/R(f,g)$ makes it straightforward to verify that the quotient map $e:Y \to Y/R(f,g)$ defines a morphism of monoids.
\end{proof}

\begin{proposition}\label{prop:coequalizer:existence:Icm}
Let $f,g:X \rightrightarrows Y$ be a pair of morphisms in $\mathbf{Icm}$. The morphism  $e:Y \to Y/R(f,g)$ defines the coequalizer of $(f,g)$ in $\mathbf{Icm}$.
\end{proposition}
\begin{proof}
The existence property of the universal property directly follows from the way the relation $R(f,g)$ is defined (see Definition \ref{def:quotient_icm}) while its uniqueness property follows from the surjectiveness of the function $e:Y \to Y/R(f,g)$ as a quotient map of sets (Proposition \ref{prop:quotient_map_icm}).
\end{proof}

\begin{remark}[Orthogonality]\label{rem:characterization_epi_coequalizer_maps}
 Let $f,g:X \rightrightarrows Y$ be a pair of morphisms in $\mathbf{Icm}$. Since the quotient map of a set by an equivalence relation is surjective, the coequalizer $e:Y \to Y/R(f,g)$ of $(f,g)$ (see Proposition \ref{prop:quotient_map_icm}) is a surjection of sets. It then follows from Example \ref{exa:Yoneda_correspondence_icmonoids} that $e:Y \to Y/R(f,g)$ is orthogonal to the ic-monoid $(B_2,+,0)$.
\end{remark}

\begin{proposition}\label{prop:pullback:coequalizer:mono}
For every morphism $f:Y \to X$ in $\mathbf{Icm}$, denote by $Y_f$ the coequalizer of the pullback pair of $f$ along itself in $\mathbf{Icm}$ (see the leftmost and middle diagrams below). The arrow $f$ coequalizes the pullback pair and the resulting universal arrow $f':Y_f \to X$ is a monomorphism.
\[
\begin{array}{l}
\xymatrix{
R \ar[d]_{u_2} \ar[r]^-{u_1}  \ar@{}[rd]|<<<{\rotatebox[origin=c]{90}{\huge{\textrm{$\llcorner$}}}} &Y\ar[d]^{f}\\
Y\ar[r]_-{f}&X
}
\end{array}
\quad\textrm{and}\quad
\begin{array}{l}
\xymatrix{
R  \ar@<+1.2ex>[r]^-{u_1}\ar@<-1.2ex>[r]_-{u_1}& Y \ar@{-->}[r]^-{q}&Y_f
}
\end{array}
\quad\Rightarrow\quad
\begin{array}{l}
\xymatrix@C+10pt@R-25pt{
&&X\\
R  \ar@<+1.2ex>[r]^-{u_1}\ar@<-1.2ex>[r]_-{u_1}& Y \ar[rd]_-{q}\ar[ru]^-{f}&\\
&&Y_f\ar@{-->}[uu]_{f'}\\
}
\end{array}
\]
\end{proposition}
\begin{proof}
To show that the arrow $f':Y_f \to X$ is a monomorphism in $\mathbf{Icm}$, we first want to show that it is a $B_2$-monomorphism and then use Proposition \ref{prop:characterization_B2_mono}. Let $a,b:B_2 \rightrightarrows Y_f$ be two arrows in $\mathbf{Icm}$ for which the equation $f' \circ a = f' \circ b$ holds. First, by Remark \ref{rem:characterization_epi_coequalizer_maps}, the coequalizer $q:Y \to Y_f$ is orthogonal to $B_2$. Thus, the arrows $a$ and $b$ admits lifts $a':B_2 \to Y$ and $b':B_2 \to Y$ along the coequalizer arrow $q:Y \to Y_f$, respectively. By assumption on the arrows $a$ and $b$, the following series of identities holds (note that the definition of $f'$ gives us the equation $f = f' \circ q$).
\[
f \circ a' = f' \circ q \circ a' = f' \circ a = f' \circ b = f' \circ q \circ b' = f \circ b'
\]
Since the ic-monoid $R$ (see the statement) is the pullback of $f$ along itself, the previous series of identities implies that there is a canonical map $h:B_2 \to R$ making the following diagrams commute.
\[
\begin{array}{l}
\xymatrix{
B_2\ar[d]_{h}\ar[r]^{a'}&Y\\
R \ar[ru]_{u_{1}}&
}
\end{array}
\quad\quad\quad\quad\quad\quad\quad
\begin{array}{l}
\xymatrix{
B_2\ar[d]_{h}\ar[r]^{b'}&Y\\
R \ar[ru]_{u_{2}}&
}
\end{array}
\]
Post-composing the previous two diagrams with the coequalizer $q:Y \to Y_f$ of $u_{1}$ and $u_{2}$ should allow the reader to see that the following identities hold.
\[
a = q \circ a' = q \circ u_{1} \circ h = q \circ u_{2} \circ h = q \circ b' = b
\]
This shows that $f':Y_f \to X$ is a $B_2$-monomorphism and hence a monomorphism.
\end{proof}

\subsection{Encoding biological mechanisms as algebraic operations}\label{ssec:biology_algebraic-operations}
In previous sections, we have made the point that the techniques proposed in this paper will extend current genotype-based analysis methods to haplotype-based analysis methods. The main difference between genotypes and haplotypes lies in their relationship to groups of individuals: genotypes are defined at the individual level, while haplotypes are defined at the population level. Specifically, a \emph{genotype} refers to a set of pairs of nucleotides describing combinations of alleles possessed by an individual at specific loci. Meanwhile, a \emph{haplotype} is a set of alleles that tend to be inherited together. In other words, genotypes are used to characterized individuals while haplotypes are used to characterized populations. For the latter, the term \emph{haplogroup} is usually used to refer to a group of individuals sharing the same haplotype.
\[
\xymatrix@C+80pt{
\fbox{a haplogroup} \ar[r]^{\textrm{has a}} & \fbox{a haplotype} \ar[r]^{\textrm{which is of a set of}}& \fbox{genotypes}
}
\]
The relationships established above will motivate the formalism introduced in the present section. Specifically, as much as a population refers to a group of individuals, we will characterize haplogroups in terms of groups of genotypes.

The diagram shown below gives an example of a set of three haplogroups for three different haplotypes (for convenience, we use a single sequence of DNA, as opposed to pairs of alleles -- but the principle stays the same). In this diagram, the bottommost haplogroups define what is commonly referred to as \emph{sub-haplogroups} for the topmost haplogroup.
\begin{equation}\label{eq:diagram:haplogroups:haplotypes:illustration}
\xymatrix@R-10pt@C-75pt{
&\fbox{$\begin{array}{|c|}\hline\cellcolor[gray]{0.8}\textrm{haplotype }\\\cellcolor[gray]{0.8}\mathtt{ACGA}\texttt{-}\mathtt{TAG}\texttt{-}\\\hline\end{array}$}\quad\quad\quad\quad\quad\,~&
\fbox{$\begin{array}{c}
\cellcolor[gray]{0.8}\textrm{haplogroup X}\\
\dots\mathtt{ACGA}\underline{\mathtt{A}}\mathtt{TAG}\underline{\mathtt{T}}\dots\\
\dots\mathtt{ACGA}\underline{\mathtt{A}}\mathtt{TAG}\underline{\mathtt{C}}\dots\\
\dots\mathtt{ACGA}\underline{\mathtt{C}}\mathtt{TAG}\underline{\mathtt{T}}\dots\\
\dots\mathtt{ACGA}\underline{\mathtt{C}}\mathtt{TAG}\underline{\mathtt{C}}\dots\\
\end{array}$}
\ar[ld]\ar[rd]&&\\\fbox{$\begin{array}{|c|}\hline\cellcolor[gray]{0.8}\textrm{haplotype }\\\cellcolor[gray]{0.8}\mathtt{ACGAATAG}\texttt{-}\\\hline\end{array}$}\quad\quad\quad\quad\quad\,~&
\fbox{$\begin{array}{c}
\cellcolor[gray]{0.8}\textrm{haplogroup  XA}\\
\dots\mathtt{ACGAATAG}\underline{\mathtt{T}}\dots\\
\dots\mathtt{ACGAATAG}\underline{\mathtt{C}}\dots\\
\end{array}$}
&&
\fbox{$\begin{array}{c}
\cellcolor[gray]{0.8}\textrm{haplogroup XC}\\
\dots\mathtt{ACGACTAG}\underline{\mathtt{T}}\dots\\
\dots\mathtt{ACGACTAG}\underline{\mathtt{C}}\dots\\
\end{array}$}
&\,~\quad\quad\quad\quad\quad\fbox{$\begin{array}{|c|}\hline\cellcolor[gray]{0.8}\textrm{haplotype }\\\cellcolor[gray]{0.8}\mathtt{ACGACTAG}\texttt{-}\\\hline\end{array}$}\\
}
\end{equation}
In this paper, we formalize the concept of haplotypes, haplogroups and genotypes as follows.

\begin{definition}[Genotypes, haplotypes and haplogroups]\label{def:genotype_haplotype_haplogroup}
Let $(\Omega,\preceq)$ be a pre-ordered set and $X:\mathbf{Seg}(\Omega) \to \mathbf{Set}$ be a functor. For every cone $\rho:\Delta_A(\tau) \Rightarrow \theta$ (section \ref{ssec:cones}) in $\mathbf{Seg}(\Omega)$, the image of $\rho$ via the functor $FX:\mathbf{Seg}(\Omega) \to \mathbf{Icm}$ defines a cone $\Delta_A(FX(\tau)) \Rightarrow FX\theta$ in $\mathbf{Icm}$. The limit adjoint of this cone is a canonical arrow of the following form in $\mathbf{Icm}$. 
\begin{equation}\label{eq:limit_adjoint:def_haplotypes_hapologroups}
FX(\tau) \to \mathsf{lim}_AFX \circ \theta
\end{equation}
In the sequel, we shall call
\begin{itemize}
\item[1)] a \emph{$\rho$-genotype} an element in the set $X(\tau)$;
\item[2)] a \emph{$\rho$-haplotype} an element in the ic-monoid $\mathsf{lim}_AFX \circ \theta$;
\item[3)] a \emph{$\rho$-haplogroup} for a certain $\rho$-haplotype $x$ an element in the ic-monoid $FX(\tau)$ whose image via (\ref{eq:limit_adjoint:def_haplotypes_hapologroups}) is $x$;
\end{itemize}
\end{definition}

For instance, a cone that could be used to characterize ``haplogroup X'' in diagram (\ref{eq:diagram:haplogroups:haplotypes:illustration}), given above, is described by the following pair of obvious arrows in $\mathbf{Seg}(\{0,1\}|8)$.
\[
\begin{array}{c|c|c}
\cellcolor[gray]{0.8}\textrm{segment indexing}&\cellcolor[gray]{0.8}\textrm{cone}&\cellcolor[gray]{0.8}\textrm{segments indexing}\\
\cellcolor[gray]{0.8}\textrm{haplogroup X}&\cellcolor[gray]{0.8}\textrm{arrows}&\cellcolor[gray]{0.8}\textrm{the haplotype}\\
\hline
&&\xymatrix@C-30pt{(\bullet&\bullet&\bullet&\bullet)&(\circ)&(\circ&\circ&\circ)&(\circ)}\\
&\rotatebox[origin=c]{25}{$\longrightarrow$}&\\
\xymatrix@C-30pt{(\bullet&\bullet&\bullet&\bullet)&(\bullet)&(\bullet&\bullet&\bullet)&(\bullet)}&&\\
&\rotatebox[origin=c]{-25}{$\longrightarrow$}&\\
&& \xymatrix@C-30pt{(\circ&\circ&\circ&\circ)&(\circ)&(\bullet&\bullet&\bullet)&(\circ)}\\
\end{array}
\]
We leave the determination of the cones characterizing haplogroups XA and XC as an exercise to reader. More intuition is given in the following example.

\begin{example}[Genotypes, haplotypes and haplogroups]\label{exa:genotype_haplotype_haplogroup}
Let $(\Omega,\preceq)$ be the Boolean pre-ordered set $\{0 \leq 1\}$ and let $(E,\varepsilon)$ be our usual pointed set $\{\mathtt{A},\mathtt{C},\mathtt{G},\mathtt{T},\varepsilon\}$. Take $\rho:\Delta_{[3]}(\tau) \Rightarrow \theta$ to be the following cone in $\mathbf{Seg}(\Omega)$.
\[
\begin{array}{cccc}
&&\xymatrix@C-30pt{
(\bullet&\bullet&\bullet&\bullet&\bullet&\bullet)&(\circ&\circ&\circ&\circ&\circ&\circ)&(\circ&\circ&\circ)
}&\\
&\rotatebox[origin=c]{40}{$\longrightarrow$}&&\\
\xymatrix@C-30pt{(\bullet&\bullet&\bullet&\bullet&\bullet&\bullet)&(\bullet&\bullet&\bullet&\bullet&\bullet&\bullet)&(\bullet&\bullet&\bullet)} &\longrightarrow & \xymatrix@C-30pt{
(\circ&\circ&\circ&\circ&\circ&\circ)&(\bullet&\bullet&\bullet&\bullet&\bullet&\bullet)&(\circ&\circ&\circ)}&\\
&\rotatebox[origin=c]{-40}{$\longrightarrow$}&&\\
 && \xymatrix@C-30pt{
(\circ&\circ&\circ&\circ&\circ&\circ)&(\circ&\circ&\circ&\circ&\circ&\circ)&(\bullet&\bullet&\bullet)
}&\\
\end{array}
\]
The elements of $F\mathbf{2}E_1^{\varepsilon}(\tau)$ given in the left-hand side column of the following table are $\rho$-genotypes\footnote{As usual, parentheses are added to the DNA sequences for clarity} that we regard as their corresponding $\rho$-haplogroups. Their images through the canonical arrow $F\mathbf{2}E_1^{\varepsilon}(\tau) \to \prod_{i \in [3]} F\mathbf{2}E_1^{\varepsilon}\theta(i)$ are given in the right-hand side column -- they correspond to their $\rho$-haplotypes.
\[
\begin{array}{c|c}
\multicolumn{1}{c|}{\cellcolor[gray]{0.8}\textrm{in }F\mathbf{2}E_1^{\varepsilon}(\tau)}&\multicolumn{1}{c}{\cellcolor[gray]{0.8}\textrm{in }\prod_{i \in [3]} F\mathbf{2}E_1^{\varepsilon}\theta(i)}\\
\hline
\mathtt{p}_{5}=\begin{array}{l}
\mathtt{(gAtCcC)(tAtGtc)(aaT)}\\
\mathtt{(aAcAgT)(cAtTgc)(atT)}\\
\end{array}&
\left(
\begin{array}{l}
\mathtt{(gAtCcC)}\\
\mathtt{(aAcAgT)}\\
\end{array}
,
\begin{array}{l}
\mathtt{(tAtGtc)}\\
\mathtt{(cAtTgc)}\\
\end{array}
,
\begin{array}{l}
\mathtt{(aaT)}\\
\mathtt{(atT)}\\
\end{array}
\right)
\\
\hline
\mathtt{p}_{6}=
\begin{array}{l}
\mathtt{(gAtCcC)(tGaTgc)(ctC)}\\
\mathtt{(gTtCcC)(tGaGtg)(aaT)}\\
\end{array}&
\left(
\begin{array}{l}
\mathtt{(gAtCgC)}\\
\mathtt{(gTtCcC)}\\
\end{array}
,
\begin{array}{l}
\mathtt{(tGaTgc)}\\
\mathtt{(tGaGtg)}\\
\end{array}
,
\begin{array}{l}
\mathtt{(ctC)}\\
\mathtt{(aaT)}\\
\end{array}
\right)\\
\hline
\mathtt{p}_{7}=
\begin{array}{l}
\mathtt{(gTtCcC)(tAtGtc)(aaT)}\\
\mathtt{(gAtCcC)(cAtTgc)(atT)} \\
\end{array}&
\left(
\begin{array}{l}
\mathtt{(gTtCcC)}\\
\mathtt{(gAtCcC)}\\
\end{array}
,
\begin{array}{l}
\mathtt{(tAtGtc)}\\
\mathtt{(cAtTgc)}\\
\end{array}
,
\begin{array}{l}
\mathtt{(aaT)}\\
\mathtt{(atT)}\\
\end{array}
\right)\\
\hline
\mathtt{p}_{8}=
\begin{array}{l}
\mathtt{(aAcAgT)(tGaTgc)(ctC)}\\
\mathtt{(gAtCcC)(tGaGtg)(aaT)}\\
\end{array}&
\left(
\begin{array}{l}
\mathtt{(aAcAgT)}\\
\mathtt{(gAtCcC)}\\
\end{array}
,
\begin{array}{l}
\mathtt{(tGaTgc)}\\
\mathtt{(tGaGtg)}\\
\end{array}
,
\begin{array}{l}
\mathtt{(ctC)}\\
\mathtt{(aaT)}\\
\end{array}
\right)\\
\end{array}
\]
Since the canonical arrow $F\mathbf{2}E_1^{\varepsilon}(\tau) \to \prod_{i \in [3]} F\mathbf{2}E_1^{\varepsilon}\theta(i)$ is a morphism of ic-monoids, we deduce, from the previous table, that the element $\mathtt{p}_{5}+\mathtt{p}_{7}$ is a $\rho$-haplogroup of the following $\rho$-haplotype. 
\[
\left(
\begin{array}{l}
\mathtt{(gAtCcC)}\\
\mathtt{(aAcAgT)}\\
\end{array}
+
\begin{array}{l}
\mathtt{(gTtCcC)}\\
\mathtt{(gAtCcC)}\\
\end{array}
,
\begin{array}{l}
\mathtt{(tAtGtc)}\\
\mathtt{(cAtTgc)}\\
\end{array}
,
\begin{array}{l}
\mathtt{(aaT)}\\
\mathtt{(atT)}\\
\end{array}
\right)
\]
Similarly, the sum $\mathtt{p}_{6}+\mathtt{p}_{8}$ is a $\rho$-haplogroup for the following $\rho$-haplotype.
\[
\left(
\begin{array}{l}
\mathtt{(gAtCcC)}\\
\mathtt{(gTtCcC)}\\
\end{array}
+
\begin{array}{l}
\mathtt{(aAcAgT)}\\
\mathtt{(gAtCcC)}\\
\end{array}
,
\begin{array}{l}
\mathtt{(tGaTgc)}\\
\mathtt{(tGaGtg)}\\
\end{array}
,
\begin{array}{l}
\mathtt{(ctC)}\\
\mathtt{(aaT)}\\
\end{array}
\right)
\]
As a result, the element $\mathtt{p}_{5}+\mathtt{p}_{7}+ \mathtt{p}_{6}+\mathtt{p}_{8}$ is a $\rho$-haplogroup of the $\rho$-haplotype shown below -- this haplotype can be regarded as a sup-haplotype of the two haplotypes shown earlier.
\begin{equation}\label{eq:haplotypes-haplogroups:exa}
\left(
\begin{array}{c}
\begin{array}{l}
\mathtt{(gAtCcC)}\\
\mathtt{(aAcAgT)}\\
\end{array}
+
\begin{array}{l}
\mathtt{(gTtCcC)}\\
\mathtt{(gAtCcC)}\\
\end{array}\\
+ \begin{array}{l}
\mathtt{(gAtCcC)}\\
\mathtt{(gTtCcC)}\\
\end{array}
+
\begin{array}{l}
\mathtt{(aAcAgT)}\\
\mathtt{(gAtCcC)}\\
\end{array}
\end{array}
,
\begin{array}{l}
\mathtt{(tAtGtc)}\\
\mathtt{(cAtTgc)}\\
\end{array}
+
\begin{array}{l}
\mathtt{(tGaTgc)}\\
\mathtt{(tGaGtg)}\\
\end{array}
,
\begin{array}{l}
\mathtt{(aaT)}\\
\mathtt{(atT)}\\
\end{array}
+
\begin{array}{l}
\mathtt{(ctC)}\\
\mathtt{(aaT)}\\
\end{array}
\right)
\end{equation}
Using a similar reasoning, we can also show that the two $\rho$-haplogroups $\mathtt{p}_{5}+\mathtt{p}_{6}$ and $\mathtt{p}_{7}+\mathtt{p}_{8}$ have identical $\rho$-haplotypes. The reader who does the computation can see that this $\rho$-haplotype is the same as the $\rho$-haplotype of $\mathtt{p}_{5}+\mathtt{p}_{7}+ \mathtt{p}_{6}+\mathtt{p}_{8}$ up to a shuffling of the pairs of DNA segments -- this is discussed further below.
\end{example}

One thing that Example \ref{exa:genotype_haplotype_haplogroup} do not quite capture is the law of segregation (the second Mendelian law), which makes the inheritance of haplotypes more complex than a mere transmission of the pairs of alleles from the parents to their progeny. For instance, if we had previously implemented such a rule in Example \ref{exa:genotype_haplotype_haplogroup}, then we could have reduced the number of terms used in the leftmost sum of (\ref{eq:haplotypes-haplogroups:exa}), because two of the terms shown in this sum equates the two others after shuffling the top and bottom DNA sequences. Also, note that Example \ref{exa:genotype_haplotype_haplogroup} do not quite yet capture homologous recombination (the third Mendelian law), but the use of the limit adjoint, sending each pair of DNA segments to a decomposition into smaller segments, already allows us to compare haplogroups up to recombination (see below for more explanation). In subsequent sections, our goal will be to model segregation and recombination events as operations emerging from the pedigrad formalism.

Biologically, segregation occurs after homologous recombination, namely during the formation of sex cells. It has even been shown that the latter influence the former in specific ways \cite{Ottolini}, however, in this paper, we will not address this relationship and only describe the two processes independently. In this respect, we can describe homologous recombination as the process of cutting each homologous chromosomes into small parts, then shuffling and reconnecting these parts together to form a new pair of chromosomes. The resulting pair can be seen as a mixture of the original homologous chromosomes such that the genomic positions of each part are essentially preserved (see diagram below).
\[
\begin{array}{c}
\includegraphics[height=3.5cm]{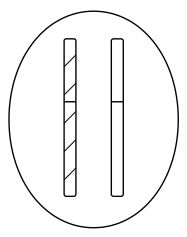}
\end{array}
\longrightarrow
\begin{array}{c}
\includegraphics[width=5cm]{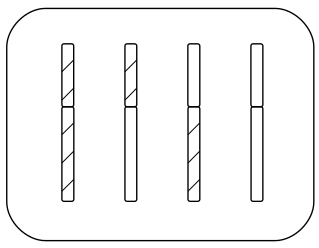}
\end{array}
\]
After homologous recombination, segregation occurs. The latter consists in separating each homologous chromosomes into separate haploid\footnote{\emph{i.e.} possessing one copy of each chromosome} sex cells.
\[
\left(
\begin{array}{l}
\mathtt{ATTAGCTACCTATAC}\\
\mathtt{ACTAGCTACATATGC}\\
\end{array}
\right)
\quad
\leadsto
\quad
\left(
\begin{array}{l}
\mathtt{ATTAGCTACCTATAC}\\
\end{array}
\right)
\quad\text{and}\quad
\left(
\begin{array}{l}
\mathtt{ACTAGCTACATATGC}\\
\end{array}
\right)
\]
Sex cells from different individuals are then meant to meet and merge to produce diploid\footnote{\emph{i.e.} possessing two copies of each chromosome} cells -- this is essentially what \emph{sexual reproduction} is. Intuitively, we could think of the second and third laws of inheritance as dual processes, in that they would visually describe ``horizontal'' and ``vertical'' shuffling operations, respectively.

Note that, over generations, these shuffling processes can lead to a substantial diversification of the genomic materials of a population. In this paper, we will capture this diversification processes through (1) the pulling back of limit adjoints of the form (\ref{eq:limit_adjoint:def_haplotypes_hapologroups}) and (2) the commutative monoid structure associated with the objects of $\mathbf{Icm}$. The definitions and remarks given in the rest of this section aims to model the law of segregation as a natural transformation in $\mathbf{Icm}$.

\begin{definition}[Law of segregation]\label{def:law-of_segregation}
Let $b$ be an element in $\Omega$. We will denote by $\mathsf{sgg}$ the natural transformation
$
\mathbf{2}E_b^{\varepsilon} \Rightarrow UFE_b^{\varepsilon}
$
in $[\mathbf{Seg}(\Omega),\mathbf{Set}]$ that is obtained, for every segment $\tau$ in $\mathbf{Seg}(\Omega)$, from the composition of the sequence of arrows displayed in (\ref{eq:first_mendelian_law}), where
\begin{itemize}
\item[-] the leftmost arrow is the product of two copies of the unit of adjunction (\ref{adjunction_set_icm});
\item[-] the middle arrow is the isomorphism making the right adjoint $U$ commute with products;
\item[-] the rightmost arrow is the image of the natural transformation of Remark \ref{exa:Addition_as_morphism} via $U$;
\end{itemize} 
\begin{equation}\label{eq:first_mendelian_law}
\xymatrix{
E_b^{\varepsilon}(\tau) \times E_b^{\varepsilon}(\tau) \ar[r]^-{\eta_{\tau} \times \eta_{\tau}} & UFE_b^{\varepsilon}(\tau) \times UFE_b^{\varepsilon}(\tau) \ar[r]^-{\cong} & U(FE_b^{\varepsilon}(\tau) \times FE_b^{\varepsilon}(\tau)) \ar[r]^-{U\mathsf{sum}} & UFE_b^{\varepsilon}(\tau)
}
\end{equation}
Note that the naturality of this arrow in $\tau$ is straightforward from the structures used.
\end{definition}

In the first paper of the present series (see \cite{Seqali}), right Kan extensions were used to extend sequence alignments on the whole category of segments. For the present article, we will use the dual notion, namely left Kan extensions.

\begin{remark}[Left Kan extensions]\label{rem:left_kan_extension}
The goal of this remark is to give a concise description of left Kan extensions for the pairs of functors associated with sequence alignments (see \cite{MacLane} for more detail on left Kan extensions).
Importantly, recall that left Kan extensions in $\mathbf{Set}$ can be defined in terms of colimits, which is the approach described below (also see \cite[Chap. X]{MacLane}). In addition, we will only consider left Kan extensions along the inclusion functors of sequence alignments. In this respect, let us consider an element $b$ in $\Omega$ and a sequence alignment $(\iota,T,\sigma)$ over $\mathbf{2}E_{b}^{\varepsilon}$ for which the inclusion $\iota$ is of the form $B \hookrightarrow \mathbf{Seg}(\Omega)$.

To start with, we shall recall a few definitions. For every object $\tau$ in $\mathbf{Seg}(\Omega)$, denote by $({\iota \downarrow \tau})$ the category whose objects are pairs $(\upsilon,f)$ where $\upsilon$ is an object in $B$ and $f$ is a morphism $\iota(\upsilon) \to \tau$ in $\mathbf{Seg}(\Omega)$ and whose arrows $(\upsilon,f) \to (\upsilon',f')$ are given by morphisms $g:\upsilon \to \upsilon'$ in $B$ that make the following square commute in $\mathbf{Seg}(\Omega)$.
\[
\xymatrix{
\tau\ar@{<-}[d]_{f}\ar@{=}[r]&\tau\ar@{<-}[d]^{f'}\\
\iota(\upsilon)\ar[r]_{\iota(g)}&\iota(\upsilon)
}
\]
The mapping $(\upsilon,f) \mapsto \upsilon$ extends to an obvious functor $\iota_{\tau}:({\iota \downarrow \tau}) \to B$ that we can be composed with the functor $T:B \to \mathbf{Set}$ to form the functor $T \circ \iota_{\tau}:({\iota \downarrow \tau}) \to \mathbf{Set}$. According to \cite[Chap. X]{MacLane}, the left Kan extension of $T$ along $\iota$ is the functor $\mathsf{Lan}_{\iota}T:\mathbf{Seg}(\Omega) \to \mathbf{Set}$ whose images are defined by the following colimit construction for every object $\tau$ in $\mathbf{Seg}(\Omega)$.
\begin{equation}\label{eq:formula_right_kan_extension}
\mathsf{Lan}_{\iota}T(\tau) := \mathsf{colim}_{({\iota \downarrow \tau})} T \circ \iota_{\tau}
\end{equation}
Let us now describe the images of $\mathsf{Lan}_{\iota}T$ on the arrows. First, recall that the functor $\iota_{\tau}:({\iota \downarrow \tau}) \to B$ is natural in $\tau$ over the opposite category $\mathbf{Seg}(\Omega)^{\mathrm{op}}$, which means that every morphism $h:\tau \to \tau'$ in $\mathbf{Seg}(\Omega)$ induces a functor $h^{*}:({\iota \downarrow \tau}) \to ({\iota \downarrow \tau'})$ for which the identity $\iota_{\tau}  = \iota_{\tau'} \circ h^{*}$ holds. Put differently, this means that the functor $h^{*}$ sends an object $(\upsilon,f)$ in $({\iota \downarrow \tau})$ to the object $(\upsilon, h \circ f)$ in $({\iota \downarrow \tau'})$. Then, the image of $\mathsf{Lan}_{\iota}T$ at the morphism $h:\tau \to \tau'$ is the comparison morphism induced by pre-composing the diagram of colimit (\ref{eq:formula_right_kan_extension}) with $h^{*}$ (as shown below).
\[
\xymatrix@C+20pt{
\mathsf{colim}_{({\iota \downarrow \tau'})} T \circ \iota_{\tau'} \circ h^{*} \ar[r]^-{\mathsf{Lan}_{\iota}T(h)} & \mathsf{colim}_{({\iota \downarrow \tau})} T \circ \iota_{\tau}
}
\]
It is straightforward to verify that the mapping $\mathsf{Lan}_{\iota}T$ defines a functor $\mathbf{Seg}(\Omega) \to \mathbf{Set}$.
\end{remark}

\begin{remark}[Second Mendelian law]\label{rem:Second_mendelian_law}
The goal of this remark is to combine the construction of Remark \ref{rem:left_kan_extension} with the morphism of Definition \ref{def:law-of_segregation} to construct a morphism in $[\mathbf{Seg}(\Omega),\mathbf{Icm}]$ that we can use to model the second Mendelian law for a given sequence alignment. Broadly, this morphism allows us to use sums of elements in ic-monoids to disentangle pairs of DNA segments, and thus model the segregation of the segments (we will see examples in section \ref{ssec:recombination_chromologies}).

To carry out our construction, we shall need to recall one more property about left Kan extensions, namely: for a given inclusion $\iota:B \hookrightarrow \mathbf{Seg}(\Omega)$, the functorial operation
\[
\mathsf{Lan}_{\iota}:[B,\mathbf{Set}] \to [\mathbf{Seg}(\Omega),\mathbf{Set}]
\]
defines a left adjoint for the pre-composition operation $F \mapsto F \circ \iota$, which is a functor of the form $[\mathbf{Seg}(\Omega),\mathbf{Set}] \to [B,\mathbf{Set}]$ (see \cite[Chapter X]{MacLane}). The counit of the adjunction at a functor $X:\mathbf{Seg}(\Omega)\to \mathbf{Set}$, call it $\nu_X:\mathsf{Lan}_{\iota}(X \circ \iota) \Rightarrow X$, is then defined by the following mapping rule where $(\upsilon,f) \in (\iota \downarrow \tau)$ and $x \in X(\upsilon)$.
\[
\nu_X:\left(
\begin{array}{cccl}
\mathsf{colim}_{({\iota \downarrow \tau})} X \circ \iota \circ \iota_{\tau}& \to & X(\tau)\\
(x,(\upsilon,f)) &\mapsto& X(f)(x)
\end{array}
\right)
\]
It follows from the adjointness property, any morphism $f:T \Rightarrow S \circ \iota$ in $[B,\mathbf{Set}]$ gives rise to a morphism $\mathsf{Lan}_{\iota}T \Rightarrow S$ in $[\mathbf{seg}(\Omega),\mathbf{Set}]$ given by the composite transformation $\nu_{S} \circ \mathsf{Lan}_{\iota}(f)$, which will be denoted as $f^{*}$ for conciseness.

In the context of this paper, we want to use the adjointness property of the left Kan extension on sequence alignments. Specifically, for every element $b \in \Omega$ and sequence alignment $(\iota,T,\sigma)$ over $\mathbf{2}E_{b}^{\varepsilon}$, the adjointness property gives us the leftmost arrow of (\ref{eq:span_phenotypic_expression}) in $[\mathbf{Seg}(\Omega),\mathbf{Set}]$, which we can also send in $[\mathbf{Seg}(\Omega),\mathbf{Icm}]$ by using the left-adjoint functor $F:\mathbf{Set} \to \mathbf{Icm}$ (as shown on the right).
\begin{equation}\label{eq:span_phenotypic_expression}
\begin{array}{l}
\xymatrix{
\mathsf{Lan}_{\iota} T \ar@{=>}[r]^-{\sigma^{*}}&\mathbf{2}E_{b}^{\varepsilon}
}
\end{array}
\quad\quad\begin{array}{l}\Rightarrow\end{array}\quad\quad
\begin{array}{l}
\xymatrix{
F\mathsf{Lan}_{\iota} T \ar@{=>}[r]^-{F\sigma^{*}}&F\mathbf{2}E_{b}^{\varepsilon}
}
\end{array}
\end{equation}
If we now focus on the construction of Definition \ref{def:law-of_segregation}, we can use the universal property of adjunction (\ref{adjunction_set_icm}) to show that the morphism $\mathsf{sgg}:\mathbf{2}E_b^{\varepsilon} \Rightarrow UFE_b^{\varepsilon}$ extends to a morphism $\mathsf{fml}:F\mathbf{2}E_b^{\varepsilon} \Rightarrow FE_b^{\varepsilon}$ (the name $\mathsf{fml}$ here stands for \emph{first Mendelian law}) in $[\mathbf{Seg}(\Omega),\mathbf{Icm}]$.
\[
\xymatrix{
\mathbf{2}E_b^{\varepsilon}\ar@{=>}[r]^-{\mathsf{sgg}}\ar@{=>}[d]_-{\eta\mathbf{2}E_b^{\varepsilon}}&UFE_b^{\varepsilon}\\
UF\mathbf{2}E_b^{\varepsilon}\ar@{==>}[ru]_-{U(\mathsf{fml})}&
}
\]
Then, we can use $\mathsf{fml}:F\mathbf{2}E_b^{\varepsilon} \Rightarrow FE_b^{\varepsilon}$ to extend the rightmost arrow of (\ref{eq:span_phenotypic_expression}) as follows.
\[
\xymatrix{
F\mathsf{Lan}_{\iota} T \ar@{=>}[r]^-{F\sigma^{*}}&F\mathbf{2}E_{b}^{\varepsilon}\ar@{=>}[r]^{\mathsf{fml}}&FE_b^{\varepsilon}
}
\]
The previous arrow models Mendel's law of segregation for our sequence alignment $T$: sums of pairs of DNA sequences are sent to sums of (sums of) DNA sequences. Later, we shall use the previous arrow to coequalize the object $FE_b^{\varepsilon}$ with respect to certain congruences (\emph{i.e.} equivalence relations in ic-monoids), which we define below, in section \ref{ssec:recombination_chromologies}.
\end{remark}

\subsection{Recombination chromologies}\label{ssec:recombination_chromologies}
We shall speak of a \emph{recombination chromology} to refer to a chromology $(\Omega,D)$ such that, for every non-negative integer $n$, the associated set $D[n]$ is finite and only contains wide spans (see Definition \ref{def:wide_spans}).

\begin{example}[Examples and non-example]
Let $\Omega$ denote the pre-ordered set $\{0 \leq 1\}$. The following diagram, living in the pre-order category of homologous segments $\mathbf{Seg}(\Omega)$, is an example of a cone that we can consider to compose a recombination chromology.
\[
\begin{array}{ccc}
&&\xymatrix@C-30pt{
(\circ&\circ&\circ)&(\circ&\circ)&(\bullet&\bullet&\bullet)&(\circ&\circ&\circ&\circ)
}\\
&\rotatebox[origin=c]{45}{$\longrightarrow$}&\\
\xymatrix@C-30pt{
(\bullet&\bullet&\bullet)&(\bullet&\bullet)&(\bullet&\bullet&\bullet)&(\bullet&\bullet&\bullet&\bullet)
} &\longrightarrow & \xymatrix@C-30pt{
(\circ&\circ&\circ)&(\circ&\circ)&(\circ&\circ&\circ)&(\bullet&\bullet&\bullet&\bullet)}
\\
&\rotatebox[origin=c]{-45}{$\longrightarrow$}&\\
 && \xymatrix@C-30pt{
(\bullet&\bullet&\bullet)&(\bullet&\bullet)&(\circ&\circ&\circ)&(\circ&\circ&\circ&\circ&)
}\\
\end{array}
\]
From a biological point of view, this type of cone could be used to specify the genomic decomposition of chromosomes in a recombination event -- each bracketed region is meant to be shuffled (horizontally).

The following diagram, living in the category of quasi-homologous segments $\mathbf{Seg}(\Omega\,|\,12)$ is another example.
\[
\begin{array}{ccc}
&&\xymatrix@C-30pt{
(\circ&\circ&\circ)&(\circ&\circ)&(\bullet&\bullet&\bullet)&(\circ&\circ&\circ&\circ)
}\\
&\rotatebox[origin=c]{45}{$\longrightarrow$}&\\
\xymatrix@C-30pt{
(\bullet)&(\bullet)&(\bullet)&(\bullet&\bullet)&(\bullet&\bullet)&(\bullet)&(\bullet&\bullet)&(\bullet&\bullet)
} &\longrightarrow & \xymatrix@C-30pt{
(\circ&\circ&\circ)&(\circ&\circ)&(\circ&\circ&\circ)&(\bullet&\bullet&\bullet&\bullet)
}\\
&\rotatebox[origin=c]{-45}{$\longrightarrow$}&\\
 && \xymatrix@C-30pt{
(\bullet&\bullet&\bullet)&(\bullet&\bullet)&(\circ&\circ&\circ)&(\circ&\circ&\circ&\circ&)
}\\
\end{array}
\]
The difference between the very first cone and the one given above is that the latter specifies a decomposition that comes from a more refined topology. This type of cones can be useful if we want to relate different topologies through the source of a collection of cones. For example, these cones can  encode various decompositions associated with recombination events through generations -- each recombination would use its own associated decomposition (\emph{i.e.} encoded as a topology for a segment) in the targets of the cone.

Below, we give an example of a cone that is not suitable for defining a recombination chromology. Indeed, the small category on which this cone is defined is a cospan $A=\{\cdot \rightarrow \cdot \leftarrow \cdot\}$, while it should be a discrete category (\emph{i.e.} a set).
\[
\begin{array}{ccccc}
&&\xymatrix@C-30pt{
(\bullet&\bullet&\bullet&\bullet)&(\circ&\circ&\circ)
}&&\\
&\rotatebox[origin=c]{45}{$\longrightarrow$}&&\rotatebox[origin=c]{-45}{$\longrightarrow$}&\\
\xymatrix@C-30pt{
(\bullet&\bullet&\bullet&\bullet)&(\bullet&\bullet&\bullet)
} & & &&\xymatrix@C-30pt{
(\circ&\circ&\circ&\circ)&(\circ&\circ&\circ)
}\\
&\rotatebox[origin=c]{-45}{$\longrightarrow$}&&\rotatebox[origin=c]{45}{$\longrightarrow$}&\\
 && \xymatrix@C-30pt{
(\circ&\circ&\circ&\circ)&(\bullet&\bullet&\bullet)
}&&\\
\end{array}
\]
\end{example}

\begin{convention}[Notation]\label{conv:definition_pi_S}
Let $(\Omega,D)$ be a recombination chromology and $X$ be a functor $\mathbf{Seg}(\Omega) \to \mathbf{Icm}$. For every wide span $\rho:\Delta_{[k]}(\tau) \Rightarrow \theta$ in $(\Omega,D)$, the limit adjoint of the cone $X(\rho):\Delta_{[k]}(X(\tau)) \Rightarrow X\theta$ in $\mathbf{Icm}$ is a product adjoint induced by the product structure defined in Remark \ref{rem:products}. We will denote this product adjoint as follows.
\[
X[\rho]:X(\tau) \to \prod_{i \in [k]} X\theta(i)
\]
\end{convention}

\begin{definition}[Recombination congruences]\label{def:Recomb_congruences}
Let $(\Omega,D)$ be a recombination chromology and $X$ be a functor $\mathbf{Seg}(\Omega) \to \mathbf{Icm}$. For every wide span $\rho:\Delta_{[k]}(\tau) \Rightarrow \theta$ in $(\Omega,D)$, we will denote by $G(X,\rho)$ the pullback of the arrow $X[\rho]$ along itself (see below) and call the resulting pair of arrows $G(X,\rho) \rightrightarrows X(\tau)$ the \emph{recombination congruence of $X$ on $\rho$}.
\[
\xymatrix{
G(X,\rho)\ar@{}[rd]|<<<{\rotatebox[origin=c]{90}{\huge{\textrm{$\llcorner$}}}}\ar[r]^{\mathsf{prj}_1}\ar[d]_{\mathsf{prj}_2}&X(\tau)\ar[d]^{X[\rho]}\\
X(\tau)\ar[r]_-{X[\rho]}&*+!L(.7){\prod_{i \in [k]} X\theta(i)}
}
\]
\end{definition}

\begin{example}[Recombination congruences]\label{exa:Relative_definition_families}
The goal of this example is to give specific instances of elements belonging to recombination congruences. We shall illustrate our example by using a functor resulting from the composition of an environment functors with the free functor $F:\mathbf{Set} \to \mathbf{Icm}$. Let $(\Omega,\preceq)$ be the Boolean pre-ordered set $\{0 \leq 1\}$ and let $(E,\varepsilon)$ be our usual pointed set $\{\mathtt{A},\mathtt{C},\mathtt{G},\mathtt{T},\varepsilon\}$. Take $\rho:\Delta_{[3]}(\tau) \Rightarrow \theta$ to be the wide span of homologous segments used in Example \ref{exa:genotype_haplotype_haplogroup} (its arrows are displayed below).
\[
\begin{array}{ll}
\rho_{1}:\xymatrix@C-30pt@R-20pt{
(\bullet&\bullet&\bullet&\bullet&\bullet&\bullet)&(\bullet&\bullet&\bullet&\bullet&\bullet&\bullet)&(\bullet&\bullet&\bullet)\ar[rr]
&\quad\quad\quad&
(\bullet&\bullet&\bullet&\bullet&\bullet&\bullet)&(\circ&\circ&\circ&\circ&\circ&\circ)&(\circ&\circ&\circ)
}&\theta(1)\\
\rho_{2}:\xymatrix@C-30pt@R-20pt{
(\bullet&\bullet&\bullet&\bullet&\bullet&\bullet)&(\bullet&\bullet&\bullet&\bullet&\bullet&\bullet)&(\bullet&\bullet&\bullet)\ar[rr]
&\quad\quad\quad&
(\circ&\circ&\circ&\circ&\circ&\circ)&(\bullet&\bullet&\bullet&\bullet&\bullet&\bullet)&(\circ&\circ&\circ)
}&\theta(2)\\
\rho_{3}:\xymatrix@C-30pt@R-20pt{
(\bullet&\bullet&\bullet&\bullet&\bullet&\bullet)&(\bullet&\bullet&\bullet&\bullet&\bullet&\bullet)&(\bullet&\bullet&\bullet)\ar[rr]
&\quad\quad\quad&
(\circ&\circ&\circ&\circ&\circ&\circ)&(\circ&\circ&\circ&\circ&\circ&\circ)&(\bullet&\bullet&\bullet)
}&\theta(3)\\
\end{array}
\]
The idea behind the notion of recombination congruence is that the ic-monoid $G(FE_1^{\varepsilon},\rho)$ contains the pairs of elements of $FE_1^{\varepsilon}(\rho)$ that are the same up to homologous recombination (for the corresponding topology associated with the cone $\rho$).
For instance, a consequence of Example \ref{exa:genotype_haplotype_haplogroup} is that the pair $(\mathsf{fml}_{\tau}(\mathtt{p}_{5}+\mathtt{p}_{6}),\mathsf{fml}_{\tau}(\mathtt{p}_{7}+\mathtt{p}_{8}))$ is an element of the ic-monoid $G(FE_1^{\varepsilon},\rho)$. Indeed, recall that the example concluded by observing that the pair $(\mathtt{p}_{5}+\mathtt{p}_{6},\mathtt{p}_{7}+\mathtt{p}_{8})$ is in the pullback $G(F\mathbf{2}E_1^{\varepsilon},\rho)$. From this, one can use the universality of pullbacks and the naturality of the morphism $\mathsf{fml}:F\mathbf{2}E_1^{\varepsilon} \Rightarrow FE_1^{\varepsilon}$ (Remark \ref{rem:Second_mendelian_law}) to show that this pair is sent to the pullback $G(FE_1^{\varepsilon},\rho)$ through the mapping rule $(x,y) \mapsto (\mathsf{fml}_{\tau}(x),\mathsf{fml}_{\tau}y))$. To discuss more complex scenarios, we shall now focus on the elements $\mathtt{p}_{8}$, $\mathtt{p}_{9}$, $\mathtt{p}_{10}$, $\mathtt{p}_{11}$, and $\mathtt{p}_{12}$. To do so, we can first compute the images of these elements through the first law of Mendelian inheritance (see Example \ref{exa:Sequence alignments}). The resulting $\rho$-haplogroups are listed in the table below.
\[
\begin{array}{l|c}
\cellcolor[gray]{0.8}&\cellcolor[gray]{0.8}\textrm{in }FE_1^{\varepsilon}(\tau)\\
\hline
\mathsf{fml}_{\tau}(\mathtt{p}_{8})&\mathtt{(aAcAgT)(tGaTgc)(ctC)} + \mathtt{(gAtCcC)(tGaGtg)(aaT)}\\
\mathsf{fml}_{\tau}(\mathtt{p}_{9})&\mathtt{(aTtAgT)(cAaGtg)(aaT)} + \mathtt{(aTtCgT)(tGaTgc)(atT)}\\
\mathsf{fml}_{\tau}(\mathtt{p}_{10})&\mathtt{(aTtCgT)(cAaTgc)(acT)} + \mathtt{(gAtCcC)(tGaTgc)(ctC)}\\
\mathsf{fml}_{\tau}(\mathtt{p}_{11})&\mathtt{(aAcAgT)(cAaTgc)(atT)} +  \mathtt{(aTtAgT)(cAaGtg)(aaT)}\\
\mathsf{fml}_{\tau}(\mathtt{p}_{12})& \mathtt{(gAtCcC)(cAaGtg)(atT)} + \mathtt{(aTtAgT)(cAaGtg)(acT)}
\end{array}
\]
We can use the previous table of $\rho$-haplogroups to compute the corresponding $\rho$-haplotypes (as  images of the canonical arrow $FE_1^{\varepsilon}[\rho]$ defined in Convention \ref{conv:definition_pi_S}), namely the following list of tuples.
\[
\begin{array}{l|c}
\cellcolor[gray]{0.8}&\multicolumn{1}{c|}{\cellcolor[gray]{0.8}\textrm{in }\prod_{i \in [3]} FE_1^{\varepsilon}\theta(i)}\\
\hline
p_{8}&\big(\,\mathtt{aAcAgT} + \mathtt{gAtCcC}\,,\,\mathtt{tGaTgc} + \mathtt{tGaGtg} \,,\,\mathtt{ctC} + \mathtt{aaT}\,\big)\\
p_{9}&\big(\,\mathtt{aTtAgT} + \mathtt{aTtCgT}\,,\,\mathtt{cAaGtg} + \mathtt{tGaTgc} \,,\,\mathtt{aaT} + \mathtt{atT}\,\big)\\
p_{10}&\big(\,\mathtt{aTtCgT} + \mathtt{gAtCcC}\,,\, \mathtt{cAaTgc} + \mathtt{tGaTgc}\,,\,\mathtt{acT} + \mathtt{ctC}\,\big)\\
p_{11}&\big(\,\mathtt{aAcAgT} + \mathtt{aTtAgT}\,,\,\mathtt{cAaTgc} + \mathtt{cAaGtg}\,,\,\mathtt{atT} + \mathtt{aaT}\,\big)\\
p_{12}&\big(\,\mathtt{gAtCcC} + \mathtt{aTtAgT}\,,\,\mathtt{cAaGtg}+\mathtt{cAaGtg}\,,\,\mathtt{atT} + \mathtt{acT}\,\big)\\
\end{array}
\]
Contrary to what was deduced for the elements $\mathtt{p}_{5}$, $\mathtt{p}_{6}$, $\mathtt{p}_{7}$, $\mathtt{p}_{8}$, we cannot use the previous table to find equations that would contain completely different terms on each side. Instead, the present situation hints toward the necessity to look at equations with repeated terms on each side (see below). For instance, we can show that the following three tuples of $\rho$-haplotypes belong to the ic-monoid $G(FE_1^{\varepsilon},\rho)$ due to the corresponding equations shown on the right.
\[
\begin{array}{lll}
(\mathsf{fml}_{\tau}(\mathtt{p}_{8}+\mathtt{p}_{9}+\mathtt{p}_{10})\,\,,&\mathsf{fml}_{\tau}(\mathtt{p}_{8}+\mathtt{p}_{9}+\mathtt{p}_{10}+\mathtt{p}_{11})),&\textrm{since }p_{8}+p_{9}+p_{10}=p_{8}+p_{9}+p_{10}+p_{11}\\
(\mathsf{fml}_{\tau}(\mathtt{p}_{10}+\mathtt{p}_{11})\,\,,&\mathsf{fml}_{\tau}(\mathtt{p}_{10}+\mathtt{p}_{11}+\mathtt{p}_{9})),&\textrm{since }p_{10}+p_{11}=p_{10}+p_{11}+p_{9}\\
(\mathsf{fml}_{\tau}(\mathtt{p}_{10}+\mathtt{p}_{11})\,\,,&\mathsf{fml}_{\tau}(\mathtt{p}_{10}+\mathtt{p}_{11}+\mathtt{p}_{12})),&\textrm{since }p_{10}+p_{11}=p_{10}+p_{11}+p_{12}\\
\end{array}
\]
Here, it is important to note that, within the scope of modeling/studying homologous recombination, it makes sense to look for pairs of the form $(t,t+x)$ in $G(FE_1^{\varepsilon},\rho)$ because they inform us that the $\rho$-haplogroup $t$ has the potential to \emph{beget} the $\rho$-haplogroup $x$. This is much less restrictive than establishing that some $\rho$-haplogroup $t$ can completely recover the whole $\rho$-haplotype of some other $\rho$-haplogroup $x$, presumably made of completely different $\rho$-genotypes. In fact, this second scenario can only be expected when the available data account for a large amount of recombination events. In practice, this is almost never the case, which is why finding pairs of the form $(t,t+x)$ in $G(FE_1^{\varepsilon},\rho)$ is more likely to happen.

The astute reader might have noticed that one cannot form a $\rho$-haplogroup made of the $\rho$-haplogroup $\mathsf{fml}_{\tau}(\mathtt{p}_{9})$, $\mathsf{fml}_{\tau}(\mathtt{p}_{10})$, $\mathsf{fml}_{\tau}(\mathtt{p}_{11})$, and $\mathsf{fml}_{\tau}(\mathtt{p}_{12})$ such that the resulting $\rho$-haplogroup begets the $\rho$-haplogroup $\mathsf{fml}_{\tau}(\mathtt{p}_{8})$. This suggests that the canonical arrow $FE_1^{\varepsilon}[\rho]$ (Convention \ref{conv:definition_pi_S}), or more precisely the cone $\rho$, is not adequate to explain the potential relationships that may exist between the $\rho$-genotype $\mathtt{p}_8$ and the $\rho$-genotypes $\mathtt{p}_{9}$, $\mathtt{p}_{10}$, $\mathtt{p}_{11}$, $\mathtt{p}_{12}$. From here on, we can choose between two approaches: either we relate the $\rho$-genotype $\mathtt{p}_8$ to so-far unconsidered $\rho$-genotypes, or we repeat the previous analysis with different cone $\rho'$. In this respect, we shall propose the following cone $\rho':\Delta_{[3]}(\tau') \Rightarrow \theta'$, for which the pair $(\mathsf{fml}_{\tau}(\mathtt{p}_{10}+\mathtt{p}_{11}),\mathsf{fml}_{\tau}(\mathtt{p}_{10}+\mathtt{p}_{11}+\mathtt{p}_{9}))$ belongs to $G(FE_1^{\varepsilon},\rho')$.
\[
\begin{array}{ll}
\rho_{1}':\xymatrix@C-30pt@R-20pt{
(\bullet&\bullet&\bullet&\bullet&\bullet&\bullet)&(\bullet&\bullet&\bullet)&(\bullet&\bullet&\bullet&\bullet&\bullet&\bullet)\ar[rr]
&\quad\quad\quad&
(\bullet&\bullet&\bullet&\bullet&\bullet&\bullet)&(\circ&\circ&\circ)&(\circ&\circ&\circ&\circ&\circ&\circ)
}&\theta'(1)\\
\rho_{2}':\xymatrix@C-30pt@R-20pt{
(\bullet&\bullet&\bullet&\bullet&\bullet&\bullet)&(\bullet&\bullet&\bullet)&(\bullet&\bullet&\bullet&\bullet&\bullet&\bullet)\ar[rr]
&\quad\quad\quad&
(\circ&\circ&\circ&\circ&\circ&\circ)&(\bullet&\bullet&\bullet)&(\circ&\circ&\circ&\circ&\circ&\circ)
}&\theta'(2)\\
\rho_{3}':\xymatrix@C-30pt@R-20pt{
(\bullet&\bullet&\bullet&\bullet&\bullet&\bullet)&(\bullet&\bullet&\bullet)&(\bullet&\bullet&\bullet&\bullet&\bullet&\bullet)\ar[rr]
&\quad\quad\quad&
(\circ&\circ&\circ&\circ&\circ&\circ)&(\circ&\circ&\circ)&(\bullet&\bullet&\bullet&\bullet&\bullet&\bullet)
}&\theta'(3)\\
\end{array}
\]
Specifically, we can verify the previous claim as follows. First, we compute the $\rho'$-haplotypes
of the $\rho'$-haplogroups induced by the elements $\mathsf{fml}_{\tau'}(\mathtt{p}_{8})$, $\mathsf{fml}_{\tau'}(\mathtt{p}_{9})$, $\mathsf{fml}_{\tau'}(\mathtt{p}_{10})$, $\mathsf{fml}_{\tau'}(\mathtt{p}_{11})$, and $\mathsf{fml}_{\tau'}(\mathtt{p}_{12})$ in a table, as shown below. Then, we use this table to observe that the $\rho'$-haplotype $p_{10}+p_{11}$ is equal to the $\rho'$-haplotype $p_{10}+p_{11}+p_{8}$. This means that the $\rho'$-haplogroup $\mathsf{fml}_{\tau'}(\mathtt{p}_{10}+\mathtt{p}_{11})$ begets the $\rho'$-haplogroup $\mathsf{fml}_{\tau'}(\mathtt{p}_{8})$.
\[
\begin{array}{c|c}
\cellcolor[gray]{0.8}&\multicolumn{1}{c|}{\cellcolor[gray]{0.8}\textrm{in }\prod_{i \in [3]} FE_1^{\varepsilon}\theta'(i)}\\
\hline
p_{8}&\big(\,\mathtt{aAcAgT} + \mathtt{gAtCcC}\,,\,\mathtt{tGa} \,,\,\mathtt{TgcctC} + \mathtt{GtgaaT}\,\big)\\
p_{9}&\big(\,\mathtt{aTtAgT} + \mathtt{aTtCgT}\,,\,\mathtt{cAa} + \mathtt{tGa} \,,\,\mathtt{GtgaaT} + \mathtt{TgcatT}\,\big)\\
p_{10}&\big(\,\mathtt{aTtCgT} + \mathtt{gAtCcC}\,,\, \mathtt{cAa} + \mathtt{tGa}\,,\,\mathtt{TgcacT} + \mathtt{TgcctC}\,\big)\\
p_{11}&\big(\,\mathtt{aAcAgT} + \mathtt{aTtAgT}\,,\,\mathtt{cAa}\,,\,\mathtt{TgcatT} + \mathtt{GtgaaT}\,\big)\\
p_{12}&\big(\,\mathtt{gAtCcC} + \mathtt{aTtAgT}\,,\,\mathtt{cAa}\,,\,\mathtt{GtgatT} + \mathtt{GtgacT}\,\big)\\
\end{array}
\]

In conclusion, this example demonstrates how the canonical arrows of the form shown in Convention \ref{conv:definition_pi_S} can be used as reasoning tools to identify regions at which \emph{crossovers} cannot have occurred (see the references given below). Specifically, this is done by looking at the cones for which two $\rho$-haplogroups of the form $t$ and $t+x$ have different $\rho$-haplotypes. However, one challenge in using this tool is that one needs to significantly restrict the set of cones on which we want to use it for reasonable computation times. Luckily for us, the recombination \emph{hot spots} at which crossovers occur in the human genome do not change much, and only their occurrence frequency does  \cite{Kauppi,Jeffreys}. For instance, there may be a noticeable difference in male and female regarding the frequency at which certain hot spots are used, but the hot spots will essentially be the same for both genders. Yet, recombination topologies will still vary from generation to generation as they are specific to each individual.
\end{example}

The kinship relations discussed in Example \ref{exa:Relative_definition_families} motivates the introduction of a formal binary relations, which we introduce in Definition \ref{def:begetting}. This relation is given in preparation for the content of section \ref{ssec:recombination_monoids}, in which we define quotients of ic-monoids by using recombination congruences. After Definition \ref{def:begetting}, we show in Remark \ref{rem:begetting} that the kinship relation is a preorder relation.

\begin{definition}[Kinship]\label{def:begetting}
Let $(M,+,0)$ be an ic-monoid. For every pair $(x,y)$ of elements in $M$, we will say that \emph{$y$ begets $x$}, and write $x \triangleright y$, whenever the equation $y = x + y$ holds in $M$.
\end{definition}

\begin{remark}[Preorder]\label{rem:begetting}
For every ic-monoid $(M,+,0)$, it is straightforward to verify that the relation $\triangleright$ defines a pre-order relation on $M$. Indeed, the relation $\triangleright$ is reflexive because $M$ is idempotent ($y = y+y$ for every $y \in M$), and it is transitive because $M$ is associative: if $x = z+x$ and $y = x +y$, then we have the equations $y = (z+x)+y = z + (x+y) = z + y$.
\end{remark}

Example \ref{exa:Relative_definition_families} has shown that it is possible to detect `good' cones (for a given dataset) by looking at the haplotypes generated by these cones. These cones should ideally be gathered within a recombination chromology. In section \ref{ssec:recombination_monoids}, we see how one can use such a chromology to give a meaning to the sentence ``up to recombination''.

\subsection{Recombination monoids}\label{ssec:recombination_monoids}
The goal of the present section is to define quotients of functors $\mathbf{Seg}(\Omega) \to \mathbf{Icm}$ with respect to families of recombination congruences (Definition \ref{def:Recomb_congruences}). We start the section with a discussion (Remark \ref{rem:functorial_quotients}) on how to form functorial quotients.

\begin{remark}[Functorial quotients]\label{rem:functorial_quotients}
Let $(\Omega,\preceq)$ be a pre-ordered set, $(\Omega,D)$ be a recombination chromology and $X$ be a functor $\mathbf{Seg}(\Omega) \to \mathbf{Icm}$. To quotient the functor $X$ with respect to its recombination congruences $G(X,\rho) \rightrightarrows X(\tau)$ over its cones $\rho:\Delta_A(\tau) \Rightarrow \theta$ in $D$, we could consider a coequalizer, call it $Q(\tau)$, of the coproduct of all its recombination congruences over $D$ (see below) -- the coequalizer diagram is given by the pair of arrows induced by the coproduct adjoints of the collections of projections $G(X,\rho) \to X(\tau)$.
\[
\xymatrix@C+18pt{
\mathop{\coprod}\limits_{\rho \in D} G(X,\rho)\ar@<-1ex>[r]_-{ \oplus_{\rho} \mathsf{prj}_2}\ar@<+1ex>[r]^-{\oplus_{\rho} \mathsf{prj}_1}& X(\tau)
}
\]
Because the resulting mapping $\tau \mapsto Q(\tau)$ is unlikely functorial, we force this functoriality by adding more pairs to the previous coproduct. Specifically, we consider the pair of arrows resulting from the composition of the arrows shown in (\ref{eq:coequalizer_mod_pedigrad}): the coproduct is, this time, taken over the finite set of pairs $(\rho,f)$ where $\rho$ is a cone of the form $\Delta_A(\upsilon) \Rightarrow \theta$ in $D$, $f$ is an arrow $\upsilon \to \tau$ in $\mathbf{Seg}(\Omega)$, and the arrow $\oplus_{\rho,f}X(f)$ is thecoproduct adjoint of the collection $(X(f))_{\rho,f}$.
\begin{equation}\label{eq:coequalizer_mod_pedigrad}
\xymatrix@C+18pt{
\mathop{\coprod}\limits_{\rho,f:\upsilon \to \tau} G(X,\rho)\ar@<-1ex>[r]_-{\coprod_{\rho,f}\mathsf{prj}_2}\ar@<+1ex>[r]^-{\coprod_{\rho,f}\mathsf{prj}_1}&\mathop{\coprod}\limits_{\rho,f:\upsilon \to \tau}  X(\upsilon) \ar[rr]^-{\oplus_{\rho,f}X(f)}&& X(\tau)
}
\end{equation}
Every coequalizer $Q(\tau)$ of the previous pair in $\mathbf{Icm}$ gives is a mapping $\tau \mapsto Q(\tau)$ that extends into a functor $\mathbf{Seg}(\Omega) \to \mathbf{Icm}$. The images of this functor on the arrows of $\mathbf{Seg}(\Omega)$ are induced as follows. For every morphism $g:\tau \to \tau'$, we have diagram (\ref{eq:coequalizer_mod_pedigrad:morphism:functoriality}), which induces an arrow $Q(g):Q(\tau) \to Q(\tau')$ in $\mathbf{Icm}$. By universality of coequalizers, we can verify that the pair of mapping rules $\tau \mapsto Q(\tau)$ (on objects) and $g \mapsto Q(g)$ (on arrows) induces a functor $\mathbf{Seg}(\Omega) \to \mathbf{Icm}$.
\begin{equation}\label{eq:coequalizer_mod_pedigrad:morphism:functoriality}
\xymatrix@C+18pt{
\mathop{\coprod}\limits_{\rho,f:\upsilon \to \tau} G(X,\rho)\ar[d]_{\subseteq}\ar@<-1ex>[r]_-{\coprod\mathsf{prj}_2}\ar@<+1ex>[r]^-{\coprod\mathsf{prj}_1}&\mathop{\coprod}\limits_{\rho,f:\upsilon \to \tau}  X(\upsilon) \ar[rr]^-{\mathop{\oplus}\limits_{\rho,f}X(f)}\ar[d]_{\subseteq}\ar@{..>}[rrd]|{\mathop{\oplus}\limits_{\rho,f}X(g \circ f)}&& X(\tau)\ar[d]^{X(g)}\\
\mathop{\coprod}\limits_{\rho,f:\upsilon \to \tau'} G(X,\rho)\ar@<-1ex>[r]_-{\coprod\mathsf{prj}_2}\ar@<+1ex>[r]^-{\coprod\mathsf{prj}_1}&\mathop{\coprod}\limits_{\rho,f:\upsilon \to \tau'}  X(\upsilon) \ar[rr]_-{\mathop{\oplus}\limits_{\rho,f}X(f)}&& X(\tau')
}
\end{equation}
\end{remark}

\begin{definition}[Recombination monoids]\label{def:Canonical_pedigrads_in_semimodules}
For every recombination chromology $(\Omega,D)$ and object $\tau$ in $\mathbf{Seg}(\Omega)$, we will denote as $DX(\tau)$ the coequalizer of (\ref{eq:coequalizer_mod_pedigrad}). The functor $DX:\mathbf{Seg}(\Omega) \to \mathbf{Icm}$ that results from forming these coequalizers will be called the \emph{recombination monoid of $(\Omega,D)$ over $X$}
\end{definition}

\begin{convention}[Coequalizer map]\label{conv:coequalizer_map}
For every functor $X:\mathbf{Seg}(\Omega) \to \mathbf{Icm}$, the natural transformation $X \Rightarrow DX$ in $[\mathbf{Seg}(\Omega),\mathbf{Icm}]$ induced by the collection of coequalizer morphisms $X(\tau) \to DX(\tau)$ (see diagram (\ref{eq:coequalizer_mod_pedigrad}) will be denoted as $q_{X}$.
\end{convention}

\begin{example}[Recombination monoids]\label{exa:Recombination_semimodule_for_DNA}
The goal of this example is to give specific instances of elements belonging to recombination monoids. In particular, this example focuses on exposing the relationships that exist between these elements. As with Example \ref{exa:Relative_definition_families}, we shall illustrate our example by using a functor resulting from the composition of an environment functor with the free functor $F:\mathbf{Set} \to \mathbf{Icm}$.  Let $(E,\varepsilon)$ be our usual pointed set $\{\mathtt{A},\mathtt{C},\mathtt{G},\mathtt{T},\varepsilon\}$ and take $(\Omega,\preceq)$ to be the Boolean pre-ordered set $\{0 \leq 1\}$. Recall that, in Example \ref{exa:Relative_definition_families}, we showed how choosing certain cones $\rho:\Delta_{[k]}(\tau) \Rightarrow \theta$ could change the way the elements of $FE_1^{\varepsilon}(\tau)$ are compared through the canonical arrow $FE_1^{\varepsilon}[\rho]:FE_1^{\varepsilon}(\tau) \to \prod_{i \in [k]} FE_1^{\varepsilon}\theta(i)$. In the case of the recombination monoids, these comparisons are realized as equations.

If we suppose that the cone $\rho:\Delta_{[3]}(\tau) \Rightarrow \theta$ given in Example \ref{exa:Relative_definition_families} is one of the cones contained in $D$, then the coequalizer map $q_{FE_{1}^{\varepsilon}}:FE_{1}^{\varepsilon}(\tau) \to DFE_{1}^{\varepsilon}(\tau)$ will identify the two $\rho$-haplogroups $\mathsf{fml}_{\tau}(\mathtt{p}_{5}+\mathtt{p}_{6})$ and $\mathsf{fml}_{\tau}(\mathtt{p}_{7}+\mathtt{p}_{8})$ as a single element in $DFE_{1}^{\varepsilon}(\tau)$.  Since the morphism $\mathsf{fml}$ takes its values in $\mathbf{Icm}$, the following identity holds in $DFE_{1}^{\varepsilon}(\tau)$.
\[
\mathsf{fml}_{\tau}(\mathtt{p}_{5}+\mathtt{p}_{6}) = \mathsf{fml}_{\tau}(\mathtt{p}_{7}+\mathtt{p}_{8})
\]
Similarly, Example \ref{exa:Relative_definition_families} has shown that the following identities hold in $DFE_{1}^{\varepsilon}(\tau)$ (which can equivalently be expressed as kinship relations, as shown on the right -- see Definition \ref{def:begetting}).
\[
\begin{array}{ll}
\mathsf{fml}_{\tau}(\mathtt{p}_{8}+\mathtt{p}_{9}+\mathtt{p}_{10})=\mathsf{fml}_{\tau}(\mathtt{p}_{8}+\mathtt{p}_{9}+\mathtt{p}_{10}+\mathtt{p}_{11}) &\Leftrightarrow\quad\quad\mathsf{fml}_{\tau}(\mathtt{p}_{8}+\mathtt{p}_{9}+\mathtt{p}_{10}) \triangleright \mathsf{fml}_{\tau}(\mathtt{p}_{11})\\
\mathsf{fml}_{\tau}(\mathtt{p}_{10}+\mathtt{p}_{11})=\mathsf{fml}_{\tau}(\mathtt{p}_{10}+\mathtt{p}_{11}+\mathtt{p}_{9})& \Leftrightarrow\quad\quad\mathsf{fml}_{\tau}(\mathtt{p}_{10}+\mathtt{p}_{11})\triangleright \mathsf{fml}_{\tau}(\mathtt{p}_{9})\\
\mathsf{fml}_{\tau}(\mathtt{p}_{10}+\mathtt{p}_{11})=\mathsf{fml}_{\tau}(\mathtt{p}_{10}+\mathtt{p}_{11}+\mathtt{p}_{12})& \Leftrightarrow\quad\quad\mathsf{fml}_{\tau}(\mathtt{p}_{10}+\mathtt{p}_{11})\triangleright \mathsf{fml}_{\tau}(\mathtt{p}_{12})\\
\end{array}
\]
Meanwhile, the previous equations do not hold in $FE_{1}^{\varepsilon}(\tau)$. For instance, by using the equations contained in the topmost table of Example \ref{exa:Relative_definition_families}, we can show that the two $\rho$-haplogroups $\mathsf{fml}_{\tau}(\mathtt{p}_{8})+\mathsf{fml}_{\tau}(\mathtt{p}_{9})+\mathsf{fml}_{\tau}(\mathtt{p}_{10})$ and $\mathsf{fml}_{\tau}(\mathtt{p}_{8})+\mathsf{fml}_{\tau}(\mathtt{p}_{9})+\mathsf{fml}_{\tau}(\mathtt{p}_{10})+\mathsf{fml}_{\tau}(\mathtt{p}_{11})$ are distinct in $FE_{1}^{\varepsilon}(\tau)$.

If we now suppose that the cone $\rho':\Delta_{[3]}(\tau') \Rightarrow \theta'$ given in Example \ref{exa:Relative_definition_families} is another cone of $D$, then the coequalizer map $q_{FE_{1}^{\varepsilon}}:FE_{1}^{\varepsilon}(\tau') \to DFE_{1}^{\varepsilon}(\tau')$ will identify the two elements $\mathsf{fml}_{\tau'}(\mathtt{p}_{10}+\mathtt{p}_{11})$ and $\mathsf{fml}_{\tau'}(\mathtt{p}_{10}+\mathtt{p}_{11}+\mathtt{p}_{8})$ as a single element in $DFE_{1}^{\varepsilon}(\tau')$, thus making the relation $\mathsf{fml}_{\tau'}(\mathtt{p}_{10}+\mathtt{p}_{11})  \triangleright \mathsf{fml}_{\tau'}(\mathtt{p}_{8})$ hold in $DFE_{1}^{\varepsilon}(\tau')$. Here again, we can verify that the relation $\mathsf{fml}_{\tau'}(\mathtt{p}_{10}+\mathtt{p}_{11})  \triangleright \mathsf{fml}_{\tau'}(\mathtt{p}_{8})$ does not hold in $FE_{1}^{\varepsilon}(\tau')$.

To conclude, this example has demonstrated how every cone $\varrho:\Delta_{[k]}(\upsilon) \Rightarrow \vartheta$ in $D$ will force the two components of a pair contained in $G(FE_{1}^{\varepsilon},\varrho)$ to be identified in the recombination monoid $DFE_{1}^{\varepsilon}(\upsilon)$ through the coequalizer of diagram (\ref{eq:coequalizer_mod_pedigrad}).
\end{example}

\begin{remark}[Mendelian monoids]\label{rem:D_ET_preciting_phenotypes}
The goal of this remark is to refine the construction of Definition \ref{def:Canonical_pedigrads_in_semimodules} to sequence alignments. In particular, this means that the equations forced by the coequalizers will involve sums of \emph{pairs} of DNA sequences, as opposed to unpaired DNA sequences, as was illustrated in Example \ref{exa:Recombination_semimodule_for_DNA} and Example \ref{exa:Relative_definition_families}.
Let $(\Omega,\preceq)$ be a pre-ordered set, let $(\Omega,D)$ be a recombination chromology, let $b$ be an element of $\Omega$, let $(E,\varepsilon)$ be a pointed set and let $(\iota, T,\sigma)$ be a sequence alignment over $\mathbf{2}E_{b}^{\varepsilon}$. Consider the arrow constructed at the end of Remark \ref{rem:Second_mendelian_law}, namely the composition of the following pair of arrows.
\[
\xymatrix{
F\mathsf{Lan}_{\iota} T \ar@{=>}[r]^-{F\sigma^{*}}&F\mathbf{2}E_{b}^{\varepsilon}\ar@{=>}[r]^-{\mathsf{fml}}&FE_b^{\varepsilon}
}
\]
Then, the composition of the resulting arrow with the coequalizer map $q_{FE_b^{\varepsilon}}:FE_b^{\varepsilon} \Rightarrow DFE_b^{\varepsilon}$ gives us the composite arrow $F\mathsf{Lan}_{\iota} T \Rightarrow DFE_b^{\varepsilon}$ shown below in (\ref{eq:span_recombination_complete}) -- we denote this arrow as $\mathsf{rec}$ (which stands for \emph{recombination}).
\begin{equation}\label{eq:span_recombination_complete}
\xymatrix{
F\mathsf{Lan}_{\iota} T \ar@{==>}@/^2pc/[rrr]^{\mathsf{rec}}\ar@{=>}[r]^-{F\sigma^{*}}&F\mathbf{2}E_{b}^{\varepsilon}\ar@{=>}[r]^-{\mathsf{fml}}&FE_b^{\varepsilon}\ar@{=>}[r]^-{q_{FE_b^{\varepsilon}}}&DFE_b^{\varepsilon}
}
\end{equation}
We shall now use arrow (\ref{eq:span_recombination_complete}) with a pullback construction. First, note that that the category $\mathbf{Icm}$ has pullbacks because it has products and coequalizers (Remark \ref{rem:products} and Proposition \ref{prop:coequalizer:existence:Icm}). As a result, the functor category $[\mathbf{Seg}(\Omega),\mathbf{Icm}]$ also has products, coequalizers and pullbacks, which are all computed pointwise (\emph{i.e.} computed in $\mathbf{Icm}$ relative to each argument in $\mathbf{Seg}(\Omega)$).

This being established, let us now denote by $R_D^T$ the pullback of $\mathsf{rec}:F\mathsf{Lan}_{\iota} T \Rightarrow DFE_b^{\varepsilon}$ along itself in $[\mathbf{Seg}(\Omega),\mathbf{Icm}]$ (see the diagram on the left-hand side of (\ref{eq:span_recombination_complete:pullback:matrix-algebra}).
\begin{equation}\label{eq:span_recombination_complete:pullback:matrix-algebra}
\begin{array}{l}
\xymatrix{
R_D^T \ar@{=>}[d]_{y_2} \ar@{=>}[r]^-{y_1}  \ar@{}[rd]|<<<{\rotatebox[origin=c]{90}{\huge{\textrm{$\llcorner$}}}} &F\mathsf{Lan}_{\iota} T\ar@{=>}[d]^{\mathsf{rec}}\\
F\mathsf{Lan}_{\iota} T \ar@{=>}[r]_-{\mathsf{rec}}&DFE_b^{\varepsilon}
}
\end{array}
\quad\quad
\quad\quad
\begin{array}{l}
\xymatrix@R-25pt{
&&DFE_b^{\varepsilon}\\
R_D^T  \ar@<+1.2ex>@{=>}[r]^-{y_1}\ar@<-1.2ex>@{=>}[r]_-{y_2}& F\mathsf{Lan}_{\iota} T \ar@{==>}[rd]^-{r}\ar@{=>}[ru]^-{\mathsf{rec}}&\\
&&D_{E}T\ar@{=>}[uu]_{\mathsf{sub}}\\
}
\end{array}
\end{equation}
If we coequalize the resulting pullback pair $(y_1,y_2)$, we obtain a coequalizer $r:F\mathsf{Lan}_{\iota} T \Rightarrow D_{E}T$, as shown on the right-hand side of (\ref{eq:span_recombination_complete:pullback:matrix-algebra}). Since the arrow $\mathsf{rec}:F\mathsf{Lan}_{\iota} T \Rightarrow DFE_b^{\varepsilon}$ also coequalizes the pullback pair $(y_1,y_2)$, the coequalizer arrow $r$ factorizes the arrow $\mathsf{rec}$ through a universal arrow $\mathsf{sub}:D_ET \Rightarrow DFE_b^{\varepsilon}$ (see the leftmost diagram of (\ref{eq:span_recombination_complete:pullback:matrix-algebra})). Since, for every segment $\tau$ in $\mathbf{Seg}(\Omega)$, the object $R_D^T(\tau)$ defines a pullback of the arrow $\mathsf{rec}_{\tau}$ along itself in $\mathbf{Icm}$ and the object $D_{E}T(\tau)$ is a coequalizer for the pair $y_{1,\tau},y_{2,\tau}:R^T_D(\tau)\rightrightarrows F\mathsf{Lan}_{\iota} T(\tau)$ in $\mathbf{Icm}$, it follows from Proposition \ref{prop:pullback:coequalizer:mono} that the arrow $\mathsf{sub}_{\tau}:D_ET(\tau) \to DFE_b^{\varepsilon}(\tau)$ is a monomorphism in $\mathbf{Icm}$.
\end{remark}

\begin{example}[Mendelian monoids]\label{exa:D_ET_preciting_phenotypes}
The goal of this example is to give specific instances of elements belonging to the images of the coequalizer functor constructed in Remark \ref{rem:D_ET_preciting_phenotypes}. To do so, we shall build on the analyses done in Example \ref{exa:Recombination_semimodule_for_DNA}. In this respect, we will keep the same notations, assumptions and equations established thereof.

Recall that that the first equation presented in Example \ref{exa:Recombination_semimodule_for_DNA} was $\mathsf{fml}_{\tau}(\mathtt{p}_{5}+\mathtt{p}_{6}) = \mathsf{fml}_{\tau}(\mathtt{p}_{7}+\mathtt{p}_{8})$. It follows from the constructions made in Remark \ref{rem:D_ET_preciting_phenotypes} that this equation is equivalent to the following equation $DFE_1^{\varepsilon}(\tau)$, where $\mathtt{p}_{5}$, $\mathtt{p}_{6}$, $\mathtt{p}_{7}$ and $\mathtt{p}_{8}$ are the elements of $F\mathsf{Lan}_{\iota} T(\tau)$ given in Example \ref{exa:Sequence alignments}.
\[
\mathsf{rec}_{\tau}(\mathtt{p}_{5}+\mathtt{p}_{6}) = \mathsf{rec}_{\tau}(\mathtt{p}_{7}+\mathtt{p}_{8})
\]
Therefore, the pair $(\mathtt{p}_{5}+\mathtt{p}_{6},\mathtt{p}_{7}+\mathtt{p}_{8})$ belongs to the pullback $R_D^T(\tau)$, which implies that the equation $\mathtt{p}_{5}+\mathtt{p}_{6}=\mathtt{p}_{7}+\mathtt{p}_{8}$ holds in the coequalizer object $D_{E}T(\tau)$. Similarly, the relations shown below on the left (all holding in $DFE_1^{\varepsilon}(\tau)$) lead to the corresponding right-hand side relations in $D_{E}T(\tau)$ (which are expressed as kinships relationships.
\[
\begin{array}{lll}
\mathsf{fml}_{\tau}(\mathtt{p}_{8}+\mathtt{p}_{9}+\mathtt{p}_{10}) \triangleright \mathsf{fml}_{\tau}(\mathtt{p}_{11}) &\quad\quad\Rightarrow\quad\quad& \mathtt{p}_{8}+\mathtt{p}_{9}+\mathtt{p}_{10} \triangleright \mathtt{p}_{11}\\
\mathsf{fml}_{\tau}(\mathtt{p}_{10}+\mathtt{p}_{11})\triangleright \mathsf{fml}_{\tau}(\mathtt{p}_{9}) & \quad\quad\Rightarrow\quad\quad & \mathtt{p}_{8}+\mathtt{p}_{9}+\mathtt{p}_{10} \triangleright \mathtt{p}_{9}\\
\mathsf{fml}_{\tau}(\mathtt{p}_{10}+\mathtt{p}_{11})\triangleright \mathsf{fml}_{\tau}(\mathtt{p}_{12}) & \quad\quad\Rightarrow\quad\quad & \mathtt{p}_{10}+\mathtt{p}_{11} \triangleright \mathtt{p}_{12}\\
\end{array}
\]
Finally, the kinship relation $\mathsf{fml}_{\tau'}(\mathtt{p}_{10}+\mathtt{p}_{11}) \triangleright \mathsf{fml}_{\tau'}(\mathtt{p}_{8})$, which was found in $DFE_1^{\varepsilon}(\tau')$, gives rise to the following kinship relation in $D_{E}T(\tau')$.
\[
\begin{array}{ccccc}
\mathtt{p}_{10}&+&\mathtt{p}_{11}&\triangleright&\mathtt{p}_{8}\\
\rotatebox[origin=c]{-90}{$
\begin{array}{l}
\mathtt{(aTtCgT)(cAa)(TgcacT)}\\
\fbox{$\mathtt{gAtCcC}$}\,\fbox{$\mathtt{tGa}$}\,\fbox{$\mathtt{TgcctC}$}\\
\end{array}$}
&+&
\rotatebox[origin=c]{-90}{$
\begin{array}{l}
\fbox{$\mathtt{aAcAgT}$}\,\mathtt{(cAa)(TgcatT)}\\
\mathtt{(aTtAgT)(cAa)}\fbox{$\mathtt{GtgaaT}$}\\
\end{array}$}
&\triangleright&
\rotatebox[origin=c]{-90}{$
\begin{array}{l}
\fbox{$\mathtt{aAcAgT}$}\,\fbox{$\mathtt{tGa}$}\,\fbox{$\mathtt{TgcctC}$}\\
\fbox{$\mathtt{gAtCcC}$}\,\fbox{$\mathtt{tGa}$}\,\fbox{$\mathtt{GtgaaT}$}\\
\end{array}$}
\end{array}
\]
Notice how the boxed DNA segments associated with the elements $\mathtt{p}_{10}$ and $\mathtt{p}_{11}$ can be rearranged into the DNA segments associated with $\mathtt{p}_{11}$. This is precisely the type of operations that we have gained in comparison to those considered in Example \ref{exa:Recombination_semimodule_for_DNA}: the main difference between the two functors $DFE_1^{\varepsilon}$ and $D_{E}T$ is the addition of the first Mendelian law, which allows us to shuffle the pair the alleles between each other. This would not have been possible if we had only considered the recombination congruences associated with the underlying functor $F\mathbf{2}E_{b}^{\varepsilon}$ of the sequence alignment $(\iota,T,\sigma)$ (see Example \ref{exa:genotype_haplotype_haplogroup}).
\end{example}

\begin{proposition}\label{prop:pedigrad:fix_points_for_D}
Let $(\Omega,\preceq)$ be a pre-ordered set, let $(\Omega,D)$ be a recombination chromology, let $b$ be an element of $\Omega$, let $(E,\varepsilon)$ be a pointed set and let $(\iota, T,\sigma)$ be a sequence alignment over $\mathbf{2}E_{b}^{\varepsilon}$. The natural transformation $q_{D_ET}:D_ET \Rightarrow D(D_ET)$ is an isomorphism in $[\mathbf{Seg}(\Omega),\mathbf{Icm}]$.
\end{proposition}
\begin{proof}
We prove the statement by showing that the two pairs of arrows associated with the coequalizer construction of $D(D_ET)$ are equal -- thus making the coequalizer construction trivial (see diagram (\ref{eq:coequalizer_mod_pedigrad} and replace $X$ with $D_ET$). First, for every cone $\rho:\Delta_{[k]}(\tau) \Rightarrow \theta$ in $D$, the natural transformation $\mathsf{sub}:D_ET \Rightarrow DFE_b^{\varepsilon}$ of Remark \ref{rem:D_ET_preciting_phenotypes} allows us to construct a natural transformation of pullback squares between the congruence relation $G(D_ET,\rho) \rightrightarrows D_ET(\tau)$ of $D_ET$ and the congruence relation $G(DFE_b^{\varepsilon},\rho) \rightrightarrows DFE_b^{\varepsilon}(\tau)$ of $DFE_b^{\varepsilon}$. Further, the naturality of $\mathsf{sub}:D_ET \Rightarrow DFE_b^{\varepsilon}$ allows us the construct a natural transformation between the version of diagram (\ref{eq:coequalizer_mod_pedigrad}) for $D_ET$ and the version for $DFE_b^{\varepsilon}$, as shown below.
\[
\xymatrix@C+18pt{
\mathop{\coprod}\limits_{\rho,f:\upsilon \to \tau} G(D_ET,\rho)\ar[d]_{}\ar@<-1ex>[r]_-{\coprod\mathsf{prj}_2}\ar@<+1ex>[r]^-{\coprod\mathsf{prj}_1}&\mathop{\coprod}\limits_{\rho,f:\upsilon \to \tau}  D_ET(\upsilon) \ar[rr]^-{\mathop{\oplus}\limits_{\rho,f}D_ET(f)}\ar[d]_{\coprod\mathsf{sub}_{\upsilon}}&& D_ET(\tau)\ar[d]^{\mathsf{sub}_{\tau}}\\
\mathop{\coprod}\limits_{\rho,f:\upsilon \to \tau'} G(DFE_b^{\varepsilon},\rho)\ar@<-1ex>[r]_-{\coprod\mathsf{prj}_2}\ar@<+1ex>[r]^-{\coprod\mathsf{prj}_1}&\mathop{\coprod}\limits_{\rho,f:\upsilon \to \tau'}  DFE_b^{\varepsilon}(\upsilon) \ar[rr]_-{\mathop{\oplus}\limits_{\rho,f}DFE_b^{\varepsilon}(f)}&& DFE_b^{\varepsilon}(\tau')
}
\]
Since the bottom pair of arrows are equal (since, by design, recombination congruences hold in the images of the coequalizer $DFE_b^{\varepsilon}$) and the rightmost arrow $\mathsf{sub}_{\tau}:D_ET(\tau) \to DFE_b^{\varepsilon}(\tau)$ is a monomorphism (see the end of Remark \ref{rem:D_ET_preciting_phenotypes}), it follows that the top pair of arrows are also equal.
\end{proof}

\subsection{Recombination schemes}\label{ssec:Recombination_schemes_and_pedigrads}
The goal of this section is to determine a set of conditions for which a recombination monoid, as given in Definition \ref{def:Canonical_pedigrads_in_semimodules}, is a pedigrad in a certain logical system of monomorphisms (see Corollary \ref{cor:D_ET_is_a_mon_pedigrad}). This will allow us to deduce strategies about how haplotypes can be encoded in a given context of recombination events.

\begin{definition}[Logical system]
We will denote by $\mathcal{W}^{\mathrm{mon}}$ the class of wide spans $\mathbf{S}=\{X \to F_i\}_{i \in [k]}$ in $\mathbf{Icm}$ whose product adjoint arrows $X \to \prod_{i \in [k]}F_i$ is a monomorphism in $\mathbf{Icm}$.
\end{definition}

\begin{remark}[Homologous recombination]
A $\mathcal{W}^{\mathrm{mon}}$-pedigrad is a functor whose recombination congruences are actually realized in its images. Another way to put it is to say that a $\mathcal{W}^{\mathrm{mon}}$-pedigrad is a functor in which homologous recombination is modeled (see Example \ref{exa:Recombination_semimodule_for_DNA}).
\end{remark}

\begin{definition}[Irreducibility]\label{def:irreducible_object}
Let $(\Omega,D)$ be a recombination chromology and $X$ be a functor $\mathbf{Seg}(\Omega) \to \mathbf{Set}$. An object $\tau$ in $\mathbf{Seg}(\Omega)$ will be said to be \emph{irreducible} for the triple $(\Omega,D,X)$ if for every arrow $f:\upsilon \to \tau$ in $\mathbf{Seg}(\Omega)$, the arrow $X(f):X(\upsilon) \to X(\tau)$ coequalizes the following pair of arrows for every cone $\rho:\Delta_A(\upsilon) \Rightarrow \theta$ in $D$.
\begin{equation}\label{eq:Irreducibility}
\xymatrix@C+18pt{
G(X,\rho)\ar@<-1ex>[r]_-{\mathsf{prj}_2}\ar@<+1ex>[r]^-{\mathsf{prj}_1}& X(\upsilon) 
}
\end{equation}
\end{definition}

\begin{remark}[Coequalizing arrows]\label{rem:Coequalizing_arrows}
By Definition \ref{def:Recomb_congruences}, the pair of arrows given in (\ref{eq:Irreducibility}) is coequalized by each canonical arrow $X(\rho_i):X(\tau) \to X(\theta(j))$ obtained from each component $\rho_i$ of the cone $\rho$. Indeed, note that, for every $j \in [k]$, the composition of the arrow $X[\rho]:X(\upsilon) \to \prod_{i \in [k]} X \circ \theta(i)$ with the product projection
\[
\prod_{i \in [k]} X \circ \theta(i) \to X \circ \theta(j)
\]
is equal to the morphism $X(\rho_j):X(\tau) \to X(\theta(j))$. Because the arrow $X[\rho]$ coequalizes the pair in (\ref{eq:Irreducibility}) (according to Definition \ref{def:Recomb_congruences}), we deduce that, for every $i \in [k]$, the arrow $X(\rho_i)$ coequalizes the pair given in (\ref{eq:Irreducibility}) for the corresponding cone $\rho$.
\end{remark}

\begin{example}[Coequalizing arrows]\label{exa:Coequalizing_arrows}
Let us consider the same setting as the one used in Example \ref{exa:Relative_definition_families}. We shall illustrate the concept of irreducibility by using the cone $\rho$ given thereof. 

First, Remark \ref{rem:Coequalizing_arrows} implies that the recombination congruence $G(FE_1^{\varepsilon},\rho) \rightrightarrows FE_1^{\varepsilon}(\tau)$ is coequalized by the image of the following arrows via the functor $FE_1^{\varepsilon}:\mathbf{Seg}(\Omega) \to \mathbf{Icm}$.
\[
\begin{array}{l}
\rho_{1}:\xymatrix@C-30pt@R-20pt{
(\bullet&\bullet&\bullet&\bullet&\bullet&\bullet)&(\bullet&\bullet&\bullet&\bullet&\bullet&\bullet)&(\bullet&\bullet&\bullet)\ar[rr]
&\quad\quad\quad&
(\bullet&\bullet&\bullet&\bullet&\bullet&\bullet)&(\circ&\circ&\circ&\circ&\circ&\circ)&(\circ&\circ&\circ)
}\\
\rho_{2}:\xymatrix@C-30pt@R-20pt{
(\bullet&\bullet&\bullet&\bullet&\bullet&\bullet)&(\bullet&\bullet&\bullet&\bullet&\bullet&\bullet)&(\bullet&\bullet&\bullet)\ar[rr]
&\quad\quad\quad&
(\circ&\circ&\circ&\circ&\circ&\circ)&(\bullet&\bullet&\bullet&\bullet&\bullet&\bullet)&(\circ&\circ&\circ)
}\\
\rho_{3}:\xymatrix@C-30pt@R-20pt{
(\bullet&\bullet&\bullet&\bullet&\bullet&\bullet)&(\bullet&\bullet&\bullet&\bullet&\bullet&\bullet)&(\bullet&\bullet&\bullet)\ar[rr]
&\quad\quad\quad&
(\circ&\circ&\circ&\circ&\circ&\circ)&(\circ&\circ&\circ&\circ&\circ&\circ)&(\bullet&\bullet&\bullet)
}
\end{array}
\]
For example, we can easily verify (see the table below) that these arrows send the pair of elements $\mathsf{fml}_{\tau}(\mathtt{p}_{5}+\mathtt{p}_{6})$ and $\mathsf{fml}_{\tau}(\mathtt{p}_{7}+\mathtt{p}_{8})$ to the same elements in $FE_1^{\varepsilon}(\rho_{1})$, $FE_1^{\varepsilon}(\rho_{2})$ and $FE_1^{\varepsilon}(\rho_{3})$, respectively.
\[
\begin{array}{|l|c|}
\hline
\cellcolor[gray]{0.8}&\cellcolor[gray]{0.8}\textrm{Images of $\mathsf{fml}_{\tau}(\mathtt{p}_{5}+\mathtt{p}_{6})$ or $\mathsf{fml}_{\tau}(\mathtt{p}_{7}+\mathtt{p}_{8})$}\\
\hline
\cellcolor[gray]{0.8}\textrm{via }FE_b^{\varepsilon}(\rho_{1})& \mathtt{gAtCcC} + \mathtt{aAcAgT}+\mathtt{gAtCcC}+\mathtt{gTtCcC}\\
\hline
\cellcolor[gray]{0.8}\textrm{via }FE_b^{\varepsilon}(\rho_{2})& \mathtt{tAtGtc}+ \mathtt{cAtTgc} + \mathtt{tGaTgc}+ \mathtt{tGaGtg}\\
\hline
\cellcolor[gray]{0.8}\textrm{via }FE_b^{\varepsilon}(\rho_{3})&
\mathtt{aaT} + \mathtt{atT} +\mathtt{ctC} +  \mathtt{aaT}\\
\hline
\end{array}
\]
This suggests that the codomains of the arrows $\rho_{1}$, $\rho_{2}$ and $\rho_{3}$ are good candidates for being irreducible objects with respect to a triple of the form $(\{0,1\},D,FE_{1}^{\varepsilon})$. This phenomenon will be discussed in Proposition \ref{prop:compare_cone_irreducible} and Remark \ref{rem:single_cone_irreducible}. Note that the irreducibility of the domains of the arrows of a cone cannot be determined solely by the shape of its associated arrows and also depends on the other cones of the underlying chromology (see Proposition \ref{prop:compare_cone_irreducible}).
\end{example}

\begin{proposition}[Irreducibility]\label{prop:compare_cone_irreducible}
Let $(\Omega,D)$ be a recombination chromology and $X$ be a functor $\mathbf{Seg}(\Omega) \to \mathbf{Icm}$. Let $\rho:\Delta_{[k]}(\tau) \Rightarrow \theta$ be a cone in $D$ and take an element $i \in [k]$ for which the following property is satisfied:
\begin{itemize}
\item[$(\ast)$] for every integer $n \geq 0$ and every cone $\varrho:\Delta_{[q]}(\upsilon) \Rightarrow \vartheta$ in $D[n]$, if there exists an arrow $\upsilon \to \theta(i)$ in $\mathbf{Seg}(\Omega)$, then there exists an element $j \in [q]$ such that there is an arrow $\vartheta(j) \to \theta(i)$ in $\mathbf{Seg}(\Omega\,|\,n)$.
\end{itemize}
It follows that the object $\theta(i)$ in $\mathbf{Seg}(\Omega)$ is irreducible for $(\Omega,D,X)$.
\end{proposition}
\begin{proof}
Let $\rho:\Delta_{[k]}(\tau) \Rightarrow \theta$ be a cone in $D$ as described in the statement. Let us take $i \in [k]$ and let $f:\upsilon \to \theta(i)$ be an arrow in $\mathbf{Seg}(\Omega)$. Let us show that, for every non-negative integer $n$ and every cone $\varrho:\Delta_{[q]}(\upsilon) \Rightarrow \vartheta$ in $D[n]$, the image $X(f):X(\upsilon) \to X(\theta(i))$ coequalizes the coequalizer $G(X,\varrho) \rightrightarrows X(\upsilon)$. By assumption, the arrow $f:\upsilon \to \theta(i)$ implies that there exists $j \in [q]$ such that there is an arrow $t:\vartheta(j) \to \theta(i)$ in $\mathbf{Seg}(\Omega\,|\,n)$. Hence, the composite $t \circ \varrho(j)$ and the arrow $f$ are two arrows of the form $\upsilon \to \theta(i)$. Because the arrow $\varrho(j):\upsilon \to \vartheta(j)$ belongs to $\mathsf{Seg}(\Omega\,|\,n)$, the composite $t \circ \varrho(j)$ also belongs to $\mathsf{Seg}(\Omega\,|\,n)$. By Lemma \ref{lem:quasi_homologous_preordered_category}, the arrow $t \circ \varrho(j)$ must be the only arrow of type $\upsilon \to \vartheta(j)$ in $\mathsf{Seg}(\Omega)$, which means that we have $t \circ \varrho(j) = f$.
By Remark \ref{rem:Coequalizing_arrows}, this means that the image $X(f):X(\upsilon) \to X(\vartheta(i))$ coequalizes the pair $G(X,\varrho) \rightrightarrows X(\upsilon)$. This shows the irreducibly of $\theta(i)$.
\end{proof}

\begin{remark}[A single cone]\label{rem:single_cone_irreducible}
Let $(\Omega,D)$ be a recombination chromology. If the chromology $D$ contains a single cone $\rho$, then this cone satisfies property $(\ast)$, stated in Proposition \ref{prop:compare_cone_irreducible}. Indeed, for every other cone $\varrho:\Delta_{[q]}(\upsilon) \Rightarrow \vartheta$ in $D$, we have the equality $\varrho = \rho$, hence the equality $\vartheta(i) = \theta(i)$ for every $i \in [k]$, which implies property $(\ast)$.
\end{remark}

\begin{proposition}\label{prop:irreducible_coequalizer_map_isomorphism}
Let $(\Omega,D)$ be a recombination chromology and $X$ be a functor $\mathbf{Seg}(\Omega) \to \mathbf{Icm}$. For every object $\tau$ in $\mathbf{Seg}(\Omega)$ that is irreducible for the triple $(\Omega,D,X)$, the coequalizer map $q_{X}:X(\tau) \to DX(\tau)$ (Convention \ref{conv:coequalizer_map})  associated with the recombination monoid over $X$ (Definition \ref{def:Canonical_pedigrads_in_semimodules}) is an isomorphism.
\end{proposition}
\begin{proof}
By Definition \ref{def:irreducible_object} and the universality of the coequalizers shown in (\ref{eq:coequalizer_mod_pedigrad}).
\end{proof}

\begin{definition}[Contexts]\label{def:recombination-context}
Let $(\Omega,D)$ be a recombination chromology and $X$ be a functor $\mathbf{Seg}(\Omega) \to \mathbf{Icm}$. For every subchromology $(\Omega,D') \subseteq (\Omega,D)$ (Convention \ref{conv:subchromologies}), we will say that the triple $(\Omega,D,X)$ defines a \emph{$D'$-context} if for every cone $\rho:\Delta_{[k]}(\tau) \Rightarrow \theta$ in $D'$ and every object $i \in [k]$, the object $\theta(i)$ is irreducible for the triple $(\Omega,D,X)$ (Definition \ref{def:irreducible_object}).
\end{definition}

\begin{remark}[Contexts]\label{rem:recombination-context}
Proposition \ref{prop:compare_cone_irreducible} suggests that we can construct recombination chromologies $(\Omega,D)$ that are $D$-contexts for any arbitrary functor $\mathbf{Seg}(\Omega) \to \mathbf{Icm}$. Indeed, if we let $(\Omega,D)$ denote a recombination chromology whose cones $\rho:\Delta_{[k]}(\tau) \Rightarrow \theta$ satisfy property $(\ast)$ for every element $i \in [k]$ (as stated in Proposition \ref{prop:compare_cone_irreducible}), then every functor $X:\mathbf{Seg}(\Omega) \to \mathbf{Icm}$ gives rise to a $D$-context $(\Omega,D,X)$.
\end{remark}

The following theorem shows that the recombination monoid resulting from a given context is a $\mathcal{W}^{\mathrm{mon}}$-pedigrad for the subchromology of the context.

\begin{theorem}\label{theo:representable_pedigrad_E_b_varepsilon}
Let $(\Omega,D,X)$ be a $D'$-context. For every cone $\rho:\Delta_{[k]}(\tau) \Rightarrow \theta$ in $D'$, the product adjoint $\iota:DX(\tau) \to \prod_{i \in [k]}DX(\theta(i))$ is a monomorphism in $\mathbf{Icm}$.
\end{theorem}
\begin{proof}
By Definition \ref{def:Recomb_congruences}, the pullback of two copies of the canonical arrow $X[\rho]$ (see Convention \ref{conv:definition_pi_S}) is the recombination congruence $G(X,\rho)\rightrightarrows X(\tau)$. If we denote by $p_1,p_2:P \rightrightarrows DX(\tau)$ the pullback of two copies of the product adjoint $\iota$ (see statement), the naturality of the coequalizer map $q_X:X \to DX$ gives us an arrow $\lambda:X(\rho) \to P$ making the following diagram commute.
\begin{equation}\label{eq:proof_mono_pedigrad}
\xymatrix@-10pt{
&G(X,\rho)\ar[dd]|\hole_<<<<<{\mathsf{prj}_2}\ar[ld]_{\mathsf{prj}_1}\ar@{-->}[rr]^{\lambda}&&P\ar[dd]_{p_2}\ar[ld]_{p_1}\\
X(\tau)\ar[rr]^>>>>>>>>{q_X}\ar[dd]_{X[\rho]}&&DX(\tau)\ar[dd]^>>>>>>{\iota}&\\
&X(\tau)\ar[ld]_{X[\rho]}\ar[rr]|\hole^<<<<<<<<<{q_X}&&DX(\tau)\ar[ld]^{\iota}\\
\prod_{i \in [k]}X(\theta(i))\ar[rr]_{\cong}^{\prod_i q_{X}}&&\prod_{i \in [k]}DX(\theta(i))&
}
\end{equation}

Let us show that $\lambda$ is an epimorphism in $\mathbf{Icm}$ by showing it is orthogonal with respect to the Boolean ic-monoid $B_2$ (see Proposition \ref{prop:characterization_epi}). 
First, because $(\Omega,D,X)$ is a $D'$-context and the cone $\rho$ belongs to the chromology $D'$, Proposition \ref{prop:irreducible_coequalizer_map_isomorphism} implies that the bottom front arrow of diagram (\ref{eq:proof_mono_pedigrad}) is an isomorphism (as shown by the symbol $\cong$ in the diagram).
Second, because $q_X:X(\tau) \to DX(\tau)$ is a coequalizer map, it is  orthogonal with respect to the Boolean ic-monoid $B_2$ (Remark \ref{rem:characterization_epi_coequalizer_maps}).
These two facts imply that, for every arrow $x:B_2 \to P$, the composite arrows $p_1 \circ x:B_2 \to DX(\tau)$ and  $p_2 \circ x:B_2 \to DX(\tau)$ admit lifts $h_1:B_2 \to X(\tau)$ and $h_2:B_2 \to X(\tau)$ along $q_X$ that make the following diagram commute.
\[
\xymatrix{
B_2\ar[r]^{h_1}\ar[d]_{h_2}&X(\tau)\ar[d]^{X[\rho]}\\
X(\tau)\ar[r]_{X[\rho]}&*+!L(.7){\prod_{i \in [k]}X(\theta(i))}
}
\]
Since $G(X,\rho)$ is the pullback of $X[\rho]$ along itself, the previous diagram and the universality of the pullback $P$ give us an arrow $h:B_2 \to G(X,\rho)$ for which the equation $\lambda \circ h = x$ holds. In other words, the arrow $\lambda$ is orthogonal to the Boolean ic-monoid $B_2$ and is hence an epimorphism by Proposition \ref{prop:characterization_epi}.

Now, because the equation $q_X \circ \mathsf{prj}_1 = q_X \circ \mathsf{prj}_2$ holds (by definition of $q_X$), and because we showed that $\lambda$ is an epimorphism, the two arrows $p_1,p_2:P \rightrightarrows DX(\tau)$ must be equal (see diagram (\ref{eq:proof_mono_pedigrad})). Because this pair of arrows is also the pullback of two copies of $\iota$, the arrow $\iota$ is a monomorphism (Proposition \ref{prop:characterization_mono}).
\end{proof}

The following result explains, in abstract terms, why Corollary \ref{cor:D_ET_is_a_mon_pedigrad} (our main result) holds.

\begin{theorem}\label{theo:morphism_to_mon_pedigrad}
Let $(\Omega,D)$ be a chromology and $X:\mathbf{Seg}(\Omega) \to \mathbf{Icm}$ be a $\mathcal{W}^{\mathrm{mon}}$-pedigrad for $(\Omega,D)$. For every morphism $f:Y \Rightarrow X$ in $[\mathbf{Seg}(\Omega),\mathbf{Icm}]$, denote as $Y_f$ the coequalizer of the pullback pair of $f$ along itself. The resulting functor $Y_f:\mathbf{Seg}(\Omega) \to \mathbf{Icm}$ is a $\mathcal{W}^{\mathrm{mon}}$-pedigrad for $(\Omega,D)$ (see the diagrams below).
\[
\begin{array}{l}
\xymatrix{
R \ar@{=>}[d]_{u_2} \ar@{=>}[r]^-{u_1}  \ar@{}[rd]|<<<{\rotatebox[origin=c]{90}{\huge{\textrm{$\llcorner$}}}} &Y\ar@{=>}[d]^{f}\\
Y\ar@{=>}[r]_-{f}&X
}
\end{array}
\quad\quad\quad\quad\quad\quad\quad
\begin{array}{l}
\xymatrix{
R  \ar@<+1.2ex>@{=>}[r]^-{u_1}\ar@<-1.2ex>@{=>}[r]_-{u_1}& Y \ar@{==>}[r]^-{q}&Y_f
}
\end{array}
\]
\end{theorem}
\begin{proof}
Since the morphism $f: Y \Rightarrow X$ coequalizes the pair $(u_1,u_2)$, there exists a unique morphism $f':Y_f \Rightarrow X$ in $[\mathbf{Seg}(\Omega),\mathbf{Icm}]$ for which the equation $f = f' \circ q$ holds. Our goal is to show that, for every cone $\rho:\Delta_A(\tau) \Rightarrow \theta$ in $D$, the limit adjoint $Y_f(\tau) \to  \mathsf{lim}_{A}Y_f \circ \theta$ is a monomorphism. To do so, notice that the natural transformation $f':Y_f \Rightarrow X$ gives us the following commutative diagram.
\[
\xymatrix@C+10pt{
Y_f(\tau) \ar[r]^{f'_{\tau}}\ar[d] &X(\tau) \ar[d]\\
\mathsf{lim}_{A}Y_f \circ \theta \ar[r]_{\mathsf{lim}_Af'_{\theta}}& \mathsf{lim}_{A}X \circ \theta
}
\]
Since the arrow $X(\tau) \to  \mathsf{lim}_{A}X \circ \theta$ is a monomorphism by assumption on $X$ (\emph{i.e.} it is a $\mathcal{W}^{\mathrm{mon}}$-pedigrad) and since the arrow $f':Y_f \Rightarrow X$ is a pointwise monomorphism (by Proposition \ref{prop:pullback:coequalizer:mono}) and limits preserve monomorphisms, the universal property of monomorphisms (section \ref{ssec:Reminder_monomorphisms_epimorphisms}) implies that the arrow $Y_f(\tau) \to  \mathsf{lim}_{A}Y_f \circ \theta$ must be a monomorphism.
\end{proof}

\begin{corollary}\label{cor:D_ET_is_a_mon_pedigrad}
Let $(E,\varepsilon)$ be a pointed set, let $(\Omega,D)$ be a recombination chromology and let $b$ be an element in $\Omega$. Suppose that either the triple $(\Omega,D,FE_{b}^{\varepsilon})$ or the triple $(\Omega,D,D_ET)$ defines a $D'$-context for some subchromology $(\Omega,D') \subseteq (\Omega,D)$. For every sequence alignment $(\iota,T,\sigma)$ over $\mathbf{2}E_{b}^{\varepsilon}$, the functor $D_ET$ (see Remark \ref{rem:D_ET_preciting_phenotypes}) is a $\mathcal{W}^{\mathrm{mon}}$-pedigrad for $(\Omega,D')$.
\end{corollary}
\begin{proof}
This proof takes different approaches depending on whether the triple $(\Omega,D,FE_{b}^{\varepsilon})$ or the triple $(\Omega,D,D_ET)$ defines a $D'$-context. Nevertheless, all our approaches rely on Theorem \ref{theo:representable_pedigrad_E_b_varepsilon}.

Suppose that the triple $(\Omega,D,FE_{b}^{\varepsilon})$ defines a $D'$-context. First, recall (from Remark \ref{rem:D_ET_preciting_phenotypes}) that the object $D_ET$ is the coequalizer of the pullback pair resulting from the pullback of the morphism $\mathsf{rec}:F\mathsf{lan}_{\iota}T \Rightarrow DFE_{b}^{\varepsilon}$ along itself. Second, it follows from Theorem \ref{theo:representable_pedigrad_E_b_varepsilon} that the codomain of this morphism $\mathsf{rec}$ is a $\mathcal{W}^{\mathrm{mon}}$-pedigrad for $(\Omega,D')$. These two facts, coupled with Theorem \ref{theo:morphism_to_mon_pedigrad}, imply that $D_ET$ is a $\mathcal{W}^{\mathrm{mon}}$-pedigrad for $(\Omega,D')$.

On the other hand, if the triple $(\Omega,D,D_ET)$ defines a $D'$-context, then the statement follows from Proposition \ref{prop:pedigrad:fix_points_for_D} and Theorem \ref{theo:representable_pedigrad_E_b_varepsilon}. Indeed, these results imply that the functor $D_ET$ is isomorphic to the functor $D(D_ET)$, which is also a $\mathcal{W}^{\mathrm{mon}}$-pedigrad for $(\Omega,D')$.
\end{proof}

\begin{remark}[Pedigrads and encodings]\label{rem:local-property:free-ic-monoids:alignments}
This example discusses the feedback and learning dynamics that can result from using context structures and pedigrad structures together. The point being made here is that pedigrads and contexts are to work together to form a calculus framework. Importantly, we will emphasize the fact that, although contexts give rise to pedigrad functors (see Theorem \ref{theo:morphism_to_mon_pedigrad}), a lack of context structure may not necessarily imply a lack of pedigrad structure. This subtle difference puts pedigrad at center of our analytical framework to study haplotypes.

Let $(E,\varepsilon)$ be a pointed set, let $(\Omega,D)$ and $(\Omega,D')$ be two recombination chromologies such that $(\Omega,D') \subseteq (\Omega,D)$ and let $b$ be an element in $\Omega$ for which the triple $(\Omega,D,FE_{b}^{\varepsilon})$ defines a $D'$-context. For every sequence alignment $(\iota,T,\sigma)$ over $\mathbf{2}E_{b}^{\varepsilon}$ and cone $\rho:\Delta_{[k]}(\tau) \Rightarrow \theta$ in $D'$, note that the functor $D_ET$ is equipped, for every $i \in [k]$, with a canonical monomorphism
\begin{equation}\label{eq:local-property:free-ic-monoids:alignments:composed}
D_ET(\theta(i)) \hookrightarrow  B_2^{E_b^{\varepsilon}(\theta(i))}
\end{equation}
resulting from the composition of the three monomorphisms given in (\ref{eq:local-property:free-ic-monoids:alignments:total}), where
\begin{itemize}
\item[-] the rightmost monomorphism is the one described in Remark \ref{rem:D_ET_preciting_phenotypes};
\item[-] the middle monomorphism results from Proposition \ref{prop:irreducible_coequalizer_map_isomorphism} and Definition \ref{def:recombination-context} as the triple $(\Omega,D,FE_{b}^{\varepsilon})$ defines a $D'$-context;
\item[-] and the leftmost monomorphism is the one described in Example \ref{exa:monomorphism-encoding}.
\end{itemize}
\begin{equation}\label{eq:local-property:free-ic-monoids:alignments:total}
\xymatrix@C+35pt{
D_ET(\theta(i)) \ar[r]^-{\mathsf{sub}_{\theta(i)}} & DFE_b^{\varepsilon}(\theta(i)) \ar[r]^-{q^{-1}_{FE_b^{\varepsilon}(\theta(i))}} & FE_b^{\varepsilon}(\theta(i)) \ar[r]^-{\varphi_{E_b^{\varepsilon}(\theta(i))}} & B_2^{E_b^{\varepsilon}(\theta(i))}
}
\end{equation}
As suggested in Example \ref{exa:monomorphism-encoding}, monomorphism (\ref{eq:local-property:free-ic-monoids:alignments:composed}) essentially gives a way to encode the ic-monoids $D_ET(\theta(i))$ in a computer. It follows from Corollary \ref{cor:D_ET_is_a_mon_pedigrad} that we can use these encodings to give an encoding for the object $D_ET(\tau)$: specifically, the fact that $D_ET$ is a $\mathcal{W}^{\mathrm{mon}}$-pedigrad for $(\Omega,D')$ gives the following monomorphism.
\[
\textstyle D_ET[\rho]:\textstyle D_ET(\tau) \to \prod_{i \in [k]} D_ET(\theta(i))
\]
By using monomorphism  (\ref{eq:local-property:free-ic-monoids:alignments:composed}), we can construct the following monomorphism, which gives a way to encode the elements of the ic-monoid $D_ET(\tau)$.
\begin{equation}\label{eq:local-property:free-ic-monoids:alignments:1}
D_ET(\tau) \hookrightarrow \prod_{i \in [k]} D_ET(\theta(i)) \hookrightarrow  \prod_{i \in [k]} B_2^{E_b^{\varepsilon}(\theta(i))} \cong B_2^{E_b^{\varepsilon}(\theta(1))+E_b^{\varepsilon}(\theta(2))+\dots+E_b^{\varepsilon}(\theta(k))}
\end{equation}
This encoding morphism is essentially the ``extra structure'' that a $\mathcal{W}^{\mathrm{mon}}$-pedigrad is meant to provide. To better appreciate this detail, remember that the object $D_ET(\tau)$ is a quotient of the free ic-monoid
\begin{equation}\label{eq:local-property:free-ic-monoids:alignments:2}
F(\mathsf{Lan}_{\iota}F(\tau)) \cong B_2^{\mathsf{Lan}_{\iota}F(\tau)},
\end{equation}
and it would usually be a non-trivial exercise to determine whether such a quotient can be injected into a free ic-monoid. Hence, the functor structure of $D_ET$ allows us to reduce the ambiguous encoding associated with a quotient of the ic-monoid shown in (\ref{eq:local-property:free-ic-monoids:alignments:2}) to the more straightforward encoding shown in (\ref{eq:local-property:free-ic-monoids:alignments:1}).

This being said, the triple $(\Omega,D,FE_{b}^{\varepsilon})$ may not always define a $D'$-context for a given subchromology $(\Omega,D')$, and the main motivation for using our pedigrad framework is to determine whether we can find some sequence of recombination subchromologies (see below) that decomposes $D'$ to permit further analysis.
\[
(D^0,\Omega)  \subseteq (D^1 ,\Omega) \subseteq \dots \subseteq (D^n ,\Omega) = (D ,\Omega)
\]
Note that such a sequence would give us a functorial filtration in $[\mathbf{Seg}(\Omega),\mathbf{Icm}]$ as shown below (this sequence results from the universality of the construction shown in Remark \ref{rem:functorial_quotients}).
\[
D^0_ET \Rightarrow D^1_ET \Rightarrow \dots \Rightarrow D^n_ET = D_ET
\]
For every pair $(j,k)$ of integers $j \leq k \leq n$, it is then of interest to determine whether the triple $(\Omega,D^k,D^k_ET)$ is a $D^j$-context. Given that all the cones of the chromology $D^k$ are `forced' to interact between each other through the recombination congruences realized into equations within the functor $D_E^kT$, the fact that the triple $(\Omega,D^k,D^k_ET)$ is a $D^j$-context means that the cones of $D^k$ integrate the information contained in the cones of $D^j$ through the pedigrad property. To understand what this means (at a practical level) for the encoding morphisms shown in (\ref{eq:local-property:free-ic-monoids:alignments:1}), let us take a cone $\varrho:\Delta_{[q]}(\upsilon) \Rightarrow \vartheta$ in $D^j$ and let us consider the resulting diagram of encoding morphisms shown below.
\[
\xymatrix{
D^{0}_ET(\upsilon)\ar[r]\ar[d]^{D^0_ET[\rho]}&\dots \ar[r]&D^j_ET(\upsilon)\ar[r]\ar[d]_{D^j_ET[\rho]}&D^{j+1}_ET(\upsilon)\ar[r]\ar[d]_{D^{j+1}_ET[\rho]}&\dots \ar[r]&D^k_ET(\upsilon)\ar[d]_{D^k_ET[\rho]}\\
\prod_iD^{0}_ET(\vartheta(i))\ar[r]&\dots \ar[r]&\prod_iD^j_ET(\vartheta(i))\ar[r]&\prod_iD^{j+1}_ET(\vartheta(i))\ar[r]&\dots \ar[r]&\prod_iD^k_ET(\vartheta(i))\\
}
\]
When the triple $(\Omega,D^k,D^k_ET)$ defines a $D^j$-context, Proposition \ref{prop:irreducible_coequalizer_map_isomorphism} suggests that the horizontal arrows $D^l_ET(\vartheta(i)) \to D^{l+1}_ET(\vartheta(i))$ and $D^{l+1}_ET(\upsilon) \to D^l_ET(\upsilon)$ of the previous diagram are isomorphisms for every integer $k-1 \geq l \geq j$. In other words, the encodings provided by the cones of $D^j$ should not change when passing from the pedigrad $D^j_ET$ to the pedigrad $D^k_ET$. If the triple $(\Omega,D^k,D^k_ET)$ did not define a $D^j$-context, then the encoding morphisms $D^{l}_ET[\rho]$ could change through the filtration. Although the encodings may change, note that the pedigrad property can still hold. In this case, the encoding of the objects $\prod_iD^l_ET(\upsilon)$ would still be guaranteed to be compatible with those of the objects $D^l_ET(\vartheta(i))$ (see Example \ref{exa:pedigrads_and_encodings}).

In practice, a filtration such as the one presented above gives us a framework to compare encoding morphisms cone by cone and understand how much information is changed through the filtration. In practice, we can do this by looking for irreducible objects through property $(\ast)$ stated in Proposition \ref{prop:compare_cone_irreducible}.
\end{remark}

\begin{example}[Pedigrads and encodings]\label{exa:pedigrads_and_encodings}
The goal of this remark is to illustrate the discussion of Remark \ref{rem:local-property:free-ic-monoids:alignments}. Let $(\Omega,\preceq)$ be the Boolean pre-ordered set $\{0 \leq 1\}$ and let $(E,\varepsilon)$ be our usual pointed set $\{\mathtt{A},\mathtt{C},\mathtt{G},\mathtt{T},\varepsilon\}$. Take $\rho:\Delta_{[3]}(\tau) \Rightarrow \theta$ and $\rho':\Delta_{[3]}(\tau') \Rightarrow \theta'$ to be the two wide spans of homologous segments given in Example \ref{exa:Relative_definition_families}. Let us take $(\Omega,D)$ to be the recombination chromology that contains the cones $\rho$ and $\rho'$ (as shown below).
\[
D[n] =
\left\{
\begin{array}{ll}
\{\rho,\rho'\} &  \textrm{ if $n = 15$.}\\
\cmemptyset & \textrm{ if $n \neq 15$.}
\end{array}
\right.
\]
Because the topologies used to define the cones $\rho$ and $\rho'$ are not ``mappable'' into each other (\emph{i.e.} they cannot factorize one another), the two cones $\rho$ and $\rho'$ satisfy property $(\ast)$ of Proposition \ref{prop:compare_cone_irreducible} for $D$.
Let now $(\iota,T,\sigma)$ be the sequence alignment over $\mathbf{2}E_1^{\varepsilon}$ defined in Example \ref{exa:Sequence alignments}. As explained in Remark \ref{rem:recombination-context}, the two triples $(\Omega,D,FE_b^{\varepsilon})$ and $(\Omega,D,D_ET)$ define two $D$-contexts. By Proposition \ref{cor:D_ET_is_a_mon_pedigrad}, this means that the recombination monoid $D_ET$ is a $\mathcal{W}^{\mathrm{mon}}$-pedigrad for $(\Omega,D)$. Let us discuss the content of Remark \ref{rem:local-property:free-ic-monoids:alignments} from the point of view of the chromology $(\Omega,D)$. First, note that, although Remark \ref{rem:local-property:free-ic-monoids:alignments} gives the monomorphism shown in (\ref{eq:local-property:free-ic-monoids:alignments:composed}) for every cone $\varrho:\Delta_{[3]}(\upsilon) \Rightarrow \vartheta$ in $D$, we can in practice find a subset
\[
S_i(\varrho) \subseteq E_1^{\varepsilon}(\vartheta(i)) = \mathbf{Set}(\mathsf{Tr}_1(\vartheta(i)),E)
\]
and a monomorphism $D_ET(\vartheta(i)) \to B_2^{S_i(\varrho)}$ through which monomorphism (\ref{eq:local-property:free-ic-monoids:alignments:composed}) factorizes.

In practice, the subset $S_i(\varrho)$ is the set of the images of the two canonical functions $T(\vartheta(i)) \to E_1^{\varepsilon}(\vartheta(i))$ resulting from composing the function $\sigma:T(\vartheta(i)) \to \mathbf{2}E_1^{\varepsilon}(\vartheta(i))$ with one of the Cartesian projections $\mathbf{2}E_1^{\varepsilon}(\vartheta(i)) \to E_1^{\varepsilon}(\vartheta(i))$. We can make this construction more explicit if we order the set $E$ according to the alphabetic order, where $\varepsilon$ is taken to be minimal (as shown below).
\[
\{\varepsilon<\mathtt{A}<\mathtt{C}<\mathtt{G}<\mathtt{T}\}
\]
In this case, the set $S_i(\varrho)$ inherits a natural lexicographic order induced by the previous ordered set. This lexicographic order then allows us to encode the elements of $D_ET(\vartheta(i))$ as tuples of elements in $B_2$ (see the tables below). To construct these tuples for the cone $\rho$, first decompose each pair of DNA sequences given in the table at the end of Example \ref{exa:Sequence alignments} according to the arrows of $\rho$, as was done in Example \ref{exa:genotype_haplotype_haplogroup}, and take each bracketed DNA segments to be the elements of the sets $S_i(\varrho)$ for the corresponding domains of $\rho$, as given below.
\[
\begin{array}{l}
S_1(\rho) =\{\mathtt{aAcAgT}<\mathtt{aAtAcC}<\mathtt{aAtCgC}<\mathtt{aTtAgT}<\mathtt{aTtCgT}<\mathtt{gAtCcC}<\mathtt{gTtCcC}\}\\
S_2(\rho) =\{\mathtt{cAaGtg}<\mathtt{cAaTgc}<\mathtt{cAtTgc}<\mathtt{tAtGtc}<\mathtt{tAtTtc}<\mathtt{tGaGtg}<\mathtt{tGaTgc}\}\\
S_3(\rho) =\{\mathtt{aaT}<\mathtt{acT}<\mathtt{atT}<\mathtt{ctC}\}\\
\end{array}
\]
This being done, we can now turn the table shown at the end Example \ref{exa:Sequence alignments} into a table of $S_i(\rho)$-indexed tuples with coefficients in $B_2$ for every $i \in [3]$. Specifically, the table given below represents mappings of the generating elements of $D_ET(\tau)$ to these tuples through the morphism represented in (\ref{eq:local-property:free-ic-monoids:alignments:1}). Note that because each individual, in the sequence alignment $T$, possesses at most two alleles for each segment $\theta(i)$, the corresponding $S_i(\rho)$-indexed tuple contains at most two symbols $\mathtt{1}$.
\[
\begin{array}{|rrr|}
\hline
\multicolumn{3}{|c|}{\cellcolor[gray]{0.8}D_ET(\tau)  \to B_2^{S_1(\rho)}|B_2^{S_2(\rho)}|B_2^{S_3(\rho)}}\\
\hline
\mathtt{p}_{1} \mapsto
\mathtt{0010010}|\mathtt{1000100}|\mathtt{1100}
&\mathtt{p}_{5} \mapsto
\mathtt{1000010}|\mathtt{0011000}|\mathtt{1010}
&\mathtt{p}_{9} \mapsto
\mathtt{0001100}|\mathtt{1000001}|\mathtt{1010}\\
\hline
\mathtt{p}_{2} \mapsto
\mathtt{0100001}|\mathtt{0010010}|\mathtt{0011}
&\mathtt{p}_{6} \mapsto
\mathtt{0000011}|\mathtt{0000011}|\mathtt{1001}
&\mathtt{p}_{10} \mapsto
\mathtt{0000110}|\mathtt{0100001}|\mathtt{0101}\\
\hline
\mathtt{p}_{3} \mapsto
\mathtt{0110000}|\mathtt{1000010}|\mathtt{1010}
&\mathtt{p}_{7} \mapsto
\mathtt{0000011}|\mathtt{0011000}|\mathtt{1010}
&\mathtt{p}_{11} \mapsto
\mathtt{1001000}|\mathtt{1100000}|\mathtt{1010}\\
\hline
\mathtt{p}_{4} \mapsto
\mathtt{0000011}|\mathtt{0010100}|\mathtt{0101}
&\mathtt{p}_{8} \mapsto
\mathtt{1000010}|\mathtt{0000011}|\mathtt{1001}
&\mathtt{p}_{12} \mapsto
\mathtt{0001010}|\mathtt{1000000}|\mathtt{0110}\\
\hline
\end{array}
\]
The encoding provided in the previous table gives an easy way to assess whether certain relationships exist between the elements of $D_ET(\tau)$. For instance, by simply looking at the tuples indexed by the set $S_2(\rho)$ (\emph{i.e.} the middle sequence of $\mathtt{0}$s and $\mathtt{1}$s), we can directly see that the image of the element $\mathtt{p}_{8}$ cannot satisfy any kinship relation with any of the images of the elements $\mathtt{p}_{1}$, $\mathtt{p}_{4}$, $\mathtt{p}_{5}$, $\mathtt{p}_{7}$, $\mathtt{p}_{9}$, $\mathtt{p}_{10}$, $\mathtt{p}_{11}$, and $\mathtt{p}_{12}$. Indeed, the middle sequence associated with the image of $\mathtt{p}_{8}$ (namely, $\mathtt{0000011}$) possesses the value $\mathtt{1}$ at its sixth component, but this is not the case for the other images listed previously. If the product adjoint
\[
D_ET[\rho]:D_ET(\tau) \to D_ET(\theta(1)) \times D_ET(\theta(2)) \times D_ET(\theta(3))
\]
turns out to be a monomorphism, then we can also assert that these relationships hold for the elements of the ic-monoid $D_ET(\tau)$ (as opposed to only being satisfied for the tuples in $B_2$).

Of course, we can proceed similarly for the other cone $\rho'$. By using the table given in Example \ref{exa:Sequence alignments}, we determine that we have the following ordered sets for the cone $\rho'$.
\[
\begin{array}{l}
S_1(\rho') =S_1(\rho)\\
S_2(\rho') =\{\mathtt{cAa}<\mathtt{cAt}<\mathtt{tAt}<\mathtt{tGa}\}\\
S_3(\rho') =\left\{\begin{array}{c}\mathtt{GtcaaT}<\mathtt{GtgaaT}<\mathtt{GtgacT}<\mathtt{GtgatT}<\mathtt{TgcatT}\\<\mathtt{TgcacT}<\mathtt{TgcctC}<\mathtt{TtcaaT}<\mathtt{TtcacT}\end{array}\right\}\\
\end{array}
\]
Then we can turn the table shown in Example \ref{exa:Sequence alignments} into a table of $S_i(\rho')$-indexed tuples with coefficients in $B_2$ for every $i \in [3]$. Since the sets $S_1(\rho')$ and $S_1(\rho)$ are equal, we only show the tuples indexed by the sets $S_2(\rho')$ and $S_3(\rho')$ in the table below.
\[
\begin{array}{|rrr|}
\hline
\multicolumn{3}{|c|}{\cellcolor[gray]{0.8}D_ET(\tau')  \to B_2^{S_2(\rho')}|B_2^{S_3(\rho')}}\\
\hline
\mathtt{p}_{1} \mapsto \mathtt{1010}|\mathtt{001000010}
&\mathtt{p}_{5} \mapsto \mathtt{0110}|\mathtt{100010000}
&\mathtt{p}_{9} \mapsto \mathtt{1001}|\mathtt{010010000}\\
\hline
\mathtt{p}_{2} \mapsto \mathtt{0101}|\mathtt{000100100}
&\mathtt{p}_{6} \mapsto \mathtt{0001}|\mathtt{010000100}
&\mathtt{p}_{10} \mapsto \mathtt{1001}|\mathtt{000001100}\\
\hline
\mathtt{p}_{3} \mapsto \mathtt{1001}|\mathtt{010100000}
&\mathtt{p}_{7} \mapsto \mathtt{0110}|\mathtt{100010000}
&\mathtt{p}_{11} \mapsto \mathtt{1000}|\mathtt{010010000}\\
\hline
\mathtt{p}_{4} \mapsto \mathtt{0110}|\mathtt{000000101}
&\mathtt{p}_{8} \mapsto \mathtt{0001}|\mathtt{010000100}
&\mathtt{p}_{12} \mapsto \mathtt{1000}|\mathtt{001100000}\\
\hline
\end{array}
\]
This time we can see that that the image of the element $\mathtt{p}_{8}$ can be begotten by the images of the elements $\mathtt{p}_{10}$ and $\mathtt{p}_{11}$. If we could subsequently show that the product adjoint $D_ET[\rho']$ is a monomorphism, then such a relationship would also hold in $D_ET(\tau')$.

In other words, one usually want to use tables as constructed above to find equations linking the elements living in the images of the functor $D_ET$. These equations would first be deduced at the level of the haplotypes and then transferred to the haplogroups by using the universal property of $\mathcal{W}^{\mathrm{mon}}$-pedigrads.

We finish this example by addressing such transfer questions (inherent in pedigrad structures) through a discussion of the different scenarios described in Remark \ref{rem:local-property:free-ic-monoids:alignments} (see more below). To make our discussion relevant to the scenarios described thereof, we will add a new cone $\rho^{\prime\prime}$ to our current example. In this respect, let us consider a cone $\rho^{\prime\prime}:\Delta_{[4]}(\tau^{\prime\prime}) \Rightarrow \theta^{\prime\prime}$ in $\mathbf{Seg}(\{0,1\}\,|\,15)$ as follows.
\[
\begin{array}{ll}
\rho_{1}^{\prime \prime}:\xymatrix@C-30pt@R-20pt{
(\bullet&\bullet&\bullet)&(\bullet&\bullet&\bullet)&(\bullet&\bullet&\bullet)&(\bullet&\bullet&\bullet&\bullet&\bullet&\bullet)\ar[rr]
&\quad\quad\quad&
(\bullet&\bullet&\bullet)&(\circ&\circ&\circ)&(\circ&\circ&\circ)&(\circ&\circ&\circ&\circ&\circ&\circ)
}&\theta^{\prime\prime}(1)\\
\rho_{2}^{\prime \prime}:\xymatrix@C-30pt@R-20pt{
(\bullet&\bullet&\bullet)&(\bullet&\bullet&\bullet)&(\bullet&\bullet&\bullet)&(\bullet&\bullet&\bullet&\bullet&\bullet&\bullet)\ar[rr]
&\quad\quad\quad&
(\circ&\circ&\circ)&(\bullet&\bullet&\bullet)&(\circ&\circ&\circ)&(\circ&\circ&\circ&\circ&\circ&\circ)
}&\theta^{\prime\prime}(2)\\
\rho_{3}^{\prime \prime}:\xymatrix@C-30pt@R-20pt{
(\bullet&\bullet&\bullet)&(\bullet&\bullet&\bullet)&(\bullet&\bullet&\bullet)&(\bullet&\bullet&\bullet&\bullet&\bullet&\bullet)\ar[rr]
&\quad\quad\quad&
(\circ&\circ&\circ)&(\circ&\circ&\circ)&(\bullet&\bullet&\bullet)&(\circ&\circ&\circ&\circ&\circ&\circ)
}&\theta^{\prime\prime}(3)\\
\rho_{4}^{\prime \prime}:\xymatrix@C-30pt@R-20pt{
(\bullet&\bullet&\bullet)&(\bullet&\bullet&\bullet)&(\bullet&\bullet&\bullet)&(\bullet&\bullet&\bullet&\bullet&\bullet&\bullet)\ar[rr]
&\quad\quad\quad&
(\circ&\circ&\circ)&(\circ&\circ&\circ)&(\circ&\circ&\circ)&(\bullet&\bullet&\bullet&\bullet&\bullet&\bullet)
}&\theta^{\prime\prime}(4)
\end{array}
\]
To understand the effect of adding the cone $\rho^{\prime\prime}$ to the chromology $D$, we want to proceed carefully and consider the two recombination chromologies $D^1$ and $D^2$ defined below.
\[
D^1[n] =
\left\{
\begin{array}{ll}
\{\rho,\rho^{\prime\prime}\} &  \textrm{ if $n = 15$.}\\
\cmemptyset & \textrm{ if $n \neq 15$.}
\end{array}
\right.
\quad\quad\quad
D^2[n] =
\left\{
\begin{array}{ll}
\{\rho,\rho',\rho^{\prime\prime}\} &  \textrm{ if $n = 15$.}\\
\cmemptyset & \textrm{ if $n \neq 15$.}
\end{array}
\right.
\]
From now on, our goal will be to determine
\begin{itemize}
\item[(1)] whether the functor $D_E^2T$ defines a $\mathcal{W}^{\mathrm{mon}}$-pedigrad for $(\Omega,D^1)$ or $(\Omega,D^2)$ through the use Corollary \ref{cor:D_ET_is_a_mon_pedigrad} or Theorem \ref{theo:morphism_to_mon_pedigrad}, and
\item[(2)] whether these pedigrads are endowed with unambiguous encodings (\emph{i.e.} free presentations for their images).
\end{itemize}
If for now we restrict ourselves to only using Corollary \ref{cor:D_ET_is_a_mon_pedigrad}, then we need to determine whether the triples $(\Omega,D^2,FE_1^{\varepsilon})$ and $(\Omega,D^2,D^2_ET)$ defines $D^1$- and $D^2$-contexts.

Our first strategy will be to use Proposition \ref{prop:compare_cone_irreducible}. In this respect, observe that the topology of $\rho^{\prime\prime}$ can be mapped to the topology of $\rho'$ (see the cone $\rho$ displayed in Example \ref{exa:Relative_definition_families}). However, the topology of $\rho^{\prime\prime}$ cannot be mapped to the topology of $\rho$ and the topologies of $\rho$ and $\rho'$ cannot be mapped to that of $\rho^{\prime\prime}$. Recall that we also had previously established that the topologies of $\rho$ and $\rho'$ were not mappable into each other. As a result, there is no morphism of the form $\tau \to \theta^{\prime\prime}(i)$, $\tau' \to \theta^{\prime\prime}(i)$, $\tau' \to \theta(i)$, and $\tau^{\prime\prime} \to \theta(i)$. Proposition \ref{prop:compare_cone_irreducible} then implies that the triples $(\Omega,D^2,FE_1^{\varepsilon})$ and $(\Omega,D^2,D^2_ET)$ are $D^1$-contexts.

Our next step is to determine whether the triple $(\Omega,D^2,FE_1^{\varepsilon})$ or $(\Omega,D^2,D^2_ET)$ defines a $D^2$-context. First, note that we have arrows $\theta'(2) \to \theta^{\prime\prime}(3)$ and $\theta'(3) \to \theta^{\prime\prime}(4)$, which by Proposition \ref{prop:compare_cone_irreducible} show that the segments $\theta'(2)$ and $\theta'(3)$ are irreducible for the two triples $(\Omega,D^2,FE_1^{\varepsilon})$ and $(\Omega,D^2,D^2_ET)$. Unfortunately, we cannot use Proposition \ref{prop:compare_cone_irreducible} on the object $\theta'(1)$. To determine whether this object is irreducible, we could directly use Definition \ref{def:irreducible_object} by determining whether the morphism $FE_1^{\varepsilon}(\tau^{\prime\prime}) \to FE_1^{\varepsilon}(\theta'(1))$ coequalizes the recombination congruence $G(FE_1^{\varepsilon},\rho^{\prime\prime}) \rightrightarrows FE_1^{\varepsilon}(\tau^{\prime\prime})$, but this fails as the functor $FE_1^{\varepsilon}$ is endowed with too much freeness. Indeed, the morphism
\[
q_{FE_b^{\varepsilon}(\theta'(1))}:FE_b^{\varepsilon}(\theta'(1)) \to D^2FE_b^{\varepsilon}(\theta'(1))
\]
is not an isomorphism, which means that the encoding method constructed in (\ref{eq:local-property:free-ic-monoids:alignments:total}) does not hold. Needless to say that the alternative option, which is to show that the morphism $D_E^2T(\tau^{\prime\prime}) \to D_E^2T(\theta'(1))$ coequalizes the recombination congruence $G(D_E^2T,\rho^{\prime\prime}) \rightrightarrows D_E^2T(\tau^{\prime\prime})$, is a cumbersome exercise (since the functor $D_E^2T$ results from applying various universal constructions). In other words, we are facing a situation where the pedigrad property does come along with straightforward encodings.

Yet, we can still try to show that $D^2_ET$ is a $\mathcal{W}^{\mathrm{mon}}$-pedigrad for $(\Omega,D^2)$ -- at least this would help us clarify the encoding of the pedigrad to some extent. For this, we shall use Theorem \ref{theo:morphism_to_mon_pedigrad}. Specifically, we can show that the functor $D^2FE_1^{\varepsilon}$ is a $\mathcal{W}^{\mathrm{mon}}$-pedigrad for $(\Omega,D^2)$, which will imply that so is $D_E^2T$. In this respect, let us show that $(\Omega,D^2,D^2FE_1^{\varepsilon})$ is a $D^2$-context\footnote{Note that this is not the same thing as showing that the triple $(\Omega,D^2,D^2_ET)$ is a $D^2$-context.}. First, recall from our previous discussion that the triple $(\Omega,D^2,D^2FE_1^{\varepsilon})$ is a $D^1$-context and that the segments $\theta'(2)$ and $\theta'(3)$ are irreducible for the triple $(\Omega,D^2,D^2FE_1^{\varepsilon})$.
It now suffices to show that $\theta'(1)$ is irreducible for the triple $(\Omega,D^2,D^2FE_1^{\varepsilon})$ by showing that the morphism $D^2FE_1^{\varepsilon}(\tau^{\prime\prime}) \to D^2FE_1^{\varepsilon}(\theta'(1))$ coequalizes the recombination congruence $G(D^2FE_1^{\varepsilon},\rho^{\prime\prime}) \rightrightarrows D^2FE_1^{\varepsilon}(\tau^{\prime\prime})$.

By Remark \ref{rem:functorial_quotients}, the object $D^2FE_1^{\varepsilon}(\theta'(1))$ is a coequalizer of the form shown in (\ref{eq:coequalizer_mod_pedigrad}), and we can use our previous discussion to show that this coequalizer is indexed by only two pairs (which results in diagram (\ref{eq:pedigrads:coequalizer:two_cones}) below), namely the pairs $(\rho,f)$ and $(\rho^{\prime\prime},g)$, where $f$ and $g$ are the unique morphisms $f:\tau\to \theta'(1)$ and $g:\tau^{\prime\prime}\to \theta'(1)$ in $\mathbf{Seg}(\Omega\,|\,15)$.
\begin{equation}\label{eq:pedigrads:coequalizer:two_cones}
G(FE_1^{\varepsilon},\rho^{\prime\prime}) \coprod G(FE_1^{\varepsilon},\rho) \rightrightarrows FE_1^{\varepsilon}(\tau^{\prime\prime}) \coprod FE_1^{\varepsilon}(\tau) \mathop{\longrightarrow}\limits^{f\oplus g} FE_1^{\varepsilon}(\theta'(1))
\end{equation}
Since the arrow $f:\tau\to \theta'(1)$ belongs to $\mathbf{Seg}(\Omega\,|\,15)$, it follows from Lemma \ref{lem:quasi_homologous_preordered_category} that the arrow $f$ is equal to the arrow $\rho_i:\tau\to \theta'(1)$, whose image via $FE_1^{\varepsilon}$ equalizes the recombination congruence $G(FE_1^{\varepsilon},\rho) \rightrightarrows FE_1^{\varepsilon}(\tau)$ (see Remark \ref{rem:Coequalizing_arrows}). In essence, this means that $D^2FE_1^{\varepsilon}(\theta'(1))$ is the coequalizer of the pair shown below.
\[
G(FE_1^{\varepsilon},\rho^{\prime\prime}) \rightrightarrows FE_1^{\varepsilon}(\tau^{\prime\prime}) \to FE_1^{\varepsilon}(\theta'(1))
\]
We can apply a similar reasoning to show that the object $D^2FE_1^{\varepsilon}(\tau^{\prime\prime})$ is also the coequalizer of the pair $G(FE_1^{\varepsilon},\rho^{\prime\prime}) \rightrightarrows FE_1^{\varepsilon}(\tau^{\prime\prime})$. Finally, it follows from diagrammatic calculations using the universality of coequalizers that the morphism $D^2FE_1^{\varepsilon}(\tau^{\prime\prime}) \to D^2FE_1^{\varepsilon}(\theta'(1))$ coequalizes the recombination congruence $G(D^2FE_1^{\varepsilon},\rho^{\prime\prime}) \rightrightarrows D^2FE_1^{\varepsilon}(\tau^{\prime\prime})$. This shows that  the object $\theta'(1)$ is irreducible for the triple $(\Omega,D^2,D^2FE_1^{\varepsilon})$, which, by our previous discussion, implies that $D_E^2T$ is a $\mathcal{W}^{\mathrm{mon}}$-pedigrad for $(\Omega,D^2)$.

Note that our conclusion about the object $\theta'(1)$ could have been different if the number of cones indexing coequalizer (\ref{eq:pedigrads:coequalizer:two_cones}) had been more important. In such a situation, we would have needed to determine the set of equations holding in the ic-monoid $D^2FE_1^{\varepsilon}(\theta'(1))$. To address such scenarios, the goal of section \ref{sec:solving_our_problem} is to present an algorithm for that purpose.
\end{example}

\section{Solving our problem and finding causal variants}\label{sec:solving_our_problem}
From a practical standpoint, the goal of the present section is to show how the formalism developed in previous sections can be used as a computational framework to discover kinship relations between sets of individuals. In addition, we will explain how this formalism can be used to learn more information about genotype-phenotype associations.
From a theoretical standpoint, the goal of this section is to develop an explicit algorithm to solve the word problem in recombination schemes (see section \ref{ssec:Recombination_schemes_and_pedigrads} above). This effort should be put in perspective of the fact that, even though the word problem is decidable for commutative semigroups, finding a general explicit algorithm for it is practically untractable (see \cite{Mayr}). To some extent, our results also cover the resolution of the word problem in well-studied cases such as in one-relation monoids \cite{Nyberg-Brodda-1} and special monoids \cite{Nyberg-Brodda-2}. Overall, the theorems provided in this section represent an effort toward tackling the word problem in general ic-monoids with an explicit algorithm.

In this respect, we shall develop a set of linear-algebra-inspired techniques to solve systems of linear equations in ic-monoids, or in fact, construct what one could see as free resolutions for ic-monoids. To define systems of linear equations in ic-monoids, we shall first introduce, in section \ref{ssec:Number_systems_for_icmonoids}, a formal substraction operation for ic-monoids. Note that this construction is different from generating a group structure, which cannot be done on ic-monoids without generating the zero group. Instead, we shall embded a given ic-monoid $M$ into its product $M \times M$ such that the formal pairings $(x,y)$ represent terms of the form $x-y$. Since these pairings will not be subject to any quotient, we will see them more as elements augmenting the ic-monoid $M$ in much the same way as rational numbers of the form $x/y = (x,y)$ augments natural numbers. To better control the algebraic properties of these ``fractions'', we will encode them as polynomials -- this will be done in section \ref{ssec:idempotent_commutative_semirings} and section \ref{ssec:formal_series-polynomials}. The resulting polynomials will be indexed by the elements of the ic-monoid and will take their values in a semiring of the form $\{0,\uparrow,\downarrow,\updownarrow\}$ -- for instance the polynomial $\uparrow \!\cdot \,X^x + \downarrow \!\cdot \, X^y$ will represent the fraction $x/y = (x,y)$. In section \ref{ssec:modulo-tensors} and section \ref{ssec:Multiplicative-atomic-structures}, we define an equivalence relation on polynomials to see them more as tensored elements of the form $\uparrow \! \otimes \, x + \downarrow \! \otimes \,y$, and hence introduce linearity properties on the exponent of the polynomials (see Theorem \ref{theo:tensor-congruence:semiring-compatiblity}). Finally, from section \ref{ssec:action-semirings} to section \ref{ssec:Calculus_and_algorithm}, we develop linear algebra techniques to find sets $I$ and $J$ of indices for which equations of the form $\sum_{i\in I} \! \uparrow \!\otimes\, x_i + \sum_{j \in J} \!\downarrow \!\otimes\, xj= \updownarrow \! \otimes \,z$ hold, meaning that the equation $\sum_{i\in I}  x_i = \sum_{j \in J}  x_j$ holds in $M$.

\subsection{Rationals for idempotent commutative monoids}\label{ssec:Number_systems_for_icmonoids}
The goal of this section is to introduce a formal subtraction operation for ic-monoids. For practical reasons, this formal subtraction operation is better developed as a multiplicative operation, as opposed to using an actual subtraction notation. In section \ref{ssec:idempotent_commutative_semirings}, we will encode these formal substractions in terms of polynomials in a semiring structure.

\begin{definition}[Rationals]
We shall denote by $\mathbb{Q}$ the functor $\mathbf{Icm} \to \mathbf{Icm}$ mapping every object $M$ in $\mathbf{Icm}$ to the Cartesian product $M \times M$ in $\mathbf{Icm}$ and sending a morphism $M \to N$ in $\mathbf{Icm}$ to the canonical morphism $M \times M \to N \times N$ associated with the mapping rule $(x,y) \mapsto (f(x),f(y))$. We will call the functor $\mathbb{Q}$ the \emph{rational endofunctor} and, for every ic-monoid $(M,+,0)$, we will call the image $\mathbb{Q}(M)$ the \emph{ic-monoid of symbolic rationals of $M$}. Every element in $\mathbb{Q}(M)$ will be called a \emph{symbolic rational}.
\end{definition}

\begin{convention}[Rationals]\label{conv:rationals:ic-monoids}
Let $(M,+,0)$ be an object in $\mathbf{Icm}$. We shall refer to every element $(x,y) \in \mathbb{Q}(M)$ by the following notations.
\[
\left[x/y\right]
\quad\quad\quad\quad\quad\quad
\left[\frac{x}{y}\right]
\]
We will also denote the monoid operation associated with $\mathbb{Q}(M)$ as $\star$ (implicitly seen as a multiplication) and we will denote the neutral element $\left[0/0\right]$ of $\mathbb{Q}(M)$ as $\mathbbm{1}_{M}$. The expressions shown below give a series of equalities resulting from these notations for every pair $(x,y) \in \mathbb{Q}(M)$ and pair $(x',y') \in \mathbb{Q}(M)$.
\[
\left[\frac{x}{y}\right]\star \left[\frac{x'}{y'}\right] = \left[\frac{x+x'}{y+y'}\right] = \left[\frac{x'}{y}\right]\star\left[\frac{x}{y'}\right] = \left[\frac{x+x'}{0}\right]\star\left[\frac{0}{y+y'}\right]
\]
For every element $x \in M$, we will also use the following notations.
\[
\frac{\mathbf{1}}{x} = \left[\frac{0}{x}\right] = \mathbf{1}/x
\quad\quad\quad\quad\quad\quad\quad\quad\quad
\mathbf{0}(x) = \left[\frac{x}{x}\right]
\]
\end{convention}

\begin{definition}[Rationals]\label{def:rationals:ic-monoids}
Let $(M,+,0)$ be an object in $\mathbf{Icm}$. It is straightforward to verify that the mapping rule $x \mapsto [x/0] $ induces a monomorphism
\[
\mathbf{1}:M \to \mathbb{Q}(M)
\]
in $\mathbf{Icm}$. It is also straightforward to check that the morphism $M \to \mathbb{Q}(M)$ is natural in $M$, which means that for every morphism $f:M \to N$ in $\mathbf{Icm}$, the following diagram commutes.
\[
\xymatrix{
M \ar[r]^-{\mathbf{1}}\ar[d]_{f}& \mathbb{Q}(M)\ar[d]^{\mathbb{Q}(f)}\\
N \ar[r]_-{\mathbf{1}}& \mathbb{Q}(N)
}
\]
From now on, we will use this monomorphism to see every element $x \in M$ as an element of $\mathbb{Q}(M)$ through the monomorphism $\mathbf{1}:M \to \mathbb{Q}(M)$. In other words, for every $x \in M$, we shall identify (or simply denote) the element $\mathbf{1}(x)$ in $\mathbb{Q}(M)$ as $x$. Specifically, this means that we have the following Cayley table for the monoid operation of $\mathbb{Q}(M)$ for every element $x \in M$.
\[
\begin{array}{c|c|c|c|c}
\cellcolor[gray]{0.7}\star&\cellcolor[gray]{0.8}\mathbbm{1}_{M}&\cellcolor[gray]{0.8}x&\cellcolor[gray]{0.8}\mathbf{1}/x&\cellcolor[gray]{0.8}\mathbf{0}(x)\\
\hline
\cellcolor[gray]{0.8}\mathbbm{1}_{M}&\mathbbm{1}_{M}&x&\mathbf{1}/x&\mathbf{0}(y)\\
\hline
\cellcolor[gray]{0.8}x&x&x&\mathbf{0}(x)&\mathbf{0}(x)\\
\hline
\cellcolor[gray]{0.8}\mathbf{1}/x&\mathbf{1}/x&\mathbf{0}(x)&\mathbf{1}/x&\mathbf{0}(x)\\
\hline
\cellcolor[gray]{0.8}\mathbf{0}(y)&\mathbf{0}(y)&\mathbf{0}(x)&\mathbf{0}(x)&\mathbf{0}(x)\\
\end{array}
\]
\end{definition}

\begin{convention}[Symmetry]\label{conv:symmetry}
Let $(M,+,0)$ be an object in $\mathbf{Icm}$. There is a symmetry isomorphism $\gamma:\mathbb{Q}(M) \to \mathbb{Q}(M)$ in $\mathbf{Icm}$, which maps a pair $(x,y) \in \mathbb{Q}(M)$ to the pair $\gamma(x,y) = (y,x) \in \mathbb{Q}(M)$ and makes diagram (\ref{eq:rationals:symmetry:diagram}) commute, where the symbols $\mathbf{0}$ and $\cdot/\mathbf{1}$ denote the obvious morphisms associated with the mapping rules $x \mapsto \mathbf{0}(x)$ and $x \mapsto \mathbf{1}/x$, respectively.
\begin{equation}\label{eq:rationals:symmetry:diagram}
\xymatrix@C+20pt{
M \ar[r]^-{\mathbf{0}}\ar[rd]_-{\mathbf{0}}&  \mathbb{Q}(M) \ar[d]^{\gamma}&M \ar[l]_-{\mathbf{1}}\ar[ld]^-{\cdot/\mathbf{1}}\\
&\mathbb{Q}(M)& \\
}
\end{equation}
For every pair $(x,y) \in \mathbb{Q}(M)$, we will use the notation $\left[x/y\right]^{-1}$ to refer to the symbolic rational $\left[y/x\right]$. As a result, for every element $z = [x/y] \in \mathbb{Q}(M)$, we have the series of equations $z\star z^{-1} = \left[(x+y)/(y+x)\right] = \mathbf{0}(x+y)$. In addition, for every element $z = [x/y] \in \mathbb{Q}(M)$, we will denote the sum $x+y$ as $\mathsf{sum}(z)$ (see Example \ref{exa:Addition_as_morphism}). This means that for every element $z \in \mathbb{Q}(M)$, we have the formula $z\star z^{-1} = \mathbf{0}(\mathsf{sum}(z))$.
\end{convention}

\subsection{Idempotent commutative semirings}\label{ssec:idempotent_commutative_semirings}
Recall that a \emph{semiring} \cite{Golan} is a set $R$ equipped with two binary operations $\star_1:R \times R \to R$ and $\star_2:R \times R \to R$ and two particular elements $e_1 \in R$ and $e_2 \in R$ such that the triple $(R,\star_1,e_1)$ defines a commutative monoid, the triple $(R,\star_2,e_2)$ defines a monoid and the following axioms are satisfied:
\begin{itemize}
\item[3)] (Distributivity) for every $x,y,z \in R$, the equation $x \star_2 (y \star_1 z) =  (x \star_2 y) \star_1 (x \star_2 z)$ holds;
\item[4)] (Distributivity) for every $x,y,z \in R$, the equation $(x \star_1 y) \star_2 z =  (x \star_2 z) \star_1 (y \star_2 z)$ holds;
\item[5)] (Annihilation) for every $x \in M$, the equations $x \star_2 e_1 = e_1 = e_1 \star_2 x$ hold.
\end{itemize}
Below, we give the axioms defining idempotent and commutative semirings.

\begin{definition}[Idempotent]
A semiring $(R,\star_1,\star_2,e_1,e_2)$ is said to be \emph{idempotent} if the commutative monoid $(R,\star_1,e_1)$ is idempotent (\emph{i.e.} an object of $\mathbf{Icm}$).
\end{definition}

\begin{definition}[Commutative]
A semiring $(R,\star_1,\star_2,e_1,e_2)$ is said to be \emph{commutative} if the monoid $(R,\star_2,e_2)$ is commutative.
\end{definition}

\begin{convention}[Notation]
We will usually follow conventions used in the literature and denote an abstract idempotent commutative semiring as $(R,+_R,\cdot_R,0_R,1_R)$, or sometimes as $(R,+,\cdot,0,1)$ for conciseness. The binary operation $+_R$ will be called the \emph{addition}, the binary operation $\cdot_R$ will be called the \emph{multiplication}, the element $0_R$ will be called the \emph{zero element} and the element $1_R$ will be called the \emph{unit element}.
\end{convention}

\begin{convention}[Naming]
For convenience, we will shorten the name \emph{idempotent commutative semiring} to \emph{ic-semiring}.
\end{convention}

\begin{example}\label{exa:semi-group:Z3Z}
The set $\overline{\mathbb{Z}}_3 =\{0,\uparrow,\downarrow,\updownarrow\}$ is equipped with an ic-semiring structure $(\overline{\mathbb{Z}}_3,+,\cdot,0,\uparrow)$ whose addition and multiplication are described (and defined) by the left and right Cayley tables of (\ref{eq:semi-group:Z3Z:tables}), respectively.
\begin{equation}\label{eq:semi-group:Z3Z:tables}
\begin{array}{c|c|c|c|c}
\cellcolor[gray]{0.7}+&\cellcolor[gray]{0.8}0&\cellcolor[gray]{0.8}\uparrow&\cellcolor[gray]{0.8}\downarrow&\cellcolor[gray]{0.8}\updownarrow\\
\hline
\cellcolor[gray]{0.8}0&0&\uparrow&\downarrow&\updownarrow\\
\hline
\cellcolor[gray]{0.8}\uparrow&\uparrow&\uparrow&\updownarrow&\updownarrow\\
\hline
\cellcolor[gray]{0.8}\downarrow&\downarrow&\updownarrow&\downarrow&\updownarrow\\
\hline
\cellcolor[gray]{0.8}\updownarrow&\updownarrow&\updownarrow&\updownarrow&\updownarrow\\
\end{array}
\quad\quad\quad\quad\quad\quad\quad\quad\quad\quad
\begin{array}{c|c|c|c|c}
\cellcolor[gray]{0.7}\cdot&\cellcolor[gray]{0.8}0&\cellcolor[gray]{0.8}\uparrow&\cellcolor[gray]{0.8}\downarrow&\cellcolor[gray]{0.8}\updownarrow\\
\hline
\cellcolor[gray]{0.8}0&0&0&0&0\\
\hline
\cellcolor[gray]{0.8}\uparrow&0&\uparrow&\downarrow&\updownarrow\\
\hline
\cellcolor[gray]{0.8}\downarrow&0&\downarrow&\uparrow&\updownarrow\\
\hline
\cellcolor[gray]{0.8}\updownarrow&0&\updownarrow&\updownarrow&\updownarrow\\
\end{array}
\end{equation}
It may be helpful to think of the elements $0$, $\uparrow$ and $\downarrow$ as the integers $0$, $+1$ and $-1$ for the multiplicative operation on the right. For example, we have the equations $\downarrow \cdot \downarrow = \uparrow$ and $\uparrow \cdot \uparrow = \uparrow$. Similarly, it may be helpful to think of the addition operation as a graphical operation that merges symbols together, where $0$ represents the empty symbol.

To show that $(\overline{\mathbb{Z}}_3,+,\cdot,0,\uparrow)$ defines a semiring, we need to find a combinatorial presentation for $\overline{\mathbb{Z}}_3$ in which the addition and multiplication trivially satisfy the axioms for semiring. Specifically, we can encode $\overline{\mathbb{Z}}_3$ as the set $\{0,1\} \times \{0,1\}$ and take the representatives $0 = (0,0)$, $\uparrow = (1,0)$, $\downarrow = (0,1)$ and $\uparrow = (1,1)$. If we see the element $0$ and $1$ as Boolean values, then we can verify that the addition and multiplication of $\overline{\mathbb{Z}}_3$ satisfy the following formulas for every $x = (x_0,x_1) \in \overline{\mathbb{Z}}_3$ and every $y = (y_0,y_1) \in \overline{\mathbb{Z}}_3$.
\[
\def\arraystretch{1.2}
\left\{
\begin{array}{ll}
x + y &= (x_0 \,\mathsf{or}\, y_0, x_1 \,\mathsf{or}\, y_1) \\
x \cdot y &= \big((x_0 \,\mathsf{and}\, y_0) \,\mathsf{or} \,(x_1 \,\mathsf{and} \,y_1), (x_1\, \mathsf{and} \,y_0) \,\mathsf{or} \,(x_0 \,\mathsf{and} \,y_1)\big)
\end{array}
\right.
\]
It follows from these formulas that the tuple $(\overline{\mathbb{Z}}_3,+,\cdot,0,\uparrow)$ defines an ic-semiring. The details are left to the reader as an exercise\footnote{Note that the formulas for $x+y$ and $x \cdot y$ are very similar to the formulas for the addition and the multiplication in the polynomial ring $\mathbb{Z}[X]/(X^2-1)$ if we replace $\mathsf{or}$ and $\mathsf{and}$ with the addition and multiplication for integers.}.
\end{example}

\begin{definition}[Action]\label{def:action:rationals}
Let $(M,+,0)$ be an object in $\mathbf{Icm}$. We will denote by $\odot$ the function $\overline{\mathbb{Z}}_3 \times \mathbb{Q}(M) \to \mathbb{Q}(M)$ defined by the mapping rule $(f,x) \mapsto f \odot x$, where $f \odot x$ satisfies the following equations.
\[
f \odot x =
\left\{
\begin{array}{ll}
\mathbbm{1}_{M}&\textrm{if } f = 0\\
x&\textrm{if } f = \uparrow\\
x^{-1}&\textrm{if } f = \downarrow\\
x\star x^{-1}&\textrm{if } f = \updownarrow\\
\end{array}
\right.
\]
Note that for every $f \in \overline{\mathbb{Z}}_3$, the equation $f \odot \mathbbm{1}_{M} = \mathbbm{1}_{M}$ holds. By Convention \ref{conv:symmetry}, we also have the equation $\updownarrow \!\odot\, x = \mathbf{0}(\mathsf{sum}(x))$. See Remark \ref{rem:semi-direct-product} for more discussion regarding the properties satisfied by the operation $\odot$.
\end{definition}

\begin{remark}[Properties]\label{rem:semi-direct-product}
Let $(M,+,0)$ be an object in $\mathbf{Icm}$. By using the idempotency property of $M$, it is straightforward to verify that the following relations hold for every $(x,y) \in \mathbb{Q}(M)$ and pair $(f,g) \in \overline{\mathbb{Z}}_3 \times \overline{\mathbb{Z}}_3$.
\[
(g\cdot f) \odot x =  g \odot (f \odot x) \quad\quad\quad\quad\quad f \odot (x\star y) = (f \odot x)\star (f \odot y)
\]
However, it is less straightforward to show that the following equation holds.
\[
(g + f) \odot x =  (g \odot x)\star (f \odot x)
\]
The proof for this equation is shown by the following equality of Cayley tables, in which we display the values for $g$ anf $f$ in the horizontal and vertical axes, respectively (note that, we use the idempotency property of $M$ to compute the values in the tables).
\[
\begin{array}{c}
\begin{array}{c|c|c|c|c}
\cellcolor[gray]{0.7}(g+f) \odot x &\cellcolor[gray]{0.8}0&\cellcolor[gray]{0.8}\uparrow&\cellcolor[gray]{0.8}\downarrow&\cellcolor[gray]{0.8}\updownarrow\\
\hline
\cellcolor[gray]{0.8}0&\mathbbm{1}_{M}&x&x^{-1}&x\star x^{-1}\\
\hline
\cellcolor[gray]{0.8}\uparrow&x&x&x\star x^{-1}&x\star x^{-1}\\
\hline
\cellcolor[gray]{0.8}\downarrow&x^{-1}&x\star x^{-1}&x^{-1}&x\star x^{-1}\\
\hline
\cellcolor[gray]{0.8}\updownarrow&x\star x^{-1}&x\star x^{-1}&x\star x^{-1}&x\star x^{-1}\\
\end{array}\\
\\
\rotatebox[origin=c]{90}{=}\\
\\
\begin{array}{c|c|c|c|c}
\cellcolor[gray]{0.7}(g\odot x)\star (f\odot x)  &\cellcolor[gray]{0.8}0&\cellcolor[gray]{0.8}\uparrow&\cellcolor[gray]{0.8}\downarrow&\cellcolor[gray]{0.8}\updownarrow\\
\hline
\cellcolor[gray]{0.8}0&\mathbbm{1}_{M}&x&x^{-1}&x\star x^{-1}\\
\hline
\cellcolor[gray]{0.8}\uparrow&x&x&x\star x^{-1}&x\star x^{-1}\\
\hline
\cellcolor[gray]{0.8}\downarrow&x^{-1}&x\star x^{-1}&x^{-1}&x\star x^{-1}\\
\hline
\cellcolor[gray]{0.8}\updownarrow&x\star x^{-1}&x\star x^{-1}&x\star x^{-1}&x\star x^{-1}\\
\end{array}
\end{array}
\]
In section \ref{ssec:action-semirings}, we will formalize the properties of the operation $\odot$ under a concept called an \emph{action of an ic-semiring on an ic-monoid}.
\end{remark}

\subsection{Formal series and polynomials}\label{ssec:formal_series-polynomials}
In this section, we define series and polynomials taking their coefficients in semirings and whose exponents run over the elements of ic-monoids. Much of the content produced in this section is inspired from the usual treatment of poylnomials in rings. In this respect, for every ic-semiring $(R,+_R,\cdot_R,0_R,1_R)$ and every ic-monoid $(M,+_M,0_M)$, we shall denote by $R[[M]]$ the product $\prod_{m \in M} R$ in $\mathbf{Icm}$ and call it the \emph{ic-monoid of formal series in $R$ over $M$}. For every tuple $r=(r_m)_m$ in $R[[M]]$, every morphism $f:M \to N$ in $\mathbf{Icm}$, and every element $n \in N$, we define the following element in $R$.
\[
r[f]_n := \textstyle \sum_{m \in f^{-1}(n)} r_m
\]
We denote by $r[f]$ the tuple $(r[f]_n)_{n \in N}$ in $R[[N]]$ and by $R[[f]]$ the morphism $R[[M]] \to R[[N]]$ in $\mathbf{Icm}$ resulting from the mapping rule $r \mapsto r[f]$. It follows from usual properties on fibers of functions that the mapping rule $f \mapsto R[[f]]$ induces a functor $\mathbf{Icm} \to \mathbf{Icm}$.

\begin{definition}[Support]\label{def:support}
Let $(R,+_R,\cdot_R,0_R,1_R)$ be an ic-semiring and $(M,+_M,0_M)$ be an ic-monoid. For every tuple $r = (r_m)_{m} \in R[[M]]$, we define the \emph{support} of $r$ as the following set.
\[
\mathsf{Sup}(r):= \{m \in M~|~r_m \neq 0\}
\]
\end{definition}

\begin{proposition}[Sums to zero]\label{prop:sum_to_zero}
Let $(M,+,0)$ be an ic-monoid. For every pair $(x,y)$ of elements in $M$, the identity $x+y = 0$ holds if, and only if, the identities $x = 0$ and $y = 0$ hold.
\end{proposition}
\begin{proof}
If the equations $x = 0$ and $y = 0$ holds, then so does the equation $x + y = 0$. To show the converse, observe that if the identity $x+y = 0$ holds, then the idempotency property in $M$ implies that we have $x = x + (x+y) = x + y = 0$. Similarly, we deduce that $y = 0$.
\end{proof}

\begin{proposition}\label{prop:support:sum}
Let $(R,+_R,\cdot_R,0_R,1_R)$ be an ic-semiring and $(M,+_M,0_M)$ be an ic-monoid. For every pair $(r,r') \in R[[M]] \times R[[M]]$, the equation $\mathsf{Sup}(r+r') = \mathsf{Sup}(r) \cup \mathsf{Sup}(r')$ holds.
\end{proposition}
\begin{proof}
By the contrapose of Proposition \ref{prop:sum_to_zero}, for every pair $(r_1,r_2)$ of elements in $R$, the following equivalence holds: the inequality $r_1 \neq 0_R$ or $r_2 \neq 0_R$ holds if, and only if, the inequality $r_1+r_2 \neq 0_R$ holds. Since the addition of $R[[M]]$ is componentwise, this implies that, for every pair $(r,r')$ of elements in $R[[M]]$, the identity $\mathsf{Sup}(r+r') = \mathsf{Sup}(r) \cup \mathsf{Sup}(r')$ must hold.
\end{proof}

\begin{convention}[Polynomials]\label{conv:polynomials:def}
For every ic-semiring $(R,+_R,\cdot_R,0_R,1_R)$ and every ic-monoid $(M,+_M,0_M)$, we will denote by $R[M]$ the subset of $R[[M]]$ containing the tuples $r$ for which the set $\mathsf{Sup}(r)$ is finite.
It follows from Proposition \ref{prop:support:sum} that if two elements $r$ and $r'$ belong to $R[M]$, then so does their sum $r+r'$. Since we also have $\mathsf{Sup}(0) = \cmemptyset$ for the particular element $0 \in R[[M]]$, we deduce that the subset $R[M]$ defines a submonoid of the ic-monoid $R[[M]]$, and is hence an object of $\mathbf{Icm}$.
\end{convention}

\begin{remark}[Polynomials]\label{rem:monomials}
Let $(R,+_R,\cdot_R,0_R,1_R)$ be an ic-semiring and $(M,+_M,0_M)$ be an ic-monoid. For every element $d \in M$, if we let $X^d$ denote the element $(r_{m})_{m \in M}$ of $R[M]$ such that $r_m = 0_R$ if $m \neq d$ and $r_m = 1_R$ if $m = d$, then we can show that each element of $R[M]$ is a finite sum of elements of the form $X^d$. Specifically, for every element $r = (r_{m})_{m \in M} \in R[M]$, we can express the element $r$ as the following finite sum.
\[
r = \textstyle \sum_{m \in \mathsf{Sup}(r)} r_m X^m
\]
Conversely, for every finite subset $S \subseteq M$ and function $a:S \to R$, the sum $\sum_{m \in S} a(m) X^m$ defines an element of $R[M]$ whose support is included in $S$.
\end{remark}

\begin{proposition}\label{prop:polynomials:semiring:functorial}
Let $(R,+_R,\cdot_R,0_R,1_R)$ be an ic-semiring. The mapping $M \mapsto R[M]$, defined for every ic-monoid $M$, induces a functor $R[\,\cdot\,]:\mathbf{Icm}\to \mathbf{Icm}$ whose images on the morphisms are the restrictions of the images of the functor $R[[\,\cdot\,]]:\mathbf{Icm} \to \mathbf{Icm}$.
\end{proposition}
\begin{proof}
For every morphism $f:M \to N$ in $\mathbf{Icm}$, the functor $R[[\,\cdot\,]]:\mathbf{Icm}^{\mathrm{op}} \to \mathbf{Icm}$ sends the morphism $f$ to the morphism $R[[f]]:R[[M]] \to R[[N]]$ that maps a tuple $r$ to the tuple $r[f]$. Let us show that if $r \in R[M]$, then $r[f] \in R[N]$. In this respect, suppose that $r \in R[M]$ and let us show that the set $\mathsf{Sup}(r[f])$. To do so, let us define the following subset of $N$.
\[
S = \{n \in N~|~f^{-1}(n) \cap \mathsf{Sup}(r)\}
\]
Since the fibers of the function $f:M \to N$ form a partition (\emph{i.e} a disjoint union) of $M$ and since $\mathsf{Sup}(r)$ is finite, the set $S$ must be finite. According to the definition of $r[f]$, the inequality $r[f]_n \neq 0$ implies that there exists $m \in f^{-1}(n)$ such that $r_m \neq 0$. Hence, we have the inclusion $\mathsf{Sup}(r[f]) \subseteq S$ and since $S$ is finite, the element $r[f]$ belongs to $R[M]$. This shows that the restriction of the morphism $R[[f]]$ on $R[N]$ lands in $R[M]$. If we denote this restriction as $R[f]:R[N] \to R[M]$, then the functoriality of $R[[\,\cdot\,]]$ implies that the mapping $f \mapsto R[f]$ is functorial.
\end{proof}

\begin{remark}[Finiteness]
For every ic-semiring $(R,+_R,\cdot_R,0_R,1_R)$ and ic-monoid $(M,+_M,0_M)$, the ic-monoid $R[[M]]$ is equal to the ic-monoid $R[M]$ whenever $M$ is finite.
\end{remark}

\begin{definition}[Coefficients]\label{def:product-coefficients}
Let $(R,+_R,\cdot_R,0_R,1_R)$ be an ic-semiring and $(M,+_M,0_M)$ be an ic-monoid. For every pair $(r,r') \in R[[M]] \times R[[M]]$, we denote by $\mathsf{Sup}(r,r')$ the set
\[
\{m+_M m' ~|~m \in \mathsf{Sup}(r), m' \in \mathsf{Sup}(r')\}
\]
and for every $d \in \mathsf{Sup}(r,r')$, we define the \emph{$d$-th product coefficient of $(r,r')$} as the following triple sum.
\[
r \boxdot_d r' :=  \sum_{d_1 \in \mathsf{Sup}(r)} \sum_{d_2 \in \mathsf{Sup}(r')}\sum_{d_1+_M d_2 = d} r_{d_1} \cdot_R r'_{d_2}
\]
\end{definition}

\begin{remark}[Supports]
Let $(R,+_R,\cdot_R,0_R,1_R)$ be an ic-semiring and $(M,+_M,0_M)$ be an ic-monoid. For every triple $(r,s,t)$ of elements in $R[[M]]$, Proposition \ref{prop:support:sum} implies that the following equations hold.
\[
\mathsf{Sup}(r,s+t) = \{m+_M m' ~|~m \in \mathsf{Sup}(r), m' \in \mathsf{Sup}(s) \cup \mathsf{Sup}(t)\} = \mathsf{Sup}(r,s) \cup \mathsf{Sup}(r,t)
\]
Since the operation $+_M$ is commutative, we also have the equation $\mathsf{Sup}(r,s) = \mathsf{Sup}(s,r)$.
\end{remark}

\begin{proposition}\label{prop:polynomials:semiring}
Let $(R,+_R,\cdot_R,0_R,1_R)$ be an ic-semiring and $(M,+_M,0_M)$ be an ic-monoid. The ic-monoid $R[M]$ can be equipped with an ic-semiring structure whose unit is given by $X^{0_M}$ and whose multiplication, say given by a map $\boxdot:R[M]\times R[M] \to R[M]$, is defined by the following equation for every pair $(r,r') \in  R[M]\times R[M]$.
\[
r \boxdot r' :=  \sum_{d \in \mathsf{Sup}(r,r')} (r \boxdot_{d} r') X^{d}
\]
\end{proposition}
\begin{proof}
This is a usual result on polynomials with coefficients in ring-like structures. Hence, we will gloss over certain details. First, we know that $R[M]$ is an ic-monoid. It follows from Definition \ref{def:product-coefficients} that each product coefficient define a commutative operation, namely we have $r \boxdot_d r' = r' \boxdot_d r$ for every $d \in \mathsf{Sup}(r,r')$. This implies that the multiplication $\boxdot$ is also commutative. Let us show that $X^{0_M}$ defines a unit element for $R[M]$. First, because we have the equation $\mathsf{Sup}(X^{0_M}) = \{0_M\}$, Definition \ref{def:product-coefficients} implies that we have the equation $\mathsf{Sup}(r,X^{0_M}) = \mathsf{Sup}(r)$. In other words, we have the equation $r \boxdot_{d} X^{0_M} = r_{d}$ for every $d \in \mathsf{Sup}(r)$ and hence the equation $r \boxdot X^{0_M} = r$. Let us now show the associativity of the operation $\boxdot$. Specifically, the fact that the operation $\boxdot$ is associative follows from the following equations.
\begin{align*}
r \boxdot_d (s \boxdot t) &= \sum_{d_1 \in \mathsf{Sup}(r)} \sum_{d_2 \in \mathsf{Sup}(s \boxdot t)}\sum_{d_1+ d_2 = d} r_{d_1} \cdot (s \boxdot_{d_2} t)\\
&= \sum_{d_1 \in \mathsf{Sup}(r)} \sum_{d_2 \in \mathsf{Sup}(s , t)}\sum_{d_1+ d_2 = d} r_{d_1} \cdot \Big(\sum_{e_1 \in \mathsf{Sup}(s)} \sum_{e_2 \in \mathsf{Sup}(t)}\sum_{e_1+ e_2 = d_2} s_{e_1} \cdot t_{e_2}\Big)\\
&= \sum_{d_1 \in \mathsf{Sup}(r)} \sum_{e_1 \in \mathsf{Sup}(s)} \sum_{e_2 \in \mathsf{Sup}(t)} \sum_{d_1+ e_1+ e_2 = d} r_{d_1} \cdot s_{e_1} \cdot t_{e_2}\\
&= \sum_{e_2 \in \mathsf{Sup}(t)} \sum_{e_3 \in \mathsf{Sup}(r,s)}\sum_{e_2+ e_3 = d} t_{e_2} \cdot \Big(\sum_{d_1 \in \mathsf{Sup}(r)} \sum_{e_1 \in \mathsf{Sup}(s)}\sum_{d_1+ e_1 = e_3} r_{d_1} \cdot s_{e_1}\Big)\\
& = (r \boxdot s) \boxdot_d t
\end{align*}
Because the multiplication of $R$ is distributive and the addition of $R[M]$ is defined componentwise, we can use the formula shown in Definition \ref{def:product-coefficients} to show that, for every triple $(r,s,t)$ of elements in $R[M]$ and every element $d \in \mathsf{Sup}(r,s+t) = \mathsf{Sup}(r,s) \cup \mathsf{Sup}(r,t)$, the following equation holds
\[
r \boxdot_d (s+t) =  r \boxdot_d s+ r \boxdot_d t
\]
This implies that the operation $\boxdot$ is distributive. Finally, we can check that the equation $r \boxdot_d 0_R = 0_R$ holds for every $r \in R[M]$, because the element $0_R$ annihilates the other elements of $R$ with respect to the multiplication of $R$. This shows that $R[M]$ is an ic-semiring.
\end{proof}

\subsection{Tensor congruence}\label{ssec:modulo-tensors}
The present section aims to define a congruence (see Definition \ref{def:tensor-congruence}) on the ic-semiring structure of Proposition \ref{prop:polynomials:semiring}. We will define this construction as the pullback of a certain morphism of ic-monoids (Convention \ref{conv:tensor-congruence}) along itself. Importantly, we will use this congruence as a way to mimic properties that are characteristic of tensorial structures (see Example \ref{exa:modulo:formula:bar_s_R} and Remark \ref{rem:conclusion-tensor-congruence}).  The resulting tensor-like structure has a straightforward digital implementation through the semirings of polynomials introduced in section \ref{ssec:formal_series-polynomials}.

\begin{convention}[Atomic subsets]\label{conv:atomic-elements}
Let $(M,+,0)$ be an ic-monoid. A subset $A \subseteq M$ will be said to be \emph{atomic} in $M$ if it is non-empty, finite and for every pair $(x,y)$ of elements in $M$, the following statements are equivalent:
\begin{itemize}
\item[1)] the relation $x+y \in A$ holds;
\item[2)] one of the two relations $x \in A$ or $y \in A$ must hold.
\end{itemize}
We will denote by $\mathsf{Atom}(M)$ the set of atomic subsets of $M$.
\end{convention}

\begin{example}[Atomic subsets]\label{exa:atomic_subsets:Z3}
Let us show that the two subsets $\{\uparrow,\updownarrow\}$ and $\{\downarrow,\updownarrow\}$ of $\overline{\mathbb{Z}}_3$ are atomic subsets of the ic-monoid $(\overline{\mathbb{Z}}_3,+,0)$. Because the proofs of each statements is very similar, we will let $r_1$ and $r_2$ be the two distinct elements of the set $\{\uparrow, \downarrow\}$. We shall let $A$ denote the set $\{r_1,\updownarrow\}$ and show that $A$ is atomic. In this respect, let $(x,y)$ be a pair of elements in $\overline{\mathbb{Z}}_3$.
Suppose that the relation $x+y \in A$ holds. Since the equation $\overline{\mathbb{Z}}_3 = A \cup \{r_2,0\}$ holds, we have either $x \in A$ or $x \in \{r_2,0\}$. If $x \notin A$, then we have either $x = 0$ or $x = r_2$. If $x = 0$, then the relation $x+y \in A$ implies that the element $y$ must be in $A$. If $x = r_2$, then Example \ref{exa:semi-group:Z3Z} implies that the equation $y = r_1$ must hold so that we have the equation $r_1+r_2 =\, \updownarrow$. This shows that, in any case, one of the two relations $x \in A$ or $y \in A$ must hold.
Conversely, suppose that $x \in A$ or $y \in A$. If $x \in A$, then the following Cayley table shows that the relation $x+y \in A$ holds.
\[
\begin{array}{c|c|c|c|c}
\cellcolor[gray]{0.7} y =  &\cellcolor[gray]{0.8}0&\cellcolor[gray]{0.8}r_1&\cellcolor[gray]{0.8}r_2&\cellcolor[gray]{0.8}\updownarrow\\
\hline
\cellcolor[gray]{0.8}x+y \in  &A&A&\{\updownarrow\} \subsetneq A&\{\updownarrow\}\subsetneq A\\
\end{array}
\]
Symmetrically, this shows that if $y \in A$, then $x+y \in A$. This shows that the set $\mathsf{Atom}(\overline{\mathbb{Z}}_3)$ contains the two subsets $\{\uparrow,\updownarrow\}$ and $\{\downarrow,\updownarrow\}$ of $\overline{\mathbb{Z}}_3$.
\end{example}

\begin{convention}[Atomic subsets]
We shall denote the two atomic subsets $\{\uparrow,\updownarrow\}$ and $\{\downarrow,\updownarrow\}$ of $\overline{\mathbb{Z}}_3$ (see Example \ref{exa:atomic_subsets:Z3}) as $\Uparrow$ and $\Downarrow$, respectively.
\end{convention}

\begin{definition}[Coefficient sets]\label{def:coefficient-sets}
Let $(R,+_R,\cdot_R,0_R,1_R)$ be an ic-semiring and $(M,+_M,0_M)$ be an ic-monoid. For every atomic subset $A \subseteq R$ and every tuple $s = (s_m)_m \in R[M]$, we define the \emph{$A$-coefficient set of $s$} as the finite set $\chi_A(s):=\{m \in \mathsf{Sup}(s)~|~s_m \in A\}$.
\end{definition}

\begin{example}[Coefficient sets]\label{exa:coefficient-sets}
Let $(M,+_M,0_M)$ be an ic-monoid and consider the following polynomial in $\overline{\mathbb{Z}}_3[M]$ for a given tuple $(m_1,m_2,m_3,m_4,m_5)$ of elements in $M$.
\[
s = \uparrow\! X^{m_1} + \downarrow\!  X^{m_2} + \uparrow\!  X^{m_3} + \updownarrow\!  X^{m_4} + \updownarrow\!  X^{m_5}
\]
According to Example \ref{exa:atomic_subsets:Z3}, the two sets $\Uparrow = \{\uparrow,\updownarrow\}$ and $\Downarrow = \{\downarrow,\updownarrow\}$ are atomic subsets of $\overline{\mathbb{Z}}_3$. We shall use these sets to illustrate the concept defined in Definition \ref{def:coefficient-sets}. Specifically, we can compute the following coefficient sets for $s$.
\begin{align*}
\chi_{\Uparrow}(s) &=\{m \in \mathsf{Sup}(s)~|~s_{m} \in \Uparrow\}\\
&=\{m \in \{m_1,m_2,m_3,m_4,m_5\}~|~s_{m} \in \{\uparrow,\updownarrow\}\}\\
&=\{m_1,m_3,m_4,m_5\}\\
\chi_{\Downarrow}(s) &=\{m \in \mathsf{Sup}(s)~|~s_{m} \in \Downarrow\}\\
&=\{m \in \{m_1,m_2,m_3,m_4,m_5\}~|~s_{m} \in \{\downarrow,\updownarrow\}\}\\
&=\{m_2,m_4,m_5\}
\end{align*}
Note that the intersection $\chi_{\Uparrow}(s) \cap \chi_{\Downarrow}(s)$ is equal to the elements of $M$ that index the coefficient $\updownarrow$ in the expression of $s$. If we take $s' = \uparrow \!X^{m_1}+\uparrow \!X^{m_2}$, then we have the following equations $\chi_{\Uparrow}(s) = \{m_1,m_2\}$ and $\chi_{\Uparrow}(s) = \cmemptyset$ and the intersection $\chi_{\Uparrow}(s) \cap \chi_{\Downarrow}(s)$ is empty.
\end{example}

\begin{definition}[Kappa]\label{def:kappa}
Let $(R,+_R,\cdot_R,0_R,1_R)$ be an ic-semiring and $(M,+_M,0_M)$ be an ic-monoid. For every atomic subset $A \subseteq R$ and every tuple $s = (s_m)_m \in R[M]$, we denote by $\kappa_A(s)$ the finite sum $\sum_{m \in \chi_A(s)} m$ in $M$.
\end{definition}

\begin{example}[Kappa]\label{exa:kappa}
Let $(M,+_M,0_M)$ be an ic-monoid and consider the following polynomial in $\overline{\mathbb{Z}}_3[M]$ for a given tuple $(m_1,m_2,m_3,m_4,m_5)$ of elements in $M$.
\[
s = \uparrow\! X^{m_1} + \downarrow\!  X^{m_2} + \uparrow\!  X^{m_3} + \updownarrow\!  X^{m_4} + \updownarrow\!  X^{m_5}
\]
According to Example \ref{exa:coefficient-sets}, we have the following equalities.
\[
\chi_{\Uparrow}(s) = m_1+m_3+m_4+m_5
\quad\quad
\chi_{\Downarrow}(s) = m_2+m_4+m_5
\]
In we now take $s' = \uparrow \!X^{m_1}+\uparrow \!X^{m_2}$, then Example \ref{exa:coefficient-sets} implies that we have the following equations (where a sum over an empty set is zero).
\[
\chi_{\Uparrow}(s') = m_1+m_2
\quad\quad\quad\quad\quad\quad\quad\quad
\chi_{\Downarrow}(s') = 0
\]
In the sequel, we will mainly rely on the previous equations and those described in to Example \ref{exa:coefficient-sets} illustrate new concepts.
\end{example}

\begin{proposition}\label{prop:extend-atomic:large-sums}
Let $(M,+,0)$ be an ic-monoid. For every element $A \in \mathsf{Atom}(M)$ and every finite collection $(x_i)_{i \in [n]}$ of elements in $M$, the relation $\sum_{i =1}^n x_i \in A$ holds if, and only if, there exists $i \in [n]$ such that the relation $x_i \in A$ holds.
\end{proposition}
\begin{proof}
The statement follows from an induction on $n$. The case $n = 1$ is obvious and the case $n=2$ directly follows from Convention \ref{conv:atomic-elements}. The case for $n>2$ is shown as follows. Since $A \in \mathsf{Atom}(M)$, the relation $\sum_{i =1}^n x_i \in A$ holds if, and only if, we have $x_n \in A$ or $\sum_{i =1}^{n-1} x_i \in A$. If $x_n \notin A$, then $\sum_{i =1}^{n-1} x_i \in A$ and the induction shows that there exists $i \in [n-1]$ such that $x_i \in A$. Conversely, if there exists $i \in [n]$ such that the relation $x_i \in A$ holds, then Convention \ref{conv:atomic-elements} implies that the element $\sum_{i =1}^n x_i = x_i + \sum_{j \in [n] \backslash\{i\}} x_i$ is in $A$.
\end{proof}

\begin{definition}[Non-vanishing atomic element]
For every ic-monoid $(M,+,0)$ and every atomic subset $A$ of $M$, we will say that $A$ is \emph{non-vanishing} if $x \in A$ implies that $x \neq 0$.
\end{definition}

\begin{definition}[Atomic structures]\label{def:atomic-structure}
For every ic-monoid $(M,+,0)$, we define an \emph{atomic structure on $M$} as a pair $(\Omega,\rho)$ where
\begin{itemize}
\item[1)] $\Omega$ is a subset of $\mathsf{Atom}(M)$ whose elements are all non-vanishing atomic subsets of $M$;
\item[2)] $\rho$ is an injective function $\Omega \to M$
\end{itemize}
such that the following equivalence is satisfied for every element $A \in \Omega$.
\[
m \in A \cap \rho(\Omega) \quad\quad \Leftrightarrow \quad\quad \rho^{-1}(m)=\{A\}
\]
\end{definition}

\begin{proposition}[Atomic structures]\label{prop:atomic-structure:restriction}
Let $(M,+,0)$ be an ic-monoid and let $(\Omega,\rho)$ an atomic structure on $M$. For every subset $\Omega' \subseteq \Omega$, the following equivalence holds  for every element $A \in \Omega$.
\[
m \in A \cap \rho(\Omega') \quad\quad \Leftrightarrow \quad\quad \rho^{-1}(m)=\{A\}\textrm{ and }A \in \Omega'
\]
\end{proposition}
\begin{proof}
If we have $m \in A \cap \rho(\Omega')$, then we also have $m \in A \cap \rho(\Omega)$, and since $(\Omega,\rho)$ is an atomic structure on $M$, this means that $\rho^{-1}(m)=\{A\}$ and $A \in \Omega$ (Definition \ref{def:atomic-structure}). Because $m \in \rho(\Omega')$ and $\rho^{-1}(m)=\{A\}$, the element $A$ must be in $\Omega'$, which means that the two relations $\rho^{-1}(m)=\{A\}$ and $A \in \Omega'$ are satisfied. Conversely, if we have $\rho^{-1}(m)=\{A\}$ and $A \in \Omega'$, then we must have $m \in A \cap \rho(\Omega)$ because $\Omega'$ is a subset of $\Omega$. This shows that we have $m \in A$ and we can show that $m$ actually belongs to $\rho(\Omega')$ because $m = \rho(A)$ and $A \in \Omega'$.
\end{proof}

\begin{example}[Atomic structure]\label{exa:atomic-structure}
The subset $\Omega := \{\Uparrow,\Downarrow\}$ of $\mathsf{Atom}(\overline{\mathbb{Z}}_3)$ defines an atomic structure on $\overline{\mathbb{Z}}_3$ when we equip it with an injective function $\rho:\Omega \to \overline{\mathbb{Z}}_3$ that maps $\Uparrow$ to $\uparrow$ and maps $\Downarrow$ to $\downarrow$. In this case, the image of $\rho$ is as follows.
\[
\rho(\Omega) = \{\uparrow,\downarrow\}
\]
If an element $A \in \Omega$ intersect with $\rho(\Omega)$ such that there exists an element $r \in A \cap \rho(\Omega)$, then we can easily verify that the set $A$ is necessarily of the form $\{r,\updownarrow\}$ such that the equation $\rho^{-1}(r) = \{A\}$ is satisfied.
Conversely, if there exists an element $r \in \overline{\mathbb{Z}}_3$ such that $\rho^{-1}(r) = \{A\}$ for some element $A \in \Omega$, then we can verify that the relation $r \in A$ holds since we have:
\[
\rho^{-1}(r) = \left\{
\begin{array}{ll}
\{\{\uparrow,\updownarrow\}\}&\textrm{if }r = \, \uparrow\\
\{\{\downarrow,\updownarrow\}\}&\textrm{if }r = \, \downarrow.\\
\end{array}
\right.
\]
This shows that the equivalence $m \in A \cap \rho(\Omega) \Leftrightarrow \rho^{-1}(m)=\{A\}$ is satisfied.
\end{example}

The following proposition refers to the counit $\mu$ of the adjunction (\ref{adjunction_set_icm}), which maps a finite subset of an ic-monoid to the sum of its elements in the ic-monoid. If the finite subset is empty, then that sum is equal to the zero element.

\begin{proposition}\label{prop:kappa-morphism}
Let $(R,+_R,\cdot_R,0_R,1_R)$ be an ic-semiring and $(M,+_M,0_M)$ be an ic-monoid. For every atomic structure $(\Omega,\rho)$ on $R$ and every element $A \in \Omega$, the mapping rule $s \mapsto \chi_A(s)$ defines a morphism $\chi_A:R[M] \to FU(M)$ in $\mathbf{Icm}$ and the mapping rule $s \mapsto \kappa_A(s)$ defines a morphism $\kappa_A:R[M] \to M$ in $\mathbf{Icm}$ such that the equation $\kappa_A = \mu_M \circ \chi_A$ holds.
\end{proposition}
\begin{proof}
The second statement follows from the first one because the morphism $\kappa_A:R[M] \to M$ would result from composing the morphism $R[M] \to FU(M)$ with the counit $FU(M) \to M$ of adjunction (\ref{adjunction_set_icm}). Let us show the first statement. First, note that the element $\chi_A(s)$ is a finite subsets of $M$ and hence defines an element in $FU(M)$. To show that $\chi_A$ is a morphism in $\mathbf{Icm}$, first note that the equation $\chi_A(0_RX^{0_M}) = \cmemptyset$ holds because the set $\mathsf{Sup}(0)$ is empty. Now, because $A$ is a non-vanishing atomic subset of $M$ and the operator $\mathsf{Sup}$ sends sums in $M$ to sums in $(FU(M),\cup,\cmemptyset)$ (see Proposition \ref{prop:support:sum}), we have the following equations for every pair $(s,s')$ of elements in $R[M]$.
\begin{align*}
\chi_A(s+s') &= \{m \in \mathsf{Sup}(s+s')~|~s_m+_Rs_m' \in A\}  \\
&= \{m \in \mathsf{Sup}(s)\cup \mathsf{Sup}(s')~|~s_m \in A \textrm{ or }s_m' \in A\}\\
&= \{m \in \mathsf{Sup}(s)\cup \mathsf{Sup}(s')~|~s_m \in A\textrm{ (with }m \in \mathsf{Sup}(s))\textrm{ or }s_m' \in A\textrm{ (with }m \in \mathsf{Sup}(s'))\}\\
& = \chi_A(s) \cup \chi_A(s')
\end{align*}
This shows that for every $A \in \mathsf{Atom}(R)$, the mapping rule $s \mapsto \chi_A(s)$ defines a morphism $R[M] \to FU(M)$ in $\mathbf{Icm}$.
\end{proof}

\begin{convention}[Atomic elements]\label{conv:atoms:wrt_RM}
Let $(R,+_R,\cdot_R,0_R,1_R)$ be an ic-semiring and $(M,+_M,0_M)$ be an ic-monoid. For every element $s \in R[M]$ and every atomic structure $(\Omega,\rho)$ on $R$, we will denote by $\Omega[s]$ the set $\{A \in \Omega~|~\chi_A(s) \neq \cmemptyset\}$.
\end{convention}

\begin{remark}[Atomic elements]\label{rem:atoms:wrt_RM}
Let $(R,+_R,\cdot_R,0_R,1_R)$ be an ic-semiring and $(M,+_M,0_M)$ be an ic-monoid. For every element $s \in R[M]$ and every atomic structure $(\Omega,\rho)$ on $R$, the equation $\kappa_A(s) = 0_M$ holds for every $A \in \mathsf{Atom}(R)\backslash \Omega[s]$ -- because we have $\chi_A(s) = \cmemptyset$.
\end{remark}

\begin{example}[Atomic elements]\label{exa:atomic-elements-Omega}
Let $(M,+_M,0_M)$ be an ic-monoid and consider the following polynomials in $\overline{\mathbb{Z}}_3[M]$ for a given tuple $(m_1,m_2,m_3,m_4,m_5)$ of elements in $M$.
\[
s = \uparrow\! X^{m_1} + \downarrow\!  X^{m_2} + \uparrow\!  X^{m_3} + \updownarrow\!  X^{m_4} + \updownarrow\!  X^{m_5}
\quad\quad\quad
s' = \uparrow \!X^{m_1}+\uparrow \!X^{m_2}
\]
According to Example \ref{exa:coefficient-sets}, we have the equalities $\Omega[s] = \{\Uparrow,\Downarrow\}$ and $\Omega[s'] = \{\Uparrow\}$.
\end{example}

\begin{proposition}\label{prop:Atom:finite}
Let $(R,+_R,\cdot_R,0_R,1_R)$ be an ic-semiring and $(M,+_M,0_M)$ be an ic-monoid. For every element $s \in R[M]$ and every atomic structure $(\Omega,\rho)$ on $R$, the set $\Omega[s]$ is finite.
\end{proposition}
\begin{proof}
It follows from Definition \ref{def:coefficient-sets} that that the inclusion $\bigcup_{A \in \mathsf{Atom}(R)} \chi_A(s) \subseteq \mathsf{Sup}(s)$ holds. Since  $s \in R[M]$, the set $\mathsf{Sup}(s)$ is finite and since the union $\bigcup_{A \in \mathsf{Atom}(R)} \chi_A(s)$ is disjoint (see Definition \ref{def:coefficient-sets}), the set of atomic elements $A \in \mathsf{Atom}(R)$ such that $\chi_A(s) \neq \cmemptyset$ must be finite.
\end{proof}

\begin{proposition}\label{prop:atom-of-elements:compatible-with-addition}
Let $(R,+_R,\cdot_R,0_R,1_R)$ be an ic-semiring and $(M,+_M,0_M)$ be an ic-monoid. For every atomic structure $(\Omega,\rho)$ on $R$, the mapping rule $s \mapsto \Omega[s]$ defines a morphism $\Omega[\cdot]:R[M] \to F(\Omega)$ in $\mathbf{Icm}$
\end{proposition}
\begin{proof}
Let us first show that the zero element is preserved. For every element $A \in \mathsf{Atom}(R)$, the set $\chi_A(0) = \{m \in \mathsf{Sup}(0)~|~0_R \in A\}$ is empty because the set $\mathsf{Sup}(0)$ is empty. By Definition \ref{conv:atoms:wrt_RM}, this means that the set $\Omega[0]$ is empty. Let us now show that the addition is preserved.
First, recall that, for every pair $(X,Y)$ of sets, if the inequality $X \cup Y \neq \cmemptyset$ holds, then we must have $X \neq \cmemptyset$ or $Y \neq \cmemptyset$. This implication and Proposition \ref{prop:kappa-morphism} allow us to show that the following equations hold for every pair $(s,s')$ of elements in $R[M]$.
\begin{align*}
\Omega[s+s'] &= \{A \in \Omega~|~\chi_A(s+s') \neq \cmemptyset \}  \\
&= \{A \in \Omega~|~\chi_A(s)\cup \chi_A(s') \neq \cmemptyset\} \\
&= \{A \in \Omega~|~\chi_A(s)\neq \cmemptyset \textrm{ or }\chi_A(s') \neq \cmemptyset\} \\
&= \Omega[s] \cup \Omega[s']
\end{align*}
The previous equalities and the equation $\Omega[0] = \cmemptyset$ (shown earlier) prove the statement.
\end{proof}

\begin{convention}[Tensor congruence morphism]\label{conv:tensor-congruence}
Let $(R,+_R,\cdot_R,0_R,1_R)$ be an ic-semiring and $(M,+_M,0_M)$ be an ic-monoid. For every atomic structure $(\Omega,\rho)$ on $R$, we will denote by $\Omega \otimes \kappa$ the morphism $R[M] \to F(\Omega) \times M^{\Omega}$ in $\mathbf{Icm}$ defined by the mapping rule $s \mapsto (\Omega[s],(\kappa_A(s))_{A \in \Omega})$.
\end{convention}

\begin{definition}[Tensor congruence]\label{def:tensor-congruence}
Let $(R,+_R,\cdot_R,0_R,1_R)$ be an ic-semiring and $(M,+_M,0_M)$ be an ic-monoid. For every atomic structure $(\Omega,\rho)$ on $R$, we define the \emph{$\Omega$-tensor congruence} as the congruence $v_1,v_2:\Omega(R,M) \rightrightarrows R[M]$ that results from pulling back the morphism $\Omega \otimes \kappa:R[M] \to F(\Omega) \times M^{\Omega}$ along itself, as shown below.
\[
\xymatrix{
\Omega(R,M) \ar@{}[rd]|<<<{\rotatebox[origin=c]{90}{\huge{\textrm{$\llcorner$}}}} \ar[r]^{v_1}\ar[d]_{v_2} & R[M]\ar[d]^{\Omega \otimes \kappa}\\
R[M] \ar[r]_{\Omega \otimes \kappa} & *+!L(.5){F(\Omega) \times M^{\Omega}}
}
\]
For every pair $(s,s')$ of elements in $R[M]$, we will write $s \equiv s'\,(\mathsf{wrt}\,\Omega)$ whenever $(s,s') \in \Omega(R,M)$. The binary relation $\cdot \equiv \cdot\,(\mathsf{wrt}\,\Omega)$ defines an equivalence relation on $R[M]$.
\end{definition}

\begin{convention}[Representatives]\label{conv:modulo:formula:bar_s_R}
Let $(R,+_R,\cdot_R,0_R,1_R)$ be an ic-semiring and $(M,+_M,0_M)$ be an ic-monoid. For every atomic structure $(\Omega,\rho)$ on $R$ and every element $s \in R[M]$, we will denote as $|s|_{\Omega}$ the following sum in $R[M]$.
\[
\textstyle \sum_{A \in \Omega[s]}  \rho(A) X^{\kappa_A(s)}
\]
The previous sum is well-defined because the set $\Omega[s]$ is finite (Proposition \ref{prop:Atom:finite}).
\end{convention}

\begin{example}[Representatives]\label{exa:modulo:formula:bar_s_R}
Let $(M,+_M,0_M)$ be an ic-monoid and consider the following polynomials in $\overline{\mathbb{Z}}_3[M]$ for a given tuple $(m_1,m_2,m_3,m_4,m_5)$ of elements in $M$.
\[
s = \uparrow\! X^{m_1} + \downarrow\!  X^{m_2} + \uparrow\!  X^{m_3} + \updownarrow\!  X^{m_4} + \updownarrow\!  X^{m_5}
\quad\quad\quad
s' = \uparrow \!X^{m_1}+\uparrow \!X^{m_2}
\]
According to Example \ref{exa:atomic-elements-Omega}, Example \ref{exa:kappa} and Example \ref{exa:atomic-structure}, we have the following equalities.
\[
|s|_{\Omega} = \uparrow\! X^{m_1+m_3+m_4+m_5}+ \downarrow\! X^{m_2+m_4+m_5}
\quad\quad
|s|_{\Omega} = \uparrow\! X^{m_1+m_2}
\]
The astute reader may have noticed that the coefficients of $s$ and $s'$ in $\overline{\mathbb{Z}}_3[M]$ determine the expressions of $|s|_{\Omega}$ and $|s'|_{\Omega}$. This is explained in more detail in Remark \ref{rem:representation-bar-s-bar-omega} below.
\end{example}

\begin{remark}[Reformulation]\label{rem:representation-bar-s-bar-omega}
Let $(R,+_R,\cdot_R,0_R,1_R)$ be an ic-semiring and $(M,+_M,0_M)$ be an ic-monoid. For every atomic structure $(\Omega,\rho)$ on $R$ and every element $s \in R[M]$, we shall let $\delta_{s}$ denote the function $\rho(\Omega[s]) \to M$ that sends an element $r$ in $\rho(\Omega[s])$ to the element $\kappa_A(s) \in M$ where $A$ is the unique element of $\Omega[s]$ such that $\rho^{-1}(r) = \{A\}$ (by definition of the set $\rho(\Omega[s])$). With these notations, it follows from Convention \ref{conv:modulo:formula:bar_s_R} that the following formula holds.
\begin{equation}\label{eq:representation-bar-s-bar-omega:rem}
|s|_{\Omega} = \sum_{r \in \rho(\Omega[s])}  r X^{\delta_{s}(r)}
\end{equation}
By Definition \ref{def:atomic-structure} and Convention \ref{conv:atoms:wrt_RM}, the set $\rho(\Omega[s])$ is a subset of $R\backslash\{0_R\}$ for every element $s \in R[M]$. We shall expand further on expression (\ref{eq:representation-bar-s-bar-omega:rem}) in Remark \ref{rem:modulo-tensor:fix-points}.
\end{remark}

The logical implication stated in Proposition \ref{prop:tensor_congruence:implies:bar_R} will be shown to be a logical equivalence in Remark \ref{rem:modulo-tensor:fix-points}. To prepare for this equivalence, we will use Remark to discuss general facts about the morphism defined in Convention \ref{conv:tensor-congruence}.

\begin{proposition}\label{prop:tensor_congruence:implies:bar_R}
Let $(R,+_R,\cdot_R,0_R,1_R)$ be an ic-semiring and $(M,+_M,0_M)$ be an ic-monoid. For every atomic structure $(\Omega,\rho)$ on $R$ and every pair $(s,s')$ of elements in $R[M]$, if the relation $s \equiv s'\,(\mathsf{wrt}\,\Omega)$ holds, then so does the equality $|s|_{\Omega} = |s'|_{\Omega}$.
\end{proposition}
\begin{proof}
By Definition \ref{def:tensor-congruence}, if the relation $s \equiv s'\,(\mathsf{wrt}\,\Omega)$ holds, then so does the equation $\Omega \otimes \kappa(s) =  \Omega \otimes \kappa(s')$. By Convention \ref{conv:tensor-congruence}, Convention \ref{conv:modulo:formula:bar_s_R}, and Remark \ref{rem:atoms:wrt_RM} we can show that this means that the equation $\sum_{A \in \Omega[s]}  \rho(A) X^{\kappa_A(s)} = \sum_{A \in \Omega[s']}  \rho(A) X^{\kappa_A(s')}$ holds.
\end{proof}

\begin{remark}[Canonical expressions]\label{rem:modulo-tensor:fix-points}
The goal of this remark is to expand on the discussion started in Remark \ref{rem:representation-bar-s-bar-omega} and propose reformulations of the parameters used in the formula provided in Convention \ref{conv:modulo:formula:bar_s_R}. In particular, we prove the formulas shown in (\ref{eq:modulo-tensor:fix-points:1}) below.

Let $(R,+_R,\cdot_R,0_R,1_R)$ be an atomic ic-semiring and $(M,+_M,0_M)$ be an ic-monoid.
For every finite subset $O \subseteq R \backslash \{0_R\}$ and every function $d:O \to M$, if we express an element $x = (x_m)_{m \in M}$ of $R[M]$ as a sum as shown on the left-hand side of (\ref{eq:modulo-tensor:fix-points:0}), then the equation shown on the right-hand side must hold for every element $m \in M$.
\begin{equation}\label{eq:modulo-tensor:fix-points:0}
x = \textstyle \sum_{r \in O} r X^{d(r)}
\quad\quad\quad\quad
\Rightarrow
\quad\quad\quad\quad
x_m = \sum_{r \in d^{-1}(m)} r 
\end{equation}
This means that, for every element $A \in \mathsf{Atom}(R)$, the $A$-coefficient set $\chi_{A}(x)$ has the following specification.
\begin{align*}
\chi_{A}(x) &= \textstyle \{m~|~\sum_{v \in d^{-1}(m)} v \in A\} & \\
& = \{m~|~\textrm{there exists }v \in d^{-1}(m)\textrm{ such that }v \in A\} &(\textrm{Proposition  \ref{prop:extend-atomic:large-sums}})\\
& = d(A)
\end{align*}
It follows from Proposition \ref{prop:kappa-morphism} ($\kappa_A = \mu \circ \chi_A$) and Convention \ref{conv:atoms:wrt_RM} ($\Omega[x] = \{A \in \Omega~|~\chi_A(x) \neq \cmemptyset\}$) that the following identities hold for every element $A \in \mathsf{Atom}(R)$.
\begin{equation}\label{eq:modulo-tensor:fix-points:1}
\kappa_{A}(x) = \mu (d(A)) \quad\quad\quad\quad\quad
\Omega[x] = \{A \in \Omega~|~A \cap O \neq \cmemptyset \}
\end{equation}
We shall use the formulas of (\ref{eq:modulo-tensor:fix-points:1}) in Proposition \ref{prop:modulo-tensor:fix-points} to show that the implication stated in Proposition \ref{prop:tensor_congruence:implies:bar_R} is an equivalence.
\end{remark}

\begin{proposition}\label{prop:modulo-tensor:fix-points}
Let $(R,+_R,\cdot_R,0_R,1_R)$ be an ic-semiring and $(M,+_M,0_M)$ be an ic-monoid. For every atomic structure $(\Omega,\rho)$ on $R$ and every pair $(s,s')$ of elements in $R[M]$, the relation $s \equiv s'\,(\mathsf{wrt}\,\Omega)$ holds if, and only if, the equality $|s|_{\Omega} = |s'|_{\Omega}$ holds. In addition, for every element $s \in R[M]$, the relation $s \equiv |s|_R\,(\mathsf{wrt}\,\Omega)$ holds.
\end{proposition}
\begin{proof}
The direct implication of the statement is shown in Proposition \ref{prop:tensor_congruence:implies:bar_R}. We shall now prove the opposite direction. In this respect, let $(s,s')$ be a pair of elements in $R[M]$ for which the equation $|s|_{\Omega} = |s'|_{\Omega}$ holds.
It follows from Remark \ref{rem:representation-bar-s-bar-omega} and Remark \ref{rem:modulo-tensor:fix-points} that the following equations hold for every element $s \in R[M]$ and every element $A \in \Omega$.
\begin{align*}
\kappa_A(|s|_{\Omega}) & = \mu (\delta_s(A)) & (\textrm{Rem. \ref{rem:representation-bar-s-bar-omega} and Rem. \ref{rem:modulo-tensor:fix-points}})\\
& = \textstyle \sum_{r \in A \cap \rho(\Omega[s])} \delta_s(r)&(\textrm{see $\mu$ in sec. \ref{ssec:ICMonoids_universal_construction}})\\
& = \textstyle  \sum_{r \in A \cap \rho(\Omega[s])} \kappa_A(s) & (\textrm{Prop. \ref{prop:atomic-structure:restriction} and Rem. \ref{rem:representation-bar-s-bar-omega}})\\
& = \kappa_A(s)&(\textrm{idempotency})
\end{align*}
Similarly, we deduce from Remark \ref{rem:representation-bar-s-bar-omega} and Remark \ref{rem:modulo-tensor:fix-points} that the following equations for every element $s \in R[M]$.
\begin{align*}
\Omega[|s|_{\Omega}] &= \{A \in \Omega~|~A \cap \rho(\Omega[s]) \neq \cmemptyset\} & (\textrm{Rem. \ref{rem:representation-bar-s-bar-omega} and Rem. \ref{rem:modulo-tensor:fix-points}})\\
&= \{A \in \Omega~|~\exists r \in A \cap \rho(\Omega[s])\} &\\
 &= \{A \in \Omega~|~\exists r \in R:\rho^{-1}(r)=\{A\}\textrm{ and }A \in \Omega[s]\} & (\textrm{Proposition \ref{prop:atomic-structure:restriction}})\\
 &= \{A \in \Omega~|~A \in \Omega[s]\} & (\textrm{take $r= \rho(A)$})\\
 & = \Omega[s]
\end{align*}
By Convention \ref{conv:tensor-congruence} and Definition \ref{def:tensor-congruence}, this means that the relation $s \equiv |s|_R\,(\mathsf{wrt}\,\Omega)$ holds.  Similarly, we deduce that the relation $s' \equiv |s'|_R\,(\mathsf{wrt}\,\Omega)$ holds. Since the equation $|s|_{\Omega} = |s'|_{\Omega}$ holds and the relation $\cdot \equiv \cdot\,(\mathsf{wrt}\,\Omega)$ is an equivalence (see Definition \ref{def:tensor-congruence}), the relation $s \equiv s'\,(\mathsf{wrt}\,\Omega)$ holds.
\end{proof}

\begin{remark}[Conclusion]\label{rem:conclusion-tensor-congruence}
Let $(R,+_R,\cdot_R,0_R,1_R)$ be an ic-semiring and $(M,+_M,0_M)$ be an ic-monoid. For every atomic structure $(\Omega,\rho)$ on $R$, the $\Omega$-tensor congruence on $R[M]$ (Definition \ref{def:tensor-congruence}) mimics a tensor structure by providing a way to canonically regroup the monomials of the element of $R[M]$ with respect to their associated value in $R$. To give an example, the following equation mimic the bilinearity of tensors with respect to their second variable for every $r_0 \in \rho(\Omega)$.
\begin{equation}\label{eq:rem:modulo-tensor}
\textstyle s = r_0X^{m_1} + r_0X^{m_2} \quad \Rightarrow \quad |s|_{\Omega} = r_0X^{m_1+m_2}
\end{equation}
More specifically, the previous equation should be compared to an equation of the form
\[
r_0 \otimes m_1 + r_0 \otimes m_1 = r_0 \otimes (m_1+m_2),
\]
which usually holds in tensors of the form $R \otimes M$. The difference between equipping the ic-monoid $R[M]$ with the relation $\cdot \equiv \cdot\,(\mathsf{wrt}\,\Omega)$ and considering an actual tensor structure is that the exponent $m_1+m_2$ shown in (\ref{eq:rem:modulo-tensor}) is computed modulo some other elements in $R$.
\end{remark}

\subsection{Multiplicative atomic structures}\label{ssec:Multiplicative-atomic-structures} The goal of this section is to determine conditions for which the tensor congruence is compatible with the semiring of polynomials on which it is defined (see Theorem \ref{theo:tensor-congruence:semiring-compatiblity}). To understand how the tensor congruence acts on the multiplications of the semiring structure, we shall first introduce the notion of \emph{multiplicative} atomic structure (Definition \ref{def:semiring:integral}). We then use this type of atomic structure to show that the tensor congruence is compatible with the multiplication (see Proposition \ref{prop:kappa:almost-multiplicative-morphism} and Proposition \ref{prop:atom-of-elements:compatible-with-boxdot}).

\begin{definition}[Multiplicative atomic structures]\label{def:semiring:integral}
Let $(R,+,\cdot,0,1)$ be an ic-semiring. An atomic structure $(\Omega,\rho)$ on $R$ will be said to be \emph{multiplicative} if, for every element $A \in \Omega$, it is equipped with a function $\mathsf{Inv}_A:\Omega \to \Omega$ such that, for every pair $(t,t') \in R \times R$ and every element $A \in \Omega$, the following implication holds.
\[
t \cdot t' \in A \quad\Leftrightarrow\quad \exists A' \in \Omega:\,t \in A'\textrm{ and }t' \in \mathsf{Inv}_A(A')
\]
\end{definition}

\begin{example}[Multiplicative atomic structure] \label{exa:Z3:multiplicative}
The goal of this example is to define the multiplicative structure on $(\Omega,\rho)$ defined in Example \ref{exa:atomic-structure}, on the ic-semiring $\overline{\mathbb{Z}}_3$. For every $A \in \Omega$, let us define a function $\mathsf{Inv}_A:\Omega \to \Omega$ according to the following mapping rules for every parameter $A \in \{\Uparrow,\Downarrow\}$.
\[
\mathsf{Inv}_{\Uparrow}:
\left(
\begin{array}{lll}
\Uparrow &\mapsto& \Uparrow\\
\Downarrow &\mapsto &\Downarrow
\end{array}
\right)
\quad\quad\quad
\mathsf{Inv}_{\Downarrow}:
\left(
\begin{array}{lll}
\Uparrow &\mapsto& \Downarrow\\
\Downarrow &\mapsto &\Uparrow
\end{array}
\right)
\]
It is straightforward to verify that the function $\mathsf{Inv}_A:\Omega \to \Omega$ is a bijection whose inverse is itself (in other words, this is an involution on $\Omega$). According to the Cayley table for the multiplication provided in Example \ref{exa:semi-group:Z3Z}, we can verify that the following equivalences hold.
\[
\left\{
\begin{array}{lll}
t \cdot t' \in \{\uparrow,\updownarrow\}
&\Leftrightarrow&
\textrm{($t \in \{\uparrow,\updownarrow\}$ and $t' \in \{\uparrow,\updownarrow\}$) or ($t \in \{\downarrow,\updownarrow\}$ and $t' \in \{\downarrow,\updownarrow\}$)} \\
t \cdot t' \in \{\downarrow,\updownarrow\}
&\Leftrightarrow&
\textrm{($t \in \{\uparrow,\updownarrow\}$ and $t' \in \{\downarrow,\updownarrow\}$) or ($t \in \{\uparrow,\updownarrow\}$ and $t' \in \{\downarrow,\updownarrow\}$)}
\end{array}
\right.
\]
The direct directions of the previous equivalences show that if $t \cdot t' \in A$, then there exists $A' \in \Omega$ such that $t \in A'$ and $t' \in \mathsf{Inv}_A(A')$. Conversely, the opposite implications  show that for every $A \in \Omega$ and every $A' \in \Omega$, if $t \in A'$ and $t' \in \mathsf{Inv}_A(A')$, then $t \cdot t' \in A$. In other words, the atomic structure $(\Omega,\rho)$ is multiplicative when equipped with the involution $\mathsf{Inv}_A:\Omega \to \Omega$ defined earlier.
\end{example}

\begin{convention}[Modified kappa]\label{conv:modified-kappa}
Let $(R,+_R,\cdot_R,0_R,1_R)$ be an ic-semiring and $(M,+_M,0_M)$ be an ic-monoid. For every multiplicative atomic structure $(\Omega,\rho,\mathsf{Inv})$ on $R$, every element $A \in \Omega$ and every element $s \in R[M]$, we will denote by
\begin{itemize}
\item[1)] $\widehat{\kappa}_A(s)$ the sum $\sum_{A' \in \Omega} \kappa_{A'}(s)$ in $M$;
\item[2)] $\widecheck{\kappa}_A(s)$ the sum $\sum_{A' \in \Omega} \kappa_{\mathsf{Inv}_A(A')}(s)$ in $M$;
\end{itemize}
If for every $A \in \Omega$, the function $\mathsf{Inv}_A$ is a bijection, then the identity $\widehat{\kappa}_A(s) = \widecheck{\kappa}_A(s)$ holds.
\end{convention}

\begin{proposition}\label{prop:kappa:almost-multiplicative-morphism}
Let $(R,+_R,\cdot_R,0_R,1_R)$ be an atomic ic-semiring and $(M,+_M,0_M)$ be an ic-monoid. For every multiplicative atomic structure $(\Omega,\rho,\mathsf{Inv})$ on $R$, every element $A \in \Omega$ and every pair $(s,s')$ of elements in $R[M]$, the following equations hold.
\[
\def\arraystretch{1.4}
\begin{array}{ll}
\kappa_A(1_RX^{0_M}) &= 0_M\\
\kappa_A(s \boxdot s') &= \widehat{\kappa}_{A}(s) +_M \widecheck{\kappa}_{A}(s')\\
\chi_A(s \boxdot s') &= \bigcup_{A' \in \Omega}\{d_1+_Md_2~|~ d_1 \in \chi_{A'}(s)\textrm{ and } d_2 \in \chi_{\mathsf{Inv}_A(A')}(s')\}
\end{array}
\]
\end{proposition}
\begin{proof}
We start with the topmost equation of the statement, namely $\kappa_A(1_RX^{0_M}) = 0_M$ for any $A \in \Omega$. This equation is a direct consequence of Remark \ref{rem:modulo-tensor:fix-points}, and more specifically, a consequence of the leftmost equation of (\ref{eq:modulo-tensor:fix-points:1}). Let us now show the bottommost equation given in the statement. We will then use this equation to show the equation $\kappa_A(s \boxdot s') = \widehat{\kappa}_{A}(s) +_M \widecheck{\kappa}_{A}(s')$ shown in the middle. To start with, notice that we can reformulate the specification of the set $\chi_A(s \boxdot s')$ by using the formula given in Definition \ref{def:product-coefficients}, as shown below.
\begin{align*}
\chi_A(s \boxdot s') &= \{d\in \mathsf{Sup}(s \boxdot s')~|~s \boxdot_d s' \in A\}  \\
&= \left\{d \in \mathsf{Sup}(s,s')~\left|~ \sum_{d_1 \in \mathsf{Sup}(s)} \sum_{d_2 \in \mathsf{Sup}(s')}\sum_{d_1+_M d_2 = d} s_{d_1} \cdot_R s'_{d_2} \in A\right.\right\}
\end{align*}
Then, we can use Proposition \ref{prop:extend-atomic:large-sums} to reformulate the previous formula as follows.
\[
\chi_A(s \boxdot s') = \{d_1+_M d_2~|~ \textrm{for every }(d_1,d_2) \in \mathsf{Sup}(s) \times \mathsf{Sup}(s'):\, s_{d_1} \cdot_R s'_{d_2} \in A\}
\]
Because the atomic structure $(\Omega,\rho)$ on $R$ is multiplicative, the relation $s_{d_1} \cdot_R s'_{d_2} \in A$ is equivalent to the existence of an element $A' \in \Omega$ such that $s_{d_1} \in A'$ and $s'_{d_2} \in \mathsf{Inv}_A(A')$. This can be reformulated into the two relations $d_1 \in \chi_{A'}(s)$ and $d_2 \in \chi_{\mathsf{Inv}_A(A')}(s')$. Hence, we deduce that the following reformulation holds -- this shows the bottommost equation of the statement.
\[
\chi_A(s \boxdot s') = \bigcup_{A' \in \Omega}\{d_1+_Md_2~|~ d_1 \in \chi_{A'}(s)\textrm{ and } d_2 \in \chi_{\mathsf{Inv}_A(A')}(s')\}
\]
By applying the counit morphism $\mu:FU(M) \to M$ of the adjunction $F \dashv U$ (see diagram (\ref{adjunction_set_icm})) on the previous unions gives the following equation (see Proposition \ref{prop:kappa-morphism}).
\[
\kappa_A(s \boxdot s') =  \sum_{A' \in \Omega}(\kappa_{A'}(s)+_M\kappa_{\mathsf{Inv}_A(A')}(s'))
\
\]
By using Convention \ref{conv:modified-kappa}, we can reformulate the previous equation as $\kappa_A(s \boxdot s') =  \widehat{\kappa}_{A}(s) +_M \widecheck{\kappa}_{A}(s')$. This finishes the proof of the statement.
\end{proof}

\begin{proposition}\label{prop:atom-of-elements:compatible-with-boxdot}
Let $(R,+_R,\cdot_R,0_R,1_R)$ be an atomic ic-semiring and $(M,+_M,0_M)$ be an ic-monoid. For every multiplicative atomic structure $(\Omega,\rho,\mathsf{Inv})$ on $R$, and every pair $(s,s')$ of elements in $R[M]$, the following equation holds.
\[
\Omega[s \boxdot s'] = \{A \in \Omega~|~\Omega[s] \cap \mathsf{Inv}_A^{-1}(\Omega[s']) \neq \cmemptyset\}
\]
\end{proposition}
\begin{proof}
Recall that, for every pair $(X,Y)$ of sets, if the inequality $X \cup Y \neq \cmemptyset$ holds, then we must have $X \neq \cmemptyset$ or $Y \neq \cmemptyset$. By applying this principle on the terms of the expression of $\chi_r(s \boxdot s')$ given in Proposition \ref{prop:kappa:almost-multiplicative-morphism}, we deduce that the following equations hold.
\begin{align*}
\Omega[s \boxdot s']  &= \{A \in \Omega~|~\chi_A(s \boxdot s') \neq \cmemptyset \}  \\
&= \{A \in \Omega~|~\exists A' \in \Omega:\,\chi_{A'}(s)\neq \cmemptyset\textrm{ and }\chi_{\mathsf{Inv}_A(A')}(s')\neq \cmemptyset\} \\
&= \{A \in \Omega~|~\Omega[s] \cap \mathsf{Inv}_A^{-1}(\Omega[s']) \neq \cmemptyset\}
\end{align*}
This implies the statement.
\end{proof}

\begin{definition}[Morphisms of semirings]\label{def:morphism_of_semrirings}
Let $(R,+_R,\cdot_R,0_R,1_R)$ and $(S,+_S,\cdot_S,0_S,1_S)$ be two semirings. We define a \emph{morphism $(R,+_R,\cdot_R,0_R,1_R) \to (S,+_S,\cdot_S,0_S,1_S)$ of semirings} \cite{Golan}  as a function $f:R \to S$ such that $f$ induces a morphism $(R,+_R,0_R) \to (S,+_S,0_S)$ of (commutative) monoids and a morphism $(R,\cdot_R,1_R) \to (S,\cdot_S,1_S)$ of monoids.
\end{definition}

\begin{theorem}[Semiring]\label{theo:tensor-congruence:semiring-compatiblity}
Let $(R,+_R,\cdot_R,0_R,1_R)$ be an atomic ic-semiring and $(M,+_M,0_M)$ be an ic-monoid. For every multiplicative atomic structure $(\Omega,\rho,\mathsf{Inv})$ on $R$, the ic-monoid $\Omega(R,M)$ defines an ic-semiring whose
\begin{itemize}
\item[1)] unit element is the pair $(1,1)$ of two copies of the unit element $1 \in R[M]$;
\item[2)] multiplication is defined by the mapping rule $((s_1,s_1'),(s_2,s_2')) \mapsto (s_1 \boxdot s_2,s_1' \boxdot s_2')$.
\end{itemize}
Additionally, the two projections $v_1:\Omega(R,M) \to R[M]$ and $v_2:\Omega(R,M) \to R[M]$ associated with the $\Omega$-tensor congruence are morphisms of semirings.
\end{theorem}
\begin{proof}
The statement is straightforward once we show that the mapping rule $((s_1,s_1'),(s_2,s_2')) \mapsto (s_1 \boxdot s_2,s_1' \boxdot s_2')$ is a stable in $\Omega(R,M)$. In other words, the statement will follow if we can show that the following implication holds.
\[
(s_1,s_1') \in \Omega(R,M)\textrm{ and }(s_2,s_2') \in \Omega(R,M)\quad\Rightarrow\quad
(s_1 \boxdot s_2,s_1' \boxdot s_2') \in \Omega(R,M)
\]
In this respect, let $(s_1,s_1')$ and $(s_2,s_2')$ be two elements in $\Omega(R,M)$. By Definition \ref{def:tensor-congruence}, this means that we have the following equations.
\[
\left\{
\begin{array}{ll}
\Omega[s_1] = \Omega[s_1']&\\
\kappa_A(s_1) = \kappa_A(s_1')&\textrm{ for every }A \in \Omega\\
\end{array}
\right.
\quad
\left\{
\begin{array}{ll}
\Omega[s_2] = \Omega[s_2']&\\
\kappa_A(s_2) = \kappa_A(s_2')&\textrm{ for every }A \in \Omega\\
\end{array}
\right.
\]
It follows from Convention \ref{conv:modified-kappa} and the previous relations that the following equations hold.
\[
\left\{
\begin{array}{ll}
\widehat{\kappa}_A(s_1) = \widehat{\kappa}_A(s_1')&\textrm{ for every }A \in \Omega\\
\widehat{\kappa}_A(s_2) = \widehat{\kappa}_A(s_2')&\textrm{ for every }A \in \Omega\\
\end{array}
\right.
\quad\quad\quad
\left\{
\begin{array}{ll}
\widecheck{\kappa}_A(s_1) = \widecheck{\kappa}_A(s_1')&\textrm{ for every }A \in \Omega\\
\widecheck{\kappa}_A(s_2) = \widecheck{\kappa}_A(s_2')&\textrm{ for every }A \in \Omega\\
\end{array}
\right.
\]
Finally, we can use the set of equations shown above with Proposition \ref{prop:kappa:almost-multiplicative-morphism} and Proposition \ref{prop:atom-of-elements:compatible-with-boxdot} to show that the following equations hold.
\[
\left\{
\begin{array}{ll}
\Omega[s_1 \boxdot s_2] = \Omega[s_1' \boxdot s_2']&\\
\kappa_A(s_1 \boxdot s_2) = \kappa_A(s_1' \boxdot s_2')&\textrm{ for every }A \in \Omega\\
\end{array}
\right.
\]
We conclude that the pair $(s_1 \boxdot s_2,s_1' \boxdot s_2')$ belongs to $\Omega(R,M)$. This finishes the proof of the statement since all the other axioms for the semiring structure of $\Omega(R,M)$ follow directly from the semiring structure of $\Omega(R,M)$.
\end{proof}

\subsection{Actions of semirings}\label{ssec:action-semirings}
The goal of this section is to formalize the properties shown in Remark \ref{rem:semi-direct-product} for the function defined in Definition \ref{def:action:rationals} as an ``action morphism'' (see Definition \ref{def:action-semi-groups}). Note that such a morphism is neither a morphism of ic-monoids nor a morphism of ic-semirings, but possesses very similar properties (see Remark \ref{rem:action-as-morphism:ic-monoids}). The main result of this section is given in Proposition \ref{prop:alpha:alpha:alpha:boxdot:diagram}, in which we relate an action of an ic-semiring on an ic-monoid to the underlying semiring of polynomials (defined in the ic-semiring over the ic-monoid).

\begin{definition}[Action]\label{def:action-semi-groups}
We define an \emph{action} of a semiring $(R,+_R,\cdot_R,0_R,1_R)$ on a commutative monoid $(M,+_M,0_M)$ is a function $m:R \times M \to M$ satisfying the following equations for every pair $(f,g) \in R \times R$ and every pair $(x,y) \in M \times M$.
\[
\left\{
\begin{array}{lll}
\alpha(f,0_M) &=& 0_M \\
\alpha(0_R,x) &=& 0_M \\
\alpha(g\cdot_R f,x) &=& \alpha(g,\alpha(f,x))\\
 &=& \alpha(f,\alpha(g,x))\\
\end{array}
\right.
\quad\quad\quad\quad
\left\{
\begin{array}{lll}
\alpha(f,x+_M y) &=& \alpha(f,x) +_M \alpha(f,y)\\
\alpha(g+_R f,x) &=& \alpha(g,x) +_M \alpha(f,x)\\
\end{array}
\right.
\]
\end{definition}

\begin{example}[Rationals]\label{exa:action:rationals:ic-monoids}
Let $(M,+,0)$ be an object in $\mathbf{Icm}$. According to Remark \ref{rem:semi-direct-product}, the function $\odot:\overline{\mathbb{Z}}_3 \times \mathbb{Q}(M) \to \mathbb{Q}(M)$ defined in Definition \ref{def:action:rationals} constitutes an action of the ic-semiring $\overline{\mathbb{Z}}_3$ on the ic-monoid $\mathbb{Q}(M)$.
\end{example}

\begin{example}[Semirings]
For every commutative semiring $(R,+,\cdot,0,1)$, the multiplication operation $\mathsf{mul}:R \times R \to R$ defined by the mapping rule $(x,y) \mapsto x \cdot y$ is an action of the semiring $(R,+,\cdot,0,1)$ on its underlying commutative monoid $(R,+,0)$ because:
\begin{itemize}
\item[-] the element $0$ annihilates every element of $R$ with respect to the multiplication;
\item[-] the multiplication is associative and commutative;
\item[-] the multiplication is distributive with respect to the addition.
\end{itemize}
See the beginning of section \ref{ssec:idempotent_commutative_semirings} for a definition of semirings.
\end{example}

\begin{example}[Morphisms]\label{exa:morphism:ic-semiring:action}
For every morphism $f:(R,+_R,\cdot_R,0_R,1_R) \to (S,+_S,\cdot_S,0_S,1_S)$ of semirings, the function $R \times S \to S$ defined by the mapping rule $(r,s) \mapsto f(r) \cdot s$ is an action of the semiring $R$ on the underlying commutative monoid $(S,+_S,0_S)$. Indeed, for every pair $(r,r') \in R \times R$ and every pair $(s,s') \in S \times S'$, we have:
\[
\left\{
\def\arraystretch{1.4}
\begin{array}{ll}
f(r) \cdot_S 0_S = 0_S& (\textrm{annihilation})\\
f(0_R) \cdot_S s = 0_S \cdot_S s = 0_S& (\textrm{annihilation})\\
f(r \cdot_R r') \cdot_S s = f(r) \cdot_S f(r') \cdot_S s & (\textrm{associativity})\\
f(r) \cdot_S (s+s') = f(r) \cdot_S s +  f(r) \cdot_S s'& (\textrm{distributivity})\\
f(r +_R r') \cdot_S s = (f(r) +_S f(r')) \cdot_S s = (f(r) \cdot_S s) +_S (f(r')\cdot_S s) &  (\textrm{distributivity})\\
\end{array}
\right.
\]
The previous equations correspond to those displayed in Definition \ref{def:action-semi-groups}.
\end{example}

\begin{remark}[Actions as morphisms of ic-monoids]\label{rem:action-as-morphism:ic-monoids}
Let $\alpha:R \times M \to R$ be action of an ic-semiring $(R,+_R,\cdot_R,0_R,1_R)$ on an ic-monoid $(M,+_M,0_M)$. We can define a function $\widetilde{\alpha}:R[M] \to M$ whose mapping rule is defined as follows.
\[
\left(
\begin{array}{ccc}
R[M]& \to & M\\
r = \sum_{m \in \mathsf{Sup}(r)}r_mX^{m} &\mapsto &\sum_{m \in \mathsf{Sup}(r)}\alpha(r_m,m)
\end{array}
\right)
\]
We can show that the function $\widetilde{\alpha}$ induces a morphism $R[M] \to M$ in $\mathbf{Icm}$. Indeed, it follows from Proposition \ref{prop:support:sum}, Definition \ref{def:action-semi-groups} and the point-wise definition of the addition of $R[M]$ that the following equation holds for every pair $(r,r') \in R[M] \times R[M]$.
\begin{equation}\label{eq:action-as-morphism:ic-monoids:RM}
\widetilde{\alpha}(r+r') = \sum_{m \in \mathsf{Sup}(r) \cup \mathsf{Sup}(r')}(\alpha(r_m,m)+_M\alpha(r_m',m))
\end{equation}
Now, observe that if $m \in \mathsf{Sup}(r)$ and $m \notin \mathsf{Sup}(r')$, then the equation $r_m' = 0_R$ holds. This means that the equation $\alpha(r_m',m) = 0_M$ holds. Similarly, if $m \notin \mathsf{Sup}(r)$ and $m \in \mathsf{Sup}(r')$, then $\alpha(r_m,m)  = 0_M$. As a result, expression (\ref{eq:action-as-morphism:ic-monoids:RM}) can be reformulated as follows.
\[
\widetilde{\alpha}(s+s') = \sum_{m \in \mathsf{Sup}(r)}\alpha(r_m,m)+_M\sum_{m \in \mathsf{Sup}(r')}\alpha(r_m',m)
\]
In other words, the equation $\widetilde{\alpha}(r+r') = \widetilde{\alpha}(r) + \widetilde{\alpha}(r')$ holds. It is also direct to observe that the equation $\widetilde{\alpha}(0) = 0_M$ holds. As a result, the function $\widetilde{\alpha}$ induces a morphism $(R[M],+,0) \to (M,+_M,0_M)$ in $\mathbf{Icm}$.
\end{remark}

The following proposition should be read with the following fact in mind: every ic-monoid $M$ can be seen as a submonoid of its associated ic-monoid $\mathbb{Q}(M)$ through the monomorphism $\mathbf{1}:M \to \mathbb{Q}(M)$. This means that the ic-semiring $\overline{\mathbb{Z}}_3[M]$ can be seen as a submonoid of the ic-monoid $\overline{\mathbb{Z}}_3[\mathbb{Q}(M)]$ through the induced monomorphism $\overline{\mathbb{Z}}_3[\mathbf{1}]:\overline{\mathbb{Z}}_3[M] \to \overline{\mathbb{Z}}_3[\mathbb{Q}(M)]$.

\begin{proposition}\label{prop:YAMC}
Let $(M,+_M,0_M)$ be an ic-monoid. For every element $s \in \overline{\mathbb{Z}}_3[M]$, the equation $\widetilde{\odot}(s) = \widetilde{\odot}(|s|_{\Omega})$ holds, where $\widetilde{\odot}$ denotes the morphism $\overline{\mathbb{Z}}_3[\mathbb{Q}(M)] \to \mathbb{Q}(M)$ in $\mathbf{Icm}$ induced from the action $\odot:\overline{\mathbb{Z}}_3 \times \mathbb{Q}(M) \to \mathbb{Q}(M)$.
\end{proposition}
\begin{proof}
Let $s$ be an element in $\overline{\mathbb{Z}}_3[M]$ and let $(\Omega,\rho)$ denote the atomic structure defined in Example \ref{exa:atomic-structure}. Note that the element $s$ can be seen as an element of $\overline{\mathbb{Z}}_3[\mathbb{Q}(M)]$. Now, for every element $r \in \overline{\mathbb{Z}}_3$, let us denote the set $\{m \in \mathsf{Sup}(s)~|~s_m = r\}$ as $K_r(s)$. As a result of this notation, we have the following formula.
\begin{equation}\label{eq:YAMC:K-formula}
s = \textstyle \sum_{r \in \overline{\mathbb{Z}}_3}  \sum_{m \in K_r(s)} rX^m
\end{equation}
Observe that the collection consisting of the sets $K_r(s)$, for all $r \in \overline{\mathbb{Z}}_3$, forms a partition of $\mathsf{Sup}(s)$.  Also, note that, for every $A \in \Omega$, the equation $K_{\rho(A)}(s) \cup K_{\updownarrow}(s) = \chi_{A}(s)$ holds.
This being established, we can use formula (\ref{eq:YAMC:K-formula}) and the construction of Remark \ref{rem:action-as-morphism:ic-monoids} to reformulate the expression of $\widetilde{\odot}(s)$ as follows -- we use the symbol $\sum$ to denote an large sum in $(\mathbb{Q}(M),\star,\mathbbm{1}_M)$.
\begin{align*}
\widetilde{\odot}(s) &= \textstyle \sum_{r \in \overline{\mathbb{Z}}_3}  \sum_{m \in K_r(s)} \widetilde{\odot}(r,m) &(\textrm{by linearity})\\
&= \textstyle \sum_{r \in \overline{\mathbb{Z}}_3} \widetilde{\odot}\big(r,\sum_{m \in K_r(s)} m\big) & (\textrm{Remark \ref{rem:action-as-morphism:ic-monoids}})\\
&= \textstyle \mathbbm{1}_{M} \star \widetilde{\odot}\big(\!\uparrow,\sum_{m \in K_{\uparrow}(s)} m\big)   \star \widetilde{\odot}\big(\!\downarrow, \sum_{m \in K_{\downarrow}(s)} m\big) \star  \widetilde{\odot}\big(\!\updownarrow, \sum_{m \in K_{\updownarrow}(s)} m\big) & (
\overline{\mathbb{Z}}_3 = \{0,\uparrow,\downarrow,\updownarrow\})
\end{align*}
If we use the formulas and relations established in Remark \ref{conv:rationals:ic-monoids}, we can reformulate the previous equality as follows.
\begin{align*}
\widetilde{\odot}(s) & = \left[\frac{\sum_{m \in K_{\uparrow}(s)} m+\sum_{m \in K_{\updownarrow}(s)} m}{\sum_{m \in K_{\downarrow}(s)} m+\sum_{m \in K_{\updownarrow}(s)} m}\right]&\left(\begin{array}{c}\textrm{Definition \ref{def:action:rationals}}\\\textrm{and }K_r(s) \subseteq M\end{array}\right)\\
& = \left[\frac{\kappa_{\Uparrow}(s)}{\kappa_{\Downarrow}(s)}\right]&(K_{\rho(A)}(s) \cup K_{\updownarrow}(s) = \chi_{A}(s))&\\
& = \widetilde{\odot}(\uparrow,\kappa_{\Uparrow}(s))  \star \widetilde{\odot}(\downarrow, \kappa_{\Downarrow}(s))&(\textrm{Definition \ref{def:action:rationals}})\\
& = \textstyle \sum_{A \in \Omega[s]}\widetilde{\odot}(\rho(A),\kappa_{A}(s)) & (\kappa_A(s) \neq 0_M \Rightarrow A \in \Omega[s])
\end{align*}
By Convention \ref{conv:modulo:formula:bar_s_R} and Remark \ref{rem:action-as-morphism:ic-monoids}, the last equation leads to the identity $\widetilde{\odot}(s) = \widetilde{\odot}(|s|_{\Omega})$.
\end{proof}

\begin{proposition}[Action]\label{prop:polynomials:action:R_on_RM}
Let $(R,+_R,\cdot_R,0_R,1_R)$ be an ic-semiring and $(M,+_M,0_M)$ be an ic-monoid. The multiplication $\boxdot$ of the ic-semiring $R[M]$ induces an action $\boxdot:R \times R[M] \to R[M]$ through the mapping rule $(r,s) \mapsto (r X^{0_M}) \boxdot s$.
\end{proposition}
\begin{proof}
Since the semiring $R[M]$ is commutative (Proposition \ref{prop:polynomials:semiring}), this statement is a direct consequence of Example \ref{exa:morphism:ic-semiring:action} -- provided that we can show that the mapping rule $r \mapsto r X^{0_M}$ defines a morphism $R \to R[M]$ of semirings (Definition \ref{def:morphism_of_semrirings}). To show this, first observe that the equations $0_R X^{0_M} = 0$ and $1_R X^{0_M} = 1$ hold. Then, we can verify that the following equations are satisfied for every pair $(r,r') \in R \times R$.
\[
(r\cdot_R r') X^{0_M} =  (r X^{0_M}) \boxdot (r' X^{0_M})
\quad\quad\quad\quad\quad
(r+_R r') X^{0_M} =  (r X^{0_M}) + (r' X^{0_M})
\]
The previous equations show that the mapping rule $r \mapsto r X^{0_M}$ induces a morphism of semirings. As a result, we can apply the scenario described in Example \ref{exa:morphism:ic-semiring:action}, which proves the statement.
\end{proof}

\begin{proposition}\label{prop:alpha:alpha:alpha:boxdot:diagram}
Let $\alpha:R \times M \to R$ be action of an ic-semiring $(R,+_R,\cdot_R,0_R,1_R)$ on an ic-monoid $(M,+_M,0_M)$. The following diagram of functions commutes.
\[
\xymatrix@C-10pt{
*+!R(.5){ R \times R[M]} \ar[d]_{\mathsf{id}_R \times \widetilde{\alpha}} \ar[r]^-{\boxdot}& R[M] \ar[d]^{\widetilde{\alpha}}\\
*+!R(.5){R \times M}\ar[r]_-{\alpha}&M
}
\]
\end{proposition}
\begin{proof}
For every $r \in R$ and every $s = \sum_{m \in \mathsf{Sup}(s)} s_m X^m$, Proposition \ref{prop:polynomials:action:R_on_RM} gives us the following reformulations of the action $r \boxdot s$.
\[
r \boxdot s = \textstyle (rX^{0_M}) \boxdot \sum_{m\in \mathsf{Sup}(s)}s_mX^{m}= \sum_{m\in \mathsf{Sup}(s)}(r \cdot_R s_m)X^{m}
\]
The previous reformulation of $r \boxdot s$ allow us to deduce the following reformulation of $\widetilde{\alpha}(r \boxdot s)$.
\[
\widetilde{\alpha}(r \boxdot s) = \textstyle \sum_{m\in \mathsf{Sup}(s)} \alpha(r \cdot s_m,m) = \alpha(r,\sum_{m\in \mathsf{Sup}(s)} \alpha(s_m,m)) = \alpha(r,\widetilde{\alpha}(s))
\]
The resulting equation $\widetilde{\alpha}(r \boxdot s) = \alpha(r,\widetilde{\alpha}(s))$ implies the statement.
\end{proof}

\subsection{Linear algebra}\label{ssec:linear algebra}
The goal of this section is to develop a type of linear algebra that will allow us to resolve systems of linear equations in ic-monoids. Most of our development is inspired from the standard linear algebra using rows, columns (Definition \ref{def:col-row}) and scalar products (Convention \ref{def:hadamard:scalar-product}) to define matrix arithmetics (see Convention \ref{conv:matrix_products} and Remark \ref{rem:linearity:matrix-product}). The main results of this section focus on the associative and homomorphic properties of matrix products (see Proposition \ref{prop:associative-product-matrix} and Proposition \ref{prop:applying-morphism:on-matrix:alpha_tilde}).

For now on, we shall replace the notations of ic-semirings $(R,+_R,\cdot_R,0_R,1_R)$ and ic-monoids $(M,+_M,0_M)$ with their non-indexed versions, namely the notations $(R,+,\cdot,0,1)$ and $(M,+,0)$.

\begin{definition}[Vector algebra]\label{def:vectors}
Let $(M,+,0)$ be an object in $\mathbf{Icm}$. For every positive integer $n$, we denote by $\mathbf{V}_{n}(M)$ the set of $n$-tuples of elements in $M$. Since $M$ is an object in $\mathbf{Icm}$, the set $\mathbf{V}_{n}(M)$ can be viewed as the product $\textstyle \prod_{i \in [n]} M$ in $\mathbf{Icm}$ (see Remark \ref{rem:products}).
\end{definition}

\begin{definition}[Formal sum]\label{def:semigroups:formal-sum}
Let $(M,+,0)$ be an object in $\mathbf{Icm}$. For every positive integer $n$, we denote by the symbol $\sum$ the canonical morphism $\mathbf{V}_{n}(M) \to M$ that sums all the coefficients of a tuple in $M$, namely the function.
\[
\sum:
\left(
\begin{array}{rll}
\mathbf{V}_{n}(M) &\to& M\\
(x_i)_{i\in [n]} && x_1 + x_2 + \dots + x_n
\end{array}
\right)
\]
We shall denote each image $\sum((x_i)_{i\in [n]})$ more conveniently as $\sum_{i \in [n]}x_i$.
\end{definition}

\begin{definition}[Vector products]\label{def:hadamard:vector-product}
Let $\alpha:R \times M \to R$ be an action of a semiring $(R,+,\cdot,0,1)$ on a commutative monoid $(M,+,0)$. For every positive integer $n$, we denote by $\mathbf{V}_{n}(\alpha)$ the obvious function $\mathbf{V}_{n}(R) \times \mathbf{V}_{n}(M) \to \mathbf{V}_{n}(M)$ resulting form the action structure of the function $\alpha$, namely the following function.
\[
\mathbf{V}_{n}(\alpha):
\left(
\begin{array}{rll}
\mathbf{V}_{n}(R) \times \mathbf{V}_{n}(M) &\to& \mathbf{V}_{n}(M)\\
((r_i)_{i\in [n]},(x_i)_{i\in [n]}) && (\alpha(r_i,x_i))_{i\in [n]})
\end{array}
\right)
\]
For every element $s \in \mathbf{V}_{n}(R)$ and element $y \in \mathbf{V}_{n}(M)$, the two mapping rules $x \mapsto \mathbf{V}_{n}(\alpha)(s,x)$ and $r \mapsto \mathbf{V}_{n}(\alpha)(r,y)$ induce morphisms $\mathbf{V}_{n}(M) \to \mathbf{V}_{n}(M)$ and $\mathbf{V}_{n}(R) \to \mathbf{V}_{n}(M)$ in $\mathbf{Icm}$.
\end{definition}

\begin{convention}[Scalar products]\label{def:hadamard:scalar-product}
Let $\alpha:R \times M \to R$ be an action of a semiring $(R,+,\cdot,0,1)$ on a commutative monoid $(M,+,0)$. For every positive integer $n$, we denote the composition of the morphism $\mathbf{V}_{n}(\alpha)$ (Definition \ref{def:hadamard:vector-product}) with the morphism $\sum$ (Definition \ref{def:semigroups:formal-sum}) as follows.
\[
\langle\cdot,\cdot\rangle^{\alpha}:
\left(
\begin{array}{ccc}
\mathbf{V}_{n}(R) \times \mathbf{V}_{n}(M) &\to& M\\
(x,y) &\mapsto& \langle x,y\rangle^{\alpha}
\end{array}
\right)
\]
According to the previous construction, for every pair $(x,y)$ of $n$-tuples $x = (x_i)_{i\in [n]} \in \mathbf{V}_{n}(R)$ and $y = (y_i)_{i\in [n]} \in \mathbf{V}_{n}(M)$, the following equation holds.
\[
\langle x,y\rangle^{\alpha} = \textstyle \sum_{i \in [n]} \alpha(x_i , y_i)
\]
Because the sum used in the previous formula is mostly formal, we will keep this notation even when the underlying ic-monoid is multiplicative (\emph{e.g.} as with ic-monoids of rationals).
\end{convention}

\begin{definition}[Matrix algebra]\label{def:matrices}
Let $(M,+,0)$ be an object in $\mathbf{Icm}$. For every pair $(n,m)$ of positive integers, we denote by $\mathbf{M}_{n,m}(M)$ the set of $(n \times m)$-matrices of elements in $M$. Since $M$ is an object in $\mathbf{Icm}$, the set $\mathbf{M}_{n,m}(M)$ can be viewed as the product $\textstyle \prod_{(i,j) \in [n]\times[m]} M$ in $\mathbf{Icm}$.
\end{definition}

\begin{convention}[Transpose]\label{conv:transpose}
Let $(M,+,0)$ be an object in $\mathbf{Icm}$. For every pair $(n,m)$ of positive integers and every $(n \times m)$-matrix $A = (a_{i,j})_{i,j} \in \mathbf{M}_{n,m}(M)$, we will denote by $A^t$ the transpose matrix  $(a_{i,j})_{j,i}$ in $\mathbf{M}_{m,n}(M)$. It is straightforward to verify that the transpose operation $A \mapsto A^t$ is an involution, which means that the identity $(A^t)^t = A$ holds.
\end{convention}

\begin{definition}[Columns and rows]\label{def:col-row}
Let $(M,+,0)$ be an object in $\mathbf{Icm}$. For every pair $(n,m)$ of positive integers and every matrix $A = (a_{i,j})_{i,j}$ in $\mathbf{M}_{n,m}(M)$, we define the following elements of $\mathbf{V}_n(M)$ and $\mathbf{V}_m(M)$ for every $i \in [n]$ and every $j \in [n]$.
\[
\mathsf{col}_j(A) := (a_{i,j})_{i \in [n]} \in \mathbf{V}_n(M)
\quad\quad\quad\quad
\mathsf{row}_i(A) := (a_{i,j})_{j \in [m]} \in \mathbf{V}_m(M)
\]
Note that the matrix $A$ can be viewed as either an $n$-tuple $(\mathsf{row}_i(A))_{i \in [n]}$ or as an $m$-tuple $(\mathsf{col}_j(A))_{j \in [m]}$.
\end{definition}

\begin{proposition}\label{prop:col-row:morphisms}
Let $(M,+,0)$ be an object in $\mathbf{Icm}$. For every pair $(n,m)$ of positive integers and every pair $(i,j) \in [n] \times [n]$, the two mapping rules $A \mapsto \mathsf{col}_j(A)$ and $A \mapsto \mathsf{row}_i(A)$ induce morphisms $\mathbf{M}_{n,m}(M) \to \mathbf{V}_n(M)$ and $\mathbf{M}_{n,m}(M) \to \mathbf{V}_n(M)$ in $\mathbf{Icm}$.
\end{proposition}
\begin{proof}
This is a straightforward consequence of the equations used in Definition \ref{def:col-row}.
\end{proof}

\begin{proposition}\label{prop:transpose:col-row}
Let $(M,+,0)$ be an object in $\mathbf{Icm}$. For every pair $(n,m)$ of positive integers and every matrix $A = (a_{i,j})_{i,j}$ in $\mathbf{M}_{n,m}(M)$, the equations $\mathsf{row}_i(A) = \mathsf{col}_i(A^t)$ and $\mathsf{col}_j(A) = \mathsf{row}_j(A^t)$ hold for every $i \in [n]$ and every $j \in [n]$.
\end{proposition}
\begin{proof}
This directly follows from Convention \ref{conv:transpose} and Definition \ref{def:col-row}.
\end{proof}

\begin{convention}[Matrix arithmetics]\label{conv:matrix_products}
Let $\alpha:R \times M \to R$ be an action of a semiring $(R,+,\cdot,0,1)$ on a commutative monoid $(M,+,0)$ and let $(n,m,p)$ be a triple of positive integers:
\begin{itemize}
\item[1)] for every matrix $A=(a_{i,j})_{i,j} \in \mathbf{M}_{n,m}(R)$ and every matrix $B=(b_{i,j})_{i,j}\in \mathbf{M}_{m,p}(M)$, we define the \emph{direct matrix product} $A \underline{\alpha} B$ as the following matrix in $\mathbf{M}_{n,p}(M)$.
\[
(A \underline{\alpha} B)_{i,j} := \langle\mathsf{row}_i(A),\mathsf{col}_j(B)\rangle^{\alpha}  = \textstyle \sum_{k \in [m]} \alpha(a_{i,k},b_{k,j})
\]
\item[2)] for every matrix $A=(a_{i,j})_{i,j} \in \mathbf{M}_{n,m}(M)$ and every matrix $B=(b_{i,j})_{i,j}\in \mathbf{M}_{m,p}(R)$, the \emph{reverse matrix product} $A\overline{\alpha} B$ as the following matrix in $\mathbf{M}_{n,p}(M)$.
\[
(A\overline{\alpha} B)_{i,j} := \langle\mathsf{col}_j(B),\mathsf{row}_i(A)\rangle^{\alpha} =  \textstyle \sum_{k \in [m]} \alpha(b_{k,j},a_{i,k})
\]
\end{itemize}
Note that if the semiring $(R,+,\cdot,0,1)$ is commutative and we take $(M,+,0)$ to be the commutative monoid $(R,+,0)$, then the identity $A \underline{\cdot} B = A\overline{\cdot} B$ holds. In this case, we will denote the resulting product matrix as $AB$.
\end{convention}

\begin{remark}[Linearity]\label{rem:linearity:matrix-product}
This remark discusses the (bi)linear properties of the matrix product defined in Convention \ref{conv:matrix_products}. Let $\alpha:R \times M \to R$ be an action of a semiring $(R,+,\cdot,0,1)$ on a commutative monoid $(M,+,0)$ and let $(n,m,p)$ be a triple of positive integers. It follows from Proposition \ref{prop:col-row:morphisms} and the constructions given in Definition \ref{def:hadamard:vector-product} and Convention \ref{def:hadamard:scalar-product} that, for every matrix $A_0 \in \mathbf{M}_{n,m}(R)$, $B_0 \in \mathbf{M}_{m,p}(M)$, $A_1 \in \mathbf{M}_{n,m}(M)$, and $B_1 \in \mathbf{M}_{m,p}(R)$, the mapping rules $B \mapsto A_0 \underline{\alpha} B$, $A \mapsto A \underline{\alpha} B_0$, $B \mapsto A_1 \overline{\alpha} B$, and $A \mapsto A \overline{\alpha} B_1$ each induce morphisms in $\mathbf{Icm}$ of the following types.
\[
\def\arraystretch{1.2}
\left\{
\begin{array}{ll}
\underline{\alpha}:\mathbf{M}_{n,m}(M)  \to \mathbf{M}_{n,p}(M)&
\underline{\alpha}:\mathbf{M}_{m,p}(R) \to \mathbf{M}_{n,p}(M)\\
\overline{\alpha}:\mathbf{M}_{n,m}(R)  \to \mathbf{M}_{n,p}(M)&
\overline{\alpha}: \mathbf{M}_{m,p}(M) \to \mathbf{M}_{n,p}(M)
\end{array}
\right.
\]
From now on, we will use these linearity properties implicitly without reference to this remark.
\end{remark}

\begin{proposition}\label{prop:transpose-product-matrix}
Let $\alpha:R \times M \to R$ be an action of a semiring $(R,+,\cdot,0,1)$ on a commutative monoid $(M,+,0)$. For every triple $(n,m,p)$ of positive integers, every matrix $A \in \mathbf{M}_{n,m}(R)$ and every matrix $B\in \mathbf{M}_{m,p}(M)$, the following equation holds.
\[
(B\overline{\alpha} A)^t = A^t\underline{\alpha}B^t
\]
\end{proposition}
\begin{proof}
Denote $A = (a_{i,j})_{i,j} \in \mathbf{M}_{n,m}(R)$ and $B=(b_{i,j})_{i,j}\in \mathbf{M}_{m,p}(M)$.
For every $i \in [n]$ and every $j \in [n]$, we can show that the following equations hold.
\begin{align*}
((B\overline{\alpha} A)^t)_{i,j} & = (B\overline{\alpha} A)_{j,i}&(\textrm{Convention \ref{conv:transpose}})\\
&= \langle\mathsf{col}_i(A),\mathsf{row}_j(B)\rangle^{\alpha}&(\textrm{Convention \ref{conv:matrix_products}})\\
&= \langle\mathsf{row}_i(A^t),\mathsf{col}_j(B^t)\rangle^{\alpha}&(\textrm{Proposition \ref{prop:transpose:col-row}})\\
&= (A^t\underline{\alpha}B^t)_{i,j}&(\textrm{Convention \ref{conv:matrix_products}})
\end{align*}
The previous identities show that the two matrices $(B\overline{\alpha} A)^t$ and $A^t\underline{\alpha}B^t$ are equal.
\end{proof}

\begin{proposition}\label{prop:associative-product-matrix}
Let $\alpha:R \times M \to R$ be an action of a commutative semiring $(R,+,\cdot,0,1)$ on a commutative monoid $(M,+,0)$. For every quadruple $(n,m,p,q)$ of positive integers, the following associativity properties hold.
\[
\def\arraystretch{1.5}
\left\{
\begin{array}{llccc}
\forall A \in \mathbf{M}_{n,m}(R), B\in \mathbf{M}_{m,p}(R), C \in \mathbf{M}_{p,q}(M) & \Rightarrow & (AB) \underline{\alpha} C &=& A\underline{\alpha}(B \underline{\alpha} C)\\
\forall A \in \mathbf{M}_{n,m}(R), B\in \mathbf{M}_{m,p}(M),C \in \mathbf{M}_{p,q}(R)  & \Rightarrow & (A\underline{\alpha} B) \overline{\alpha} C &=& A\underline{\alpha}(B \overline{\alpha} C)\\
\forall A \in \mathbf{M}_{n,m}(M), B\in \mathbf{M}_{m,p}(R),C \in \mathbf{M}_{p,q}(R) & \Rightarrow & (A \overline{\alpha} B) \overline{\alpha} C &=& A \overline{\alpha} (B C)\\
\end{array}
\right.
\]
\end{proposition}
\begin{proof}
We shall denote $A = (a_{u,k})_{u,k} \in \mathbf{M}_{n,m}(R)$, $B=(b_{k,j})_{k,j}\in \mathbf{M}_{m,p}(M)$ and $C=(c_{j,v})_{j,v}\in \mathbf{M}_{p,q}(M)$. The bottommost equation of the statement is a direct consequence of the topmost equation according to the correspondence established in Proposition \ref{prop:transpose-product-matrix} via the transpose involution $X \mapsto X^t$. To show the topmost equation shown in the statement, let us observe that for every pair of matrices $A \in \mathbf{M}_{n,m}(R)$ and $B\in \mathbf{M}_{m,p}(R)$ and every $u \in [n]$, the following equations hold.
\begin{equation}\label{eq:associativity:part1}
\mathsf{row}_u(AB) = ((AB)_{u,j})_{j \in [p]} = (\langle\mathsf{row}_u(A),\mathsf{col}_j(B)\rangle^{\mathsf{mul}})_{j \in [p]} =\big(\textstyle\sum_{k \in [m]}a_{u,k}\cdot b_{k,j}\big)_{j \in [p]}
\end{equation}
It also follows from Definition \ref{def:col-row} that, for every matrix $C \in \mathbf{M}_{p,q}(M)$ and every $v \in [q]$, we have the equation $\mathsf{col}_{v}(C) =  (c_{j,v})_{j \in [p]}$. Hence, we deduce that the following equations hold for every $u \in [n]$ and every $v \in [q]$.
\begin{align*}
((AB) \underline{\alpha} C)_{u,v} & = \langle\mathsf{row}_u(AB),\mathsf{col}_{v}(C)\rangle^{\alpha} & (\textrm{Convention \ref{conv:matrix_products}})\\
& = \sum_{j \in [p]} \alpha(\sum_{k \in [m]}a_{u,k}\cdot b_{k,j},c_{j,v}) & (\textrm{Equation (\ref{eq:associativity:part1}))})\\
& = \sum_{j \in [p]} \sum_{k \in [m]} \alpha(a_{u,k}\cdot b_{k,j},c_{j,v}) & (\textrm{Definition \ref{def:action-semi-groups}})\\
& = \sum_{k \in [m]} \sum_{j \in [p]} \alpha(a_{u,k}, \alpha( b_{k,j},c_{j,v})) & (\textrm{Definition \ref{def:action-semi-groups}})\\
& =  \sum_{k \in [m]} \alpha(a_{u,k}, \sum_{j \in [p]} \alpha( b_{k,j},c_{j,v})) & (\textrm{Definition \ref{def:action-semi-groups}})\\
& =  \sum_{k \in [m]} \alpha(a_{u,k}, \mathsf{col}_v(B\underline{\alpha}C)) & (\textrm{Definition \ref{def:action-semi-groups}})\\
& = \langle\mathsf{row}_u(A),\mathsf{col}_v(B\underline{\alpha}C)\rangle^{\alpha} &
\end{align*}
The previous series of equations show that the coefficient of the matrix $(AB) \underline{\alpha} C$ are the same as those of $A\overline{\alpha} (B \underline{\alpha} C)$.
To show the middle equation of the statement, let us observe that, for every pair of matrices $A \in \mathbf{M}_{n,m}(R)$ and $B\in \mathbf{M}_{m,p}(B)$ and every $u \in [n]$, the following equations hold.
\begin{equation}\label{eq:associativity:part2}
\mathsf{row}_u(A\underline{\alpha} B) = ((A\underline{\alpha} B)_{u,j})_{j \in [p]} = (\langle\mathsf{row}_u(A),\mathsf{col}_j(B)\rangle^{\alpha})_{j \in [p]} =\big(\textstyle\sum_{k \in [m]}\alpha(a_{u,k},b_{k,j})\big)_{j \in [p]}
\end{equation}
Similarly, for every matrix $C \in \mathbf{M}_{p,q}(M)$ and every $v \in [q]$, we have the following equations.
\begin{equation}\label{eq:associativity:part3}
\mathsf{col}_v(B\overline{\alpha}C) = ((B\overline{\alpha}C)_{k,v})_{k \in [m]} = (\langle\mathsf{col}_v(C),\mathsf{row}_k(B)\rangle^{\alpha})_{k \in [m]} =\big(\textstyle\sum_{j \in [p]} \alpha(c_{j,v},b_{k,j})\big)_{k \in [m]}
\end{equation}
We then deduce that the following equations holds for every $u \in [n]$ and every $v \in [q]$.
\begin{align*}
((A\underline{\alpha} B) \overline{\alpha} C)_{u,v} & = \langle\mathsf{col}_v(C),\mathsf{row}_{u}(A\underline{\alpha}B)\rangle^{\alpha} & (\textrm{Convention \ref{conv:matrix_products}})\\
& = \sum_{j \in [p]} \alpha(c_{j,v},\sum_{k \in [m]}\alpha(a_{u,k},b_{k,j})) & (\textrm{Equation (\ref{eq:associativity:part2})})\\
& = \sum_{j \in [p]} \sum_{k \in [m]} \alpha(c_{j,v},\alpha(a_{u,k},b_{k,j})) & (\textrm{Definition \ref{def:action-semi-groups}})\\
& = \sum_{j \in [p]} \sum_{k \in [m]} \alpha(c_{j,v} \cdot a_{u,k},b_{k,j})) & (\textrm{Definition \ref{def:action-semi-groups}})\\
& =  \sum_{k \in [m]} \sum_{j \in [p]}   \alpha(a_{u,k}, \alpha(c_{j,v},b_{k,j})) & (\textrm{Definition \ref{def:action-semi-groups}})\\
& =  \sum_{k \in [m]}  \alpha(a_{u,k}, \sum_{j \in [p]} \alpha(c_{j,v},b_{k,j})) & (\textrm{Definition \ref{def:action-semi-groups}})\\
& = \langle\mathsf{row}_u(A),\mathsf{col}_v(B\overline{\alpha}C)\rangle^{\alpha} & (\textrm{Equation (\ref{eq:associativity:part3})})
\end{align*}
The previous series of equations show that the coefficient of the matrix $(A\underline{\alpha} B) \overline{\alpha} C$ are the same as those of $A\underline{\alpha} (B \overline{\alpha} C)$. This finishes the proof of the statement.
\end{proof}

\begin{convention}[Applying functions]\label{conv:applying-morphism:on-matrix}
Let $(M,+,0)$ and $(N,+,0)$ be two objects in $\mathbf{Icm}$. For every function $f:M \to N$, every pair $(n,m)$ of positive integers, and every matrix $A = (a_{i,j})_{i,j} \in \mathbf{M}_{n,m}(M)$, we will denote as $f(A)$ the matrix resulting from applying the function $f:M \to N$ on each coefficient of $A$, namely the matrix $(f(a_{i,j}))_{i,j}$ in $\mathbf{M}_{n,m}(N)$.
\end{convention}

\begin{proposition}[Homomorphism]\label{prop:applying-morphism:on-matrix:alpha_tilde}
Let $\alpha:R \times M \to R$ be an action of an ic-semiring $(R,+,\cdot,0,1)$ on an ic-monoid $(M,+,0)$. For every triple $(n,m,p)$ of positive integers, every matrix $A \in \mathbf{M}_{n,m}(R[M])$ and every matrix $B \in \mathbf{M}_{m,p}(R)$, the following equation holds.
\[
\widetilde{\alpha}(A\, \overline{\boxdot}\, B) = \widetilde{\alpha}(A) \,\overline{\alpha}\, B
\]
\end{proposition}
\begin{proof}
First, for every pair $(i,j) \in [n] \times [m]$, let us recall that the $(i,j)$-coefficient of the matrix $A\overline{\boxdot} B$ satisfies the following equations.
\begin{equation}\label{eq:applying-alpha:on-matrix:prep}
(A\,\overline{\boxdot}\, B)_{i,j} = \langle\mathsf{col}_j(B),\mathsf{row}_i(A)\rangle^{\boxdot} =  \textstyle \sum_{k \in [m]} b_{k,j} \boxdot a_{i,k}
\end{equation}
From this, we deduce that the $(i,j)$-coefficient of the matrix $\widetilde{\alpha}(A\overline{\alpha} B)$ is as follows.
\begin{align*}
(\widetilde{\alpha}(A\,\overline{\boxdot}\, B))_{i,j} & = \widetilde{\alpha}((A\,\overline{\boxdot} \,B)_{i,j}) &(\textrm{Convention \ref{conv:applying-morphism:on-matrix}})\\
& = \widetilde{\alpha}(\sum_{k \in [m]} b_{k,j} \boxdot a_{i,k}) &(\textrm{Equation \ref{eq:applying-alpha:on-matrix:prep}})\\
& = \sum_{k \in [m]} \widetilde{\alpha}( b_{k,j} \boxdot a_{i,k}) & (\textrm{Remark \ref{rem:action-as-morphism:ic-monoids}})\\
& = \sum_{k \in [m]} \alpha( b_{k,j}, \widetilde{\alpha}(a_{i,k})) & (\textrm{Proposition \ref{prop:alpha:alpha:alpha:boxdot:diagram}})\\
& = \langle\mathsf{col}_j(B),\mathsf{row}_i(\widetilde{\alpha}(A))\rangle^{\alpha} = (\widetilde{\alpha}(A) \,\overline{\alpha}\, B)_{i,j}&
\end{align*}
These equations show that the matrix $\widetilde{\alpha}(A \,\overline{\boxdot}\, B)$ is equal to the matrix
$\widetilde{\alpha}(A) \,\overline{\alpha}\, B$.
\end{proof}

\subsection{Skew linear algebra}\label{secc:skew:linear-algebra}
The goal of this section is to define a calculus environment in which the questions raised in the introduction of this paper can be formulated. Subsequently, in section \ref{secc:up-down:linear-algebra}, we shall introduce a slightly more familiar environment in which calculations can be carried over in the same fashion as in standard linear algebra. The definitions and propositions stated in the present section all relate to the ic-monoid structure $(\mathbb{Q}(M),\star,\mathbbm{1}_M)$ defined in Convention \ref{conv:rationals:ic-monoids} and the action $\odot:\overline{\mathbb{Z}}_3 \times \mathbb{Q}(M) \to  \mathbb{Q}(M)$ described in Example \ref{def:action:rationals} for any ic-monoid $(M,+,0)$.

\begin{definition}[Skew congruence]\label{def:action-modulo}
Let $(M,+_M,0_M)$ be an ic-monoid. For every pair $(m_1,m_2)$ of elements in $\mathbb{Q}(M)$, we use the expression $m_1 \Rrightarrow m_2$ to mean that there exists $m \in \mathbb{Q}(M)$ such that the relation $m_1 = m_2 \star (\updownarrow \!\odot\, m)$ holds in $\mathbb{Q}(M)$.
\end{definition}

\begin{proposition}
Let $(M,+_M,0_M)$ be an ic-monoid. The relation $\cdot \Rrightarrow  \cdot$ defines a pre-order relation on the elements of the ic-monoid $\mathbb{Q}(M)$.
\end{proposition}
\begin{proof}
First, the relation $\Rrightarrow$ is reflexive because we have the equations $m = m \star ( \updownarrow \!\odot\, \mathbbm{1}_M)$. Let us now show the transitive property. In this respect, let $(m_1,m_2,m_3)$ be a triple of elements of $\mathbb{Q}(M)$ such that the two relations $m_1 \Rrightarrow m_2$ and $m_2 \Rrightarrow m_3$ hold. Let us prove that the relation $m_1 \Rrightarrow m_3$ holds. By Definition \ref{def:action-modulo}, there exists a pair $(m,m')$ of elements in $\mathbb{Q}(M)$ such that $m_1 = m_2 \star (\updownarrow \!\odot\, m)$ and $m_2 = m_3 \star ( \updownarrow \!\odot\, m')$. As a result, we have the following equations.
\[
m_1 = m_2 \star (\updownarrow \!\odot\, m) = m_3 \star ( \updownarrow \!\odot\, m') \star ( \updownarrow \!\odot\, m) = m_3 \star ( \updownarrow \!\odot\, (m'+m))
\]
This shows that the relation $m_1 \Rrightarrow m_3$ holds and that the binary relation $\Rrightarrow  $ is transitive. Hence, the relation $\Rrightarrow$ is a pre-order.
\end{proof}

\begin{convention}[Skew equivalence]\label{conv:skew-equality}
Let $(M,+_M,0_M)$ be an ic-monoid. For every pair $(n,m)$ of positive integers and every pair $A = (a_{i,j})_{i,j}$ and $B=(b_{i,j})_{i,j}$ of matrices in $\mathbf{M}_{n,m}(\mathbb{Q}(M))$, we define the relation $A \Rrightarrow B$ if, and only if, for every pair $(i,j) \in [n] \times [m]$, the relation $a_{i,j} \Rrightarrow b_{i,j}$ holds in $\mathbb{Q}(M)$.
\end{convention}

\begin{remark}[Skew equivalence]\label{rem:skew-equality:reformulation}
The goal of this remark is to reformulate the definition of the relation $\Rrightarrow$ given in Definition \ref{def:action-modulo}. Specifically, we can show that, for every ic-monoid $(M,+,0)$, every pair $(n,m)$ of positive integers, and every pair $A = (a_{i,j})_{i,j}$ and $B=(b_{i,j})_{i,j}$ of matrices in $\mathbf{M}_{n,m}(\mathbb{Q}(M))$, the relation $A \Rrightarrow B$ holds if, and only if, there exists an element $\lambda_{i,j} \in M$ for which the following relation holds in $\mathbb{Q}(M)$.
\begin{equation}\label{eq:skew-equality:reformulation}
a_{i,j} = \mathbf{0}(\lambda_{i,j}) \star b_{i,j}
\end{equation}
To show this, first recall (from Definition \ref{def:action-modulo}) that the relation $A \Rrightarrow B$ holds if, and only if, for every pair $(i,j) \in [n] \times [m]$, there exists an element $c_{i,j} \in \mathbb{Q}(M)$ for which the equality $a_{i,j} =  b_{i,j} \star (\updownarrow \!\odot \,c_{i,j})$ holds. By Convention \ref{conv:symmetry}, we can show that the latter is equivalent to the equality $a_{i,j} =  b_{i,j} \star \mathbf{0}(\mathsf{sum}(c_{i,j}))$ for every pair $(i,j) \in [n] \times [m]$. Equation (\ref{eq:skew-equality:reformulation}) follows if we take $\lambda_{i,j} = \mathsf{sum}(c_{i,j})$ for every pair $(i,j) \in [n] \times [m]$.
\end{remark}

\begin{convention}[Indices]\label{conv:indices}
Let $n$ be a positive integer. For every matrix $Z = (z_{k,1})_{k,1}$ in $\mathbf{M}_{n,1}(\overline{\mathbb{Z}}_3)$ and every element $r\in \overline{\mathbb{Z}}_3$, we will denote as $\mathsf{Ind}_Z(r)$ the subset $\{k~|~z_{k,1} = r\}$ of $[n]$. Because the mapping rule $k \mapsto z_{k1}$ defines a function $[n] \to \overline{\mathbb{Z}}_3$ whose fiber at an element $r\in \overline{\mathbb{Z}}_3$ corresponds to the set $\mathsf{Ind}_Z(r)$, the following equation holds.
\[
[n] = \mathsf{Ind}_Z(0) \cup \mathsf{Ind}_Z(\uparrow) \cup \mathsf{Ind}_Z(\downarrow) \cup \mathsf{Ind}_Z(\updownarrow)
\]
More specifically, the previous union forms a partition (\emph{i.e.} a disjoint union) of the set $[n]$.
\end{convention}

\begin{remark}[Skew systems of equations]\label{rem:unification}
The goal of this remark is establish a correspondence between systems of linear equations and the relation defined in Convention \ref{conv:skew-equality}. In particular, we want to understand this relation of Convention \ref{conv:skew-equality} from the viewpoint of standard linear algebra operations.

Let $(M,+,0)$ be an ic-monoid and $A = (a_{i,j})_{i,j}$ be a matrix in $\mathbf{M}_{n,m}(M)$. Let $\mathbf{1}(A)$ be the matrix resulting from applying the morphism $\mathbf{1}:M \to \mathbb{Q}(M)$ on each coefficient of $A$ (see Convention \ref{conv:applying-morphism:on-matrix} and Definition \ref{def:rationals:ic-monoids}). Hence, the matrix $\mathbf{1}(A)$ is in $\mathbf{M}_{n,m}(\mathbb{Q}(M))$. Let now $Z = (z_{i,1})_{i,1}$ be a matrix in $\mathbf{M}_{n,1}(\overline{\mathbb{Z}}_3)$ satisfying a relation of the form $\mathbf{1}(A) \overline{\odot} Z \Rrightarrow 0$ in $\mathbf{M}_{m,1}(\mathbb{Q}(M))$. By Remark \ref{rem:skew-equality:reformulation}, this means that, for every pair $i \in [n]$, there exists an element $\lambda_{i} \in M$ such that the following equation holds.
\begin{equation}\label{eq:unification:1}
\sum_{k \in [m]} (z_{k,1} \odot a_{i,k}) = \mathbf{0}(\lambda_{i})
\end{equation}
We can reformulate equation (\ref{eq:unification:1}) by only involving the coefficient of $A$. To do so, let us define the following elements in $M$.
\[
u_i = \sum_{k \in \mathsf{Ind}_X(\updownarrow)} a_{i,k} + \sum_{k \in \mathsf{Ind}_X(\uparrow)} a_{i,k}
\quad\quad\quad
d_i = \sum_{k \in \mathsf{Ind}_X(\updownarrow)} a_{i,k} + \sum_{k \in \mathsf{Ind}_X(\downarrow)} a_{i,k}
\]
It follows from  Definition \ref{def:action:rationals} that the element $\sum_{k \in [m]} (x_{k,1} \odot a_{i,k})$ is equal to the rational $\left[u_i/d_i\right]$. As a result, equation (\ref{eq:unification:1}) can be reformulate as the following two equations.
\begin{equation}\label{eq:unification:2}
\sum_{k \in \mathsf{Ind}_X(\updownarrow)} a_{i,k} + \sum_{k \in \mathsf{Ind}_X(\uparrow)} a_{i,k} = \lambda_i = \sum_{k \in \mathsf{Ind}_X(\updownarrow)} a_{i,k} + \sum_{k \in \mathsf{Ind}_X(\downarrow)} a_{i,k}
\end{equation}
Since equation (\ref{eq:unification:2}) hold for every $i \in [n]$, we deduce that, for every matrix $A \in \mathbf{M}_{n,m}(M)$, any relation of the form $\mathbf{1}(A) \overline{\odot} Z \Rrightarrow 0$ in $\mathbf{M}_{m,1}(\mathbb{Q}(M))$ is equivalent to the following equation.
\begin{equation}\label{eq:unification:3}
\sum_{k \in \mathsf{Ind}_X(\updownarrow)} \mathsf{col}_k(A) + \sum_{k \in \mathsf{Ind}_X(\uparrow)} \mathsf{col}_k(A) = \sum_{k \in \mathsf{Ind}_X(\updownarrow)} \mathsf{col}_k(A) + \sum_{k \in \mathsf{Ind}_X(\downarrow)} \mathsf{col}_k(A)
\end{equation}
Equation (\ref{eq:unification:3}) should be compared to the types of equations discussed in the introduction of this paper (see equation (\ref{eq:intro:icm-equations:haplotypes(recomb)})). We explain this link further in Remark \ref{rem:unification:2}.
\end{remark}

\begin{convention}[Balanced vectors]
Let $n$ be a positive integer. We shall say that a matrix vector $Z = (z_{k,1})_{k,1}$ in $\mathbf{M}_{n,1}(\overline{\mathbb{Z}}_3)$ is \emph{balanced} if it satisfies the following equivalence.
\[
\mathsf{Ind}_Z(\uparrow) = \cmemptyset \quad\quad \Leftrightarrow \quad\quad \mathsf{Ind}_Z(\downarrow) = \cmemptyset
\]
The set of balanced matrix vectors in $\mathbf{M}_{n,1}(\overline{\mathbb{Z}}_3)$ will be denoted as $\mathbf{Bal}_n$
\end{convention}

\begin{convention}[Skew null space]\label{conv:skew-null-space}
Let $(M,+_M,0_M)$ be an ic-monoid. For every pair $(n,m)$ of positive integers and every matrix $A \in \mathbf{M}_{n,m}(M)$, we define the \emph{skew null space of $A$} as the following set.
\[
\mathsf{SNull}(A) = \{Z \in \mathbf{Bal}_m~|~\mathbf{1}(A) \,\overline{\odot}\, Z \Rrightarrow  0 \textrm{ in }\mathbf{M}_{n,1}(\mathbb{Q}(M))\}
\]
\end{convention}

The following remark prepares toward the introduction and development of a general method to find linear relationships linking any given set of elements in an ic-monoid. In particular, this method will address the resolution of equations of the form discussed in the introduction of this paper -- see equation (\ref{eq:intro:icm-equations:haplotypes(recomb)}).

\begin{remark}[Finding linear equations]\label{rem:unification:2}
The goal of this remark is to refine Remark \ref{rem:unification} and pinpoint a context in which the elements of a skew null space (Convention \ref{conv:skew-null-space}) can be described in terms of linear equations in the underlying ic-monoid, as opposed to linear equations involvingd columns of matrices (as in Remark \ref{rem:unification}).

Let $(M_i)_{i \in [n]}$ be a finite collection of ic-monoids $M_i$ and let $M$ denote the product $\prod_{i \in [n]} M_i$ in $\mathbf{Icm}$. We shall denote by $\pi_i$ the projection morphism $M \to M_i$. Note that the ic-monoid $M_i$ can be seen as a submonoid of $M$ through the monomorphism $x \mapsto (\delta_{i,k}(x))_{k \in [n]}$ where we define $\delta_{i,k}(x)$ as follows for every pair $(i,k)$ of elements in $[n]$.
\[
\delta_{i,k}(x) =
\left\{
\begin{array}{ll}
x&\textrm{if }k = i\\
0_{M_i}&\textrm{if }k \neq i\\
\end{array}
\right.
\]
For every finite collection $e = (e_{j})_{j \in [m]}$ of elements in $M$, let us denote by $e^i_j$ the element $\pi_i(e_j)$ sent to the ic-monoid $M$ through the monomorphism $M_i \hookrightarrow M$ (see above). Let us denote by $\mathsf{A}(e)$ the resulting matrix $(e^i_j)_{i,j}$ in $\mathbf{M}_{n,m}(M)$. Note that because $M =\prod_{i \in [n]} M_i$, the vector
\[
\mathsf{col}_j(\mathsf{A}(e))  = (\pi_i(e_j))_{i \in [n]}
\]
can be identified with the element $e_j$ of the ic-monoid $M$. In particular, any linear equation satisfied by the vectors $\mathsf{col}_j(\mathsf{A}(e))$, for every $j \in [m]$, is also satisfied by the vectors $e_j$, for every $j \in [m]$, through the obvious correspondence $\mathsf{col}_j(\mathsf{A}(e)) \mapsto e_j$. As a result, it follows from Remark \ref{rem:unification}, and more specifically equation (\ref{eq:unification:3}), that for every element $Z \in \mathsf{SNull}(\mathsf{A}(e))$, the following equation holds.
\[
\sum_{k \in \mathsf{Ind}_Z(\updownarrow)} e_k + \sum_{k \in \mathsf{Ind}_Z(\uparrow)} e_k = \sum_{k \in \mathsf{Ind}_Z(\updownarrow)} e_k + \sum_{k \in \mathsf{Ind}_Z(\downarrow)} e_k
\]
Conversely, let us consider a triple $(U_0,U_1,U_{-1})$ of disjoint subsets of $[n]$ for which the following relations hold.
\begin{equation}\label{eq:unification:linear-relation:recovered}
\def\arraystretch{1.5}
\left\{
\begin{array}{l}
U_1 = \cmemptyset\quad\Leftrightarrow\quad U_{-1} = \cmemptyset\\
\sum_{k \in U_0} e_k + \sum_{k \in U_1} e_k = \sum_{k \in U_0} e_k + \sum_{k \in U_{-1}} e_k
\end{array}
\right.
\end{equation}
It follows from the equivalence of (\ref{eq:unification:linear-relation:recovered}) (top row) that we can define a matrix $Z = (z_{i,1})_{i,1}$ in $\mathbf{Bal}_{m}$ whose coefficients satisfy the following relations for every $i \in [m]$.
\[
z_{i,1} =
\left\{
\begin{array}{ll}
\updownarrow&\textrm{if }i \in U_0\\
\uparrow&\textrm{if }i \in U_1\\
\downarrow&\textrm{if }i \in U_{-1}\\
0&\textrm{otherwise.}\\
\end{array}
\right.
\]
By using the previous equations, as well as Definition \ref{def:action:rationals} and equation (\ref{eq:unification:linear-relation:recovered}), we can show that the product $\mathbf{1}(\mathsf{A}(e)) \overline{\odot} Z$ is equal to a matrix $Y = (y_{i})_i \in\mathbf{M}_{n,1}(\mathbb{Q}(M))$ whose coefficients satisfy the following equations for every $i \in [n]$.
\[
y_i = \left[ \frac{\sum_{k \in U_0} e_k^i + \sum_{k \in U_1} e_k^i}{\sum_{k \in U_0} e_k^i + \sum_{k \in U_{-1}} e_k^i} \right] = \mathbf{0}(\sum_{k \in U_0} e_k^i + \sum_{k \in U_1} e_k^i)
\]
It follows from Remark \ref{rem:skew-equality:reformulation} that the relation $\mathbf{1}(\mathsf{A}(e))\overline{\odot} Z \Rrightarrow 0$ holds in ic-monoid $\mathbf{M}_{n,1}(\mathbb{Q}(M))$. Since the matrix $Z$ belongs to $\mathbf{Bal}_{m}$, it is also an element of $\mathsf{SNull}(\mathsf{A}(e))$. In other words, the present remark shows that there is a bijection of the following form.
\[
\mathsf{SNull}(\mathsf{A}(e)) \cong \{(U_0,U_1,U_{-1})~|~\textrm{such that (\ref{eq:unification:linear-relation:recovered}) is satisfied.}\}
\]
\end{remark}

\subsection{Up-down linear algebra}\label{secc:up-down:linear-algebra}
The formulation of the problem stated through the concept of skew null spaces (in Convention \ref{conv:skew-null-space}) is not practical because the action used for the matrix product does not allow us to take advantage of the methods of standard linear algebra to find solutions for it. If we want to create opportunities to do so, it is necessary to embed the underlying \emph{acted-on} ic-monoid into the underlying \emph{acting} ic-semiring such that the action used for the matrix product is induced by the semiring multiplication.
To this end, the present section will provide a preliminary step to the reformulation of \emph{skew} null spaces (Convention \ref{conv:skew-null-space}) into \emph{relative} null spaces (see Convention \ref{conv:relative-null-space} below). In this respect, the main objective of this section is to find a context skew null spaces and relative null spaces define the same type of problems (see Theorem \ref{theo:correspondence:null-spaces:1} and Theorem \ref{theo:correspondence:null-spaces:2}).

\begin{definition}[Canceling matrix]\label{def:weakly-canceling-porperty}
Let $(M,+,0)$ be an ic-monoid and let $(n,m)$ be a pair of positive integers. We will say that a matrix $A = (a_{i,j})$ in $\mathbf{M}_{n,m}(\overline{\mathbb{Z}}_3[M])$ is \emph{canceling} if for every $(i,j) \in [n] \times [m]$, there exists an element $\lambda_{i,j} \in \overline{\mathbb{Z}}_3[M]$ such that equation $a_{i,j} = \,\updownarrow \!\boxdot\, \lambda_{i,j}$ holds (see Proposition \ref{prop:polynomials:action:R_on_RM}). We shall denote by $\mathbf{Can}_{n,m}(M)$ the set of canceling matrices in $\mathbf{M}_{n,m}(\overline{\mathbb{Z}}_3[M])$.
\end{definition}

\begin{proposition}[Canceling property]\label{prop:cancleling-property:skew-relation:matrix}
Let $(M,+,0)$ be an ic-monoid and let $(n,m)$ be a pair of positive integers. For every matrix $A \in \mathbf{Can}_{n,m}(M)$, the relation $\widetilde{\odot}(A) \Rrightarrow 0$ holds.
\end{proposition}
\begin{proof}
Let us denote the matrix $A$ as $(a_{i,j})_{i,j}$. By Definition \ref{def:weakly-canceling-porperty}, for every pair $(i,j) \in [n] \times [m]$, there exists an element $\lambda_{i,j} \in \overline{\mathbb{Z}}_3[M]$ such that the equation $a_{i,j} = \, \updownarrow \! \boxdot \, \lambda_{i,j}$ holds. By Proposition \ref{prop:alpha:alpha:alpha:boxdot:diagram} and Definition \ref{def:action:rationals}, we have the following equations.
\[
\widetilde{\odot}(a_{i,j}) = \widetilde{\odot}(\updownarrow \! \boxdot \, \lambda_{i,j}) = \updownarrow \! \odot \, \widetilde{\odot}(\lambda_{i,j}) = \mathbf{0}(\widetilde{\odot}(\lambda_{i,j}))
\]
By Remark \ref{rem:skew-equality:reformulation}, this means that the relation $\widetilde{\odot}(A) \Rrightarrow 0$ holds.
\end{proof}

\begin{convention}[Equivalence]\label{conv:skew-equality-matrix}
Let $(R,+,\cdot,0,1)$ be an ic-semiring and let $(M,+,0)$ be an ic-monoid. For every multiplicative atomic structure $(\Omega,\rho,\mathsf{Inv})$ on $R$, every pair $(n,m)$ of positive integers and every pair $A = (a_{i,j})_{i,j}$ and $B=(b_{i,j})_{i,j}$ of matrices in $\mathbf{M}_{n,m}(R[M])$, we define the relation $A \equiv B\,(\mathsf{wrt}\,\Omega)$ if, and only if, for every pair $(i,j) \in [n] \times [m]$, the relation $a_{i,j} \equiv b_{i,j} \,(\mathsf{wrt}\,\Omega)$ holds in $R[M]$.
\end{convention}

\begin{remark}[Equivalence relation]
Let $(R,+,\cdot,0,1)$ be an integral ic-semiring, let $(M,+,0)$ be an ic-monoid and let $(n,m)$ be a pair of positive integers. It follows from the properties of the $\Omega$-tensor congruence (Definition \ref{def:tensor-congruence}) that the relation $\cdot \equiv \cdot\,(\mathsf{wrt}\,\Omega)$ defines an equivalence relation on $\mathbf{M}_{n,m}(R[M])$.
\end{remark}

\begin{convention}[Relative null space]\label{conv:relative-null-space}
Let $(M,+_M,0_M)$ be an ic-monoid and let $(\Omega,\rho,\mathsf{Inv})$ be the multiplicative atomic structure on $\overline{\mathbb{Z}}_3$ defined in Example \ref{exa:Z3:multiplicative}. For every pair $(n,m)$ of positive integers and every matrix $A \in \mathbf{M}_{n,m}(\overline{\mathbb{Z}}_3[M])$, we define the \emph{relative null space of $A$} as the following set.
\[
\mathsf{RNull}(A) = \{Z \in \mathbf{Bal}_{m}~|~\exists Y \in \mathbf{Can}_{n,1}(M):A \,\overline{\boxdot} \,Z \equiv Y \,(\mathsf{wrt}\,\Omega)\}
\]
\end{convention}

\begin{convention}[Canonical embedding]\label{conv:canonical-embedding:emb}
Let $(M,+_M,0_M)$ be an ic-monoid. We shall denote as $\mathsf{emb}$ the function that maps every element $m \in M$ to the element $\uparrow \! X^{\mathbf{1}(m)} \in \overline{\mathbb{Z}}_3[\mathbb{Q}(M)]$. For convenience,  we shall follow Convention \ref{def:rationals:ic-monoids} and identify $\mathbf{1}(m)$ with the element $m$ itself such that we will write $\mathsf{emb}(m) = \uparrow \! X^{m}$.
\end{convention}

The following remark uses the notion of coefficient sets introduced in Definition \ref{def:coefficient-sets}. Hence, we suggest the reader to review the examples of section \ref{ssec:modulo-tensors} before working on the content of Remark \ref{rem:tensor-congurence:emb} below.

\begin{remark}[Tensor congruence]\label{rem:tensor-congurence:emb}
The goal of this remark is to show that the function $\mathsf{emb}$, defined in Convention \ref{conv:canonical-embedding:emb}, possesses homomorphic properties, which will be important for the conclusion of this paper (section \ref{ssec:Calculus_and_algorithm}).
Let $(M,+_M,0_M)$ be an ic-monoid and let $(\Omega,\rho,\mathsf{Inv})$ be the multiplicative atomic structure on $\overline{\mathbb{Z}}_3$ defined in Example \ref{exa:Z3:multiplicative}. For every element $m \in M$, we have the identities $\chi_{\Uparrow}(\mathsf{emb}(m)) = \{m\}$ and $\chi_{\Downarrow}(\mathsf{emb}(m)) = \cmemptyset$, which give us the expression $|\mathsf{emb}(m)|_{\Omega} = \uparrow X^{m}$ (see Convention \ref{conv:modulo:formula:bar_s_R}). On the other hand, we have the following identities.
\[
\chi_{\Uparrow}(\mathsf{emb}(m)+\mathsf{emb}(m')) = \chi_{\Uparrow}(\uparrow X^{m} + \uparrow X^{m'}) =\left\{
\begin{array}{ll}
\{m\} &\textrm{if }m= m'\\
\{m,m'\} &\textrm{if }m \neq m'\\
\end{array}
\right.
\]
This means that the equation $\chi_{\Uparrow}(\mathsf{emb}(m)+\mathsf{emb}(m')) = \{m,m'\}$ holds in general, and it is straightforward to show that $\chi_{\Downarrow}(\mathsf{emb}(m)+\mathsf{emb}(m')) = \cmemptyset$. As a result, we deduce that the equation $|\mathsf{emb}(m)+\mathsf{emb}(m')|_{\Omega} = \uparrow X^{m+m'}$ holds (see Convention \ref{conv:modulo:formula:bar_s_R}). Since the beginning of this remark has also provided the equation $|\mathsf{emb}(m+m')|_{\Omega} = \uparrow\!X^{m+m'}$, we deduce that the equation $|\mathsf{emb}(m)+\mathsf{emb}(m')|_{\Omega} = |\mathsf{emb}(m+m')|_{\Omega}$ holds. By using Proposition \ref{prop:modulo-tensor:fix-points}, this equation can be rephrased as the following relation.
\[
\mathsf{emb}(m)+\mathsf{emb}(m') \equiv \mathsf{emb}(m+m')\,(\mathsf{wrt}\,\Omega)
\]
Even though the function $\mathsf{emb}$ is compatible with the addition operation (via the previous relation), it does not define a morphism of ic-monoids, because the element $\mathsf{emb}(0) = \uparrow\!X^{0_M}$ is not equivalent to the zero element $0X^{0_M} \in \overline{\mathbb{Z}}_3[M]$ through the relation $\cdot \equiv \cdot \,(\mathsf{mod}\,\Omega)$.
\end{remark}

\begin{theorem}[Correspondence I]\label{theo:correspondence:null-spaces:1}
Let $(M,+,0)$ be an ic-monoid. For every pair $(n,m)$ of positive integers and every matrix $A \in \mathbf{M}_{n,m}(M)$, the following inclusion holds (see Convention \ref{conv:canonical-embedding:emb} and Convention \ref{conv:applying-morphism:on-matrix}).
\[
\mathsf{RNull}(\mathsf{emb}(A)) \subseteq \mathsf{SNull}(A)
\]
\end{theorem}
\begin{proof}
Before proving the statement, recall that the matrix $\mathsf{emb}(A)$ is equal to the matrix $(\uparrow \! X^{\mathbf{1}(a_{i,j})})_{i,j}$ in $\mathbf{M}_{n,m}(\overline{\mathbb{Z}}_3[\mathbb{Q}(M)])$ (see Convention \ref{conv:applying-morphism:on-matrix} and Convention \ref{conv:canonical-embedding:emb}). This implies that the equation $\widetilde{\odot}(\mathsf{emb}(A)) = \mathbf{1}(A)$ holds. Let us now show the inclusion given in the statement. In this respect, let $Z \in \mathbf{Ban}_m$ be an element of $\mathsf{RNull}(\mathsf{emb}(A))$ and let us show that $Z$ belongs to the set $\mathsf{SNull}(A)$.  There exists a matrix $Y \in \mathbf{Can}_{n,1}(M)$ such that the equation $\mathsf{emb}(A)  \,\overline{\boxdot}\,Z \equiv Y \,(\mathsf{wrt}\,\Omega)$ holds. As a result, we have the following equations.
\begin{align*}
\widetilde{\odot}(\mathsf{emb}(A) Z) & = \widetilde{\odot}(|\mathsf{emb}(A) Z|_{|\Omega})&(\textrm{Proposition \ref{prop:YAMC}})\\
& = \widetilde{\odot}(|Y|_{\Omega}) &(\textrm{Proposition \ref{prop:tensor_congruence:implies:bar_R}})\\
& = \widetilde{\odot}(Y)&(\textrm{Proposition \ref{prop:YAMC}})\\
& \Rrightarrow 0 &(\textrm{Proposition \ref{prop:cancleling-property:skew-relation:matrix}})
\end{align*}
We can also use Proposition \ref{prop:applying-morphism:on-matrix:alpha_tilde} to rewrite the expression of the matrix $\widetilde{\odot}(\mathsf{emb}(A) \,\overline{\boxdot}\, Z)$ in $\mathbf{M}_{n,1}(M)$. Specifically, we have the following equations in  $\mathbf{M}_{n,1}(M)$.
\begin{align*}
\widetilde{\odot}(\mathsf{emb}(A) \,\overline{\boxdot}\, Z) & = \widetilde{\odot}(\mathsf{emb}(A)) \,\overline{\odot} \,Z &(\textrm{Proposition \ref{prop:applying-morphism:on-matrix:alpha_tilde}})\\
& = \mathbf{1}(A) \,\overline{\odot}\,Z&(\widetilde{\odot}(\mathsf{emb}(A)) = \mathbf{1}(A))
\end{align*}
In other words, we have shown that the relation $A \overline{\boxdot} Z \Rrightarrow 0$ holds, which shows that $Z$ is an element of $\mathsf{SNull}(A)$.
\end{proof}

\begin{theorem}[Correspondence II]\label{theo:correspondence:null-spaces:2}
Let $(M,+,0)$ be an ic-monoid such that we can express $M$ as a (finite) product $\prod_{i \in [n]} M_i$ in $\mathbf{Icm}$. For every finite collection $e = (e_{j})_{j \in [m]}$ of elements in $M$, the following equation holds (see the notations of Remark \ref{rem:unification:2}).
\[
\mathsf{RNull}(\mathsf{emb}(\mathsf{A}(e))) = \mathsf{SNull}(\mathsf{A}(e))
\]
\end{theorem}
\begin{proof}
By Theorem \ref{theo:correspondence:null-spaces:1}, the inclusion $\mathsf{RNull}(\mathsf{emb}(\mathsf{A}(e))) \subseteq \mathsf{SNull}(\mathsf{A}(e))$ is satisfied. There only remains to show the opposite inclusion.

Let $Z  = (z_{k,1})_{k,1}$ be an element of $\mathsf{SNull}(\mathsf{A}(e)) \subseteq \mathbf{Ban}_{m}$ and let us show that the element $Z$ belongs to $\mathsf{RNull}(\mathsf{emb}(\mathsf{A}(e)))$. We shall denote by $Y = (y_i)_{i \in [n]}$ the matrix $\mathsf{emb}(\mathsf{A}(e)) \overline{\boxdot} Z$ in $\mathbf{M}_{n,1}(\overline{\mathbb{Z}}_3[M])$. In other words, for every $i \in [n]$, we have the following equations.
\[
y_i = \sum_{k \in [m]} (z_{k,1} X^{0_M}) \boxdot (\uparrow X^{e^i_{k}}) = \sum_{k \in [m]} z_{k,1}   X^{e^i_{k}}
\]
If we now use the fact that $y_i$ can be written as $\sum_{m \in \mathsf{Sup}(y_i)} y_{i,m}X^m$, then we can express $y_{i,m}$ as the following sum, where we denote $Q(m|e) = \{k~|~e^i_k = m\}$.
\[
y_{i,m} = \sum_{k \in Q(m|e)}z_{k,1}
\]
The previous formula implies that we have the following equations for every element $U \in \Omega$.
\begin{align*}
\chi_{U}(y_i) &= \{m~|~y_{i,m} \in U\}&\\
&= \{m~|~\textrm{there exists }k \in Q(m|e): z_{k,1} \in U\}&(\textrm{Proposition \ref{prop:extend-atomic:large-sums}})\\
&= \textstyle \{e^i_k~|~k \in \bigcup_{u \in U}\mathsf{Ind}_Z(u)\} & (\textrm{Convention \ref{conv:indices}})
\end{align*}
Since the elements of $\Omega$ are $\Uparrow = \{\uparrow,\updownarrow\}$ and $\Downarrow = \{\downarrow,\updownarrow\}$, the previous equations give us the following specifications.
\[
\chi_{\Uparrow}(y_i) = \{e^i_{k}~|~k \in \mathsf{Ind}_Z(\uparrow) \cup \mathsf{Ind}_Z(\updownarrow)\}
\quad\quad\quad
\chi_{\Downarrow}(y_i) = \{e^i_{k}~|~k \in \mathsf{Ind}_Z(\downarrow) \cup \mathsf{Ind}_Z(\updownarrow)\}
\]
Furthermore, since the relation $\mathbf{1}(A(e)) \,\overline{\odot}\, Z \Rrightarrow Y$ is satisfied (because $Z \in \mathsf{SNull}(\mathsf{A}(e))$), Remark \ref{rem:unification} -- and more specifically equation (\ref{eq:unification:2}) --  implies that there exists an element $\lambda_i \in M$ for which the equations shown below, in (\ref{eq:correspondenceII:2}), hold for every element $i \in [n]$.
\begin{equation}\label{eq:correspondenceII:2}
\kappa_{\Uparrow}(y_i) = \sum_{k \in \mathsf{Ind}_X(\uparrow)} e^i_k + \sum_{k \in \mathsf{Ind}_X(\updownarrow)} e^i_k = \lambda_i = \sum_{k \in \mathsf{Ind}_Z(\downarrow)} e^i_k + \sum_{k \in \mathsf{Ind}_X(\updownarrow)} e^i_k = \kappa_{\Downarrow}(y_i)
\end{equation}
Now, recall that the following equations hold.
\begin{equation}\label{eq:correspondenceII:1}
|y_i|_{\Omega} =  \left\{
\begin{array}{ll}
\uparrow X^{\kappa_{\Uparrow}(y_i)} +  \downarrow X^{\kappa_{\Downarrow}(y_i)}&\textrm{if }\chi_{\Uparrow}(y_i) \neq \cmemptyset\textrm{ and } \chi_{\Downarrow}(y_i) \neq \cmemptyset\\
\uparrow X^{\kappa_{\Uparrow}(y_i)} &\textrm{if }\chi_{\Uparrow}(y_i) \neq \cmemptyset\textrm{ and } \chi_{\Downarrow}(y_i) = \cmemptyset\\
\downarrow X^{\kappa_{\Downarrow}(y_i)}&\textrm{if }\chi_{\Uparrow}(y_i) = \cmemptyset\textrm{ and } \chi_{\Downarrow}(y_i) \neq \cmemptyset\\
0&\textrm{if }\chi_{\Uparrow}(y_i) = \cmemptyset\textrm{ and } \chi_{\Downarrow}(y_i) = \cmemptyset
\end{array}
\right.
\end{equation}
Since the matrix vector $Z$ is balanced (\emph{i.e.} $\mathsf{Ind}_Z(\uparrow) \neq \cmemptyset \Leftrightarrow \mathsf{Ind}_Z(\downarrow) \neq \cmemptyset$), we deduce from equation (\ref{eq:correspondenceII:2}) that the relations shown in (\ref{eq:correspondenceII:1}) can be reformulated as follows.
\begin{align*}
|y_i|_{\Omega} & =
\left\{
\begin{array}{ll}
\uparrow X^{\lambda_i} +  \downarrow X^{\lambda_i} &\textrm{if }\chi_{\Uparrow}(y_i) \neq \cmemptyset \textrm{ and } \chi_{\Downarrow}(y_i) \neq \cmemptyset\\
0&\textrm{otherwise.}
\end{array}
\right.\\
& =
\left\{
\begin{array}{ll}
\updownarrow X^{\lambda_i} &\textrm{if }\chi_{\Uparrow}(y_i) \neq \cmemptyset \textrm{ and } \chi_{\Downarrow}(y_i) \neq \cmemptyset\\
\updownarrow \boxdot \,0&\textrm{otherwise.}
\end{array}
\right.
\end{align*}
The previous equations show that if we let $Y'$ denote the matrix $(|y_i|_{\Omega})_{i,1}$ in $\mathbf{M}_{n,1}(\overline{\mathbb{Z}}_3[M])$, then we have $Y' \in  \mathbf{Can}_{n,1}(M)$. Also, by Proposition \ref{prop:modulo-tensor:fix-points}, the fact that we defined $Y'$ as the matrix $(|y_i|_{\Omega})_{i,1}$ implies that the relation $Y \equiv Y'\,(\mathsf{wrt}\,\Omega)$ holds. Hence, we have the relation $(\mathsf{emb}(\mathsf{A}(e)) \,\overline{\boxdot}\, Z \equiv Y'\,(\mathsf{wrt}\,\Omega)$ where $Y' \in  \mathbf{Can}_{n,1}(M)$. This means that the element $Z$ is in the relative null space $\mathsf{RNull}(\mathsf{emb}(\mathsf{A}(e)))$.
\end{proof}

\subsection{Calculus and algorithm}\label{ssec:Calculus_and_algorithm}
The goal of this section is to discuss an algorithmic procedure (see Remark \ref{rem:algogirthm:linear-systems} and Example \ref{exa:pedigrad-unification}) that returns a description of the null spaces discussed in Theorem \ref{theo:correspondence:null-spaces:2}. To this end, we shall reformulate the specification of these null spaces one last times solely in terms of operations in the semiring $\overline{\mathbb{Z}}_3$ (see Definition \ref{conv:null-space}). In order to make this reformulation, we shall need the homomorphic properties discussed in Remark \ref{rem:tensor-congurence:emb} and the concept of a basis for an ic-monoid (Definition \ref{def:basis}).

As mentioned earlier, we shall focus on the situation presented in Theorem \ref{theo:correspondence:null-spaces:2}. In this respect, we shall let $(M,+,0)$ denote an ic-monoid that can be expressed as a (finite) product $\prod_{i \in [n]} M_i$ in $\mathbf{Icm}$. Additionally, we shall assume that, for every $i \in [n]$, there exists a positive integer $n_i$ such that, if we let $S_i$ denote the finite set $[n_i]$, then the ic-monoid $M_i$ is (isomorphic to) the ic-monoid $B_2^{S_i}$. We suggest that the reader reviews Example \ref{exa:products:powers-of-B_2} as a refresher on certain notations and intuitions.

\begin{definition}[Basis]\label{def:basis}
Let $(N,+,0)$ be an ic-monoid. We will say that a collection $(h_i)_{i \in I}$ of elements in $N$ is a \emph{basis} for $N$ if for every element $e \in N$, there exists a unique finite subset $U \subseteq I$ such that we can express $e$ as a sum $\sum_{i \in U} h_i$.
\end{definition}

\begin{convention}[Basis]
For every element $i \in [n]$ and every element $x \in S_i$, we denote as $\mathsf{h}^i(x)$ the element of $M_i = B_2^{S_i}$ represented by the tuple $(\mathsf{h}^i_t(x))_{t \in S_i}$ whose coefficients are defined as follows.
\[
\mathsf{h}^i_t(x) =
\left\{
\begin{array}{ll}
1&\textrm{if }t = x\\
0&\textrm{otherwise.}
\end{array}
\right.
\]
For example, if we take $S_i = [4]$, then we have $\mathsf{h}^i(1) = \mathtt{1000}$, $\mathsf{h}^i(2) = \mathtt{0100}$, $\mathsf{h}^i(3) = \mathtt{0010}$ and $\mathsf{h}^i(4) = \mathtt{0001}$. It is straightforward to verify that the collection $(\mathsf{h}^i(x))_{x \in S_i}$ defines a basis for $M_i$, namely: for every element $e \in M_i$, there exists a unique subset $U \subseteq S_i$ such that the following equation holds.
\[
e = \sum_{x \in U}\mathsf{h}^i(x)
\]
For example, the element $\mathsf{1011}$ of $B_2^{S_i}$ can be expressed as $ \mathsf{h}^i(1) + \mathsf{h}^i(3)+\mathsf{h}^i(4)$ such that its associated set $U$ is equal to $\{1,3,4\}$.
\end{convention}

\begin{convention}[Coproduct and cardinality]\label{conv:coproduct:S_ast}
We shall denote as $S_{\ast}$ the coproduct $\coprod_{i \in [n]} S_i$. As is common practice in set theory, the elements of the set $S_{\ast}$ will be seen as pairs $(i,x)$ where $i \in [n]$ and $x \in S_i$. We shall denote as $n_{\ast}$ the cardinal of the set $S_{\ast}$. Note that the integer $n_{\ast}$ is equal to the sum $n_1+n_2+\dots+n_n$, where $n_i$ is the cardinal of the set $S_i$ for every $i \in [n]$. The definition of $n_{\ast}$ implicitely gives a bijection $\omega:S_{\ast} \to [n_{\ast}]$ sending an element $(i,x) \in S_{\ast}$ to the integer $\sum_{j = 1}^{i-1} n_j + x$.
\end{convention}

\begin{remark}[Submonoid and basis]\label{remark:basis:M_M_i}
As shown in Remark \ref{rem:unification:2}, the ic-monoid $M_i$ can be seen as a submonoid of $M$. As such, for every $i \in [n]$ and every $x \in S_i$, the element $\mathsf{h}^i(x) \in M_i$ can be seen as an element of $M$. Then, we can use Definition \ref{def:basis} and the projections $\pi_i:M \to M_i$ to show that the collection of elements $\mathsf{h}^i(x)$, where the elements $i$ and $x$ run over the sets $[n]$ and $S_i$, respectively, constitutes a basis for $M$. Indeed, for every element $e \in M$, we have the following expressions.
\[
\quad\quad\quad e = \sum_{i \in [n]}\pi_i(e) = \sum_{i \in [n]} \sum_{x \in U_i}\mathsf{h}^i(x)\quad\quad (\textrm{assuming that }\pi_i(e) = \sum_{x \in U_i}\mathsf{h}^i(x)\textrm{ for some }U_i)
\]
These expressions give us the expression $e = \sum_{(i,x) \in U} \mathsf{h}^i(x)$ where $U$ denotes the subset $\coprod_{i \in [n]} U_i$ of $S_{\ast}$. If there was another subset $V \subseteq S_{\ast}$ for which the formula
\[
e = \sum_{(i,x) \in V} \mathsf{h}^i(x)
\]
was satisfied, then, for every $i \in [n]$, the projection $\pi_i:M \to M_i$ would give us the equality $\sum_{x \in V_i} \mathsf{h}^i(x) = \pi_i(e) = \sum_{x \in U_i} \mathsf{h}^i(x)$ where $V_i = \{x~|~(i,x) \in V\}$. By universality of the basis $(\mathsf{h}^i(x))_{x \in S_{i}}$ for $M_i$, we would obtain the equation $U_i = V_i$. This implies that the collection $(\mathsf{h}^i(x))_{(i,x) \in S_{\ast}}$ constitutes a basis for the ic-monoid $M$.
\end{remark}

\begin{convention}[Notation]
We shall denote as $\mathsf{h}$ the basis of the ic-monoid $(M,+,0)$ defined by the collection $(\mathsf{h}^i(x))_{(i,x) \in S_{\ast}}$ (see Remark \ref{remark:basis:M_M_i} above).
\end{convention}

\begin{convention}[Generated set]
Let $(N,+_N,0_N)$ be an ic-monoid and let $(R,+_R,\cdot_R,0_R,1_R)$ be an ic-semiring. For every basis $h = (h_i)_{i \in I}$ of $N$, we will denote as $R[N|h]$ the set of elements $x \in R[N]$ that can be expressed as a sum $\sum_{i \in U} r_iX^{h_i}$ where $U$ is a finite subset of $I$ and $(r_i)_{i \in U}$ is a $U$-indexed collection of elements $r_i$ in $R$.
\end{convention}

\begin{proposition}\label{prop:canceling:updown:equivalence_through_bar_omega}
Let $(\Omega,\rho,\mathsf{Inv})$ be the multiplicative atomic structure on $\overline{\mathbb{Z}}_3$ defined in Example \ref{exa:Z3:multiplicative} and let $s$ be an element of $\overline{\mathbb{Z}}_3[M|\mathsf{h}]$. The following statements are equivalent.
\begin{itemize}
\item[1)] there exists $\lambda$ in $\overline{\mathbb{Z}}_3[M]$ such that $s = \updownarrow \boxdot\, \lambda$;
\item[2)] there exists $\lambda'$ in $\overline{\mathbb{Z}}_3[M]$ such that $|s|_{\Omega} = \updownarrow \boxdot\, \lambda'$;
\end{itemize}
\end{proposition}
\begin{proof}
Suppose that there exists an element $\lambda \in \overline{\mathbb{Z}}_3[M]$ such that $s = \updownarrow \boxdot\, \lambda$. By using the Cayley table shown in Example \ref{exa:semi-group:Z3Z}, we can show that each coefficient $s_m$ composing the tuple $s \in \overline{\mathbb{Z}}_3[M]$ is equal to either $0$ or $\updownarrow$. This means that the following equations hold (see Definition \ref{def:coefficient-sets}).
\[
\quad\quad\quad\quad\quad\quad\quad\quad\quad\kappa_{\Uparrow}(s) = K_{\updownarrow}(s) = \kappa_{\Downarrow}(s)\quad\textrm{where we denote } K_{\updownarrow}(s) = \{m~|~s_m = \updownarrow\}
\]
We deduce that if $s \neq 0$, then we have $|s|_{\Omega} = \,\uparrow X^{\kappa_{\Uparrow}(s)} + \downarrow X^{\kappa_{\Downarrow}(s)} = \,\updownarrow \! X^{K_{\updownarrow}(s)}$ and if $s = 0$, then we have $|s|_{\Omega} = 0 = \,\updownarrow \boxdot \,0$. In other words, we conclude that there exists an element $\lambda' \in \overline{\mathbb{Z}}_3[M]$ such that $|s|_{\Omega} = \,\updownarrow \boxdot\, \lambda'$.

Conversely, let us now assume that there exists an element $\lambda' \in \overline{\mathbb{Z}}_3[M]$ such that $|s|_{\Omega} = \,\updownarrow \boxdot\, \lambda'$. By Definition \ref{conv:modulo:formula:bar_s_R}, this can only happen if the identity $\kappa_{\Uparrow}(s) = \kappa_{\Downarrow}(s)$ holds. Note that this identity can be reformulated as follows (see Proposition \ref{prop:kappa-morphism}).
\begin{equation}\label{eq:canceling:updown:equivalence_through_bar_omega}
\textstyle \sum_{m \in \chi_{\Uparrow}(s)} m = \sum_{m \in \chi_{\Downarrow}(s)} m
\end{equation}
Since we have $s \in \overline{\mathbb{Z}}_3[M|\mathsf{h}]$, each index $m$ involved in equation (\ref{eq:canceling:updown:equivalence_through_bar_omega}) is an element of the basis $\mathsf{h} = (\mathsf{h}^i(x))_{(i,x) \in S_{\ast}}$. It follows from the universal property associated with $\mathsf{h}$ (see Definition \ref{def:basis}) and equation (\ref{eq:canceling:updown:equivalence_through_bar_omega}) that the relation $m \in \chi_{\Uparrow}(s)$ holds if, and only if, the relation $m \in \chi_{\Downarrow}(s)$ holds. In other words, the equation $\chi_{\Uparrow}(s) = \chi_{\Downarrow}(s)$ holds. This can only happen if the coefficients encoding the tuple $s \in \overline{\mathbb{Z}}_3[M]$ are equal to $0$ or $\updownarrow$ (for instance, if $s_m = \uparrow$, then $m \in \chi_{\Uparrow}(s) \backslash \chi_{\Downarrow}(s)$). We therefore deduce that the element $s \in \overline{\mathbb{Z}}_3[M]$ is of the form $\updownarrow \boxdot\, \lambda$ for some element $\lambda \in \overline{\mathbb{Z}}_3[M]$.
\end{proof}

\begin{convention}[Decomposition]\label{conv:decomposition_dec_op}
Let $e$ be an element of $M$ such that, for every $i \in [n]$, the projection $\pi_i(e)$ in $M_i$ admits a (unique) decomposition of the form $\sum_{x \in U_i} \mathsf{h}^i(x)$. For every such element $e$, we define the following elements in $\overline{\mathbb{Z}}_3[M|h]$.
\[
\textstyle
\mathsf{dec}_i(e) = \sum_{x \in U_i} \mathsf{emb}(\mathsf{h}^i(x))
\quad\quad\quad\quad
\mathsf{dec}(e) = \sum_{i \in [n]} \mathsf{dec}_i(e)
\]
As a result, the element $\mathsf{dec}(e)$ can also be expressed as a sum $\sum_{(i,x) \in U } \mathsf{emb}(\mathsf{h}^i(x))$ in $\overline{\mathbb{Z}}_3[M|h]$ where we denote as $U$ the subset $\prod_{i \in [n]} U_i$ of $S_{\ast}$.
\end{convention}

\begin{proposition}[Correspondence]\label{prop:dec-emb:correspondence}
Let $(\Omega,\rho,\mathsf{Inv})$ be the multiplicative atomic structure on $\overline{\mathbb{Z}}_3$ (defined in Example \ref{exa:Z3:multiplicative}). For every element $e \in M$, the following relation holds.
\[
\mathsf{emb}(e) \equiv \mathsf{dec}(e)\,(\mathsf{wrt}\,\Omega)
\]
\end{proposition}
\begin{proof}
Let $e$ be an element in $M$ and let $\sum_{(i,x) \in U} \mathsf{h}^i(x)$ be the unique expression of $e$ in $M$ in terms of the basis $(\mathsf{h}^i(x))_{(i,x) \in S_{\ast}}$. For every $i \in [n]$, we shall also denote as $U_i$ the sets $\{x~|~(i,x) \in U\}$ -- this means that the equation $U = \coprod_{i \in [n]} U_i$ holds. According to Remark \ref{rem:tensor-congurence:emb} and Theorem \ref{theo:tensor-congruence:semiring-compatiblity}, we have the following relations.
\begin{align*}
\mathsf{emb}(e) & \textstyle =  \mathsf{emb}(\sum_{(i,x) \in U} \mathsf{h}^i(x)) &\\
&\textstyle \equiv \sum_{(i,x) \in U} \mathsf{emb}(\mathsf{h}^i(x)) \,(\mathsf{wrt}\,\Omega) &(\textrm{Remark \ref{rem:tensor-congurence:emb}})\\
&\textstyle \equiv \mathsf{dec}(e) \,(\mathsf{wrt}\,\Omega)&(\textrm{Convention \ref{conv:decomposition_dec_op}})
\end{align*}
The previous series of relations proves the statement.
\end{proof}

\begin{proposition}[Correspondence]
Let $(\Omega,\rho,\mathsf{Inv})$ be the multiplicative atomic structure on $\overline{\mathbb{Z}}_3$ (defined in Example \ref{exa:Z3:multiplicative}). For every matrix $A$ in $\mathbf{M}_{n,m}(\overline{\mathbb{Z}}_3[M])$, the following equation holds.
\[
\mathsf{RNull}(\mathsf{emb}(A)) = \mathsf{RNull}(\mathsf{dec}(A))
\]
\end{proposition}
\begin{proof}
It follows from Proposition \ref{prop:dec-emb:correspondence} (and the fact that the relation $\cdot \equiv \cdot \,(\mathsf{wrt}\,\Omega)$ is symmetric) that the following relation holds.
\[
\mathsf{dec}(A) \equiv \mathsf{emb}(A) \,(\mathsf{wrt}\,\Omega)
\]
It then follows from Theorem \ref{theo:tensor-congruence:semiring-compatiblity} that we can multiply the previous relation with any matrix. Specificaly, for every matrix $Z$ in $\mathbf{M}_{m,1}(\overline{\mathbb{Z}}_3)$, the following relation holds.
\[
\mathsf{dec}(A)\,\overline{\boxdot}\,Z \equiv \mathsf{emb}(A)\,\overline{\boxdot}\,Z \,(\mathsf{wrt}\,\Omega)
\]
If we take $Z$ to be in $\mathsf{RNull}(\mathsf{emb}(A))$, then there exists $Y \in \mathbf{Can}_{n,1}(M)$ such that $\mathsf{emb}(A)\,\overline{\boxdot}\,Z \equiv Y \,(\mathsf{wrt}\,\Omega)$. As a result, we obtain the relation $\mathsf{dec}(A)\,\overline{\boxdot}\,Z \equiv Y \,(\mathsf{wrt}\,\Omega)$, which shows that $Z$ belongs to $\mathsf{RNull}(\mathsf{dec}(A))$. Conversely, if we take $Z$ to be in $\mathsf{RNull}(\mathsf{dec}(A))$, then there exists $Y \in \mathbf{Can}_{n,1}(M)$ such that $\mathsf{dec}(A)\,\overline{\boxdot}\,Z \equiv Y \,(\mathsf{wrt}\,\Omega)$, which shows that $\mathsf{emb}(A)\,\overline{\boxdot}\,Z \equiv Y \,(\mathsf{wrt}\,\Omega)$ and hence $Z \in \mathsf{RNull}(\mathsf{emb}(A))$. This shows the equality of the statement.
\end{proof}

The following definition should be compared to Definition \ref{def:weakly-canceling-porperty}. Note that the definition below only involves the semiring $\overline{\mathbb{Z}}_3$ while Definition \ref{def:weakly-canceling-porperty} involves semirings of polynomials.

\begin{definition}[Canceling]\label{def:canceling:coef_in_Z3}
Let $(n,m)$ be a pair of positive integers. We will say that a matrix $A = (a_{i,j})$ in $\mathbf{M}_{n,m}(\overline{\mathbb{Z}}_3)$ is \emph{canceling} if for every $(i,j) \in [n] \times [m]$, the relation $a_{i,j} \in \{0, \updownarrow\}$ holds. We shall denote by $\mathbf{Can}_{n,m}$ the set of canceling matrices in $\mathbf{M}_{n,m}(\overline{\mathbb{Z}}_3)$.
\end{definition}

\begin{proposition}[Canceling property]\label{prop:canceling-property}
Let $(n,m,p)$ be a triple of positive integers. For every matrix $B \in \mathbf{M}_{p,n}(\overline{\mathbb{Z}}_3)$ and every matrix $A \in \mathbf{Can}_{n,m}$, the product $BA$ is in $\mathbf{Can}_{p,m}$.
\end{proposition}
\begin{proof}
We shall let $A$ and $B$ be of the form $(a_{i,j})_{i,j}$ and $(b_{i,j})_{i,j}$, respectively. Since the product $BA$ is already in $\mathbf{M}_{p,m}(\overline{\mathbb{Z}}_3)$, there only remains to verify the property stated in Definition \ref{def:canceling:coef_in_Z3}. Let us denote the matrix $BA$ as $(c_{i,j})_{i,j}$. With these notations, we have the following expression for every $(i,j) \in [p]\times [m]$.
\[
c_{i,j} = \textstyle  \sum_{k \in [n]}b_{i,k} \cdot a_{k,j}
\]
Now, we can use Definition \ref{def:canceling:coef_in_Z3} to establish that, for every $(k,j) \in [n] \times [m]$, the relation $a_{k,j}  \in \{0,\updownarrow\}$ holds. According to the Cayley table shown in Example \ref{exa:semi-group:Z3Z}, this means that the relation $b_{i,k} \,\boxdot\, a_{k,j} \in \{0,\updownarrow\}$ holds for every triple $(i,k,j) \in [p] \times [n] \times [m]$.
It follows from the axioms for ic-semirings that the element $c_{i,j}$ belongs to the set $\{0,\updownarrow\}$ for every $(i,j) \in [p] \times [m]$.
\end{proof}

\begin{definition}[Null space]\label{conv:null-space}
For every pair $(n,m)$ of positive integers and every matrix $A \in \mathbf{M}_{n,m}(\overline{\mathbb{Z}}_3)$, we define the \emph{null space of $A$} as the following set.
\[
\mathsf{Null}(A) = \{Z \in \mathbf{Bal}_{m}~|~AZ \in \mathbf{Can}_{n,1}\}
\]
\end{definition}

\begin{proposition}[Null space]\label{prop:cnull-space:anceling-property}
Let $(n,m)$ be a pair of positive integers. For every matrix $B \in \mathbf{M}_{n,n}(\overline{\mathbb{Z}}_3)$ and every matrix $A \in \mathbf{M}_{n,m}(\overline{\mathbb{Z}}_3)$, the inclusion $\mathsf{Null}(A) \subseteq \mathsf{Null}(BA)$ holds.
\end{proposition}
\begin{proof}
Directly follows from Definition \ref{conv:null-space} and Proposition \ref{prop:canceling-property}.
\end{proof}

\begin{example}[Swapping operation]\label{exa:swapping-operations}
Recall that in standard linear algebra (over rings and fields), we can swap rows of a matrix by using a matrix product operation. In much the same way, one can also use a matrix product operation to swap the rows of a matrix with coefficients in the ic-monoid $\overline{\mathbb{Z}}_3$. Specifically, let $(n,m)$ be a pair of positive integers and let $A$ be a matrix in $\mathbf{M}_{n,m}(\overline{\mathbb{Z}}_3)$. We can swap
the $k_1$-th row of $A$ with the $k_2$-th row of $A$ (where $k_1 < k_2$) by multiplying the matrix $A$ with the matrix $\mathsf{B}_{k_1}^{k_2} = (b_{i,j})_{i,j}$ of $\mathbf{M}_{n,n}(\overline{\mathbb{Z}}_3)$ defined as follows.
\[
b_{i,j} = \left\{
\begin{array}{ll}
\uparrow &\textrm{if }i = j\textrm{ and }j \notin \{k_1,k_2\}\\
\uparrow &\textrm{if }i = k_1\textrm{ and }j = k_2\\
\uparrow &\textrm{if }i = k_2\textrm{ and }j = k_1\\
0 &\textrm{otherwise.}\\
\end{array}
\right.
\]
In this case, we can check that we have the following equations.
\begin{equation}\label{eq:swapping:matrix_product}
\mathsf{B}_{k_1}^{k_2}A =
\left(
\begin{array}{c}
(\mathsf{row}_i(A))_{i = 1,\dots,k_1-1}\\
\mathsf{row}_{k_2}(A)\\
(\mathsf{row}_i(A))_{i = k_1+1,\dots,k_2-1}\\
\mathsf{row}_{k_1}(A)\\
(\mathsf{row}_i(A))_{i = k_2+1,\dots,n-1}\\
\end{array}
\right)
\quad\quad\quad\mathsf{B}_{k_2}^{k_1}\mathsf{B}_{k_1}^{k_2} = \left(
\begin{array}{ccccc}
\uparrow & 0 & 0 &\dots & 0\\
0 & \uparrow & 0 &\dots & 0\\
0 & 0 & \uparrow &\dots & 0\\
\vdots & \vdots & \vdots &\dots & \vdots\\
0 & 0 & 0 &\dots & \uparrow\\
\end{array}
\right)= \mathsf{Id}_{n}
\end{equation}
It follows from the right-hand side equation of (\ref{eq:swapping:matrix_product}), combined with Proposition \ref{prop:cnull-space:anceling-property}, that swapping rows in a matrix does not change its underlying null space. In other words, if $B$ denotes a product of matrices  of the form $\mathsf{B}_{k_1}^{k_2}$, then the equation $\mathsf{Null}(A) = \mathsf{Null}(BA)$ holds.
\end{example}

The following definition makes use of the morhpism $\omega$ introduced in Convention \ref{conv:coproduct:S_ast}.

\begin{definition}[Coefficient matrix]
Let $e = (e_{j})_{j \in [m]}$ be a finite collection of elements in $M$. For every $j \in [m]$ and every $i \in [n]$, we shall let $U_i(j)$ denote the unique subset of $[n_i]$ indexing the sum composing the expression of the element $\pi_i(e_j) \in M_i$ in terms of the basis $(\mathsf{h}^i(x))_{x \in S_i}$ for $M_i$  (see below).
\[
\pi_i(e_j) = \sum_{x \in U_i(j)}\mathsf{h}^i(x)
\]
Then, we define the matrix $\mathsf{Coef}(e)$ in $\mathbf{M}_{n_{\ast},m}(\overline{\mathbb{Z}}_3)$ as follows (see Convention \ref{conv:coproduct:S_ast}).
\[
\mathsf{Coef}(e) = (c_{k,j})_{k,j}
\quad\quad
c_{k,j} = \left\{
\begin{array}{ll}
\uparrow&\textrm{if }\omega^{-1}(k) = (i,x)\textrm{ and }x \in U_i(j)\\
0&
\end{array}
\right.
\]
We will give an example for the operation $\mathsf{Coef}$ in Example \ref{exa:pedigrad-unification}.
\end{definition}

\begin{theorem}[Correspondence III]\label{theo:correspondence:null-spaces:3}
Let $(\Omega,\rho,\mathsf{Inv})$ be the multiplicative atomic structure on $\overline{\mathbb{Z}}_3$ (defined in Example \ref{exa:Z3:multiplicative}). For every matrix $A$ in $\mathbf{M}_{n,m}(\overline{\mathbb{Z}}_3[M])$, the following equation holds.
\[
\mathsf{Null}(\mathsf{Coef}(e)) = \mathsf{RNull}(\mathsf{dec}(\mathsf{A}(e)))
\]
\end{theorem}
\begin{proof}
Throughout this proof, we shall denote $\mathsf{Coef}(e) = (c_{k,j})_{k,j}$. Before proving the statement, we shall make some comment on the coefficients of the matrix $\mathsf{dec}(\mathsf{A}(e))$.
First, by Convention \ref{conv:decomposition_dec_op}, we have the following equations for every $i \in [n]$ and $j\in [m]$.
\begin{align*}
\mathsf{dec}(e_j^i) & = \mathsf{dec}_i(e_j)&\\
& = \textstyle\sum_{i \in [n]}\sum_{x \in U_i(j)} \mathsf{emb}(\mathsf{h}^i(x))&\\
& = \textstyle \sum_{i \in [n]}\sum_{x \in U_i(j)} \uparrow\! X^{\mathsf{h}^i(x)}&\\
& = \textstyle \sum_{i \in [n]}\sum_{x \in S_i} c_{\omega(i,x),j} X^{\mathsf{h}^i(x)}&
\end{align*}
Let $Z = (z_{j,1})_{j,1}$ denote a matrix vector in $\mathbf{Bal}_{m}$. We shall denote the $(i,1)$-coefficient of the matrix product $\mathsf{dec}(\mathsf{A}(e))\,\overline{\boxdot}\,Z$ in $\mathbf{M}_{n,1}(\overline{\mathbb{Z}}_3[M])$ as $\alpha_{i}$ and, for every $k \in [n_{\ast}]$, we shall denote the $(k,1)$-coefficient of the matrix product $\mathsf{Coef}(e)Z$ in $\mathbf{M}_{n,1}(\overline{\mathbb{Z}}_3)$ as $\beta_{k}$. The series of identities shown below establishes a relationship between these coefficients for every $i \in [n]$.
\begin{align*}
\alpha_{i} & = \textstyle \sum_{j \in [m]}z_{j,1} \boxdot \mathsf{dec}(e_j^i) &\\
& = \textstyle \sum_{j \in [m]}z_{j,1} \boxdot(\sum_{i \in [n]}\sum_{x \in S_i} c_{\omega(i,x),j} X^{\mathsf{h}^i(x)})&\\
& = \textstyle \sum_{i \in [n]}\sum_{x \in S_i} (\sum_{j \in [m]}z_{j,1} \cdot c_{\omega(i,x),j}) X^{\mathsf{h}^i(x)}&\\
& = \textstyle \sum_{i \in [n]}\sum_{x \in S_i} \beta_{\omega(i,x)} X^{\mathsf{h}^i(x)}&(\textstyle\beta_{k} = \sum_{j \in [m]}z_{j,1} \cdot c_{k,j})
\end{align*}
In other words, we have the following equation for every $i \in [n]$.
\begin{equation}\label{eq:correspondence:3}
\alpha_{i} = \sum_{i \in [n]}\sum_{x \in S_i} \beta_{\omega(i,x)} X^{\mathsf{h}^i(x)},
\end{equation}
We shall now show the equality of the statement by using a double-inclusion argument. Suppose that the matrix $Z$ belongs to $\mathsf{Null}(\mathsf{Coef}(e))$. This means that, for every $k \in  [n_*]$, we have the relation $\beta_k \in \{0,\updownarrow\}$. Since we have equation (\ref{eq:correspondence:3})
it follows that for every $i \in [n]$, there exists an element $\lambda_i$ in $\overline{\mathbb{Z}}_3[M]$ such that the coefficient $\alpha_{i}$ is of the from $\updownarrow \!\boxdot \,\lambda_i$. This means that the matrix $\mathsf{dec}(\mathsf{A}(e))\,\overline{\boxdot}\,Z$ is in $\mathbf{Can}_{n,1}(M)$. Since the relation $\cdot \equiv \cdot\,(\mathsf{wrt}\,\Omega)$ is reflexive, it follows that the matrix vector $Z$ belongs to $\mathsf{RNull}(\mathsf{dec}(\mathsf{A}(e)))$.

Conversely, if the matrix $Z$ belongs to $\mathsf{RNull}(\mathsf{dec}(\mathsf{A}(e)))$, then for every $i \in [n]$, there exists an element $\lambda_i$ in $\overline{\mathbb{Z}}_3[M]$ such that the relation $\alpha_{i} \equiv \updownarrow \!\boxdot \,\lambda_i$ holds. By Proposition \ref{prop:modulo-tensor:fix-points}, this means that we have the following equation.
\[
|\alpha_{i}|_{\Omega} = |\!\!\updownarrow \!\boxdot \,\lambda_i|_{\Omega}
\]
By Proposition \ref{prop:canceling:updown:equivalence_through_bar_omega}, there exists an element $\lambda_i' \in \overline{\mathbb{Z}}_3[M]$ for which the equation $|\alpha_{i}|_{\Omega} = \updownarrow \!\boxdot \,\lambda_i'$ holds. Another application of Proposition \ref{prop:canceling:updown:equivalence_through_bar_omega} shows that there exists an element $\lambda_i^{\prime\prime}$ in $\overline{\mathbb{Z}}_3[M]$ for which the equation $\alpha_{i} = \updownarrow \!\boxdot \,\lambda_i^{\prime\prime}$ holds. Since we have equation (\ref{eq:correspondence:3}) and the elements of the basis $\mathsf{h}$ are all distinct, it follows that we have the relation $\beta_{\omega(i,x)} \in \{0,\updownarrow\}$ for every $i \in [n]$ and every $x \in S_i$. Because $\omega$ is a bijection $S_{\ast} \to [n_{\ast}]$ (Convention \ref{conv:coproduct:S_ast}), this means that the relation $\beta_{k} \in \{0,\updownarrow\}$ holds for every $k \in [n_{\ast}]$. In other words, the matrix vector $\mathsf{Coef}(e)Z$ is in $\mathbf{Can}_{n,1}$.
\end{proof}

\begin{convention}[Reduced row echelon form]\label{conv:reduced_row_echelon_form}
For every pair $(n,m)$ of positive integers, a matrix $A = (a_{i,j})_{i,j}$ in $\mathbf{M}_{n,m}(\overline{\mathbb{Z}}_3)$ will be said to admit a \emph{reduced row echelon form} for a non-negative integer $n' \leq n$ and an non-decreasing function $\phi:[n'] \to [m]$ if it satisfies the following conditions.
\[
a_{i,j} = \left\{
\begin{array}{ll}
\uparrow &j = \phi(i)\textrm{ and }i \in [n']\\
0 &j < \phi(i)\textrm{ and }i \in [n']\\
0 &i \notin [n']\\
\end{array}
\right.
\]
If we refer to an arbitrary choice of an element in $\overline{\mathbb{Z}}_3$ by a symbol $\ast$, then the matrix shown below, on the left, describes a matrix in reduced row echelon form for the integer $4$ and the function $\phi$ displayed on the right.
\[\left(
\begin{array}{cccccc}
\uparrow & \ast & \ast & \ast & \ast& \ast\\
0 & \uparrow & \ast & \ast & \ast & \ast\\
0 & 0 & 0 & \uparrow & \ast & \ast\\
0 & 0 & 0 & \uparrow & \ast & \ast\\
0 & 0 & 0 & 0 & 0 & 0\\
\end{array}
\right)
\quad\quad\quad
\phi:
\left(
\begin{array}{ccc}
[4] & \to &[6]\\
1&\mapsto&1\\
2&\mapsto&2\\
3&\mapsto&4\\
4&\mapsto&4\\
\end{array}
\right)
\]
\end{convention}

\begin{convention}[Solution]\label{conv:solution_ud}
It follows from the Cayley table of Example \ref{exa:semi-group:Z3Z} for the addition that for every element $y \in \overline{\mathbb{Z}}_3$, the relation $x + y  \in \{0,\updownarrow\}$ always admits a solution $x$. Specifically, let $(x,y)$ be a pair of elements in $\overline{\mathbb{Z}}_3$. By Proposition \ref{prop:sum_to_zero}, we have $x+y = 0$ if, and only if, we have $x = y = 0$. It is also straightforward to verify that we have the equation $x+y = \updownarrow$ if, and only, if we have the following relations.
\[
x \in \left\{\begin{array}{ll}
\{\updownarrow\}&\textrm{if }y=0\\
\Downarrow&\textrm{if }y=\uparrow\\
\Uparrow&\textrm{if }y=\downarrow\\
\overline{\mathbb{Z}}_3&\textrm{if }y=\updownarrow\\
\end{array}
\right.
\]
To conclude, the relation $x+y \in \{0,\updownarrow\}$  holds if, and only, if we have the following parametrization of $x$ in terms of $y$:
\[
\quad\quad\quad\quad\quad x \in \langle y \rangle \quad\textrm{where }\quad
\langle y \rangle := \left\{\begin{array}{ll}
\{0,\updownarrow\}&\textrm{if }y=0,\\
\Downarrow&\textrm{if }y=\uparrow,\\
\Uparrow&\textrm{if }y=\downarrow,\\
\overline{\mathbb{Z}}_3&\textrm{if }y=\updownarrow.\\
\end{array}
\right.
\]
This relational parametrization is the closest to what we could see as a solution (or a root) of a linear equation $x+y = \updownarrow$. Note that if we have several equations of the form $\{x+y_i = \updownarrow \}_{i \in [n]}$, then we can take $x$ to be in the intersection $\bigcap_{i \in [n]} \langle y_i \rangle$, which always contains the element $\updownarrow$.
\end{convention}

\begin{remark}[Resolution algorithm for linear systems]\label{rem:algogirthm:linear-systems}
Let $(n,m)$ be a pair of positive integers and let $A = (a_{i,j})_{i,j}$ be a matrix in $\mathbf{M}_{n,m}(\overline{\mathbb{Z}}_3)$. Suppose that $A$ admits a reduced row echelon form for a pair $(n',\phi)$. By Definition \ref{conv:null-space}, a matrix vector $Z = (z_{i,1})_{i,1} \in \mathbf{M}_{m,1}(\overline{\mathbb{Z}}_3)$ is in $\mathsf{Null}(A)$ if, and only if, it is in $\mathbf{Bal}_{m}$ and the following relation holds for every $i \in [n']$.
\[
z_{\phi(i),1} + z_{\phi(i)+1,1} \cdot a_{i,\phi(i)+1} + \dots + z_{m,1} \cdot a_{i,m} \in \{0,\updownarrow\}
\]
To resolve such a system, we need to distinguish the independent variables from the dependent variables. Following Convention \ref{conv:solution_ud}, we can parametrize the leftmost variable $z_{\phi(j),1}$ in terms of the other variables $z_{\phi(j)+1,1}, \dots ,z_{m,1}$ (displayed on its right-hand side). Specifically, we have the following relation for every $i \in [n']$.
\[
z_{\phi(i),1} \in \left\langle z_{\phi(i)+1,1} \cdot a_{i,\phi(i)+1} + \dots + z_{m,1} \cdot a_{i,m}\right\rangle
\]
These relations suggest an algorithm as follows:
\begin{itemize}
\item[1)] (Independent variables) for every $i \in [m] \backslash \phi([n'])$, choose a value for $z_{i,1}$ in $\overline{\mathbb{Z}}_3$;
\item[2)] for every $j \in [n']$, once the value for $z_{k,1}$ is determined for all $k > \phi(j)$:
\begin{itemize}
\item[a)] find all the elements of the set $E(j) = \{j' \in [n']~|~\phi(j') = \phi(j)\}$, and
\item[b)] take $z_{\phi(j),1}$ to be in the following non-empty intersection.
\[
\bigcap_{j' \in E(j)} \left\langle z_{\phi(j')+1,1} \cdot a_{j',\phi(j')+1} + z_{\phi(j')+2,1} \cdot a_{j',\phi(j')+2} +\dots + z_{m,1} \cdot a_{j',m}\right\rangle
\]
\end{itemize}
\item[3)] repeat until the variable $z_{\phi(1),1}$ is given a value.
\end{itemize}
\end{remark}

We conclude this section with the following example, which extends to the discussion ended in
Example \ref{exa:pedigrads_and_encodings}. In particular, we show how pedigrads can be used to analyze interactions between genetic variants and how these interactions affect phenotypes.

\begin{example}[Unification]\label{exa:pedigrad-unification}
The goal of this example is to unify the various concepts developed throughout the paper within a single example and show how the questions raised in the introduction can be addressed by using the content of section \ref{sec:Pedigrads_in_idempotent_commutative_monoids} and section \ref{sec:solving_our_problem}.
Let $(\Omega,\preceq)$ be the Boolean pre-ordered set $\{0 \leq 1\}$ and let $(E,\varepsilon)$ be our usual pointed set $\{\mathtt{A},\mathtt{C},\mathtt{G},\mathtt{T},\varepsilon\}$. Take $\rho:\Delta_{[3]}(\tau) \Rightarrow \theta$ and $\rho':\Delta_{[3]}(\tau') \Rightarrow \theta'$ to be the two wide spans of homologous segments used in Example \ref{exa:Relative_definition_families} and take $(\Omega,D)$ to be the recombination chromology defined in Example \ref{exa:pedigrads_and_encodings} with respect to these cones (as shown below).
\[
D[n] =
\left\{
\begin{array}{ll}
\{\rho,\rho'\} &  \textrm{ if $n = 15$.}\\
\cmemptyset & \textrm{ if $n \neq 15$.}
\end{array}
\right.
\]
Let now $(\iota,T,\sigma)$ be the sequence alignment over $\mathbf{2}E_1^{\varepsilon}$ defined in Example \ref{exa:Sequence alignments}. As explained in Example \ref{exa:pedigrads_and_encodings}, the recombination monoid $D_ET$ is a $\mathcal{W}^{\mathrm{mon}}$-pedigrad for $(\Omega,D)$. If we denote as $M(\rho)$ the product $\prod_{i \in [3]} D_ET(\theta(i))$, the pedigrad structure of $D_ET$ gives a monomorphism $D_ET[\rho]:D_ET(\tau) \to M(\rho)$. This means that a finite collection $e = (e_{j})_{j \in [m]}$ of elements in $D_ET(\tau)$ will satisfy a linear relation in $D_ET(\tau)$ (such as the one shown on the left-hand side of (\ref{eq:conclusion:unification:1})) if, and only if, the images $D_ET[\rho](e_j)$ of the elements $e_j$ through the morphism $D_ET[\rho]:D_ET(\tau) \to M(\rho)$ satisfy the corresponding relationship in $M(\rho)$ (see right-hand side of (\ref{eq:conclusion:unification:1})).
\begin{equation}\label{eq:conclusion:unification:1}
\sum_{j \in A} e_j = \sum_{j \in B} e_j \quad \Leftrightarrow \quad \sum_{j \in A} D_ET[\rho](e_j) = \sum_{j \in B} D_ET[\rho](e_j)
\end{equation}
While Remark \ref{rem:unification:2} gives an equational description of the elements of the set $\mathsf{SNull}(\mathsf{A}(D_ET[\rho](e)))$ in terms of balanced linear relations of the form shown on the right of (\ref{eq:conclusion:unification:1}), we want to use Theorem \ref{theo:correspondence:null-spaces:2} and Theorem \ref{theo:correspondence:null-spaces:3} show that the elements of $\mathsf{SNull}(\mathsf{A}(D_ET[\rho](e)))$ can be computed as the element of the set $\mathsf{Null}(\mathsf{Coef}(D_ET[\rho](e)))$. After this, we want to use Proposition \ref{prop:cnull-space:anceling-property} and Example \ref{exa:swapping-operations} to show that the set $\mathsf{Null}(\mathsf{Coef}(D_ET[\rho](e)))$ can be computed as another null space $\mathsf{Null}(A)$ where $A$ is a permutation of the row of the matrix $\mathsf{Coef}(D_ET[\rho](e))$. If the matrix $A$ admits a reduced row echelon form, then Remark \ref{rem:algogirthm:linear-systems} shows that we can describe the elements of the set $\mathsf{Null}(A)$ -- which is equal to $\mathsf{SNull}(\mathsf{A}(D_ET[\rho](e)))$ -- algorithmically.

To illustrate the previous procedure, we shall let the finite collection $e$ be the collection $(\mathtt{p}_{2},\mathtt{p}_{3},\mathtt{p}_{4},\mathtt{p}_{6},\mathtt{p}_{8},\mathtt{p}_{9},\mathtt{p}_{10},\mathtt{p}_{12})$ of elements in $T(\tau)$. In such a situation, the matrix $\mathsf{A}(D_ET[\rho](e))$ is of the following form (see the table relative to the cone $\rho$ given in Example \ref{exa:pedigrads_and_encodings}).
\[
\prematrix{cccccccc}
{\mathtt{p}_{2}&\mathtt{p}_{3}&\mathtt{p}_{4}&\mathtt{p}_{6}&\mathtt{p}_{8}&\mathtt{p}_{9}&\mathtt{p}_{10}&\mathtt{p}_{12}}
{
\mathtt{0100001}&\mathtt{0110000}&\mathtt{0000011}&\mathtt{0000011}&\mathtt{1000010}&\mathtt{0001100}&\mathtt{0000110}&\mathtt{0001010}\\

\mathtt{0010010}&\mathtt{1000010}&\mathtt{0010100}&\mathtt{0000011}&\mathtt{0000011}&\mathtt{1000001}&\mathtt{0100001}&\mathtt{1000000}\\

\mathtt{0011}&\mathtt{1010}&\mathtt{0101}&\mathtt{1001}&\mathtt{1001}&\mathtt{1010}&\mathtt{0101}&\mathtt{0110}\\

}
\]
As explained above, computing the skew null space of the previous matrix amounts to computing the null space of the matrix $\mathsf{Coef}(D_ET[\rho](e))$ -- we display this matrix below, on the left. We can then swap the rows of the matrix $\mathsf{Coef}(D_ET[\rho](e))$ to obtain the matrix $A$ in row echelon form, as shown below, on the right.
\[
\prematrix{cccccccc}
{\mathtt{p}_{2}&\mathtt{p}_{3}&\mathtt{p}_{4}&\mathtt{p}_{6}&\mathtt{p}_{8}&\mathtt{p}_{9}&\mathtt{p}_{10}&\mathtt{p}_{12}}
{
0         &0        &0        &0        &\uparrow &0        &0        &0        \\
\uparrow  &\uparrow &0        &0        &0        &0        &0        &0        \\
0         &\uparrow &0        &0        &0        &0        &0        &0        \\
0         &0        &0        &0        &0        &\uparrow &0        &\uparrow \\
0         &0        &0        &0        &0        &\uparrow &\uparrow &0        \\
0         &0        &\uparrow &\uparrow &\uparrow &0        &\uparrow &\uparrow \\
\uparrow  &0        &\uparrow &\uparrow &0        &0        &0        &0        \\
0         &\uparrow &0        &0        &0        &\uparrow &0        &\uparrow \\
\uparrow  &0        &0        &0        &0        &0        &\uparrow &0        \\
\uparrow  &0        &\uparrow &0        &0        &0        &0        &0        \\
0         &0        &0        &0        &0        &0        &0        &0        \\
0         &0        &\uparrow &0        &0        &0        &0        &0        \\
0         &\uparrow &0        &\uparrow &\uparrow &0        &0        &0        \\
0         &0        &0        &\uparrow &\uparrow &\uparrow &\uparrow &0        \\
0         &\uparrow &0        &\uparrow &\uparrow &\uparrow &0        &0        \\
0         &0        &\uparrow &0        &0        &0        &\uparrow &\uparrow \\
\uparrow  &\uparrow &0        &0        &0        &\uparrow &0        &\uparrow \\
\uparrow  &0        &\uparrow &\uparrow &\uparrow &0        &\uparrow &0        \\
}
\quad
\to
\quad\prematrix{cccccccc}
{\mathtt{p}_{2}&\mathtt{p}_{3}&\mathtt{p}_{4}&\mathtt{p}_{6}&\mathtt{p}_{8}&\mathtt{p}_{9}&\mathtt{p}_{10}&\mathtt{p}_{12}}
{
\uparrow  &0        &0        &0        &0        &0        &\uparrow &0        \\
\uparrow  &0        &\uparrow &0        &0        &0        &0        &0        \\
\uparrow  &0        &\uparrow &\uparrow &0        &0        &0        &0        \\
\uparrow  &0        &\uparrow &\uparrow &\uparrow &0        &\uparrow &0        \\
\uparrow  &\uparrow &0        &0        &0        &0        &0        &0        \\
\uparrow  &\uparrow &0        &0        &0        &\uparrow &0        &\uparrow \\

0         &\uparrow &0        &0        &0        &0        &0        &0        \\
0         &\uparrow &0        &0        &0        &\uparrow &0        &\uparrow \\
0         &\uparrow &0        &\uparrow &\uparrow &0        &0        &0        \\
0         &\uparrow &0        &\uparrow &\uparrow &\uparrow &0        &0        \\

0         &0        &\uparrow &0        &0        &0        &0        &0        \\
0         &0        &\uparrow &0        &0        &0        &\uparrow &\uparrow \\
0         &0        &\uparrow &\uparrow &\uparrow &0        &\uparrow &\uparrow \\

0         &0        &0        &\uparrow &\uparrow &\uparrow &\uparrow &0        \\

0         &0        &0        &0        &\uparrow &0        &0        &0        \\

0         &0        &0        &0        &0        &\uparrow &0        &\uparrow \\
0         &0        &0        &0        &0        &\uparrow &\uparrow &0        \\

0         &0        &0        &0        &0        &0        &0        &0        \\

}
\]

Note that the matrix $A$ admits a reduced row echelon form for an obvious non-decreasing function $\phi:[17] \to [8]$ whose images are $\phi([17]) = \{1,2,3,4,5,6\}$. By Remark \ref{rem:algogirthm:linear-systems}, we can describe every element $Z = (z_1,z_2,z_3,z_4,z_5,z_6,z_7,z_8)$ of $\mathsf{Null}(A)$ algorithmically. Specifically, by following each steps of the algorithm presented thereof, we obtain the following algorithm for our specific situation:
\begin{itemize}
\item[1)] choose a value for the variables $z_7$ and $z_8$ in $\overline{\mathbb{Z}}_3$ -- these are our ``independent variables'';
\item[2)] choose a value for $z_6$ in $\langle z_7 \rangle \cap \langle z_8 \rangle$;
\item[3)] choose a value for $z_5$ in $\langle 0 \rangle$, namely either $0$ or $\updownarrow$;
\item[4)] choose a value for $z_4$ in the following intersection\footnote{We indicate that the sum $z_6+z_7$ is in $\{0,\updownarrow\}$ because we have $z_6 \in \langle z_7 \rangle$}:
\[
\langle  z_5+\underbrace{z_6+z_7}_{\in \{0,\updownarrow\}}\rangle.
\]
\item[5)] choose a value for $z_3$ in the following intersection:
\[
\langle z_4+z_5+z_7+z_8 \rangle \cap \langle z_7+z_8 \rangle \cap \langle 0\rangle
\]
\item[6)] choose a value for $z_2$ in the following intersection\footnote{We indicate that the sum $z_6+z_8$ is in $\{0,\updownarrow\}$ because we have $z_6 \in \langle z_8 \rangle$}:
\[
\langle z_4+z_5+z_6 \rangle \cap \langle z_4+z_5 \rangle \cap \langle \underbrace{z_6+z_8}_{\in \{0,\updownarrow\}} \rangle \cap \langle 0\rangle.
\]
\item[7)] choose a value for $z_1$ in the following intersection\footnote{We indicate that the sum $z_2+z_6+z_8$ is in $\{0,\updownarrow\}$ because we have $z_2 \in \langle z_6+z_8 \rangle$}:
\[
\langle \underbrace{z_2+z_6+z_8}_{\in \{0,\updownarrow\}} \rangle \cap \langle z_2 \rangle \cap \langle z_3+z_4+z_5+z_7 \rangle \cap \langle z_3+z_4 \rangle \cap \langle z_3 \rangle \cap \langle z_7 \rangle
\]
\end{itemize}
Note that the parametrization shown above can be used to reason about the general form of the elements $Z$ of $\mathsf{SNull}(\mathsf{A}(D_ET[\rho](e)))$. For example, if we pick $z_8 = \updownarrow$ and $z_7 = 0$, then we must necessarily have $z_6 \in \langle 0 \rangle$ since $\langle 0 \rangle \subseteq \langle \updownarrow \rangle$. Such specifications then allow us to restrict the ``space'' of choices for the values of $z_5$, $z_4$, $z_3$, $z_2$ and $z_1$, as shown below.
\[
\begin{array}{c}
z_5 \in \langle 0 \rangle \quad\quad z_4 \in \langle z_5+z_6 \rangle \quad\quad z_3 \in  \langle 0 \rangle\quad\quad z_2 \in \langle z_4+z_5+z_6 \rangle \cap \langle z_4+z_5 \rangle \cap \langle 0\rangle\\
z_1 \in \langle z_2 \rangle \cap \langle z_3+z_4+z_5 \rangle \cap \langle z_3+z_4 \rangle \cap \langle z_3 \rangle \cap \langle z_7 \rangle
\end{array}
\]
If we take $z_6 = 0$, $z_5 = \updownarrow$ and $z_3 = 0$, then we can take $z_4 = \uparrow$ and $z_2 = 0$. Finally, the previous values imply that we must have $z_1 \in \langle 0 \rangle \cap \langle \updownarrow \rangle$, which means that $z_1 = \updownarrow$. In other words, we have showed that the following equation holds.
\[
\quad\quad\quad\quad\mathtt{p}_{2} + \mathtt{p}_{8} + \mathtt{p}_{12} = \mathtt{p}_{2} + \mathtt{p}_{8} + \mathtt{p}_{12} +\mathtt{p}_{6}
\]
The previous equation provides the following kinship relationship -- we display the corresponding genotypes under each individual, as was done in Example \ref{exa:D_ET_preciting_phenotypes}.
\begin{equation}\label{eq:final:resolution}
\begin{array}{ccccccc}
\mathtt{p}_{2}&+&\mathtt{p}_{8}&+&\mathtt{p}_{12}&\triangleright&\mathtt{p}_{6}\\
\rotatebox[origin=c]{-90}{$
\begin{array}{l}
\fbox{$\mathtt{gTtCcC}$}\,\fbox{$\mathtt{tGa}$}\,\mathtt{(GtgatT)}\\
\mathtt{(aAtAcC)(cAt)}\,\fbox{$\mathtt{TgcctC}$}\\
\end{array}$}
&+&
\rotatebox[origin=c]{-90}{$
\begin{array}{l}
\mathtt{(aAcAgT)}\fbox{$\mathtt{tGa}$}\,\fbox{$\mathtt{TgcctC}$}\\
\mathtt{(gAtCcC)}\fbox{$\mathtt{tGa}$}\,\fbox{$\mathtt{GtgaaT}$}\\
\end{array}$}
&+&
\rotatebox[origin=c]{-90}{$
\begin{array}{l}
\fbox{$\mathtt{gAtCcC}$}\,\mathtt{(cAa)(GtgatT)}\\
\mathtt{(aTtAgT)(cAa)(GtgacT)}\\
\end{array}$}
&\triangleright&
\rotatebox[origin=c]{-90}{$
\begin{array}{l}
\fbox{$\mathtt{gAtCcC}$}\,\fbox{$\mathtt{tGa}$}\,\fbox{$\mathtt{TgcctC}$}\\
\fbox{$\mathtt{gTtCcC}$}\,\fbox{$\mathtt{tGa}$}\,\fbox{$\mathtt{GtgaaT}$}\\
\end{array}$}
\end{array}
\end{equation}
Interestingly, if we translate the previous relation into the corresponding phenotypes for each individual, we obtain the following relation.
\[
\mathsf{AD}+\mathsf{AD}+\mathsf{CD} \triangleright \mathsf{AB}
\]
The occurrence of the phenotype $\mathsf{B}$ on the right-hand side (while not being present on the left-hand side) suggests that the phenotype $\mathsf{B}$ interacts combinatorially with the phenotype $\mathsf{D}$ such that some of the variants responsible for the phenotype $\mathsf{B}$ must be complementary to variants associated with the phenotype $\mathsf{B}$. For example, at the level of single nucleotide variants (\emph{e.g.} SNPs), these complementary patterns would be as follows:
\begin{itemize}
\item[1)] we would have genotypes of the form $\mathtt{XX}$ in $\mathtt{p}_{6}$, while the genotypes found in $\mathtt{p}_{2}$, $\mathtt{p}_{8}$, $\mathtt{p}_{12}$ should be of the form $\mathtt{XY}$, $\mathtt{YX}$, or $\mathtt{YY}$;
\item[2)] we would have genotypes of the form  $\mathtt{XY}$ or $\mathtt{YX}$ in $\mathtt{p}_{6}$, while the genotypes found in $\mathtt{p}_{2}$, $\mathtt{p}_{8}$, $\mathtt{p}_{12}$ should be of the form $\mathtt{XX}$ or $\mathtt{YY}$.
\end{itemize}
This type of complementary patterns would also extend to multiple nucleotide variations (\emph{e.g} haplotypes) in a straightforward manner, except that these patterns would involve many more combinations (as opposed to just comparing genotypes of the form $\mathtt{X}$ and $\mathtt{Y}$). In this respect, in the case of relation (\ref{eq:final:resolution}), we can deduce that:
\begin{itemize}
\item[1)] the middle 3-nucleotide-long segments cannot \emph{solely explain} the interactions between $\mathsf{B}$ and $\mathsf{D}$ because we have the pairing $(\mathtt{tGa},\mathtt{tGa})$ on both sides;
\item[2)] the bottom 6-nucleotide-long segments cannot \emph{solely explain} the interactions between $\mathsf{B}$ and $\mathsf{D}$ because we have the pairing $(\mathtt{GtgatT},\mathtt{TgcctC})$ on both sides.
\end{itemize}
The previous facts suggest that the phenotypes $\mathsf{B}$ and $\mathsf{D}$ could potentially be caused by combinations of variants solely located in the topmost segment -- which turns out to be the case according to the four rules determining the phenotypes $\mathsf{A}$, $\mathsf{B}$, $\mathsf{C}$, and $\mathsf{D}$ (see section \ref{ssec:motivations}). However, we need to keep in mind that this is  a lucky circumstance in the case of our example and that the phenotypes $\mathsf{B}$ and $\mathsf{D}$ could have -- for instance -- been caused by variants \emph{also} belonging to the bottommost segment (note that phenotype $\mathtt{A}$ is an example of such a situation).

To resolve the previous uncertainties, we should also extend the previous reasoning to combinations of segments (as opposed to considering single segments). Since these combinations would involve more variations, we would certainly need to involve more individuals in our equations. Furthermore, to be able to apply the techniques developed in this paper, we would need to consider cones that would translate such combinations. This would ultimately allow us to see any combination of segments as a formal single segment. For example, the cone displayed in (\ref{eq:final_segment}) below combines the leftmost and rightmost segments of the topology used to defined the cone $\rho$.
\begin{equation}\label{eq:final_segment}
\begin{array}{l}
\rho_{1}^{\prime\prime}:\xymatrix@C-30pt@R-20pt{
(\bullet&\bullet&\bullet&\bullet&\bullet&\bullet)&(\bullet&\bullet&\bullet&\bullet&\bullet&\bullet)&(\bullet&\bullet&\bullet&\bullet&\bullet&\bullet)\ar[rr]
&\quad\quad\quad&
(\bullet&\bullet&\bullet&\bullet&\bullet&\bullet)&(\circ&\circ&\circ&\circ&\circ&\circ)&(\bullet&\bullet&\bullet&\bullet&\bullet&\bullet)
}\\
\rho_{2}^{\prime\prime}:\xymatrix@C-30pt@R-20pt{
(\bullet&\bullet&\bullet&\bullet&\bullet&\bullet)&(\bullet&\bullet&\bullet&\bullet&\bullet&\bullet)&(\bullet&\bullet&\bullet&\bullet&\bullet&\bullet)\ar[rr]
&\quad\quad\quad&
(\circ&\circ&\circ&\circ&\circ&\circ)&(\bullet&\bullet&\bullet&\bullet&\bullet&\bullet)&(\circ&\circ&\circ&\circ&\circ&\circ)
}
\end{array}
\end{equation}
Unfortunately, in practice, combinations of segments can give rise to long encoding vectors in $B_2$, mostly because DNA segments generally contain somatic mutations, which contributes to their diversity. As a result, it may sometimes be more effective to consider cones that also ignore certain positions for a given group of individuals, as illustrated below.
\[
\begin{array}{l}
\rho_{1}^{\prime\prime}:\xymatrix@C-30pt@R-20pt{
(\bullet&\bullet)&(\bullet)&(\bullet)&(\bullet)&(\bullet)&(\bullet&\bullet&\bullet&\bullet&\bullet&\bullet)&(\bullet&\bullet&\bullet)&(\bullet)&(\bullet&\bullet)\ar[rr]
&\quad\quad\quad&
(\bullet&\bullet)&(\circ)&(\bullet)&(\circ)&(\bullet)&(\circ&\circ&\circ&\circ&\circ&\circ)&(\bullet&\bullet&\bullet)&(\circ)&(\bullet&\bullet)
}\\
\rho_{2}^{\prime\prime}:\xymatrix@C-30pt@R-20pt{
(\bullet&\bullet)&(\bullet)&(\bullet)&(\bullet)&(\bullet)&(\bullet&\bullet&\bullet&\bullet&\bullet&\bullet)&(\bullet&\bullet&\bullet)&(\bullet)&(\bullet&\bullet)\ar[rr]
&\quad\quad\quad&
(\circ&\circ)&(\circ)&(\circ)&(\circ)&(\circ)&(\bullet&\bullet&\bullet&\bullet&\bullet&\bullet)&(\circ&\circ&\circ)&(\circ)&(\circ&\circ)
}
\end{array}
\]
Once an adequate cone is defined, we can reapply the techniques used to analyze equations of the form shown in (\ref{eq:final:resolution}), but for the corresponding cone. Each cone has the potential to give specific information about existing genetic variation interactions.

In the context of GWAS, the techniques and reasoning steps described above can be applied to (1) find genetic markers that are linked to observed phenotypes and (2) to unravel the interactions that lead to those phenotypes.
\end{example}

\section{Conclusion}\label{sec:conclusion}

Throughout this paper, we have demonstrate how pedigrads enriched in the category of idempotent commutative monoids could be used as a symbolic and computational framework
to (1) understand population stratification structures in patient data through the resolution of the word problem in ic-monoids and (2) localize genomic regions associated with complex genetic effect interactions. For the sake of exposition, we have illustrated the use of our framework through a specific problem, presented in section \ref{ssec:Main_example} (Main example), which was built on intuition developed in section \ref{ssec:motivations} (Motivations).

In a first step, we explained how we can reformulate the specifics of our problem within the pedigrad framework. To do so, we showed in Example \ref{exa:Sequence alignments} how to encode the data associated with our problem into a sequence alignment functor -- a concept introduced in earlier work. Then, we showed in Example \ref{exa:genotype_haplotype_haplogroup} and Example \ref{exa:Relative_definition_families} how we can use this sequence alignment functor to relate individuals to their haplotypes. Finally, through Example \ref{exa:Recombination_semimodule_for_DNA} and Example \ref{exa:D_ET_preciting_phenotypes}, we showed how to make these relationships more biologically accurate by taking into account segregation mechanisms. The underlying model for linking individuals was then integrated into an algebraic structure possessing pedigradic properties and called a recombination scheme.

In Example \ref{exa:Coequalizing_arrows} and Example \ref{exa:pedigrads_and_encodings}, we showed how recombination schemes and their pedigradic properties can be used to give computerized representation of genomic datasets and their underlying mechanisms. This was specifically achieved through Theorem \ref{theo:representable_pedigrad_E_b_varepsilon}
 and Theorem \ref{theo:morphism_to_mon_pedigrad}. Then, throughout section \ref{sec:solving_our_problem}, we developed a matrix algebra framework to find linear equations in idempotent commutative monoids. To do so, we defined a semiring structure (Theorem \ref{theo:tensor-congruence:semiring-compatiblity}) that allowed us to mimic standard linear algebra techniques: we used this semiring structure in a series of steps to define gradually more practical notions of null spaces and resolve linear relationships in idempotent commutative monoids (see Theorem
\ref{theo:correspondence:null-spaces:1}, Theorem \ref{theo:correspondence:null-spaces:2} and Theorem \ref{theo:correspondence:null-spaces:3}). Finally, we illustrated in Example \ref{exa:pedigrad-unification} how to apply the results of section \ref{sec:solving_our_problem}. Specifically, we used the computational framework provided by the recombination scheme of Example \ref{exa:pedigrads_and_encodings} to learn information about population stratification and find genomic regions demonstrating evidence of complex interactions between genetic markers associated with specific phenotypes.

Overall, our results pave the way for a framework extending the GWAS paradigm to a paradigm taking into account complex genetic interactions. In this respect, future work will investigate how our framework can be used in a statistical fashion for causal inference in situations where genetic signals are confounded by external factors and population stratification biases.



\bibliographystyle{plain}

\begin{thebibliography}{10}



\bibitem{Mendel_history}  C. Auffray, D. Noble, (Dec 2022), \emph{Gregor Mendel at the source of genetics and systems biology: Celebrating the relevance of Gregor Mendel’s experiments on the development of hybrid plants on the occasion of his bicentenary}, Biological Journal of the Linnean Society, Volume 137, Issue 4, pp 720–-736. \url{https://doi.org/10.1093/biolinnean/blac105}

\bibitem{LD5} K. Ardlie, L. Kruglyak, M. Seielstad, (Apr 2002), \emph{Patterns of linkage disequilibrium in the human genome}, Nature Reviews Genetics, Volume 3, pp 299–-309.
\url{https://doi.org/10.1038/nrg777}


\bibitem{LD3} L. R. Cardon, G. R. Abecasis, (Mar 2003), \emph{Using haplotype blocks to map human complex trait loci}, Trends in Genetics, Volume 19, Issue 3, pp 135--40.
\url{https://doi.org/10.1016/S0168-9525(03)00022-2}

\bibitem{PRS_tutorial} S. W. Choi, T. S. Mak, P. F. O'Reilly, (Sep 2020), \emph{Tutorial: a guide to performing polygenic risk score analyses}, Nature Protocols, Volume 15, Issue 9, pp 2759--2772.
\url{https://doi.org/10.1038/s41596-020-0353-1}


\bibitem{recomb1} M. J. Daly, J. D. Rioux, S. F. Schaffner, T. J. Hudson , E. S. Lander, (Oct 2001)\emph{High-resolution haplotype structure in the human genome}, Nature Genetics, Volume 29, Issue 2, pp 229--32.
\url{https://doi.org/10.1038/ng1001-229}

\bibitem{LD2} B. Devlin, N. Risch, (Sep 1995), \emph{A comparison of linkage disequilibrium measures for fine-scale mapping}, Genomics, Volume 29, Issue 2, pp 311--22.
\url{https://doi.org/10.1006/geno.1995.9003}


\bibitem{non-linear2} M. Elgart, G. Lyons, S. Romero-Brufau, (Aug 2022), \emph{Non-linear machine learning models incorporating SNPs and PRS improve polygenic prediction in diverse human populations}, Communications Biology, Volume 5, Issue 856. \url{https://doi.org/10.1038/s42003-022-03812-z}



\bibitem{Intermediate_Alz} D. Ferreira, A. Nordberg, E. Westman, (Mar 2020), \emph{Biological subtypes of Alzheimer disease: A systematic review and meta-analysis}, Neurology, Volume 94, Issue 10, pp 436--448.
\url{https://doi.org/10.1212/WNL.0000000000009058}


\bibitem{Golan} J. S. Golan, (1999), \emph{Semirings and their Applications}, Springer Netherlands, xii+382 pp.
\url{https://doi.org/10.1007/978-94-015-9333-5}

\bibitem{Pleiotropy2} J. Gratten, P. M. Visscher, (Jul 2016), \emph{Genetic pleiotropy in complex traits and diseases: implications for genomic medicine}, Genome Medicine, Volume 8, Issue 78. \url{https://doi.org/10.1186/s13073-016-0332-x}


\bibitem{Haldane} J. B. S. Haldane, (1919), \emph{The combination of linkage values, and the calculation of distance between the loci of linked factors}, Journal of Genetics, Volume 8, Issue 4, pp 299--309.

\bibitem{Pleiotropy3} G. Hemani, J. Bowden, G. D. Smith, (Aug 2018), \emph{Evaluating the potential role of pleiotropy in Mendelian randomization studies}, Human Molecular Genetics, Volume 27, Issue R2, pp R195--R208.
\url{https://doi.org/10.1093/hmg/ddy163}


\bibitem{PRS_survey} R. P. Igo Jr, T. G. Kinzy, J. N. C. Bailey, (Dec 2019), \emph{Genetic Risk Scores}, Current Protocols in Human Genetics, Volume 104, Issue 1.
\url{https://doi.org/10.1002/cphg.95}

\bibitem{Hap_map} International HapMap Consortium, (Oct 2005), \emph{A haplotype map of the human genome}, Nature, Volume 437, Issue 7063, pp 1299--320.
\url{https://doi.org/10.1038/nature04226}

\bibitem{Jeffreys} A. J. Jeffreys, L. Kauppi, R. Neumann, (Oct 2001), \emph{Intensely punctate meiotic recombination in the class II region of the major histocompatibility complex}, Nature Genetics, Volume 29, Issue 2, pp  217--22.
\url{https://doi.org/10.1038/ng1001-217}

\bibitem{LD4} L. B. Jorde, (Oct 2000), \emph{Linkage disequilibrium and the search for complex disease genes}, Genome Research, Volume 10, Issue 10, pp 1435--44.
\url{https://doi.org/10.1101/gr.144500}


\bibitem{Kauppi} L. Kauppi, A. J. Jeffreys, S. Keeney, (2004), \emph{Where the crossovers are: recombination distributions in mammals}, Nature Reviews Genetics, Volume 5, pp 413--424.
\url{https://doi.org/10.1038/nrg1346}

\bibitem{Dictionary} R. C. King, W. D. Stansfield, and P. K. Mulligan, (2007), \emph{A Dictionary of Genetics}, 7th Edition, Oxford University Press.
\url{https://doi.org/10.1093/acref/9780195307610.001.0001}


\bibitem{Lander_intro} E. Lander, (Feb 2011), \emph{Initial impact of the sequencing of the human genome}, Nature, Volume 470, pp 187–-197.
\url{https://doi.org/10.1038/nature09792}


\bibitem{recomb2} N. Li, M. Stephens, (Dec 2003), \emph{Modeling linkage disequilibrium and identifying recombination hotspots using single-nucleotide polymorphism data}, Genetics, Volume 165, Issue 4, pp 2213--33. Erratum in: Genetics, (Jun 2004), Volume 167, Issue 2, p 1039.
\url{https://doi.org/10.1093/genetics/165.4.2213}

\bibitem{GWAS_prob1} B. Liang, H. Ding, L. Huang, H. Luo, X. Zhu, (Feb 2020), \emph{GWAS in cancer: progress and challenges}, Molecular Genetics and Genomics, Volume 295, pp 537–561.
\url{https://doi.org/10.1007/s00438-020-01647-z}


\bibitem{MacLane}  S. Mac Lane, (1998), \emph{Categories for the working mathematician}, Graduate texts in mathematics, Springer, New York.

\bibitem{Mayr} E. W. Mayr, A. R. Meyer, (Dec 1982), \emph{The complexity of the word problems for commutative semigroups and polynomial ideals}, Advances in Mathematics, Volume 46, Issue 3, pp 305--329.
\url{https://doi.org/10.1016/0001-8708(82)90048-2}

\bibitem{non-linear1} Z. R. McCaw, T. Colthurst, T. Yun, N. A. Furlotte, A. Carroll, B. Alipanahi, C. Y. McLean, F. hormozdiari, (Jan 2022), \emph{DeepNull models non-linear covariate effects to improve phenotypic prediction and association power}. Nature Communications, Volume 13, Issue 241. \url{https://doi.org/10.1038/s41467-021-27930-0}

\bibitem{McPeek} M. S. McPeek, (1996), \emph{An Introduction to Recombination and Linkage Analysis}. In: T. Speed, M.S. Waterman  (eds) Genetic Mapping and DNA Sequencing (The IMA Volumes in Mathematics and its Applications), Volume 81, Springer, New York, NY.

\bibitem{Mendel_nature} I. Miko, (2008), \emph{Gregor Mendel and the principles of inheritance}, (2008), Nature Education, Volume 1, Issue 1, p 134.


\bibitem{non-linear3} C. Niel, C. Sinoquet, C. Dina, G. Rocheleau, (Sep 2015), \emph{A survey about methods dedicated to epistasis detection}, Frontiers in Genetics, Volume 6. \url{https://doi.org/10.3389/fgene.2015.00285}

\bibitem{Nyberg-Brodda-1} C-F Nyberg-Brodda, (Aug 2021), \emph{The word problem for one-relation monoids: a survey}, Semigroup Forum, Volume 103, pp 297--355.
\url{https://doi.org/10.1007/s00233-021-10216-8}


\bibitem{Nyberg-Brodda-2} C-F Nyberg-Brodda, (May 2022), \emph{On the word problem for special monoids}, Semigroup Forum, Volume 105, pp 295–-327.
\url{https://doi.org/10.1007/s00233-022-10286-2}


\bibitem{Ottolini} C. Ottolini, L. Newnham, A. Capalbo, et al., (May 2015), \emph{Genome-wide maps of recombination and chromosome segregation in human oocytes and embryos show selection for maternal recombination rates}, Nature Genetics, Volume 47, pp 727–735.
\url{https://doi.org/10.1038/ng.3306}


\bibitem{Pareigis} B. Pareigis, H. Rohrl, (May 2013), \emph{Remarks on Semimodules}, Version 2,
arXiv:1305.5531, \url{https://arxiv.org/abs/1305.5531}.
\url{https://doi.org/10.48550/arXiv.1305.5531}

\bibitem{Intermediate2} G. A. Preston, D. R. Weinberger, (Jun 2005), \emph{Intermediate phenotypes in schizophrenia: a selective review}, Dialogues in Clinical Neuroscience, Volume 7, Issue 2, pp 165--79.
\url{https://doi.org/10.31887/DCNS.2005.7.2/gpreston}

\bibitem{LD1} J. K. Pritchard, M. Przeworski, (Jul 2001), \emph{Linkage Disequilibrium in Humans: Models and Data}, The American Journal of Human Genetics, Volume 69, Issue 1, pp 1--14.
\url{https://doi.org/10.1086/321275}


\bibitem{Intermediate1} R. Rasetti, D. R. Weinberger, (Mar 2011), \emph{Intermediate phenotypes in psychiatric disorders}, Current Opinion in Genetics \& Development, Volume 21, Issue 3, pp 340--348.
\url{https://doi.org/10.1016/j.gde.2011.02.003}


\bibitem{Smith03} G. D. Smith, S. Ebrahim, (Feb 2003), \emph{`Mendelian randomization’: can genetic epidemiology contribute to understanding environmental determinants of disease?}, International Journal of Epidemiology, Volume 32, Issue 1, pp 1--22.
\url{https://doi.org/10.1093/ije/dyg070}

\bibitem{Smith05} G. D. Smith, S. Ebrahim, (May 2005). \emph{What can mendelian randomisation tell us about modifiable behavioural and environmental exposures?}, British Medical Journal, Volume 330, Issue 7499, pp 1076--1079.
\url{https://doi.org/10.1136/bmj.330.7499.1076}

\bibitem{Intermediate3} G. D. Smith GD, S. Ebrahim, (2008), \emph{Mendelian Randomization: Genetic Variants as Instruments for Strengthening Causal Inference in Observational Studies}. In: National Research Council (US) Committee on Advances in Collecting and Utilizing Biological Indicators and Genetic Information in Social Science Surveys, M. Weinstein,J. W. Vaupel, K. W. Wachter, editors. Biosocial Surveys. Washington (DC): National Academies Press (US); Chapter 16. Available from: \url{https://www.ncbi.nlm.nih.gov/books/NBK62433/}.

\bibitem{Smith14} G. D. Smith, G. Hemani, (Sep 2014), \emph{Mendelian randomization: genetic anchors for causal inference in epidemiological studies}, Human Molecular Genetics, Volume 23, Issue R1, pp R89--R98.
\url{https://doi.org/10.1093/hmg/ddu328}

\bibitem{Speed_GMF} T. P. Speed, (Jul 2005), \emph{Genetic Map Functions}, Encyclopedia of Biostatistics, 2nd Edition, John Wiley \& Sons, Ltd.
\url{https://doi.org/10.1002/0470011815.b2a05037}

\bibitem{SpivakBook} D. I. Spivak, (2014), \emph{Category theory for the sciences}, MIT Press, Cambridge, MA, viii+486 pp.

\bibitem{GWAS_2} B. E. Stranger, E. A. Stahl, T. Raj, (Feb 2011), \emph{Progress and promise of genome-wide association studies for human complex trait genetics}, Genetics, Volume 187, Issue 2, pp 367--83.
\url{https://doi.org/10.1534/genetics.110.120907}


\bibitem{GWAS_prob2} V. Tam., N. Patel, M. Turcotte, Y. Boss\'{e}, G. Par\'{e}, D. Meyre, (May 2019), \emph{Benefits and limitations of genome-wide association studies}, Nature Reviews Genetics, Volume 20, pp 467--484.
\url{https://doi.org/10.1038/s41576-019-0127-1}

\bibitem{Seqali} R. Tuy\'{e}ras, (Dec 2018), \emph{Category theory for genetics I: mutations and sequence alignments}, Theory and Applications of Categories, Volume 33, Issue 40, pp 1266-1314.


\bibitem{GWAS_1} E. Uffelmann, Q. Q. Huang, N. S. Munung, J. de Vries, Y. Okada, A. R. Martin, H. C. Martin, T. Lappalainen, D. Posthuma, (Aug 2021), \emph{Genome-wide association studies}, Nature Reviews Methods Primers, Volume 1, Issue 59.
\url{https://doi.org/10.1038/s43586-021-00056-9}


\bibitem{Hap_LD} J. Wall, J. Pritchard, (Aug 2003), \emph{Haplotype blocks and linkage disequilibrium in the human genome}, Nature Reviews Genetics, Volume 4, pp 587–-597.
\url{https://doi.org/10.1038/nrg1123}

\bibitem{Pleiotropy1} K. Watanabe, S. Stringer, O. Frei, M. U. Mirkov, C. de Leeuw, T. J. C. Polderman, S. van der Sluis, O. A. Andreassen, B. M. Neale, D. Posthuma, (Aug 2019), \emph{A global overview of pleiotropy and genetic architecture in complex traits}, Nature Genetics, Volume 51, pp 1339–1348.
\url{https://doi.org/10.1038/s41588-019-0481-0}



\bibitem{Zhao_mapF} H. Zhao, T. P. Speed, (Apr 1996), \emph{On genetic map functions}, Genetics, Volume 142, Issue 4, pp 1369--77.
\url{https://doi.org/10.1093/genetics/142.4.1369}





%
%
%
%














\end{thebibliography}

\end{document}